\documentclass[11pt, reqno, oneside]{article}
\usepackage[margin=1 in]{geometry} 
\usepackage[affil-it]{authblk}

\usepackage{amsmath, amsthm, amsfonts, amssymb, amsbsy, bigstrut, graphicx, enumerate,  upref, longtable, comment, booktabs, array, caption, subcaption, upref,eufrak, mismath}

\usepackage[breaklinks]{hyperref}

\usepackage[numbers, sort&compress]{natbib} 

\usepackage{cleveref}

\newtheorem{thm}{Theorem}[subsection]
\newtheorem{lmm}[thm]{Lemma}
\newtheorem{prop}[thm]{Proposition}
\newtheorem{defn}[thm]{Definition}
\newtheorem{assumption}[thm]{Assumption}
\theoremstyle{definition}
\newtheorem{remark}[thm]{Remark}
\newtheorem{ex}[thm]{Example}

\numberwithin{equation}{subsection}

\usepackage[utf8]{inputenc}
\usepackage{amsmath}
\usepackage{amsfonts}
\usepackage{bbm}
\usepackage{mathtools}
\usepackage{tikz}
\usetikzlibrary{decorations.pathreplacing}
\usepackage{amsthm}
\usepackage[english]{babel}
\usepackage{fancyhdr}
\usepackage{amssymb}
\usepackage{enumitem}
\usepackage{qtree}
\usepackage{mathrsfs}

\renewcommand{\P}{\mathbb{P}}

\renewcommand{\d}{\mathrm{d}}

\allowdisplaybreaks

\date{}
\newtheorem{theorem}{Theorem}[section]

\newtheorem{corollary}[theorem]{Corollary}

\newcommand{\continuation}{??}

\Crefname{equation}{}{}
\Crefname{theorem}{Theorem}{Theorems}
\Crefname{assumption}{Assumption}{Assumptions}
\Crefname{remark}{Remark}{Remarks}
\Crefname{lemma}{Lemma}{Lemmas}
\Crefname{lemma}{Lemma}{Lemmas}
\Crefname{enumi}{}{}

\begin{document}
 \title{On the stability of solutions to random optimization problems under small perturbations}
\author{Sourav Chatterjee\thanks{Department of Statistics, Stanford University, 390 Jane Stanford Way, Stanford, CA 94305, USA. Email: \href{mailto:souravc@stanford.edu}{\tt souravc@stanford.edu}.  
}}
\affil{Stanford University}
\author{Souvik Ray\thanks{School of Data Science and Society, University of North Carolina, 211 Manning Drive, Chapel Hill, NC 27599, USA. Email: \href{mailto:souvikr@unc.edu}{\tt souvikr@unc.edu}. 
}}
\affil{ University of North Carolina at Chapel Hill}
\maketitle

\begin{abstract}
Consider the Euclidean traveling salesman problem with $n$ random points on the plane. Suppose that one of the points is shifted to a new random location. This gives us a new optimal path. Consider such shifts for each of the $n$ points. Do we get $n$ very different optimal paths? In this article, we show that this is not the case --- in fact, the number of truly different paths can be at most $\bigO(1)$ as $n\to \infty$. The proof is based on a general argument which allows us to prove similar stability results in a number of other settings, such as branching random walk, the Sherrington--Kirkpatrick model of mean-field spin glasses, the Edwards--Anderson model of short-range spin glasses, and the Wigner ensemble of random matrices. \newline
\end{abstract}



\tableofcontents 

\section{Introduction}
\label{intro}
\subsection{Motivation}
Let $X_1, \ldots, X_n$ be  independently distributed uniform random points in the $[0,1]^d$, for some $d \geq 2$. Consider the traveling salesman problem (TSP), i.e., the problem of finding the tour along these $n$ points, starting and ending at the same point and visiting every other points exactly once, with the shortest Euclidean length. Since the distribution of the random points is absolutely continuous with respect to the Lebesgue measure, the shortest tour will be unique almost surely. Let us denote this optimal tour by $\widehat{G}=\widehat{G}(X_1, \ldots, X_n)$. Now take $i \in \{1, \ldots,n\}$ and replace the $i$-th point $X_i$ with a copy $X_i^{\prime}$, independent of $(X_1, \ldots, X_n)$. The new set of $n$ points $\{X_1, \ldots, X_{i-1}, X_i^{\prime}, X_{i+1}, \ldots, X_n\}$ will generate another optimal tour, which we denote by $\widehat{G}_i$. Varying $i$ over $\{1, \ldots,n\}$, we get a set consisting of the $n$ optimal tours $\{\widehat{G}_1, \ldots, \widehat{G}_n \}$. Are these paths likely to be all very different than each other? The following result, which is one of the main theorems of this paper, shows that this is not the case.

\begin{thm}{\label{clean:tsp}}
Fix $d \geq 2$ and take $X_1,\ldots, X_n \stackrel{i.i.d.}{\sim} \mathrm{Uniform}([0,1]^d)$. For each $i \in \{1, \ldots,n\}$, take a copy $X_i^{\prime}$ of $X_i$, independent of $(X_1, \ldots,X_n)$. Define $X_j^{(i)}$ to be equal to $X_i^{\prime}$ if $j=i$ and to be equal to $X_j$ if $j \neq i$. Let $\widehat{G}_i$ be the (almost surely unique) shortest traveling salesman tour through $\{X_1^{(i)}, \ldots, X_n^{(i)}\}$. In other words,
$$ \widehat{G}_i := \arg \min_{G \in \mathcal{G}_n} \sum_{\{j,k\} \in E(G)} \|X_j^{(i)} - X_k^{(i)} \|_2,$$
where $\|\cdot \|_2$ denotes the $\ell_2$-norm, $\mathcal{G}_n$ is the collection of all Hamiltonian tours in a complete graph with $n$ points and $E(G)$ denotes the set of edges for any graph $G$. We can metrize the space $\mathcal{G}_n$ using the normalized graph distance $d_n$ defined as 
 $$ d_n(G_1,G_2) := \dfrac{1}{n} \operatorname{card} ( E(G_1) \Delta E(G_2)), \; \forall \; G_1, G_2 \in \mathcal{G}_n, $$
 where $\Delta$ denotes the symmetric difference between two sets. Then for any $\varepsilon, \delta >0$, there exists $K=K(\varepsilon,\delta) \in \mathbb{N}$ such that for any $n \geq 1$ with probability at least $1-\delta$, there exists a (random) subset $S \subseteq \{1, \ldots, n\}$  with cardinality at least $(1-\varepsilon)n$ satisfying the property that the set $\{\widehat{G}_i : i \in S\}$  can be covered by $K$ many balls (with respect to metric $d_n$) of radius $\varepsilon$. 
\end{thm}



In other words, most of the new tours $\widehat{G}_i$ are asymptotically located in a small neighborhood of a finite number of graphs.  \Cref{clean:tsp} is a special case of the general version in  \Cref{thm:tsp}. The traveling salesman problem is not special in this regard. Given any random optimization problem, we can ask the same question about the ``stability" of its solutions  in the following sense: \textit{If we replace small parts of the random input data by independent copies, can the set of optimizers be well-approximated by a finite set with high probability?}  In this article, we develop a general theory that allows us to prove versions of the above theorem for a plethora  of optimization problems. 


The paper is organized in the following fashion. In \Cref{litreview}, we discuss some motivations behind our investigation of ``stability" in the sense described above. In that context, we mention some other notions of stability that exist in the current literature and how they compare to the above paradigm. In \Cref{stochopt}, we formally define random optimization problem and the notion of stability that we are interested in. We also describe the general strategy to prove stability for a given random optimization problem. In \Cref{sec:thm}, we describe a couple of general  theorems which provide sufficient conditions on the random optimization problems to satisfy a stability property. In \Cref{sec:lin}, we apply these results to prove stability for a list of examples from statistical physics, computational geometry, combinatorial optimization, and random matrix theory, to name a few. Discussion of relevant works form mathematical and statistical literature are appropriately distributed throughout the entire paper. Throughout \Cref{stochopt} and \Cref{sec:thm}, we shall use the TSP as a demonstrative example to elucidate the definitions and notations introduced, while the concrete analysis of TSP in our context is deferred to \Cref{sec:nonlin}. Finally, \Cref{appnd} includes supplementary materials and digressions which are instrumental to our discussions in \Cref{stochopt} and \Cref{sec:lin}.

\subsection{Literature review}{\label{litreview}}

It is natural to inquire about the motivation for considering stability under the perturbation scheme described above, i.e., replacing one (or an appropriate small fraction) of the input data entries by independent copies. For example, a popular method for perturbing the data entails taking a completely different dataset from a slightly perturbed version of the original probability measure and then evaluate the change in statistical properties of the optimal solution as the size of the dataset grows large. This technique, known as \textit{influence functions}, is a standard way in statistical literature to quantify the robustness of statistical inference procedures under contaminated datasets, see \cite{robust} for detailed discussion of this aspect. One of earliest examples of considering the stability of random optimization problem under perturbations in the probability literature came in \cite{dubins}  
which considered the stability of optimal strategies in gambling problems under small changes in the house strategy. \cite{newman} discusses stability of ground states for spin glass models under perturbations of a single bond disorder, which somewhat resembles our perturbation scheme.  

Our perturbation scheme is distinct from all of these above-mentioned procedure since it changes a portion of the data points while \textit{keeping the distribution of the data input unchanged}. Perturbations of this form come up quite often in the literature on concentration inequalities. As an example, suppose we want to quantify the concentration of $f(X_1, \ldots, X_n)$ around its expected value, where $X_1, \ldots, X_n$ are i.i.d.~random variables and $f$ is some function. To measure the influence of the $j$-th co-ordinate of the argument on the output, we take an independent copy $X_j^{\prime}$ and consider the difference $\Delta_jf(X) := f(X_1, \ldots, X_n)-f(X_1, \ldots, X_{j-1}, X_j^{\prime}, X_{j+1}, \ldots, X_n)$. Some distributional properties (for example, conditional variance, entropy etc.) of the random variables $\left\{\Delta_jf(X) : j=1, \ldots, n \right\}$  can then be combined to obtain concentration inequalities for $f(X_1, \ldots,X_n)$. Standard concentration inequalities like the Azuma--Hoeffding bound and the bounded difference inequality fall under this category. Similar ideas lie behind the tensorization method, which is extensively used in high-dimensional applications to obtain dimension-free concentration inequalities by tensorizing concentration inequalities from lower dimensions. See the exposition by \citet*{boucheron} for details. These ideas have been hugely influential in the theoretical computer science literature, especially in the study of Boolean functions where one changes one co-ordinate of the argument (keeping others fixed) to measure the influence of that co-ordinate. We defer to the excellent book by \citet{ryan} for a survey. 

Another example of this perturbation method comes from proving the central limit theorems (CLTs). For example, \citet{cha} devised a version of Stein's method for proving CLTs for functions of independent random variables  using  properties of small perturbations such functions. More generally, there is a deep connection between the stability of solutions to random optimization problems and central limit theorems for the optimum values.  A seminal work in this direction is by~\citet{penrose:yukich}, who used stability under perturbations to prove CLTs for complex geometric functionals defined on Poisson and Binomial point process. The notion of ``stability" used there is defined in terms of ``add one cost", which means the increment in the value of the functional caused by inserting a point at the origin.  This approach to proving CLTs for geometric functionals has been subsequently developed in \cite{penrose:yukich2, penrose:yukich3, penrose:yukich4,penrose}. Some recent development on this front can be found in \citet*{lachieze} and~\citet{khanh}.


Leaving the discussion of the rationale behind our notion of stability, let us briefly mention the main strategy to prove stability in our context and how it relates to existing works in the literature. A key concept in our strategy is the notion of \textit{near optimal solutions}, which are  basically, as the name suggests, the points for which the value of the objective function in the optimization problem is close to its optimum value; the degree of closeness varies depending upon the problem. For example, in case of TSP on $[0,1]^d$ as introduced in \Cref{clean:tsp}, we consider as the optimal solutions all such Hamiltonian tours with tour length $\bigO(n^{-1/d})$ away from the optimum tour length.  We then establish that with high probability this set of near optimal solutions can be covered by a finite number of balls of radius $\varepsilon$ in the normalized Hamming metric. Upon further showing that most of the different optimizers obtained from the perturbed datasets (where we replace small part of the original dataset by independent copies) lie in the near optimal set for the original input data, we get the validity of our notion of stability for that particular problem. 

The investigation of the near-optimal solution set can be thought of as another approach to inspect stability of random optimization problems since we expect that ``stable" optimization problems would result in near-optimal solution sets whose elements are not too different from each other. One such approach is the concept of \textit{almost essential uniqueness} (AEU), introduced by David Aldous~\cite{aldous}, which can be described as follows: \textit{Any solution to the random optimization problem, which is a positive distance (with respect to some properly normalized metric like $d_n$ in case of TSP) away from the optimal solution, results in the value of the objective function which is typically a positive fraction away from the optimal value, regardless the size of input data}. In other words, the only near optimal solutions are those points which are very close to the optimal points and thus the optimum point is \textit{essentially unique}. The  AEU property was proved by Aldous~\cite{aldous} for the random assignment model, and was conjectured to be true for the Euclidean TSP.  The AEU property also holds true for minimum spanning tree (MST) problem as shown in \cite{aldous:mst}. 

A related but opposite concept is the \textit{multiple valley property} (MVP) (also called the multiple peaks property) introduced by Chatterjee~\cite{chabook}. An optimization problem  is said to satisfy MVP if \textit{with high probability there are large number of points which are some positive distance away from each other but the values of the objective functions evaluated at those points are only an infinitesimal fraction away from the optimal value} and thus in essence there are many substantially different near-optimal solutions.  The MVP property was demonstrated to be closely related to the phenomena of \textit{superconcentration} and \textit{chaos} as introduced in \cite{chabook}. MVP was also  proved for a number of problems, including the directed polymer model, the Sherrington--Kirkpatrick model of spin glasses, and the largest eigenvector of GOE and GUE random matrices in~\cite{chabook}.  Since then, MVP (or absence of it) and related concepts of superconcentration and chaos have been established for a number of interesting examples: See \cite{ding} for MVP in Gaussian fields,  \cite{kuo-chen} and \cite{kuo-chen2} for mixed $p$-spin models, \cite{ganguly2020} for dynamical last passage percolation, \cite{cha:surface} for growing random surfaces, \cite{ahlberg} for first-passage percolation and \cite{chaea} for the Edwards--Anderson model.

Compared to  AEU and MVP, we look at near-optimal solutions  in a much smaller window. This contrast is due to the nature of the perturbations that we are considering. We only care about near-optimal solutions in an window size of the same order as the typical amount by which the optimal value changes when one random input is replaced by an independent copy. This, in general, is much smaller than typical size of the optimal value of the problem, which is the typical window size while considering AEU and MVP. This contrast needs to be kept in mind to not get confused by the disparity of language between our work and \cite{chabook}. An example of this disparity is the  statement ``the Sherrington--Kirkpatrick model is chaotic under small perturbations", proved in \cite{chabook}, where being ``chaotic" is equivalent (in some sense) to satisfying MVP. On the other hand, in this paper we will prove the statement ``the Sherrington--Kirkpatrick model is stable under small perturbations".  However, in view of the  different nature of  perturbations involved, these two statements are not contradictory. In \Cref{stochopt}, we give more technical account of AEU and MVP and their comparison with our results.

For other recent works on the structure and geometry of the near optimal solution sets of random optimization problems, see \cite{ding2014} and \cite{biskup} for discrete Gaussian free fields, \cite{bates} for Gaussian disordered systems at low temperature, \cite{mordant} for Gaussian random assignment field and \cite{ganguly2023} for Gaussian polymer model.  Another recent work on the topic of near-optimal solution of random optimization problems, although not directly relevant to our work,  is the paper of Gamarnik et al.~\cite{gamarnik}, where the authors have established the inadequacy of two popular algorithms in producing near-optimal solutions in a broad class of random optimization problems.


\subsection{Notation}
 We shall write $[n]$ for the set $\left\{1, \ldots, n\right\}$. The  $\ell_p$  norm  of a vector  $x \in \mathbb{R}^d$ will be denoted by $\|x\|_p$. The inner product of two vectors $x, y \in \mathbb{R}^d$ will be denoted by $x \cdot y$. The notation $X^{\top}$ will be used to denote the transpose of a matrix $X$, whereas $\|X \|_{\mathrm{op}}$ will denote its operator norm, i.e., for an $m \times n$ matrix $X$, we have  $\|X \|_{\mathrm{op}} := \sup \left\{\|Xu\|_2 : u \in \mathbb{R}^n, \|u\|_2 \leq 1\right\}$.   For two sets $A$ and $B$, $A \Delta B$ will denote their symmetric difference. For two real symmetric matrices $A$ and $B$, we shall write $A \preceq B$ if $B-A$ is non-negative definite. The $k$-th order derivative of a real-valued function $g$ defined on some open interval of the real line will be denoted by $g^{(k)}$. For  $g$  defined on some open set in $\mathbb{R}^d$ with $d \geq 2$, its gradient, second order derivative and Laplacian will be denoted by $\nabla g$,$ \nabla^2 g$ and $\Delta g$, respectively.  Note that same symbols are used to denote set-difference and the Laplacian, but it will be clear from the context which meaning applies at each specific part of this article. For a bijective function $g$, its inverse will be denoted by $g^{-1}$ as usual, whereas if the function $g : \mathbb{R} \to \mathbb{R}$ is non-decreasing, the notation $g^{\leftarrow}$ will denote its (left continuous) inverse defined as $g^{\leftarrow}(y) := \inf \left\{x : g(x) \geq y\right\}$. 
   
   For two sequences of real numbers $\{a_n\}$ and $\{b_n\}$, we will use the notation $a_n \sim b_n$ to mean that $a_n/b_n\to 1$ as $n \to \infty$.  The notation $a_n=o(b_n)$ will mean that $a_n/b_n \to 0$ as $n \to \infty$, and $a_n=\mathcal{O}(b_n)$ will mean that the sequence $\{a_n/b_n : n \geq 1\}$ is bounded. On the other hand, $a_n \gg b_n$ will imply that $a_n/b_n \to \infty$ as $n \to \infty$.
   
   For a random variable $X$, the notation $X \sim \mathcal{P}$ will mean that $X$ has law $\mathcal{P}$.  The notations $\stackrel{a.s.}{\longrightarrow}, \stackrel{p}{\longrightarrow}$ and $\stackrel{d}{\longrightarrow}$ will stand for almost sure convergence, convergence in probability, and weak convergence, respectively. Similar to the order notations for non-random sequences, for sequences random variables $\{X_n : n \geq 1\}$ and $\{Y_n : n \geq1\}$, we will write $X_n=o_p(Y_n)$ to mean that $X_n/Y_n \stackrel{p}{\longrightarrow} 0$ as $n \to \infty$, and $X_n=\mathcal{O}_p(Y_n)$ to mean that the sequence $\{X_n/Y_n : n \geq 1\}$ is tight. 
   
   The notation $\mathbbm{1}(A)$ or $\mathbbm{1}_A$ will denote the indicator function for an event $A$. For any probability space $(\Omega,\mathcal{F},\mathbb{P})$, the notation $\overline{\mathcal{F}}$ will denote the completion of the $\sigma$-algebra $\mathcal{F}$ with respect to $\mathbb{P}$. The notation $X \perp\!\!\!\perp Y$ will mean that two random variables (or vectors) $X$ and $Y$ are independent. For a random variable $X$, we shall denote its $L^p$-norm $(\mathbb{E}|X|^p)^{1/p}$ by $\|X\|_{L^p}$. $\mathrm{Binomial}(n,p)$, $\mathrm{Multinomial}(n;p_1, \ldots,p_k)$ and $\mathrm{Poisson}(\lambda)$ will respectively denote Binomial, Multinomial and Poisson distributions with the corresponding parameters. $\operatorname{Symm}(n)$ will  denote the symmetric group of order $n$, i.e., the group of all permutations of $[n]$. For $d \geq 1$, we shall refer by $B_d(\mathbf{x},r)$ to the closed ball of radius $r$ around the point $\mathbf{x} \in \mathbb{R}^d$. Also, we shall permit the dual indices in \textit{Holder's inequality} to be $(1, \infty)$ or $(\infty,1)$.


\subsection{Theoretical framework}\label{stochopt}
Consider a sequence of random optimization problems $\{\mathscr{P}_n : n \geq 1\}$. Without loss of generality we only consider minimization problems. Each problem $\mathscr{P}_n$ then has the following components. 

\begin{enumerate}[label=(Com:\Alph*)]{\label{ass:structure}}
\item \label{item:assA} A common sample space for the random input, denoted by $\mathcal{X}$, along with a $\sigma$-algebra $\mathcal{A}$ on it. Usually $\mathcal{X}$ would be a subset of $\mathbb{R}^d$ for some $d \geq 1$ and $\mathcal{A}$ will be the Borel $\sigma$-algebra on this space.
\item \label{item:asBA} The input distribution $\mathcal{P}$, which is a probability measure on the space $(\mathcal{X}, \mathcal{A})$.
\item  \label{item:assC} An index set for the random inputs, denoted by $I_n$, with $|I_n| := k_n \in \mathbb{N}$.
\item \label{item:assD} A non-empty parameter space for the problem, denoted by $\mathcal{S}_n$  and equipped with a metric $d_n$. We shall only consider cases where $\mathcal{S}_n$ is compact and the metric $d_n$ is normalized in such a way that $\operatorname{diam}(\mathcal{S}_n)$ is uniformly bounded in $n$.
\item \label{item:assE} The objective function $\psi_n : \mathcal{X}^{I_n} \times \mathcal{S}_n \rightarrow \mathbb{R}$. We assume that 
$\psi_n$ is measurable when we endow the space $\mathcal{X}^{I_n} \times \mathcal{S}_n$ with the $\sigma$-algebra $ \overline{\mathcal{A}^{\otimes I_n}} \otimes \mathcal{B}_n$,  where $\overline{\mathcal{A}^{\otimes I_n}}$ is the completion of the $\sigma$-algebra $\mathcal{A}^{\otimes I_n}$ with respect to the probability measure $\mathcal{P}^{\otimes I_n}$ and   $\mathcal{B}_n$ is the Borel $\sigma$-algebra on $\mathcal{S}_n$ generated by the open sets of the topology induced by the metric $d_n$.
\end{enumerate}
The minimization problem $\mathscr{P}_n$ is defined as follows. 

\begin{defn}{\label{defran}}
Suppose that $\{X_{n,i} : i \in I_n\}$ is an i.i.d. collection of random variables 
having distribution $\mathcal{P}$. Then 
\begin{equation}
\mathscr{P}_n := \text{ Minimize } \psi_n ( (X_{n,i} )_{i \in I_n}; \omega ) \text{ over } \omega \in \mathcal{S}_n. 
\end{equation}
\end{defn}
Let $\psi_{n, \mathrm{opt}} : \mathcal{X}^{I_n} \rightarrow [-\infty, \infty)$ be the function denoting the optimized value in problem $\mathscr{P}_n$, that is, 
\[
\psi_{n,\mathrm{opt}} ( (x_i)_{i \in I_n}) := \inf_{\omega \in \mathcal{S}_n} \psi_n((x_i)_{i \in I_n}; \omega) \in [-\infty, \infty),
\]
for $(x_i)_{i \in I_n} \in \mathcal{X}^{I_n}$.  We shall restrict our attention only to the the situations, as stated in \Cref{unique}, when the minimum in the problem $\mathscr{P}_n$ is attained at a unique parameter value. It is to be noted that,  the optimum value in the problem $\mathscr{P}_n$ is always attained if $\mathcal{S}_n$ is finite or $\psi_n$ is continuous in its second argument (since $\mathcal{S}_n$ is compact). 

\begin{assumption}{\label{unique}}
	Consider the following set:
	$$\mathscr{X}_n := \{ (x_i)_{i \in I_n} \mid \; \exists \textup{ unique } \omega^* \in \mathcal{S}_n \textup{ such that } \psi_n((x_i)_{i \in I_n}; \omega^*) = \psi_{n, \mathrm{opt}}((x_i)_{i \in I_n}) \} \subseteq \mathcal{X}^{I_n}.$$
	We assume that $\mathscr{X}_n^c$ is a null-set with respect to $\mathcal{P}^{\otimes I_n}$, and hence is in the completed $\sigma$-algebra $\overline{\mathcal{A}^{\otimes I_n}} $ on $\mathcal{X}^{I_n}$. 
	In other words, if $\{X_{n,i} : i \in I_n\}$ is an i.i.d.~collection of random variables 
	having distribution $\mathcal{P}$, then 
	$$ \mathbb{P} [\mathbf{X}^n:=(X_{n,i})_{i \in I_n} \in \mathscr{X}_n^c ] = 0, \; \forall \; n \geq 1.$$
\end{assumption}


Without further assumptions on $\mathscr{P}_n$, we cannot guarantee measurability of the function $\psi_{n, \mathrm{opt}}$. Nonetheless, we can define the optimizer for the problem $\mathscr{P}_n$ as a 
function $\widehat{\omega}_n  : (\mathcal{X}^{I_n}, \mathcal{A}^{I_n}) \to (\mathcal{S}_n, \mathcal{B}_n)$ such that 
\[
\psi_n((x_{i})_{i \in I_n}; \widehat{\omega}_n((x_i)_{i \in I_n})) = \psi_{n, \mathrm{opt}}((x_i)_{i \in I_n}).
\]
The function $\widehat{\omega}_n$ is well-defined on $\mathscr{X}_n$. Moreover, if  $\widehat{\omega}_n : \mathscr{X}_n \to \mathcal{S}_n$ is measurable (with respect to the $\sigma$-algebra on $\mathscr{X}_n$ induced by $\overline{\mathcal{A}^{\otimes I_n}} $), we can define a measurable extension of $\widehat{\omega}_n$ to the whole of $\mathcal{X}^{I_n}$. This requires some mild assumptions $\psi_n$. To avoid this technical issue, for now we include measurability of $\psi_{n, \mathrm{opt}}$ and $\widehat{\omega}_n$ as assumptions, along with the assumption of integrability of the optimum value in the problem $\mathscr{P}_n$. Later, we will see these measurability  assumptions to be easily verified in examples (see Remark~\ref{mescont}).

\begin{assumption}{\label{ass:meas}} We assume the following.
\begin{enumerate}[leftmargin=*,label=(A\arabic*)]
\item \label{item:A1} The map $\psi_{n, \mathrm{opt}} : (\mathcal{X}^{I_n}, \overline{\mathcal{A}^{\otimes I_n}}) \rightarrow ([-\infty, \infty), \mathcal{B}_{[-\infty, \infty)})$ is measurable, where $\mathcal{B}_{[-\infty, \infty)}$ is the Borel $\sigma$-algebra on $[-\infty, \infty)$.
\item \label{item:A2} The map	$\widehat{\omega}_n : \mathscr{X}_n \to (\mathcal{S}_n, \mathcal{B}_n)$ is measurable, when $\mathscr{X}_n$ is equipped with  the $\sigma$-algebra induced on it by $\overline{\mathcal{A}^{\otimes I_n}} $.
\item \label{item:A3} If $\mathbf{X}^n = (X_{n,i})_{i \in I_n} \sim \mathcal{P}^{\otimes I_n}$, then 
	$$  \mathbb{E} \biggl[ \inf_{\omega \in \mathcal{S}_n} \psi_n( \mathbf{X}^n; \omega)\biggr] = \mathbb{E} [  \psi_{n, \mathrm{opt}}( \mathbf{X}^n)] \in \mathbb{R}, \;\; \text{ for all } n \geq 1.$$
\end{enumerate}

\end{assumption}

\begin{ex}{\label{exam:tsp}}
For understanding the notations introduced above, consider TSP on Euclidean spaces. Here $\mathcal{X}=\mathbb{R}^d$ for some $d \geq 1$, equipped with its Borel $\sigma$-algebra. We have $n$ i.i.d.~points $X_{n,1},\ldots,X_{n,n}$ in $\mathcal{X}$, drawn from the common distribution $\mathcal{P}$. Thus, $I_n=[n]$ and $k_n=n$. The parameter space $\mathcal{S}_n$ is the set of all Hamiltonian cycles on the complete graph with $n$ vertices, equipped with the metric 
$$ d_n(G_1,G_2) := \dfrac{\operatorname{card}( E(G_1) \Delta E(G_2))}{n}, \; \; \forall\; G_1,G_2 \in \mathcal{S}_n,$$
where $E(G)$ denotes the set of edges of the graph $G \in \mathcal{S}_n$ and $\operatorname{card}(A)$ denotes the cardinality of a set $A$. The normalization in the definition of $d_n$ ensures that $0 \leq d_n \leq 2$. The goal in TSP is to find the Hamiltonian cycle with smallest Euclidean length, and thus,
$$ \psi_n( (x_{n,i})_{i \in [n]}; G ) := \sum_{\{i,j\} \in E(G)} \|x_{n,i}-x_{n,j}\|_2, \; \forall \; G \in \mathcal{S}_n,  (x_{n,i})_{i \in [n]} \in \mathcal{X}^n.$$
The random optimization problem in Definition~\ref{defran} then corresponds to the TSP. 
\end{ex}

\begin{remark}{\label{mescont}}
	Suppose that $\mathcal{X}$ is equipped with some topology and $\mathcal{X}^{I_n} \times \mathcal{S}_n$ is endowed with the corresponding product topology; recall that the topology on $\mathcal{S}_n$ is generated by the open balls corresponding to the metric $d_n$.  For example, if $\mathcal{X}$ is a subset of $\mathbb{R}^d$, then we can take the standard topology generated by the open balls. If we assume $\psi_n$ to be jointly continuous in its arguments (with respect to the previously mentioned product topology), then the assumption that the parameter space $\mathcal{S}_n$ is a compact metric space, hence separable, guarantees that the map $\psi_{n, \mathrm{opt}}$ is measurable as required in \Cref{item:A1} of \Cref{ass:meas}. Moreover, under \Cref{unique}, compactness of $\mathcal{S}_n$ and continuity of $\psi_n$ also guarantee that the map $\widehat{\omega}_n : \mathscr{X}_n \to \mathcal{S}_n$ is continuous (see \Cref{argmincont} for the proof) and hence measurable. 
Thus, the hypotheses in \ref{item:A1} and  \ref{item:A2} are automatically satisfied. Since  in all the examples discussed in this work $\psi_n$ is continuous,  we do not bother to check the statements \ref{item:A1} and \ref{item:A2} of \Cref{ass:meas}.
\end{remark}

 Our objective is to demonstrate some  notion of ``stability" for the optimal solution $\widehat{\omega}_n(\mathbf{X}^n)$ under small perturbations of the input data $\mathbf{X}^n$. In this paper, we shall consider stability under a specific kind of perturbation; namely what happens to the optimal  solution if we replace some of the random inputs by  i.i.d.~copies. To be more precise, consider 
a collection of subsets $\{J_{n,l} : l \in \mathcal{L}_n\}$ of $I_n$, with 
\[
m_n:=\operatorname{card}(\mathcal{L}_n) \in \mathbb{N},
\]
satisfying $\cup_{l \in \mathcal{L}_n} J_{n,l}=I_n$. We emphasize that $\mathcal{J}_n:=\{J_{n,l} : l \in \mathcal{L}_n\}$ need not form a partition of $I_n$. The elements of $\mathcal{J}_n$ will be our perturbation blocks. In other words,  for all $l \in \mathcal{L}_n$, we take an independent copy of $\mathbf{X}^n$, say $\mathbf{X}^{n,l}:=(X_{n,i}^{(l)})_{i \in I_n}$, and define the $l$-th perturbed input $\mathbf{X}_{l}^n $ as follows.
\begin{equation}{\label{perturb}}
(\mathbf{X}_l^n)(i) := \begin{cases}
X_{n,i}, & \text{if } i \in I_n \setminus J_{n,l} \\
X_{n,i}^{(l)}, & \text{if } i \in J_{n,l}.
\end{cases}
\end{equation}
It should be noted that the only assumption we made on the joint distribution of $(\mathbf{X}^n, \mathbf{X}^{n,l} : l \in \mathcal{L}_n)$ is that $\mathbf{X}^{n,l}$ is an independent copy of $\mathbf{X}^n$, for all $l \in \mathcal{L}_n$.

Let us briefly get back to the TSP example to understand the perturbation scheme described above. In the context of TSP, we will take $\mathcal{L}_n = [n]$, with $m_n=n$ and $J_{n,l} = \{l\}$, for  $l \in [n]=\mathcal{L}_n$. According to the scheme described in \eqref{perturb}, the $l$-th perturbed input is obtained from $\mathbf{X}^n$ by replacing $X_{n,l}$ with $X_{n,l}^{(l)}$, an independent copy of $X_{n,l}$. This is in accordance with the perturbation scheme described in the beginning of \Cref{intro}.



Assumptions \ref{unique}  and \ref{ass:meas} guarantee that $\widehat{\omega}_n(\mathbf{X}_l^n)$ is a well-defined random variable for all $l \in \mathcal{L}_n$. Our notion of stability will mean that \textit{there are not too many drastically different optimal solutions for different perturbed inputs}, i.e., the ``size" of the set $\{\widehat{\omega}_n(\mathbf{X}_l^n) : l \in \mathcal{L}_n\}$ is ``small" compared to $m_n$, where ``small'' means that \textit{for any $\varepsilon >0$, with high probability we can throw away approximately $\varepsilon m_n$ many points from the set  $\{\widehat{\omega}_n(\mathbf{X}_l^n) : l \in \mathcal{L}_n\}$ and the remaining points can be covered by finitely many balls of radius $\varepsilon$, the number of such balls being uniformly bounded in $n$.} In order to rigorously write down this notion, we introduce the following notation. Let 
\begin{align}\label{onadef}
 \mathcal{O}_{n,A} := \{\widehat{\omega}_n(\mathbf{X}_l^n) : l \in A\}, \; \forall \; A \subseteq \mathcal{L}_n.
 \end{align}
Let us briefly recall the concepts of covering numbers and packing numbers for metric spaces.
\begin{defn}
Let $(S,d_S)$ be a metric space. 
\begin{enumerate}
	\item For any $T \subseteq S$ and $\delta >0$, we say $T^{\prime} \subseteq S$ forms an $\delta$-cover of $T$ if for any $t \in T$ there exists an $t^{\prime} \in T^{\prime}$ such that $d_S(t,t^{\prime}) \leq \delta.$ The $\delta$-covering number of $T$, denoted by $N(T,d_S,\delta)$, is the size of the smallest $\delta$-cover of $T$, i.e.,
	$$ N(T,d_S,\delta) = \inf \{k \geq 1 : \; \exists \;\; \text{a $\delta$-cover } T^{\prime} \text{ of } T \text{ such that } \operatorname{card}(T^{\prime})=k\}.$$
	\item For any $T \subseteq S$ and $\delta >0$, we say that $T^{\prime} \subseteq T$ forms an $\delta$-packing of $T$ if $d_S(t_1,t_2) > \delta$ for all $t_1 \neq t_2 \in T^{\prime}$. The $\delta$-packing number of $T$, denoted by $P(T,d_S,\delta)$, is the size of the largest $\delta$-packing of $T$, i.e.,
	$$ P(T,d_S,\delta) = \sup \{k \geq 1 : \; \exists \;\; \text{a $\delta$-packing } T^{\prime} \text{ of } T \text{ such that } \operatorname{card}(T^{\prime})=k\}.$$
\end{enumerate}

\end{defn}

 Note that, an $\varepsilon$-cover of $T \subseteq S$ is permitted to contain points from $S \setminus T$ and we make use of the following well-known fact (see \cite[Section 4.2]{hdp} for related results), 
\begin{equation}{\label{pnp}}
P(T, d_S, 2\delta) \leq N(T, d_S, \delta) \leq P(T, d_S, \delta), \; \forall \; T \subseteq S, \; \delta >0.
\end{equation}

\begin{defn}{\label{stable}}
We say that the random optimization problem sequence $\{\mathscr{P}_n : n \geq 1\}$ is \emph{stable under small perturbations} with perturbation blocks $\{\mathcal{J}_n : n \geq 1 \}$ if for any $\varepsilon, \delta >0$, there exists $M(\varepsilon, \delta) < \infty$ such that 
\begin{equation}{\label{maindef}}
\limsup_{n \to \infty} \mathbb{P} \biggl[ \min_{A \subseteq \mathcal{L}_n : |A| > (1-\varepsilon)m_n} P ( \mathcal{O}_{n,A}, d_n, \varepsilon )  \geq M(\varepsilon, \delta)\biggr] \leq \delta,
\end{equation}  
where $\mathcal{O}_{n,A}$ is defined as in \eqref{onadef}.
\end{defn}

In other words, Definition \ref{stable} says that for any $\varepsilon, \delta >0$, there exists a finite number $M(\varepsilon, \delta)$ (not depending on $n$) such that for all large enough $n$, with high probability (at least $1-2\delta$) we can find a (random) subset $\mathscr{A}_n$ of $\mathcal{L}_n$, which is  obtained by removing at most an $\varepsilon$-fraction of  elements from $\mathcal{L}_n$ such that the (random) subset $\mathcal{O}_{n,\mathscr{A}_n}$, formed by optimal solutions $\widehat{\omega}_n(\mathbf{X}_l^n)$ corresponding to the remaining points $l \in \mathscr{A}_n$, can be covered by at most $M(\varepsilon, \delta)$ many balls (with respect to metric $d_n$) of radius $\varepsilon$.  
As a technical note, \Cref{maindef} is well-defined since $P(\mathcal{O}_{n,A}, d_n, \varepsilon)$ is measurable for all $A \subseteq \mathcal{L}_n$; this is guaranteed by the finiteness of $A$. 

Our main objective in this paper is to derive conditions under which a problem sequence $\{\mathscr{P}_n\}$ is stable under small perturbations with perturbation blocks $\{\mathcal{J}_n \}$, and then to establish  these conditions for some well-known examples. Let us now briefly recall our main strategy which we shortly mentioned in \Cref{litreview}. A useful concept in this regard is the notion of \textit{near optimal solutions}. For 
any \textit{window length} $\theta >0$, we define $\mathcal{N}_{n, \theta}$ to be the set of points in $\mathcal{S}_n$ where the objective function has values at most $\theta$ away from the optimal value. That is, 
\begin{equation}
\mathcal{N}_{n, \theta} := \{ \omega \in \mathcal{S}_n \mid \psi_n(\mathbf{X}^n,\omega) \leq \psi_{n, \mathrm{opt}}(\mathbf{X}^n) + \theta\}, \; \forall \; n \geq 1.
\end{equation}
The related  concept of \textit{Almost Essential Uniqueness} (AEU), introduced by David Aldous~\cite{aldous}, is defined as follows. We say that the random optimization problem sequence $\{\mathscr{P}_n : n \geq 1\}$ satisfies the AEU property if for any $\delta >0$ there exists $\varepsilon(\delta) >0$ such that for any  (random) sequence $\{\widetilde{\omega}_n : n \geq 1\}$ with $\widetilde{\omega}_n \in \mathcal{S}_n$  satisfying $\mathbb{E}d_n(\widehat{\omega}_n, \widetilde{\omega}_n) \geq \delta $ for all $n \geq 1$, we have 
\[
\liminf_{n \to \infty} \frac{\mathbb{E}(\psi_n(\mathbf{X}^n; \widetilde{\omega}_n))}{\mathbb{E}\psi_{n,\mathrm{opt}}(\mathbf{X}^n)} \geq 1+ \varepsilon(\delta).
\]
The related but opposite concept of \textit{Multiple Valley Property} (MVP), introduced in \cite{chabook}, is defined as follows:  An optimization problem sequence  $\{\mathscr{P}_n : n \geq 1\}$ is said to satisfy MVP if there are sequences $\varepsilon_n$ and $\delta_n$ tending to zero, a sequence $K_n \to\infty$,  and some $c>0$, such that for all $n$, with probability at least $1-\delta_n$, there exists a subset $\mathcal{A}_n \subseteq \mathcal{S}_n$ of size $\ge K_n$, such that 
\[
\inf_{\omega, \omega^{\prime} \in \mathcal{A}_n} d_n(\omega, \omega^{\prime}) \ge c, \; \; \text{ and } 
\sup_{ \omega \in \mathcal{A}_n} \biggl|\frac{\psi_n(\mathbf{X}^n; \omega)}{\psi_{n, \mathrm{opt}}(\mathbf{X}^n)} - 1\biggr|\le \varepsilon_n.
\]
For AEU and MVP, we look at near-optimal solutions in a window of size $o(\psi_{n, \mathrm{opt}}(\mathbf{X}^n))$, whereas in this paper we will be looking at near-optimal solutions in a much smaller window,  the typical amount by which the optimal value changes when one random input is replaced by an independent copy. We have already mentioned in \Cref{litreview}, how these two related and similarly worded concepts might create confusions. 


Our basic approach for proving stability in the sense of Definition~\ref{stable} can be outlined briefly as follows. For a suitably chosen sequence $\{\theta_n : n \geq 1 \}$, we will establish that the set $\mathcal{N}_{n, \theta_n}$ has uniformly bounded metric entropy with high probability,  along with the fact that most of the points in $\mathcal{O}_{n,\mathcal{L}_n}$ can be well-approximated by the elements of $\mathcal{N}_{n, \theta_n}$. The strategy is rigorously formulated in the following theorem. 

\begin{thm}{\label{strat}}
Consider a sequence of random minimization problems $\{\mathscr{P}_n : n \geq 1\}$ satisfying Assumption \ref{unique} and Assumption \ref{ass:meas} along with  perturbation block $\mathcal{J}_n$ for the problem $\mathscr{P}_n$. Suppose that there exists a sequence of positive real numbers $\{\theta_n : n \geq 1\}$ satisfying the following. 
\begin{enumerate}[label=(\Alph*)] 
\item \label{item:A} 
The sequence $\{P(\mathcal{N}_{n,c\theta_n}, d_n, \varepsilon) : n \geq 1\}$ is a tight sequence of random variables for any $c \in (0, \infty)$ and $\varepsilon >0$. 
\item \label{item:B} There are random elements $\{\omega_{n,l}^* : l \in \mathcal{L}_n \}$, measurable with respect to the $\sigma$-algebra generated by $( \mathbf{X}^n, \mathbf{X}^{n}_l : l \in \mathcal{L}_n)$, such that 
\begin{equation} {\label{most}}
\begin{split}
&\max_{l \in \mathcal{L}_n}  d_n(\widehat{\omega}_n(\mathbf{X}_l^n),\omega^*_{n,l}) \stackrel{p}{\longrightarrow} 0 \; \text{ and }  \\
&\lim_{c \to \infty} \limsup_{n \to \infty} \dfrac{1}{m_n}\sum_{l \in \mathcal{L}_n} \mathbb{P} [\psi_n(\mathbf{X}^n, \omega^*_{n,l}) -  \psi_{n, \textup{opt}}( \mathbf{X}^n) > c\theta_n ] =0.
\end{split}
\end{equation}

\end{enumerate}
Then $\{\mathscr{P}_n : n \geq 1\}$ is stable under small perturbations with perturbation blocks $\mathcal{J}_n$.

\end{thm}

To put Theorem~\ref{strat} into words, $\omega^*_{n,l}$ is a ``sister solution'' to the optimization problem $\mathscr{P}_n$ with input $\mathbf{X}^n_l$ in the sense that it approximates (according to the metric $d_n$) the optimal solution $\widehat{\omega}_n(\mathbf{X}_l^n)$. 
Display~\eqref{most} asserts that most of these ``sister solutions'' are near-optimal solutions for the problem with original input $\mathbf{X}^n$. Since \Cref{item:A} guarantees that the near-optimal solution space is not too large, we can conclude that there exists a subset of $\mathcal{O}_{n,\mathcal{L}_n}$ which is not too large despite containing most of the near-optimal solutions. For a number of examples that we shall discuss later, it is enough to take the sister solution $\omega^*_{n,l}$ to be the actual optimal solution $\widehat{\omega}_n(\mathbf{X}_l^n)$; see \Cref{mostprop} for a general result in this direction.  On the contrary, we refer to the proofs of \Cref{thm:mst} and \Cref{thm:tsp} for examples where taking $\omega^*_{n,l}= \widehat{\omega}_n(\mathbf{X}_l^n)$ does not work and hence we needed some explicit construction for $\omega^*_{n,l}$. 

\begin{proof}[Proof of Theorem \ref{strat}]
Fix $\varepsilon, \delta >0$ and get $c=c(\varepsilon,\delta) \in (0, \infty)$ such that
$$ \limsup_{n \to \infty} \dfrac{1}{m_n}  \sum_{l \in \mathcal{L}_n} \mathbb{P} [   \psi_n(\mathbf{X}^n, \omega^*_{n,l}) -  \psi_{n, \mathrm{opt}}( \mathbf{X}^n) \geq c  \theta_n]  \leq  \delta\varepsilon /2.$$ Assumption \ref{item:B} guarantees the existence of such a $c$. Assumption \ref{item:A} guarantees the existence of $M_1(\varepsilon,\delta) \in (0, \infty)$ such that
\begin{equation}{\label{neartight}}
 \sup_{n \geq 1} \mathbb{P} [ P ( \mathcal{N}_{n,c\theta_n}, d_n, \varepsilon/2) \geq M_1 (\varepsilon, \delta) ] \leq \delta/2.
\end{equation}
  On the event $\{\upsilon_n := \max_{l \in \mathcal{L}_n} d_n(\widehat{\omega}_n(\mathbf{X}_l^n),\omega^*_{n,l})< \varepsilon/4 \}$, any $\varepsilon$-packing  of $\mathcal{O}_{n,A}$ generates a $\varepsilon/2$-packing 
  of $\mathcal{O}^*_{n,A} := \{\omega^*_{n,l} : l \in A\}$ for any $A \subseteq \mathcal{L}_n$, and hence, 
  \[
  P(\mathcal{O}_{n,A}, d_n, \varepsilon) \leq P(\mathcal{O}^*_{n,A}, d_n, \varepsilon/2).
  \]
Setting $A_n := \{l \in \mathcal{L}_n : \omega_{n,l}^* \in \mathcal{N}_{n, c\theta_n}\}$ we can write the following:
\begin{align}
& \mathbb{P} \biggl[ \min_{A \subseteq \mathcal{L}_n : |A| > (1-\varepsilon)m_n} P ( \mathcal{O}_{n,A}, d_n, \varepsilon )  \geq M_1(\varepsilon, \delta)\biggr] \notag \\
 &\leq \mathbb{P} \bigl( P ( \mathcal{O}_{n,A_n}, d_n, \varepsilon )  \geq M_1(\varepsilon, \delta) \bigr)+ \mathbb{P}(|A_n| \leq (1-\varepsilon) m_n) \nonumber \\
& \leq  \mathbb{P} \bigl( P ( \mathcal{O}^*_{n,A_n}, d_n, \varepsilon/2 )  \geq M_1(\varepsilon, \delta)\bigr) + \mathbb{P}(\upsilon_n \geq \varepsilon/4)  + \mathbb{P}(|A_n| \leq (1-\varepsilon) m_n) \nonumber \\
&\leq   \mathbb{P} \bigl( P ( \mathcal{N}_{n,c\theta_n}, d_n, \varepsilon/2)  \geq M_1(\varepsilon, \delta)\bigr)  + \mathbb{P}(\upsilon_n \geq \varepsilon/4) + \mathbb{P}(|A_n| \leq (1-\varepsilon) m_n) \label{eq3}. 
\end{align}
The limit superior of the first term in the right hand side of \Cref{eq3} is at most $\delta/2$, whereas the second term converges to $0$ by the first assumption in \Cref{most}. 
As for the third term, observe that
\begin{align*}
&\mathbb{P}(|A_n| \leq (1-\varepsilon) m_n) \leq \mathbb{P} \Biggl[ \sum_{l \in \mathcal{L}_n} \mathbbm{1}( \omega^*_{n,l} \notin \mathcal{N}_{n, c\theta_n} ) > \varepsilon m_n \Biggr] \\
& \leq \dfrac{1}{\varepsilon m_n} \sum_{l \in \mathcal{L}_n} \mathbb{P} [ \psi_n(\mathbf{X}^n, \omega^*_{n,l}) -  \psi_{n, \mathrm{opt}}( \mathbf{X}^n) > c\theta_n ],
\end{align*}
which shows that the third term in the right hand side of \Cref{eq3} also has limit superior at most $\delta/2$. This completes the proof. 
\end{proof}

\begin{remark}{\label{rem:strateasy}}
By an application of Markov's inequality, it is enough to establish any of the following two conditions, where obviously the latter is stronger than the former, in order to guarantee the second assertion in \Cref{most}:
\begin{align*}
&\sum_{l \in \mathcal{L}_n} \mathbb{E} [ \psi_n(\mathbf{X}^n, \omega^*_{n,l}) -  \psi_{n, \mathrm{opt}}( \mathbf{X}^n)] = \bigO(m_n\theta_n), \\
&\sup_{l \in \mathcal{L}_n} \mathbb{E}[ \psi_n(\mathbf{X}^n, \omega^*_{n,l}) -  \psi_{n, \mathrm{opt}}( \mathbf{X}^n)] = \bigO(\theta_n).
\end{align*}
 Indeed, this will be the case in all the examples which we shall discuss in \Cref{sec:lin}.
\end{remark}

\begin{remark}
The statement in \Cref{item:A} requires us to choose $\theta_n$ as small as possible whereas \Cref{item:B} forces $\theta_n$ to be large. Remark~\ref{rem:strateasy} gives us an idea about the suitable choice of $\theta_n$. Since $\omega_{n,l}^*$ is very close to $\widehat{\omega}_n(\mathbf{X}_l^n)$ and the input data $\mathbf{X}_l^n$ is just a small perturbation of original input data $\mathbf{X}^n$, we can see from Remark~\ref{rem:strateasy} that $\theta_n$ has to be at least of the same order as the typical value of $|\psi_{n,\mathrm{opt}}(\mathbf{X}^n_l) - \psi_{n,\mathrm{opt}}(\mathbf{X}^n)|$, which is the typical amount by which the optimal value changes under the small perturbation. While this is only an heuristic observation, our  theorems in  \Cref{sec:thm}  provide the proper choice of $\theta_n$.
\end{remark}

\begin{remark}
	To make \ref{item:A} mathematically sensible, we need to ensure that $P(\mathcal{N}_{n,\theta_n},d_n, \varepsilon)$ is measurable. To accomplish this goal, we consider the space of all compact metric spaces $\mathscr{M}$, identified up to isometry and equipped with the Gromov--Hausdorff metric $d_{GH}$; see \cite[Sections~7.3--4]{bbi} for details on Gromov--Hausdorff distance. Provided that either $\mathcal{S}_n$ is finite or $\omega \mapsto  \psi_n(\cdot, \omega)$ is continuous (which will be the case in all the examples discussed in \Cref{sec:lin}), we can argue that $\mathcal{N}_{n, \theta}$ is a random element in $\mathcal{M}$; here we recall that $\mathcal{S}_n$ is assumed to be compact. It is easy to see that for any metric space $(X,d_X)$, the map $Y \mapsto P(Y,d_X, \varepsilon)$ is lower semi-continuous, when defined on the space of compact subsets of $X$ equipped with Gromov--Hausdorff metric, hence guaranteeing the measurability of  $P(\mathcal{N}_{n,\theta_n},d_n, \varepsilon)$.
\end{remark}


In most of  the examples, the challenging task will be to prove \Cref{item:A}. Throughout the rest of the paper we will consider $\mathcal{X} \subseteq \mathbb{R}^d$ for some $d \geq 1$. The technique employed to prove \Cref{item:A} can be summarized as follows. For any fixed $n$, we take a small Markovian step, with stationary distribution $\mathcal{P}^{\otimes I_n}$, starting from $\mathbf{X}^n$ to arrive at $\widetilde{\mathbf{X}}^n$. We then perform a Taylor-like approximation of $\psi_n(\widetilde{\mathbf{X}}^n; \omega) - \psi_n(\mathbf{X}^n; \omega)$ and show that the minimum of these differences over $\omega \in \mathcal{N}_{n, \theta_n}$ takes a large negative value on the event where $P(\mathcal{N}_{n,\theta_n},d_n, \varepsilon)$ is large. This type of implication will be justified by an application of the \textit{Sudakov's minoration inequality} for Gaussian processes, stated below.  

\begin{defn}
Let $(\mathcal{T},d_{\mathcal{T}})$ be a non-empty  pseudo-metric space. A collection of random variables $\{X_t : t \in \mathcal{T}\}$ indexed by $\mathcal{T}$ is called a centered Gaussian process on $(\mathcal{T},d_{\mathcal{T}})$ if  
\begin{enumerate}
\item $(X_{t_1}, \ldots, X_{t_k})$ is a centered Gaussian random vector for any $k \geq 1$ and any $t_1, \ldots, t_k \in \mathcal{T}$, and 
\item $\|X_t-X_s\|_{L^2} := \bigl(\mathbb{E}\bigl[ (X_t-X_s)^2\bigr]\bigr)^{1/2} = d_{\mathcal{T}}(t,s), \;\; \forall \; s,t \in \mathcal{T}.$
\end{enumerate}
\end{defn}

\begin{prop}[\textbf{Sudakov minoration}]{\label{sudakov}}
Let $(\mathcal{T},d_{\mathcal{T}})$ be a non-empty  pseudo-metric space and suppose $\{X_t : t \in \mathcal{T}\}$ is a Gaussian process on $(\mathcal{T},d_{\mathcal{T}})$. Then for any $\varepsilon >0$, 
$$ \mathbb{E} \inf_{t \in \mathcal{T}} X_t \leq -A^* \varepsilon \sqrt{ \log P(\mathcal{T},d_{\mathcal{T}},\varepsilon)},$$
 where $A^*>0$ is a universal constant.

\end{prop}


A proof of Proposition~\ref{sudakov} can be found in \cite[Theorem 7.4.1]{hdp}. Taking the Metropolis--Hastings algorithm as the Markovian step and assuming a Taylor-like expansion for the objective function $\psi_n$ along with some smoothness and growth conditions on the components of the Taylor expansion, we prove in \Cref{genthm} a general result of the form \Cref{item:A}. 

When working out applications of Theorem \ref{genthm}, we first  consider examples where $\psi_n(\cdot, \omega)$ is linear for every $\omega \in \mathcal{S}_n$. 
These include the branching random walk, the Sherrington--Kirkpatrick Model of mean-field spin glasses, the Edwards--Anderson model of lattice spin glasses, eigenvectors of Wigner matrices, the random assignment model, and symmetric optimization problems on complete weighted graphs. After that, we consider the setting where the objective function is nonlinear, for example the largest eigenvectors of Wishart matrices, the traveling salesman problem and minimum spanning tree on Euclidean spaces. 


\begin{remark}
Note that a statement like \Cref{item:A} has nothing to do with the particular choice of perturbation blocks for the concerned problem. Therefore, deriving a result of this form with the optimal choice of $\theta_n$ is of independent interest. It is not hard to show that a sequence of random compact metric spaces $\{(\mathcal{C}_n,d_n) : n \geq 1 \}$ is tight (when considered as $(\mathscr{M},d_{GH})$-valued random elements) if and only if the sequences of real-valued random variables $\{\operatorname{diam}(\mathcal{C}_n) : n \geq 1\}$ and $\{P(\mathcal{C}_n,d_n,\delta) : n \geq 1\}$ are tight for all $\delta >0$. Since we have already assumed $\mathcal{S}_n$ is compact, a statement like \Cref{item:A} implies that $\{(\mathcal{N}_{n,\theta_n}, d_n) : n \geq 1\}$ is tight. One can then try to figure out whether this sequence of random compact sets indeed converges weakly or not, since a weak limit will provide us with an idea of the geometry of the near-optimal  solution space  for this particular problem, as input size grows to infinity. For example, if the weak limit is a singleton set, then one can say that the optimal solution is ``unique" in some limiting sense.  Establishing the existence and characterizing such weak limits is another interesting avenue for future research.
\end{remark}

\section{Main abstract results}
{\label{sec:thm}}

\subsection{General results to prove condition~\ref{item:A} of Theorem \ref{strat}}
\label{sec:thm1}
The main purpose of this section is to introduce the two theorems about tightness of the sequence $\{P(\mathcal{N}_{n, \theta_n}, d_n, \varepsilon) : n \geq 1\}$.  As mentioned earlier, both  involve taking a small Markovian step starting from our given random inputs. The difference lies in the choice for this Markov process; Theorem~\ref{genthm} uses Metropolis--Hastings Algorithm, whereas \Cref{genthm:g} involves an application of Langevin dynamics. We defer the detailed discussion of the qualitative differences between these two approaches  to the end of this  section.

 Theorem~\ref{genthm} provides us with a range of suitable choices for $\{\theta_n : n \geq 1\}$ for making the sequence $\{P(\mathcal{N}_{n, \theta_n}, d_n, \varepsilon) : n \geq 1\}$ tight for a generic objective function. It assumes only a Taylor-like expansion for the objective function $\psi_n$, as given in Assumption~\ref{ass:psi} below, along with a series of growth conditions on the constituent components of that expansion.
 
 \begin{assumption}{\label{ass:psi}}
Suppose that the random inputs take values in $\mathcal{X} = \mathbb{R}^d$, equipped with its Borel $\sigma$-algebra, and there exist $L \in \mathbb{N}$,  measurable functions $H_{n,i} : \mathcal{X}^{I_n} \times \mathcal{S}_n  \to \mathbb{R}^d$ 
for all $i \in I_n$; $R_{n,1}, \ldots, R_{n,L} : \mathcal{X}^{I_n} \times \mathcal{X}^{I_n} \times \mathcal{S}_n \to \mathbb{R}_{\geq 0}$ and $\widetilde{R}_{n,1}, \ldots, \widetilde{R}_{n,L} : \mathcal{X}^{I_n} \to \mathbb{R}_{\geq 0}\;$ for which the following conditions hold true for any $\mathbf{x}^n = (x_i)_{i \in I_n}, \mathbf{c}^n = (c_i)_{i \in I_n} \in \mathcal{X}^{I_n} $. 
\begin{enumerate} [leftmargin=*,label=(B\arabic*)]
\item \label{item:B1}For any $\omega \in \mathcal{S}_n$, 
 		$$ \psi_n(\mathbf{x}^n + \mathbf{c}^n; \omega) - \psi_n(\mathbf{x}^n; \omega) \leq \sum_{i \in I_n} c_i \cdot H_{n,i}(\mathbf{x}^n; \omega) + \sum_{l=1}^L R_{n,l}(\mathbf{x}^n, \mathbf{c}^n; \omega) \widetilde{R}_{n,l}(\mathbf{x}^n+\mathbf{c}^n),  $$
 		for almost every $\mathbf{x}^n,\mathbf{c}^n$ with respect to Lebesgue measure on $\mathcal{X}^{I_n}$.
 	\item \label{item:B2} For $j \in [L]$, there exists a constant $\upsilon_j \in (0, \infty)$ such that  for all  $s>0$ we have
 	$$ R_{n,j}(\mathbf{x}^n, s \mathbf{c}^n; \omega) = s^{1+\upsilon_j}R_{n,j}(\mathbf{x}^n, \mathbf{c}^n; \omega), \; \forall \; \omega \in \mathcal{S}_n.$$
\end{enumerate}
 		 The condition \Cref{item:B2} is sometimes  replaced by a stronger version.
 		 \begin{enumerate}[label=(B\arabic*)$'$]
 		 \setcounter{enumi}{1}
 		 \item \label{item:B2prime} For all  $s>0$ and $\widetilde{\mathbf{c}}^n = (\widetilde{c}_{i})_{i \in I_n} \in \mathcal{X}^{I_n}$ satisfying $\|\widetilde{c}_{i}\|_2 \leq \|c_i \|_2$ for all $i \in I_n$, we have
 		  	$$ R_{n,j}(\mathbf{x}^n, s \widetilde{\mathbf{c}}^n; \omega) \leq s^{1+\upsilon_j}R_{n,j}(\mathbf{x}^n, \mathbf{c}^n; \omega), \; \forall \; \omega \in \mathcal{S}_n.$$
 		 \end{enumerate}
 \end{assumption}

For an  example of the Taylor-like expansion in Assumption~\ref{ass:psi}, let us go back to the example of TSP  (\Cref{exam:tsp}).  Recall that in that example, we had $\mathcal{X}=\mathbb{R}^d$ with $d \geq 2$, $I_n=[n]$ and $k_n=n$ with 
 $$ \psi_n( (x_{n,i})_{i \in [n]}; G ) := \sum_{\{i,j\} \in E(G)} \|x_{n,i}-x_{n,j}\|_2, \; \forall \; G \in \mathcal{S}_n,  (x_{n,i})_{i \in [n]} \in \mathcal{X}^n,$$
 where $\mathcal{S}_n$ is the set of all Hamiltonian cycles on the complete graph with $n$ vertices and $E(G)$ is the set of edges of $G \in \mathcal{S}_n$. It is clear that $\psi_n$ is jointly continuous and convex; and hence we can apply a first order approximation as follows: For (almost) any $\mathbf{x}^n=(x_i)_{i \in [n]}, \mathbf{c}^n=(c_i )_{i \in [n]} \in (\mathbb{R}^d)^{[n]}$,
 \begin{equation}{\label{ex:first}}
 \psi_n(\mathbf{x}^n+\mathbf{c}^n; G)-\psi_n(\mathbf{x}^n; G) \leq \sum_{i \in [n]} c_i \cdot \partial_i \psi_n(\mathbf{x}^n+\mathbf{c}^n; G),
 \end{equation}
 where $\partial_i$ refers to the (partial) gradient with respect to the $i$-th input. Computing the gradient and plugging it into (\ref{ex:first}), we obtain the following. We write $\Delta x_{ij}$ and $\Delta c_{ij}$ to denote $x_i-x_j$ and $c_i-c_j$ respectively. 
 \begin{align}
&\psi_n(\mathbf{x}^n+\mathbf{c}^n; G)-\psi_n(\mathbf{x}^n; G)  \leq  \sum_{i \in [n]} c_i \cdot \Biggl[ \sum_{\substack{j : \{i,j \} \in E(G)}} \dfrac{\Delta x_{ij} + \Delta c_{ij}}{\big\|\Delta x_{ij}+\Delta c_{ij}\big\|_2}\Biggr] \nonumber \\
& = \sum_{\substack{\left\{i,j\right\} \\ \left\{i,j\right\} \in E(G)}} \dfrac{\Delta c_{ij} \cdot (\Delta x_{ij}+\Delta c_{ij})}{{\big\|\Delta x_{ij}+\Delta c_{ij}\big\|_2}} \nonumber \\
& \leq  \sum_{i \in [n]} c_i \cdot \Biggl[ \sum_{\substack{j : \{i,j \} \in E(G)}} \dfrac{\Delta x_{ij}}{\big\|\Delta x_{ij}\big\|_2}\Biggr] + \sum_{\substack{\left\{i,j\right\} \\ \left\{i,j\right\} \in E(G)}} \big\| \Delta c_{ij} \big\|_2 \;\; \Bigg\|\dfrac{\Delta x_{ij}+\Delta c_{ij}}{{\big\|\Delta x_{ij}+\Delta c_{ij}\big\|_2}} - \dfrac{\Delta x_{ij}}{\big\|\Delta x_{ij}\big\|_2}\Bigg\|_2. \label{ex:second}
 \end{align}
Applying triangle inequality and H\"older's Inequality appropriately to the second term in \Cref{ex:second}, we arrive at the following : 
\begin{align*}
  &\psi_n(\mathbf{x}^n+\mathbf{c}^n; G)-\psi_n(\mathbf{x}^n; G) \\
 	& \leq  \sum_{i \in [n]} c_i \cdot \Biggl[ \sum_{\substack{j : \{i,j \} \in E(G)}} \dfrac{\Delta x_{ij}}{\big\|\Delta x_{ij}\big\|_2}\Biggr]
+ 2\biggl[ \sum_{i \in [n]} \sum_{j : \{i,j\} \in E(G)} \|\Delta c_{ij}\|_2^6 \biggr]^{1/3}\\
	&\hspace{3 in}  \cdot\Biggl[ \sum_{i\in [n]} \left(\min_{j \in [n] : j \neq i}\|\Delta x_{ij} + \Delta c_{ij}\|_2 \right)^{-3/2} \Biggr]^{2/3}.
 \end{align*}
 We refer the equation \Cref{choice5} for a detailed proof of this expansion, 
 which amounts to \Cref{item:B1} with $L=1$ and 
\begin{align}
& H_{n,i}(\mathbf{x}^n;G) = \sum_{\substack{j : \{i,j \} \in E(G)}} \dfrac{x_i-x_j}{\|x_i-x_j\|_2}, \forall\; i \in [n], \label{express1}\\ 
& R_{n,1}(\mathbf{x}^n,\mathbf{c}^n;G) = 2\Biggl[ \sum_{i \in [n]} \sum_{j : \{i,j\} \in E(G)} \|c_i-c_j\|_2^6 \Biggr]^{1/3}, \label{express2}\\
& \widetilde{R}_{n,1}(\mathbf{x}^n) =   \Biggl[ \sum_{i \in [n]}  \left(\min_{j \in [n] : j \neq i}\|x_i-x_j\|_2 \right)^{-3/2} \Biggr]^{2/3}. \label{express3}
\end{align}
 Clearly, \Cref{item:B2} is then also satisfied with $\upsilon_1 =1$. 

Coming back to our general discussion, the Markovian step in the proof of Theorem~\ref{genthm} below is obtained from the Metropolis--Hastings Algorithm with Gaussian proposal density. 
 To make the proof work with this choice of the Markovian step, we need to make sure that the Metropolis--Hastings proposal for the Gaussian density is not rejected too many times. Assumption~\ref{ass:p} restricts the choice of $\mathcal{P}$ (the distribution for the random inputs) so that we indeed have a small rejection probability in the MH step.

\begin{assumption}{\label{ass:p}}
The input probability measure $\mathcal{P}$ on $\mathbb{R}^d$ has a density $f$ with respect to Lebesgue measure on $\mathbb{R}^d$ such that for $X \sim \mathcal{P}, W \sim N_d(\mathbf{0},I_d)$ with $X \perp\!\!\!\perp W$ and for any $p,s \geq 0$, we have 
		$$ \mathbb{E} \biggl[\|W\|_2^{p} \biggl(1-\dfrac{f(X+sW)}{f(X)} \biggr)_+\biggr] \leq C(f,p)s,$$
		where $C(f,p)$ is a finite positive constant depending on $f$ and $p$ only. 
\end{assumption}

At first glance it might not be apparent why the condition in Assumption~\ref{ass:p} actually reduces the rejection probability in the MH step. To further illustrate this point, see the sufficient condition for Assumption~\ref{ass:p} given in \Cref{f1}. For a reasonably well-behaved stationary density $f$, it is generally true that as the initial point of the MH step moves towards the boundary of the support of $f$ the rejection probability tends to increase. The second condition in Lemma~\ref{f1} dictates that the boundary region of the support of $\mathcal{P}$ has small mass; thus controlling the rejection rate.

\begin{lmm}{\label{f1}}
Let $S$ be the support of the probability measure $\mathcal{P}$ on $\mathbb{R}^d$ and $f$ be its density with respect to Lebesgue measure. Then  $f$ satisfies Assumption~\ref{ass:p} whenever the following two conditions hold true.
\begin{enumerate}

\item $f$ is differentiable in $\operatorname{int}(S)$ and $\|\nabla f\|_2 \in L^1(\mathbb{R}^d)$, upon setting $\nabla f \equiv 0$ outside $\operatorname{int}(S)$.
\item There exists a finite constant $C(f)$, depending only on $f$, such that 
$$ \mathcal{P}( S_r) = \int_{S_r} f(x)\, dx \leq C(f)r\; \; \forall \; r \geq 0,$$
where $S_r:= \{x \in \mathbb{R}^d : \inf_{y \in S^c} \|x-y\|_2 \leq r\}$, with $S_r$ being empty if $S=\mathbb{R}^d$. 
\end{enumerate}
\end{lmm} 

\begin{lmm}{\label{f2}}
Suppose $\mathcal{P}$ can be written as $\mathcal{P}_1 \otimes \mathcal{P}_2$, where $\mathcal{P}_i$ is a probability measure on $\mathbb{R}^{d_i}$ with density $f_i$ with respect the Lebesgue measure on $\mathbb{R}^{d_i}$, for $i=1,2$ and $d=d_1+d_2$. If $\mathcal{P}_1$ and $\mathcal{P}_2$ both satisfies Assumption~\ref{ass:p} then so does $\mathcal{P}$. Moreover, if $\mathcal{P}$ satisfies Assumption~\ref{ass:p}, then so does any non-constant affine transformation of $\mathcal{P}$.
\end{lmm}

We refer to Section~\ref{appnd} for the proofs of Lemmas~\ref{f1} and~\ref{f2}. These two lemmas provide a large number of examples for densities $f$ satisfying Assumption~\ref{ass:p}. These examples include the standard distributions like the uniform distribution on a measurable set with non-zero finite Lebesgue measure and a well-behaved measure-zero boundary (e.g., the boundary is a closed $C^2$-curve in $\mathbb{R}^d$), the Gaussian distribution,  products of Gamma distributions, products of Beta distributions with parameters greater than or equal to $1$, and Dirichlet distributions with all the parameters greater than or equal to $1$, to name a few. 

Before stating \Cref{genthm}, it is useful to recall the concept of sub-Gamma variables, which will also be instrumental in the discussion of \Cref{genthm:g}, the other main abstract result of this section.

\begin{defn}{\label{subGamma}}
A real-valued integrable random variable $X$ is said to be sub-Gamma with variance proxy $\sigma^2 \geq 0$ and scale parameter $c \geq 0$ if 
$$ \log \mathbb{E} \exp [ \lambda ( X- \mathbb{E}X)] \leq \dfrac{\lambda^2\sigma^2}{2(1-c|\lambda|)}, \; \forall \; |\lambda| < 1/c.$$
\end{defn}

The collection of such random variables is denoted by $\mathscr{G}(\sigma^2,c)$. If $c=0$, then $X$ is called sub-Gaussian with variance proxy $\sigma^2$. A random vector is called sub-Gamma with variance proxy $\sigma^2$ and scale parameter $c$ if all of its coordinates are in $\mathscr{G}(\sigma^2,c)$. The collection of all $p$-dimensional sub-Gamma vectors with variance proxy $\sigma^2$ and scale parameter $c$ is denoted by $\mathscr{G}(\sigma^2,c;p)$. We refer to \cite[Sections 2.3--2.5]{boucheron} for a detailed exposition on sub-Gaussian and sub-Gamma random variables and their properties. An important property of sub-Gamma vectors which we shall use frequently needs an earlier mention : whenever $\sigma_1^2 \geq \sigma_2^2 \geq 0$ and $c_1 \geq c_2 \geq 0$, we have $\mathscr{G}(\sigma_2^2,c_2;p) \subseteq \mathscr{G}(\sigma_1^2,c_1;p)$.  \Cref{subgamma} contains all the relevant materials about this topic that will be needed in this article. 
We are now ready to state the first main result of this section.
\begin{thm}{\label{genthm}}
	Consider a sequence of random minimization problems $\{\mathscr{P}_n : n \geq 1\}$ with random inputs in $\mathcal{X}=\mathbb{R}^d$ and  satisfying Assumptions \ref{unique}, \ref{ass:meas} and \ref{ass:psi} (with either \Cref{item:B2} or \cref{item:B2prime}). 
	Suppose that the probability measure $\mathcal{P}$  (for the inputs)  satisfies Assumption~\ref{ass:p}. 
	Further assume the existence of  a sequence of positive real numbers 
$\{\theta_n\}_{n \geq 1} $ satisfying the following conditions.
	\begin{enumerate} [leftmargin=*,label=(C\arabic*)]
		\item  \label{item:C1} Assume that for some $\lambda \in [0,\infty)$, there exists a sequence of positive real numbers $\{\varsigma_{n,\lambda} : n \geq 1\}$ satisfying 
			$$ \biggl( \mathbb{E} \sup_{\omega \in \mathcal{N}_{n,\theta_n}} \sum_{i \in I_n} \|  H_{n,i} ( \mathbf{X}^n; \omega) \|_2^{(1+\lambda)/\lambda} \biggr)^{\lambda/(1+\lambda)} = \bigO(\varsigma_{n,\lambda}), \; \text{ as } n \to \infty.$$
		For the case $\lambda=0$, the  above condition is understood as 
		$$ \bigg \| \sup_{ \omega \in \mathcal{N}_{n,\theta_n}} \sup_{i \in I_n} 	\|  H_{n,i} ( \mathbf{X}^n; \omega) \|_2 \bigg \|_{L^{\infty}}= \bigO(\varsigma_{n,0}), \; \text{ as } n \to \infty.$$
			
		\item \label{item:C2} 
		For all $j \in [L]$, assume that there exists $\gamma_j \in [0, \infty]$ and  sequences of positive real numbers $\{\nu_{n,j} : n \geq 1\}$ and $\{\widetilde{\nu}_{n,j} : n \geq 1\}$ 
		such that the following holds true. 
		$$  \biggl[ \mathbb{E}  \biggl( \sup_{\omega \in \mathcal{N}_{n, \theta_n}} R_{n,j}( \mathbf{X}^n, \mathbf{\widetilde{W}}^n; \omega )^{(1+\gamma_j)/\gamma_j} \biggr) \biggr]^{\gamma_j/(1+\gamma_j)} = \bigO(\nu_{n,j}),$$
		and 
		$$ [\mathbb{E}(\widetilde{R}_{n,j}(\mathbf{X}^n)^{1+\gamma_j}) ]^{1/(1+\gamma_j)} = \bigO (\widetilde{\nu}_{n,j}),$$
		as $n \to \infty$. Here $\mathbf{\widetilde{W}}^n := (\widetilde{W}_{n,i})_{i \in I_n}$ is a  collection of mean-zero sub-Gaussian vectors with variance proxy $1$ (i.e., $\widetilde{W}_{n,i} \in \mathscr{G}(1,0;d)$ for all $i \in I_n$) such that $\{(X_{n,i},\widetilde{W}_{n,i})\}_{i \in I_n}$ is an i.i.d.~collection. If Assumption~\ref{ass:psi} is satisfied with the stronger assumption \ref{item:B2prime}, instead of \ref{item:B2}, then it is enough to consider $\widetilde{\mathbf{W}}^n$ to be collection of independent $d$-dimensional standard Gaussian vectors, independent of $\mathbf{X}^n$. 
	
		\item \label{item:C3} Define the the following (random) pseudo-metric on $\mathcal{S}_n$: 
		$$ d_{2,n}(\omega_1, \omega_2) := \biggl(\sum_{i \in I_n} \| H_{n,i}(\mathbf{X}^n;\omega_1)-H_{n,i}(\mathbf{X}^n;\omega_2) \|_2^2 \biggr)^{1/2},\; \forall \; \omega_1, \omega_2 \in \mathcal{S}_n.$$
		Assume the existence of a positive real number $\varepsilon_0$ and a sequence of positive real numbers $\{\tau_n : n \geq 1  \}$ such that for any $\varepsilon \in (0, \varepsilon_0)$ and $\delta>0$ there exists $\kappa_{\varepsilon, \delta} >0$ and $K_{\varepsilon, \delta} \in \mathbb{N}$ such that 
		$$ \limsup_{n \to \infty} \mathbb{P} [ P (\mathcal{N}_{n,\theta_n},d_n, \varepsilon) > K_{\varepsilon,\delta}P(\mathcal{N}_{n,\theta_n}, d_{2,n}, \tau_n \kappa_{\varepsilon,\delta})] \leq \delta.$$
		
	\end{enumerate}
	Under the above conditions, $\{P(\mathcal{N}_{n,\theta_n}, d_n, \varepsilon) : n \geq 1\}$ is a tight family of random variables for any $\varepsilon >0$, provided $\{\theta_n : n \geq 1\}$ satisfies the following growth condition:
	\begin{equation}{\label{item:C4}}
	\theta_n = \bigO \biggl( \tau_n \min \biggl\{  \dfrac{\tau_n^{1+\lambda}}{k_n\varsigma_{n,\lambda}^{1+\lambda}},\;\;  \min_{j=1}^L \biggl( \dfrac{\tau_n}{\nu_{n,j}\widetilde{\nu}_{n,j}}\biggr)^{1/\upsilon_j} \biggr\}\biggr).
	\end{equation}
\end{thm}


To summarize the nature of the assumptions in Theorem~\ref{genthm}, note that 
\Cref{item:C1,item:C2} impose growth conditions on the averages of the first and second-order terms in the Taylor-like expansion of $\psi_n$ assumed in Assumption~\ref{ass:psi}, while \ref{item:C3} connects the metric $d_n$ to the objective function $\psi_n$ through the pseudo-metric $d_{2,n}$, where we emphasize that $d_{2,n}$ may not be a proper metric. Theorem~\ref{genthm} is rather a recipe for proving results similar to its final assertion than an end-product by itself. 

\begin{remark}{\label{boundedH}}
	In many examples that we shall consider in this article, we shall have for all $n \geq 1$,
	\begin{equation}{\label{uniupperH}}
	\sup_{ \omega \in  \mathcal{S}_n}  \sup_{i \in I_n}  \|H_{n,i}( \mathbf{X}^n; \omega) \|_2 \leq \varsigma_n < \infty, \;\text{almost surely},
	\end{equation}
	for some non-random sequence $\{\varsigma_n : n \geq 1\}$ and $X_{n,i} \stackrel{i.i.d.}{\sim} f$. In such situation, it is clear that the assumption \Cref{item:C1} is satisfied with $\lambda = 0$ and $\varsigma_{n,0} = \varsigma_n$ and thus the restriction on $\theta_n$ in \Cref{genthm} simplifies to the following :
	 \begin{equation}{\label{item:C4special}}
	 	\theta_n = \bigO \biggl( \tau_n \min \biggl\{  \dfrac{\tau_n}{k_n\varsigma_n},\;\;  \min_{j=1}^L \biggl( \dfrac{\tau_n}{\nu_{n,j}\widetilde{\nu}_{n,j}}\biggr)^{1/\upsilon_j} \biggr\}\biggr).
	 	\end{equation}
 	As an example, TSP on Euclidean spaces satisfies the condition \Cref{uniupperH} with $\varsigma_n \equiv 1$. 
\end{remark}

Let us now briefly return to the example of TSP on Euclidean spaces (\Cref{exam:tsp}) for a demonstration of Theorem~\ref{genthm}. We list the key components and the final result here; the details can be found in the proof of Theorem~\ref{nearoptgraph} in \Cref{sec:nonlin}. First note that the expression for $H_{n,i}$ in \Cref{express1} for TSP ensures that $\|H_{n,i}(\mathbf{x}^n;G)\| \leq 2$ for any $\mathbf{x}^n$ and any Hamiltonian cycle $G$; therefore we can apply the conclusion made in \Cref{boundedH} to guarantee that the condition in  \Cref{item:C1}  is satisfied for the choices $\lambda =0$ and $\varsigma_{n,0} \equiv 2$. Reflecting upon the expression for $R_{n,1}$ presented in \Cref{express2}, one can see that the condition \Cref{item:B2} is trivially satisfied with $\upsilon_1=1$; moreover we also have, 
\begin{equation*}
R_{n,1}(\mathbf{x}^n,\mathbf{c}^n;G)^3 \leq 8\sum_{i \in [n]} \sum_{j : \{i,j\}\in E(G)} 2^5(\|c_i\|_2^6 + \|c_j\|_2^6 ) = 2^{10} \sum_{i \in [n]} \|c_i\|_2^6.
\end{equation*}
Therefore, for a collection of random elements $\widetilde{\mathbf{W}}^n$ as described in \Cref{item:C2}, we have the following for any~$\theta_n$:
$$ \mathbb{E} \biggl( \sup_{\omega \in \mathcal{N}_{n, \theta_n}} R_{n,j}( \mathbf{X}^n, \mathbf{\widetilde{W}}^n; \omega )^{3} \biggr) \leq 2^{10} \sum_{i \in [n]} \mathbb{E} \|\widetilde{W}_{n,i}\|_2^6 = \bigO(n).$$
Thus  the first condition in \Cref{item:C2} is satisfied for $\gamma _1= 1/2, \nu_{n,1}= n^{1/3}$. Similarly, from the expression in~\Cref{express3}, we have 
$$ \widetilde{R}_{n,1}(\mathbf{X}^n)^{3/2} = \sum_{i \in [n]} \biggl(\min_{j \in [n] : j \neq i} \|X_{n,i}-X_{n,j}\|_2 \biggr)^{-3/2}.$$
In the above expression, the term within the parentheses is the nearest-neighbor distance for the input point $X_{n,i}$. It is well-known that the nearest neighbor distance for any input point from a collection of $n$ independent and identically distributed points in $\mathbb{R}^d$ is of order $n^{-1/d}$, provided the sampling density is `nice'. Following this argument, we can establish that 
$$ \mathbb{E}  \widetilde{R}_{n,1}(\mathbf{X}^n)^{3/2} = \sum_{i \in [n]} \mathbb{E} \biggl(\min_{j \in [n] : j \neq i} \|X_{n,i}-X_{n,j}\|_2 \biggr)^{-3/2} = \bigO (n^{1+ \frac{3}{2d}}).$$
Thus, the second condition in \Cref{item:C2} is satisfied for $\widetilde{\nu}_{n,1} = n^{2/3+1/d}$. Finally, let us present a short heuristic argument behind \Cref{item:C3} in the case of TSP. If we take two (random) Hamiltonian cycles $G_{1,n}, G_{2,n} \in \mathcal{N}_{n,\theta_n}$ with $d_n(G_{1,n},G_{2,n}) > \varepsilon$, then there exists at least $\varepsilon n/4$ many vertices which have different neighbor sets in $G_{1,n}$ and $G_{2,n}$. Using a nearest-neighbor argument, it can be shown that for any such vertex $i$, the quantity $\|H_{n,i}(\mathbf{X}^n;G_{1,n})-H_{n,i}(\mathbf{X}^n;G_{2,n}) \|_2$ is bounded away from $0$ with high probability. This guarantees that $n^{-1/2}d_{2,n}(G_{1,n},G_{2,n})$ is bounded away from $0$ with high probability and hence the condition \Cref{item:C3} will be satisfied for $\tau_n=\sqrt{n}$. 
Plugging-in these quantities in \Cref{item:C4special} and applying Theorem~\ref{genthm}, we can conclude that for TSP on $\mathbb{R}^d$, choosing $\theta_n= \bigO(n^{-1/d})$ yields that  $\{P(\mathcal{N}_{n,\theta_n}, d_n, \varepsilon) : n \geq 1\}$ is a tight sequence for any $\varepsilon >0$.

A special case of interest is when the objective function $\psi_n((x_{n,i})_{i \in I_n})$ is linear in $(x_{n,i})_{i \in I_n}$, as described in \Cref{ass:linear} below. This structure of the objective function is present in many problems of interest, some of which will be analyzed in \Cref{sec:lin}. 

\begin{assumption}[Linearized version of \Cref{ass:psi}]{\label{ass:linear}}
	There exist measurable functions $ \; H_{n,i} : (\mathcal{S}_n, \mathcal{B}_n) \to (\mathbb{R}^d, \mathcal{B}_{\mathbb{R}^d})$, for $i \in I_n$, such that 
	\begin{equation}
		\psi_n ( (x_{n,i})_{i \in I_n}; \omega) = \sum_{i \in I_n} x_{n,i} \cdot H_{n,i}(\omega), \; \forall \; \omega \in \mathcal{S}_n, \; (x_{n,i})_{i \in I_n} \in \mathcal{X}^{I_n}.
	\end{equation}
\end{assumption}

In this particular scenario, we can take $L=1, R_{n,1} \equiv \widetilde{R}_{n,1} \equiv 0$ and $\nu_{n,1} \equiv \widetilde{\nu}_{n,1} \equiv 0$ in statement of \Cref{genthm}, simplifying it to the following form.

\begin{thm}[Linearized version of Theorem \ref{genthm}]{\label{genthm2}}
	Consider a sequence of random minimization problems $\{\mathscr{P}_n : n \geq 1\}$ with random inputs in $\mathcal{X}=\mathbb{R}^d$ and  satisfying Assumptions~\ref{unique}, \ref{ass:meas} and~\ref{ass:linear}. Suppose that the probability measure $\mathcal{P}$ (for the random inputs)   has a density $f$ with respect to Lebesgue measure on $\mathbb{R}^d$ satisfying Assumption~\ref{ass:p}. Further assume the following two conditions.
		\begin{enumerate} [leftmargin=*,label=(D\arabic*)]
			\item  \label{item:D1} Assume that for some $\lambda \in [0,\infty)$, there exists a sequence of positive real numbers $\{\varsigma_{n,\lambda} : n \geq 1\}$ satisfying 
				$$ \biggl( \sup_{\omega \in \mathcal{S}_{n}} \sum_{i \in I_n} \|  H_{n,i} (  \omega) \|_2^{(1+\lambda)/\lambda} \biggr)^{\lambda/(1+\lambda)} = \bigO(\varsigma_{n,\lambda}), \; \text{ as } n \to \infty.$$
			For the case $\lambda=0$, the  above condition is understood as 
			$$ \sup_{ \omega \in \mathcal{S}_{n}} \sup_{i \in I_n} 	\|  H_{n,i} ( \omega) \|_2 = \bigO(\varsigma_{n,0}), \; \text{ as } n \to \infty.$$
			
			\item  \label{item:D2} Consider the the following  pseudo-metric on $\mathcal{S}_n$: 
					$$ d_{2,n}(\omega_1, \omega_2) := \biggl(\sum_{i \in I_n} \| H_{n,i}(\omega_1)-H_{n,i}(\omega_2)\|_2^2 \biggr)^{1/2},\; \forall \; \omega_1, \omega_2 \in \mathcal{S}_n.$$
					Assume the existence of a positive real number $\varepsilon_0$ and a sequence of positive real numbers $\{\tau_n : n \geq 1 \}$ such that for any $\varepsilon \in (0, \varepsilon_0)$ there exists $\kappa_{\varepsilon} >0$ and $K_{\varepsilon} \in \mathbb{N}$ such that 
					\begin{equation*}
					 P(\mathcal{B},d_n, \varepsilon) \leq  K_{\varepsilon}P(\mathcal{B}, d_{2,n}, \tau_n \kappa_{\varepsilon}),\;\; \forall \; \mathcal{B} \subseteq \mathcal{S}_n.
					\end{equation*}
			\end{enumerate}
	Then  $\{P(\mathcal{N}_{n,\theta_n}, d_n, \varepsilon) : n \geq 1\}$ is a tight family of random variables for any $\varepsilon >0$, provided that $\theta_n = \bigO ( \tau_n^{2+\lambda}k_n^{-1}\varsigma_{n,\lambda}^{-1-\lambda})$ as $n \to \infty$.
\end{thm}


Let us spend a few words on the assumption in \Cref{item:D2} which is the counterpart of \Cref{item:C3} in the simplified context of \Cref{genthm2}; \Cref{item:D2} being easier to prove since the metric $d_{2,n}$ is non-random in \Cref{genthm2}. In some examples satisfying \Cref{ass:linear}, we would have $d_n=\tau_n^{-1}d_{2,n}$ and hence \Cref{item:D2} would be trivial with $\kappa_{\varepsilon}=\varepsilon$ and $K_{\varepsilon}=1$. This will be the situation for symmetric graph optimization problems on weighted complete graphs, discussed in \Cref{graphcom},  and for the random assignment model, discussed in \Cref{ram}. In other cases, e.g., the Sherrington--Kirkpartick Model in \Cref{skmodel},  proving \Cref{item:D2} needs some short arguments. Proving \Cref{item:C3} needs considerable effort due to randomness of the metric $d_{2,n}$; an example of this situation can be found in the proof of \Cref{nearoptgraph}. 

\begin{remark}{\label{uniqueremark}}
Consider an optimization problem sequence satisfying \Cref{ass:linear}. If the input distribution $\mathcal{P}$ has density with respect to Lebesgue measure on $\mathbb{R}^d$, then for any $\omega_1, \omega_2 \in \mathcal{S}_n$, 
\begin{align*}
 \mathbb{P}(\psi_n(\mathbf{X}^n;\omega_1) = \psi_n(\mathbf{X}^n;\omega_2)) >0 &\Rightarrow \mathbb{P} \biggl(\sum_{i \in I_n} X_{n,i} \cdot (H_{n,i}(\omega_1)-H_{n,i}(\omega_2)) =0 \biggr) >0 \\
 & \Rightarrow H_{n,i}(\omega_1) = H_{n,i}(\omega_2) \; \forall \; i \in I_n \Rightarrow d_{2,n}(\omega_1, \omega_2)=0.
\end{align*}
Provided that $d_{2,n}$ is a proper metric on $\mathcal{S}_n$ and $\mathcal{S}_n$ is finite or countably infinite, we can therefore guarantee that  \Cref{unique} is satisfied. 
\end{remark}

\begin{remark}{\label{meanupper}}
Consider an optimization problem sequence satisfying \Cref{ass:linear} and \Cref{item:D1} with the choice $\lambda=0$. . If the input distribution $\mathcal{P}$ is integrable, we have 
$$ \mathbb{E}  \bigl| \inf_{\omega \in \mathcal{S}_n} \psi_n(\mathbf{X}^n; \omega) \bigr| \leq \mathbb{E} \sup_{\omega \in \mathcal{S}_n} \biggl| \sum_{i \in I_n} X_{n,i} \cdot H_{n,i}(\omega) \biggr| \leq k_n \varsigma_{n,0} \mathbb{E}_{\mathcal{P}}\|X\|_2 < \infty, $$
and hence the integrability condition in \Cref{item:A3} of \Cref{ass:meas} is satisfied.  
\end{remark}

\begin{proof}[Proof of Theorem~\ref{genthm}]
Fix $\varepsilon \in (0, \varepsilon_0)$ and $ \delta >0$.	We start by taking a small Metropolis--Hastings step from $\mathbf{X}^n$ with the choice of proposal step being  $d$-dimensional standard Gaussian distribution. To be more precise, fix a sequence $\{t_n : n \geq 1\}$ of positive real numbers. Generate an i.i.d.~collection of standard $d$-dimensional Gaussian random vectors $\{W_{n,i} : i \in I_n\}$ and an i.i.d.~collection of $\text{Uniform}[0,1]$ variables $\{U_{n,i} : i \in I_n\}$, both of these collection being independent of each other and jointly independent of $\mathbf{X}^n$. Define
	$$ \widetilde{X}_{n,i} := X_{n,i} + t_nW_{n,i}\mathbbm{1}\biggl\{U_{n,i} \leq \dfrac{f(X_{n,i}+t_nW_{n,i})}{f(X_{n,i})} \biggr\} = X_{n,i}+t_n \widetilde{W}_{n,i}, \; \; \forall \; i \in I_n,$$
	with $\widetilde{\mathbf{W}}^n :=(\widetilde{W}_{n,i})_{i \in I_n}$ having been defined accordingly. Setting $\widetilde{\mathbf{X}}^n := (\widetilde{X}_{n,i})_{i \in I_n}$, we observe that it is a collection of i.i.d.~random variables having density $f$. This follows from the validity of Metropolis--Hastings algorithm 
	and the fact that the standard Gaussian measure is symmetric around the origin; see \cite[Chapter 3.1]{mcmc} for more details on MH algorithm. It also implies that the random vectors $\widetilde{W}_{n,i}$'s are mean-zero, even without any integrability assumption on the probability measure induced by the  density $f$.   
	Therefore, 
	\begin{equation}{ \label{lower21}}
		\mathbb{E} \biggl[ \inf_{\omega \in \mathcal{S}_n} \psi_n(\widetilde{\mathbf{X}}^n; \omega) \biggr] = \mathbb{E} \biggl[ \inf_{\omega \in \mathcal{S}_n} \psi_n(\mathbf{X}^n; \omega) \biggr],
	\end{equation}
	whereas by \Cref{item:B1} and the weaker assumption \Cref{item:B2}, we have that for any $\omega \in \mathcal{S}_n$, 
	\begin{align}
		&\psi_n( \widetilde{\mathbf{X}}^n; \omega) - \psi_n(\mathbf{X}^n; \omega) = \psi_n( \mathbf{X}^n + t_n\widetilde{\mathbf{W}}^n; \omega ) -\psi_n(\mathbf{X}^n; \omega) \nonumber \\
		&\leq t_n\sum_{i \in I_n} \widetilde{W}_{n,i} \cdot H_{n,i}( \mathbf{X}^n; \omega) +  \sum_{j=1}^L  R_{n,j}( \mathbf{X}^n, t_n\widetilde{\mathbf{W}}^n; \omega ) \widetilde{R}_{n,j}( \mathbf{X}^n + t_n\widetilde{\mathbf{W}}^n) \nonumber \\
		& \leq t_n \sum_{i \in I_n} W_{n,i} \cdot H_{n,i} ( \mathbf{X}^n; \omega) - t_n\sum_{i \in I_n}  (W_{n,i}-\widetilde{W}_{n,i}) \cdot H_{n,i} ( \mathbf{X}^n; \omega) \nonumber \\
		& \hspace{ 1.2 in}  + \sum_{j=1}^L t_n^{1+\upsilon_j} R_{n,j}( \mathbf{X}^n, \mathbf{\widetilde{W}}^n; \omega ) \widetilde{R}_{n,j}( \mathbf{X}^n + t_n\widetilde{\mathbf{W}}^n ). \label{taylor}
	\end{align}
	 This  enables us to write the following estimate: 
	\begin{align}
	&\inf_{\omega \in \mathcal{N}_{n,\theta_n}} [\psi_n (\widetilde{\mathbf{X}}^n; \omega) - \psi_n (\mathbf{X}^n; \omega)  ]  \nonumber \\
	 &\hspace{0.5 in} \leq t_n \inf_{\omega \in \mathcal{N}_{n,\theta_n}}  \sum_{i \in I_n} W_{n,i} \cdot H_{n,i} ( \mathbf{X}^n; \omega) + t_n\sup_{\omega \in \mathcal{N}_{n,\theta_n}}  \biggl| \sum_{i \in I_n}  (W_{n,i}-\widetilde{W}_{n,i}) \cdot H_{n,i} ( \mathbf{X}^n; \omega) \biggr| \nonumber  \\
	 &\qquad \hspace{1 in}  +  \sum_{j=1}^L t_n^{1+\upsilon_j} \inf_{\omega \in \mathcal{N}_{n,\theta_n}}   R_{n,j}( \mathbf{X}^n, \widetilde{\mathbf{W}}^n; \omega) \widetilde{R}_{n,j}( \mathbf{X}^n + t_n\widetilde{\mathbf{W}}^n ). \label{maxest}
	\end{align}
	From the definition of $\mathcal{N}_{n,\theta_n}$, we have 
	\begin{equation}{\label{maxest2}}
  \sup_{\omega \in \mathcal{N}_{n, \theta_n}} 	\psi_n (\mathbf{X}^n; \omega) \leq  \inf_{\omega \in \mathcal{S}_n} \psi_n(\mathbf{X}^n; \omega) + \theta_n.
	\end{equation}
	Combining \Cref{maxest} and \Cref{maxest2}, we get the following:
	\begin{align}
		& \inf_{\omega \in \mathcal{S}_n} \psi_n(\widetilde{\mathbf{X}}^n; \omega) \nonumber \\ &\leq  \inf_{\omega \in \mathcal{N}_{n,\theta_n}}  \psi_n (\widetilde{\mathbf{X}}^n; \omega) \nonumber \\ 
		& = \inf_{\omega \in \mathcal{N}_{n,\theta_n}}  [\psi_n (\widetilde{\mathbf{X}}^n; \omega) - \psi_n (\mathbf{X}^n; \omega) + \psi_n (\mathbf{X}^n; \omega) ] \nonumber \\
		& \leq  \inf_{\omega \in \mathcal{N}_{n,\theta_n}}  [\psi_n (\widetilde{\mathbf{X}}^n; \omega) - \psi_n (\mathbf{X}^n; \omega)  ] + \sup_{\omega \in \mathcal{N}_{n,\theta_n}} 	\psi_n (\mathbf{X}^n; \omega)   \nonumber \\
		& \leq t_n\inf_{\omega \in \mathcal{N}_{n,\theta_n}}  \sum_{i \in I_n} W_{n,i} \cdot H_{n,i} ( \mathbf{X}^n; \omega) + t_n\sup_{\omega \in \mathcal{N}_{n,\theta_n}}   \biggl| \sum_{i \in I_n}  (W_{n,i}-\widetilde{W}_{n,i}) \cdot H_{n,i} ( \mathbf{X}^n; \omega) \biggr|  \nonumber \\
		&\qquad +  \sum_{j=1}^L t_n^{1+\upsilon_j} \sup_{\omega \in \mathcal{N}_{n,\theta_n}}    R_{n,j}( \mathbf{X}^n, \widetilde{\mathbf{W}}^n; \omega) \widetilde{R}_{n,j}( \mathbf{X}^n + t_n\widetilde{\mathbf{W}}^n )    +  \inf_{\omega \in \mathcal{S}_n} \psi_n(\mathbf{X}^n; \omega) + \theta_n. \label{ineq}
	\end{align}
	Since the $W_{n,i}$'s are i.i.d.~Gaussian random variables, and are independent of $\mathbf{X}^n$, the process 
	\[
	\biggl\{\sum_{i \in I_n} W_{n,i} \cdot H_{n,i} ( \mathbf{X}^n; \omega) : \omega \in \mathcal{S}_n \biggr\}
	\]
	is Gaussian on the pseudo-metric space $(\mathcal{S}_n,d_{2,n})$, conditional on $\mathbf{X}^n$. Therefore, we can employ \Cref{sudakov} (Sudakov minoration) to conclude that for any $\varepsilon^{\prime}>0$, 
	\begin{align}{\label{first term}}
		\mathbb{E} \biggl[ \inf_{\omega \in \mathcal{N}_{n,\theta_n}}  \sum_{i \in I_n} W_{n,i} \cdot H_{n,i} ( \mathbf{X}^n; \omega) \biggl| \mathbf{X}^n \biggr] \leq -A^*\varepsilon^{\prime} \sqrt{\log P( \mathcal{N}_{n,\theta_n}, d_{2,n},\varepsilon^{\prime})}.
	\end{align}
	We now concentrate on the second term on the right side in \Cref{ineq}. Defining
	$$ \Upsilon_{n,i} := \biggl(1-\dfrac{f(X_{n,i}+t_nW_{n,i})}{f(X_{n,i})} \biggr)_{+}\; \; \forall \; i \in I_n,$$
	observe that for all $i \in I_n$, 
	\begin{align*}
	\mathbb{E}\bigl [ \|W_{n,i}-\widetilde{W}_{n,i}\|_2^{1+\lambda} \bigl| \mathbf{X}^n, \mathbf{W}^n \bigr] &= \|W_{n,i}\|_2^{1+\lambda}\mathbb{P} \biggl( U_{n,i} > \frac{f(X_{n,i}+t_nW_{n,i})}{f(X_{n,i})} \biggl| \mathbf{X}^n, \mathbf{W}^n  \biggr) \\
	&= \|W_{n,i}\|_2^{1+\lambda} \Upsilon_{n,i}.
	\end{align*}
This aids us in deriving the following upper bound when $\lambda \in (0, \infty)$ upon repeated application of H\"older's inequality.  
\begin{align}
		&\mathbb{E} \biggl[ \sup_{\omega \in \mathcal{N}_{n,\theta_n}}  \biggl| \sum_{i \in I_n}  (W_{n,i}-\widetilde{W}_{n,i}) \cdot H_{n,i} ( \mathbf{X}^n; \omega) \biggr|  \biggr]  \notag \\
		& \leq \mathbb{E} \biggl[ \sup_{\omega \in \mathcal{N}_{n,\theta_n}}  \sum_{i \in I_n}    \| W_{n,i}-\widetilde{W}_{n,i} \|_2 \|  H_{n,i} ( \mathbf{X}^n; \omega) \|_2  \biggr] \nonumber \\
		& \leq \mathbb{E} \biggl[ \sup_{\omega \in \mathcal{N}_{n,\theta_n}} \biggl(  \sum_{i \in I_n}    \| W_{n,i}-\widetilde{W}_{n,i} \|_2^{1+\lambda} \biggr)^{1/(1+\lambda)} \biggl( \sum_{i \in I_n} \|  H_{n,i} ( \mathbf{X}^n; \omega) \|_2^{(1+\lambda)/\lambda} \biggr)^{\lambda/(1+\lambda)}  \biggr] \nonumber \\
		& \leq    \biggl( \mathbb{E} \sum_{i \in I_n}    \| W_{n,i}-\widetilde{W}_{n,i} \|_2^{1+\lambda} \biggr)^{1/(1+\lambda)} \biggl( \mathbb{E} \sup_{\omega \in \mathcal{N}_{n,\theta_n}} \sum_{i \in I_n} \|  H_{n,i} ( \mathbf{X}^n; \omega) \|_2^{(1+\lambda)/\lambda} \biggr)^{\lambda/(1+\lambda)}  \nonumber \\
		& \leq    \biggl( \mathbb{E} \sum_{i \in I_n}    \| W_{n,i} \|_2^{1+\lambda} \Upsilon_{n,i} \biggr)^{1/(1+\lambda)} \biggl( \mathbb{E} \sup_{\omega \in \mathcal{N}_{n,\theta_n}} \sum_{i \in I_n} \|  H_{n,i} ( \mathbf{X}^n; \omega) \|_2^{(1+\lambda)/\lambda} \biggr)^{\lambda/(1+\lambda)}  \nonumber \\
	& = ( \bigO(k_nt_n) )^{1/(1+\lambda)} \bigO(\varsigma_{n,\lambda}) = \bigO( k_n^{1/(1+\lambda)} t_n^{1/(1+\lambda)}\varsigma_{n,\lambda}), \label{second2}
	\end{align}
	where 
	\Cref{second2} follows from \Cref{item:C1} and \Cref{ass:p}. 
	The final assertion in \Cref{second2} is also valid for $\lambda=0$, since 
	\begin{align}
			&  \mathbb{E} \biggl[ \sup_{\omega \in \mathcal{N}_{n,\theta_n}}  \sum_{i \in I_n}    \| W_{n,i}-\widetilde{W}_{n,i} \|_2 \|  H_{n,i} ( \mathbf{X}^n; \omega) \|_2  \biggr] \nonumber \\
			& \leq  \mathbb{E}  \biggl(  \sum_{i \in I_n}    \| W_{n,i}-\widetilde{W}_{n,i} \|_2 \biggr) \biggl \| \sup_{ \omega \in \mathcal{N}_{n,\theta_n}} \sup_{i \in I_n} 	\|  H_{n,i} ( \mathbf{X}^n; \omega) \|_2 \biggr \|_{L^{\infty}} \nonumber \\
			& \leq  \mathbb{E}  \biggl(  \sum_{i \in I_n}    \| W_{n,i} \|_2 \Upsilon_{n,i} \biggr) \biggl \| \sup_{ \omega \in \mathcal{N}_{n,\theta_n}} \sup_{i \in I_n} 	\|  H_{n,i} ( \mathbf{X}^n; \omega) \|_2 \biggr \|_{L^{\infty}}  = \bigO(k_nt_n\varsigma_{n,0}). \label{second3}
		\end{align}
For the third term on the right side  in \Cref{ineq}, first note that 
		\begin{equation}{\label{contraction}}
		\|\widetilde{W}_{n,i}\|_2 = \|W_{n,i}\|_2\mathbbm{1}\biggl\{U_{n,i} \leq \dfrac{f(X_{n,i}+t_nW_{n,i})}{f(X_{n,i})} \biggr\} \leq \|W_{n,i}\|_2,  \text{ for all } i \in I_n,
		\end{equation} 
and hence $\widetilde{W}_{n,i} \in \mathscr{G}(1,0;d)$ for all $i$. By construction it is clear that $\{ (X_{n,i},\widetilde{W}_{n,i})\}_{i \in I_n}$ is an i.i.d.~collection and hence the following estimates can be obtained from \Cref{item:C2} upon  application of H\"older's inequality:
	\begin{align}
		& \mathbb{E} \biggl[ \sup_{\omega \in \mathcal{N}_{n,\theta_n}} R_{n,j}( \mathbf{X}^n, \widetilde{\mathbf{W}}^n; \omega) \widetilde{R}_{n,j}( \mathbf{X}^n + t_n\widetilde{\mathbf{W}}^n ) \biggr] \nonumber \\
		&\leq  \biggl[ \mathbb{E}  \biggl(\sup_{\omega \in \mathcal{N}_{n,\theta_n}} R_{n,j}( \mathbf{X}^n, \widetilde{\mathbf{W}}^n; \omega)^{(1+\gamma_j)/\gamma_j}\biggr)\biggr]^{\gamma_j/(1+\gamma_j)}  [ \mathbb{E}(\widetilde{R}_{n,j}( \mathbf{X}^n + t_n\widetilde{\mathbf{W}}^n)^{1+\gamma_j}  )]^{1/(1+\gamma_j)} \nonumber \\
	& =  \biggl[ \mathbb{E}  \biggl(\sup_{\omega \in \mathcal{N}_{n,\theta_n}} R_{n,j}( \mathbf{X}^n, \widetilde{\mathbf{W}}^n; \omega)^{(1+\gamma_j)/\gamma_j}\biggr)\biggr]^{\gamma_j/(1+\gamma_j)}  [ \mathbb{E}(\widetilde{R}_{n,j}( \mathbf{X}^n)^{1+\gamma_j} )]^{1/(1+\gamma_j)} \nonumber\\
		&=\bigO(\nu_{n,j}\widetilde{\nu}_{n,j}), \label{thirdeq}
	\end{align} 
	where the first equality in  \Cref{thirdeq} follows from the fact that $\mathbf{X}^n$ has the same law as $\widetilde{\mathbf{X}}^n=\mathbf{X}^n+t_n\widetilde{\mathbf{W}}^n$, whereas the second equality  is derived using  \Cref{item:C2}. If the stronger assumption \Cref{item:B2prime} is satisfied, we could have used \Cref{contraction} to argue that 
	$$  R_{n,j}( \mathbf{X}^n, \widetilde{\mathbf{W}}^n; \omega) \leq  R_{n,j}( \mathbf{X}^n, \mathbf{W}^n; \omega), \; \forall \; \omega \in \mathcal{S}_n, \; i \in I_n,$$
	and could have derived \Cref{thirdeq} similarly from \Cref{item:C2} using only the weaker assumption that 
	$$ \biggl[ \mathbb{E}  \biggl(\sup_{\omega \in \mathcal{N}_{n,\theta_n}} R_{n,j}( \mathbf{X}^n, \widetilde{\mathbf{W}}^n; \omega)^{(1+\gamma_j)/\gamma_j}\biggr)\biggr]^{\gamma_j/(1+\gamma_j)} = \bigO(\nu_{n,j}).$$ 
	Taking expectations on both sides of  \Cref{ineq} and using \Cref{lower21}, \Cref{first term}, \Cref{second2} and \Cref{thirdeq}, 
	we get that for any  $\varepsilon^{\prime} >0$. 
	\begin{align}
		A^* \varepsilon^{\prime} \mathbb{E} \sqrt{\log P( \mathcal{N}_{n,\theta_n},d_{2,n},\varepsilon^{\prime})} \leq \dfrac{\theta_n}{t_n} +   \bigO( k_n^{1/(1+\lambda)} t_n^{1/(1+\lambda)}\varsigma_{n,\lambda}) +  \sum_{j=1}^L t_n^{\upsilon_j}\bigO( \nu_{n,j}\widetilde{\nu}_{n,j}), \label{new1}
	\end{align} 
	or equivalently, if we replace $\varepsilon^{\prime}$ in  \Cref{new1} by $\tau_n \kappa_{\varepsilon, \delta}$ (which we can since $\tau_n,\kappa_{\varepsilon, \delta} >0$ by \Cref{item:C3}) and divide both sides by $A^*\tau_n$, we obtain the following:  
	\begin{align*}
		\kappa_{\varepsilon, \delta} \mathbb{E} \sqrt{\log P(\mathcal{N}_{n, \theta_n}, d_{2,n}, \tau_n\kappa_{\varepsilon, \delta})} \leq \dfrac{\theta_n}{A^*t_n\tau_n} + \dfrac{C_1 k_n^{1/(1+\lambda)} t_n^{1/(1+\lambda)}\varsigma_{n,\lambda}}{\tau_n} +  C_1\sum_{j=1}^L \dfrac{t_n^{\upsilon_j} \nu_{n,j}\widetilde{\nu}_{n,j}}{\tau_n}.
	\end{align*} 
	 Now set $t_n = \theta_n/\tau_n$. The growth conditions assumed in \Cref{item:C4} guarantees that
	\begin{align}
\kappa_{\varepsilon, \delta} \mathbb{E} \sqrt{\log P(\mathcal{N}_{n, \theta_n}, d_{2,n}, \tau_n\kappa_{\varepsilon, \delta})}	& \leq \dfrac{1}{A^*} + \bigO \biggl(\dfrac{ k_n^{1/(1+\lambda)} \theta_n^{1/(1+\lambda)}\varsigma_{n,\lambda}}{\tau_n^{(2+\lambda)/(1+\lambda)}}\biggr) +  \sum_{j=1}^L \bigO\biggl(\dfrac{\theta_n^{\upsilon_j} \nu_{n,j}\widetilde{\nu}_{n,j}}{\tau_n^{1+\upsilon_j}}\biggr) \nonumber \\
&= \bigO(1). 
	\end{align}
	Therefore, $\{P(\mathcal{N}_{n, \theta_n}, d_{2,n}, \tau_n\kappa_{\varepsilon, \delta}) : n \geq 1\}$ is a tight sequence of random variables for any $\varepsilon \in (0, \varepsilon_0)$ and $ \delta >0$. Assumption \Cref{item:C3}, together with this tightness, now guarantees that
	\begin{align*}
		&\lim_{N \to \infty} \limsup_{n \to \infty} \mathbb{P} [ P(\mathcal{N}_{n, \theta_n}, d_{n}, \varepsilon) \geq NK_{\varepsilon,\delta}] \nonumber \\
		& \leq \limsup_{n \to \infty} \mathbb{P} \biggl( \dfrac{P (\mathcal{N}_{n,\theta_n},d_n, \varepsilon)}{P(\mathcal{N}_{n,\theta_n}, d_{2,n}, \tau_n \kappa_{\varepsilon,\delta})} > K_{\varepsilon,\delta}\biggr) \\
		& \hspace{ 0.5 in} + \lim_{N \to \infty} \sup_{n \geq 1} \mathbb{P} [ P(\mathcal{N}_{n, \theta_n}, d_{2,n}, \tau_n\kappa_{\varepsilon, \delta}) \geq N]  \leq \delta.
	\end{align*}
This establishes that $\{P(\mathcal{N}_{n, \theta_n}, d_{n}, \varepsilon) : n \geq 1\}$ is a tight sequence for any $\varepsilon \in (0, \varepsilon_0)$ and consequently for any $\varepsilon >0$. This completes the proof.
	\end{proof}

Unfortunately, it turns out that \Cref{genthm} fails to provide the optimal choice for the sequence $\{\theta_n : n \geq 1\}$ in many examples. Though the optimal order of $\theta_n$ is not always necessary for our purpose in this article, it is nice to have a different technique which can achieve a better outcome. To understand this deficiency of \Cref{genthm}, we need to look at the computations carried out to arrive at \Cref{second2}. Due to the ``discontinuous" nature of Metropolis-Hastings algorithm, the random variable $\|W_{n,i}-\widetilde{W}_{n,i} \|_2$ does not scale well with respect to higher order moments, in the sense that $\mathbb{E} \|W_{n,i}-\widetilde{W}_{n,i} \|_2^p = \bigO(t_n)$ for any $p \geq 0$ instead of being $\bigO(t_n^p)$. As a consequence, we obtain the final growth rate bound in \Cref{second2} to be $\bigO(k_n^{1/(1+\lambda)}t_n^{1/(1+\lambda)}\varsigma_{n,\lambda})$ instead of $\bigO(k_n^{1/(1+\lambda)}t_n\varsigma_{n,\lambda})$, and thus losing out a tighter upper bound whenever we take $\lambda>0$.  The remedy for this problem  is to use the same approach as in \Cref{genthm}, but with a different Markovian step which provides us with a ``continuous" perturbation of the original data.   Using \textit{Langevin dynamics} as the Markovian step we can achieve this goal, although its application imposes further limitations on the choice of input probability density function $f$. We shall now restrict our attention only to the cases when the probability measure $\mathcal{P}$ for the inputs satisfy the following conditions stated in \Cref{mu}. 



\begin{assumption}{\label{mu}}
We will say that the input probability measure $\mathcal{P}$ satisfies this assumption for the pair $(\rho,g)$ if there exists $p \in \mathbb{N}$, $\rho : \mathbb{R}^{p} \to \mathbb{R}$ and $g : \mathbb{R}^{p} \to \mathbb{R}^d$ satisfying the following conditions.
\begin{enumerate}[label=(\arabic*)]
\item \label{ass:F1} $\exp(-\rho)$ is integrable on $\mathbb{R}^{p}$. 
\item \label{ass:F2} $\rho$ is almost everywhere (with respect to Lebesgue measure) differentiable on $\mathbb{R}^{p}$ and the gradient $\nabla \rho$ is Lipschitz. 
\item \label{ass:F3}If $Y$ is a random vector having density (with respect to Lebesgue measure) proportional to $\exp(-\rho)$, then $\nabla \rho (Y)$ is a sub-Gamma vector with variance proxy $\sigma^2>0$ and scale parameter $c \geq 0$, i.e., $\nabla \rho(Y) \in \mathscr{G}(\sigma^2,c;p)$.
\item \label{ass:F4} $g$ has uniformly bounded first and second order derivatives and $g(Y) \sim \mathcal{P}$  where $Y$ is as in \Cref{ass:F3} above.
\end{enumerate}
The condition \Cref{ass:F2} can be replaced by either one of the following.
\begin{enumerate}[label=(\arabic*)$'$]
\setcounter{enumi}{1}
\item \label{ass:F2p} $\rho$ is twice continuously differentiable and $\|\nabla \rho\|_2^2 - \frac{1}{2} \Delta \rho$ is bounded below. Here $\Delta$ denotes the Laplacian operator.
\end{enumerate}
\begin{enumerate}[label=(\arabic*)$''$]
\setcounter{enumi}{1}
\item \label{ass:F2pp} $\rho$ is twice continuously differentiable and there exists $a,b \in \mathbb{R}$ such that $x \cdot \nabla \rho(x) \geq -a\|x\|_2^2 -b$ for all $x \in \mathbb{R}^{p}$.
\end{enumerate}
\end{assumption}
As a notational clarification, we mention that $\nabla \rho$ will be considered as row or column vector-valued function depending upon the context. On the contrary, $\nabla g$ will be always considered as $d \times p$ matrix-valued function. In \Cref{ass:F4} above, the uniform boundedness of a vector-valued function refers to uniform-boundedness of its individual  coordinates.


\Cref{mu} enables us to make use of \textit{overdamped Langevin dynamics}, also known as \textit{Kolmogorov process}, to take the Markovian step from a sample drawn from $\mathcal{P}$. We refer to \cite[Section 2.2.2]{sobolev} for a detailed exposition on Kolmogorov processes. To summarize the literature for our current application, consider the multi-dimensional differential operator $\mathcal{L}$ defined as 
$$ \mathcal{L} \varphi := - \nabla \rho \cdot \nabla \varphi + \Delta \varphi.$$
Under either of the assumptions \Cref{ass:F2}, \Cref{ass:F2p} or \Cref{ass:F2pp},   the Markov semi-group corresponding to the diffusion with infinitesimal generator $\mathcal{L}$ has stationary measure $\mathcal{Q}_{\rho}$, the probability measure on $\mathbb{R}^{p}$ with density proportional to $\exp(-\rho)$. See \cite[Chapter 5.9, 5.15]{bhatt} and \cite[Lemma 2.2.23]{sobolev} for these kinds of results. Moreover, if $Y \sim \mathcal{Q}_{\rho}$ and $\{W_t : t \geq 0\}$ is a standard $p$-dimensional Brownian Motion started from $0$, independent of $Y$, then the stochastic differential equation 
$$ dY_t = -\nabla \rho(Y_t) \; dt + \sqrt{2} dW_t,$$
with  boundary condition $Y_0:= Y$ has a unique continuous solution $t \mapsto Y_t$ defined on $\mathbb{R}_{\geq 0}$ almost surely; see \cite[Theorem 7.2.4, 7.2.5]{bhatt} and \cite[Thorem 2.2.19]{sobolev}. The stationarity condition  guarantees that 
\begin{equation}{\label{markovstep}}
Y_t = Y_0 - \int_{0}^t \nabla \rho(Y_s) \; ds + \sqrt{2}W_t \sim \mathcal{Q}_{\rho}, \; \forall \; t \geq 0.
\end{equation} 
The stationary Markov process, as described in \Cref{markovstep}, will provide the crucial Markovian steps needed in the proof of \Cref{genthm:g}, which is the second main result of this section.

\begin{thm}{\label{genthm:g}}
Consider a sequence of random minimization problems $\{\mathscr{P}_n : n \geq 1\}$ with random inputs in $\mathcal{X}=\mathbb{R}^d$, satisfying  Assumptions~\ref{unique}, \ref{ass:meas} and \Cref{ass:psi} with condition \ref{item:B2}. 
Assume that the distribution of the random inputs, $\mathcal{P}$, satisfies Assumption~\ref{mu} for the pair $(\rho,g)$. By virtue of this assumption, we can take the random inputs as $X_{n,i}=g(Y_{n,i})$, where $\{Y_{n,i} : i \in I_n\}$ is a collection of i.i.d.~random variables from the distribution $\mathcal{Q}_{\rho}$. 	Further assume the existence of  a sequence of positive real numbers 
$\{\theta_n\}_{n \geq 1} $ satisfying the following conditions.
	\begin{enumerate} [leftmargin=*,label=(E\arabic*)]
		\item  \label{item:E1} Assume that for some $\lambda \in [0,\infty)$, there exists a sequence of positive real numbers $\{\varsigma_{n,\lambda} : n \geq 1\}$ satisfying 
			$$ \biggl( \mathbb{E} \sup_{\omega \in \mathcal{N}_{n,\theta_n}} \sum_{i \in I_n} \|  H_{n,i} ( \mathbf{X}^n; \omega) \|_2^{(1+\lambda)/\lambda} \biggr)^{\lambda/(1+\lambda)} = \bigO(\varsigma_{n,\lambda}), \; \text{ as } n \to \infty.$$
		For the case $\lambda=0$, the  above condition is understood as 
		$$ \biggl \| \sup_{ \omega \in \mathcal{N}_{n,\theta_n}} \sup_{i \in I_n} 	\|  H_{n,i} ( \mathbf{X}^n; \omega) \|_2 \biggr \|_{L^{\infty}}= \bigO(\varsigma_{n,0}), \; \text{ as } n \to \infty.$$
			
		\item \label{item:E2} 
		For all $j \in [L]$, assume that there exists $\gamma_j \in [0, \infty]$ and  sequences of positive real numbers $\{\nu_{n,j} : n \geq 1\}$ and $\{\widetilde{\nu}_{n,j} : n \geq 1\}$ 
		such that the following holds true:
		$$  \biggl[ \mathbb{E}  \biggl( \sup_{\omega \in \mathcal{N}_{n, \theta_n}} R_{n,j}( \mathbf{X}^n, \mathbf{Z}^n; \omega)^{(1+\gamma_j)/\gamma_j} \biggr) \biggr]^{\gamma_j/(1+\gamma_j)} = \bigO(\nu_{n,j}),$$
		and 
		$$ [\mathbb{E} (\widetilde{R}_{n,j}(\mathbf{X}^n)^{1+\gamma_j} )]^{1/(1+\gamma_j)} = \bigO (\widetilde{\nu}_{n,j}),$$
		as $n \to \infty$. Here $\mathbf{Z}^n := (Z_{n,i})_{i \in I_n}$ is any  collection 
		of sub-Gamma vectors with variance proxy $1$ and scale parameter $1$ (i.e., $Z_{n,i} \in \mathscr{G}(1,1;d)$ for all $i \in I_n$) such that $\{(Y_{n,i},Z_{n,i})\}_{i \in I_n}$ is an i.i.d.~collection.
	
		\item \label{item:E3} Define the the following (random) pseudo-metric on $\mathcal{S}_n$:
		$$ d_{g,n}(\omega_1, \omega_2) := \biggl(\sum_{i \in I_n} \| (\nabla g(Y_{n,i}))^{\top}(H_{n,i}(\mathbf{X}^n;\omega_1)-H_{n,i}(\mathbf{X}^n;\omega_2) ) \|_2^2 \biggr)^{1/2},\; \forall \; \omega_1, \omega_2 \in \mathcal{S}_n.$$
		Assume the existence of a positive real number $\varepsilon_0$ and a sequence of positive real numbers $\{\tau_n : n \geq 1 \}$ such that for any $\varepsilon \in (0, \varepsilon_0)$ and $\delta>0$ there exists $\kappa_{\varepsilon, \delta} >0$ and $K_{\varepsilon, \delta} \in \mathbb{N}$ such that 
		$$ \limsup_{n \to \infty} \mathbb{P} [ P (\mathcal{N}_{n,\theta_n},d_n, \varepsilon) > K_{\varepsilon,\delta}P(\mathcal{N}_{n,\theta_n}, d_{g,n}, \tau_n \kappa_{\varepsilon,\delta})] \leq \delta.$$
		
	\end{enumerate}
		Under the above conditions, $\{P(\mathcal{N}_{n,\theta_n}, d_n, \varepsilon) : n \geq 1\}$ is a tight family of random variables for any $\varepsilon >0$, provided that $\{\theta_n : n \geq 1\}$ satisfies the following growth condition.
		\begin{equation}{\label{item:E4}}
		\theta_n = \bigO \biggl( \tau_n \min \biggl\{ 1,\;\; \dfrac{\tau_n}{\varsigma_{n,\lambda} k_n^{1/(1+\lambda)}},\;\;  \min_{j=1}^L \biggl( \dfrac{\tau_n}{\nu_{n,j}\widetilde{\nu}_{n,j}}\biggr)^{1/\upsilon_j} \biggr\}\biggr).
		\end{equation}
\end{thm}

\begin{remark}{\label{boundedH2}}
Similar to the case discussed in \Cref{boundedH}, if our problem satisfies the condition~\Cref{uniupperH}, we have the assumption~\Cref{item:E1} satisfied with $\lambda = 0$ and $\varsigma_{n,0} = \varsigma_n$, and thus the growth condition \eqref{item:E4} simplifies to the following:
	 \begin{equation}{\label{item:E4special}}
	 	\theta_n = \bigO \biggl( \tau_n \min \biggl\{1,\;\;  \dfrac{\tau_n}{k_n\varsigma_n},\;\;  \min_{j=1}^L \biggl( \dfrac{\tau_n}{\nu_{n,j}\widetilde{\nu}_{n,j}}\biggr)^{1/\upsilon_j} \biggr\}\biggr).
	 	\end{equation}
We observe that in this particular case, as expected, \Cref{genthm} and \Cref{genthm:g} give same growth bound on $\theta_n$.
\end{remark}

Similar to \Cref{genthm2}, we can also write a corollary from \Cref{genthm:g} for the particular case when the optimization problem satisfies the linearity assumption \ref{ass:linear}.

\begin{thm}[Linearized version of Theorem \ref{genthm:g}]{\label{genthm:g2}}
	Consider a sequence of random minimization problems $\{\mathscr{P}_n : n \geq 1\}$ with random inputs in $\mathcal{X}=\mathbb{R}^d$ and  satisfying Assumptions~\ref{unique}, \ref{ass:meas} and~\ref{ass:linear}. Suppose that the probability measure $\mathcal{P}$ (for the random inputs)   has a density $f$ with respect to Lebesgue measure on $\mathbb{R}^d$ satisfying Assumption~\ref{mu}  for the pair $(\rho,g)$. By virtue of this assumption, we can take the random inputs as $X_{n,i}=g(Y_{n,i})$, where $\{Y_{n,i} : i \in I_n\}$ is a collection of i.i.d.~random variables from the distribution $\mathcal{Q}_{\rho}$.  Further assume the following two conditions.
		\begin{enumerate} [leftmargin=*,label=(F\arabic*)]
			\item  \label{item:F1} Assume that for some $\lambda \in [0,\infty)$, there exists a sequence of positive real numbers $\{\varsigma_{n,\lambda} : n \geq 1\}$ satisfying 
				$$ \biggl( \sup_{\omega \in \mathcal{S}_{n}} \sum_{i \in I_n} \|  H_{n,i} (  \omega) \|_2^{(1+\lambda)/\lambda} \biggr)^{\lambda/(1+\lambda)} = \bigO(\varsigma_{n,\lambda}), \; \text{ as } n \to \infty.$$
			For the case $\lambda=0$, the  above condition is understood as 
			$$ \sup_{ \omega \in \mathcal{S}_{n}} \sup_{i \in I_n} 	\|  H_{n,i} ( \omega) \|_2 = \bigO(\varsigma_{n,0}), \; \text{ as } n \to \infty.$$
			\item  \label{item:F2}Consider the the following  pseudo-metric on $\mathcal{S}_n$: 
					$$ d_{g,n}(\omega_1, \omega_2) := \biggl(\sum_{i \in I_n}\| \nabla g(Y_{n,i})^{\top} ( H_{n,i}(\omega_1)-H_{n,i}(\omega_2)) \|_2^2 \biggr)^{1/2},\; \forall \; \omega_1, \omega_2 \in \mathcal{S}_n.$$
				Assume the existence of a positive real number $\varepsilon_0$ and a sequence of positive real numbers $\{\tau_n : n \geq 1  \}$ such that for any $\varepsilon \in (0, \varepsilon_0)$ and $\delta>0$ there exist $\kappa_{\varepsilon, \delta} >0$ and $K_{\varepsilon, \delta} \in \mathbb{N}$ such that 
						$$ \limsup_{n \to \infty} \mathbb{P} [ P (\mathcal{N}_{n,\theta_n},d_n, \varepsilon) > K_{\varepsilon,\delta}P(\mathcal{N}_{n,\theta_n}, d_{g,n}, \tau_n \kappa_{\varepsilon,\delta})] \leq \delta.$$
			\end{enumerate}
	Then  $\{P(\mathcal{N}_{n,\theta_n}, d_n, \varepsilon) : n \geq 1\}$ is a tight family of random variables for any $\varepsilon >0$, provided $\theta_n = \bigO ( \min \{\tau_n, \tau_n^{2}\varsigma_{n,\lambda}k_n^{-1/(1+\lambda)}\})$ as $n \to \infty$.
\end{thm}
Before diving into the proof of \Cref{genthm:g}, let us spend some time discussing which probability distributions satisfy \Cref{mu}. It is easy to see that an analogue of \Cref{f2} also holds in this situation, and hence we shall mainly focus on finding probability densities on the real line satisfying \Cref{mu}, i.e., $d=1$. Indeed, this will encompass many situations where \Cref{genthm:g} will be applied. 

\begin{ex}{\label{exmu}}
The following list describes some techniques to check whether any particular density on the real line satisfies \Cref{mu} or not, along with a non-exhaustive list of such standard distributions.
\begin{enumerate}[label=(\alph*)]
\item \label{item:a} If $\mathcal{P}$ satisfies \Cref{mu} for the pair $(\rho,g_1)$ and $g_2$ has uniformly bounded first and second order derivatives, then clearly $\mathcal{P} \circ g_2^{-1}$ satisfies \Cref{mu} for the pair $(\rho,g_2 \circ g_1)$. 
\item \label{item:b} The observation made in \Cref{item:a} ensures that any affine transformation of a distribution satisfying \Cref{mu} also satisfies \Cref{mu}. Similarly, convolutions of a finite number of distributions, where each of them satisfies \Cref{mu} with either  \Cref{ass:F2}, or \Cref{ass:F2p} or \Cref{ass:F2pp}, also satisfies \Cref{mu}. The same observation can be made about their products (if they are distributions on real line) provided all of them have bounded support.
\item \label{item:c} It is easy to check that $\rho(x):=C|x|^r$ for all $x \in \mathbb{R}$ and for some constant $C \in (0, \infty)$ satisfies \Cref{mu} if $r \geq 2$; in particular it satisfies \Cref{ass:F2p} instead of \Cref{ass:F2}. Since we can take $g \equiv 1$, all of these densities are valid candidates for $\mathcal{P}$. The case of $r=2$ is of particular interest since it yields the Gaussian distribution. 
\item \label{item:d} If $\mathcal{P}$ has  distribution function $F_{\mathcal{P}}$, then $F_{\mathcal{P}}^{\leftarrow} \circ F_{\rho}(X) \sim \mathcal{P}$ where $X \sim \mathcal{Q}_{\rho}$ and $F_{\rho}$ is its distribution function. Recall that $F_{\mathcal{P}}^{\leftarrow}$ denotes the left continuous inverse of $F_{\mathcal{P}}$. 
Thus if $F_{\mathcal{P}}^{\leftarrow} \circ F_{\rho}$ has continuous and bounded first and second order derivatives, we can conclude that $\mathcal{P}$ satisfies \Cref{mu}. This technique yields a class of distributions satisfying \Cref{mu} as demonstrated in \Cref{condrho} below. As an extra advantage, under the conditions of \Cref{condrho}, we can take $g$ to be strictly increasing and $\mathcal{Q}_{\rho} \circ (\rho^{(1)})^{-1}$ to have sub-Gaussian tail. Recall that the notation $g^{(k)}$ denotes the $k$-th order derivative for the function $g$. Some standard distributions  covered by \Cref{condrho} include the uniform distribution on any bounded interval, and Beta, Gamma, Exponential, Gaussian and Gumbel distributions, to name a few. 

\item  \label{item:e} Since Dirichlet distribution can be simulated by multiplying affine transformations of independent Beta distributions (for example using the ``string cutting" procedure), the discussion in \Cref{item:b} and \Cref{item:d} ensures that Dirichlet distributions are also included in the scope of \Cref{mu}.

\end{enumerate}
\end{ex}

\begin{lmm}{\label{condrho}}

Let the input distribution $\mathcal{P}$ on the real line have a continuous distribution function $F$ and density $f$. 
Suppose that for some $- \infty \leq a_- < a_+ \leq \infty$, the density $f$ is strictly positive and continuously differentiable on $(a_-,a_+)$ and is equal to zero outside this interval, satisfying the following two conditions.
\begin{enumerate}
\item At least one of the following three conditions is satisfied. 
\begin{enumerate}
\item \label{n1}$1/f(x) + |f^{(1)}(x)| = \bigO(1)$ as $x \downarrow a_-$ (in which case $a_-$ has to be finite).
\item \label{n2}$f(x)$ converges to $0$ or $\infty$ and $\log f(x)/\log F(x)$ converges to $\beta_- \in \mathbb{R}$ as $x \downarrow a_-$. Also $F(x)/f(x) = \bigO(1)$ as $x \downarrow a_-$. The last condition is trivially satisfied if $f(x)$ converges to $\infty$ as $x \downarrow - \infty$. 
\item \label{n3}There exists $\gamma_{-} >0$ satisfying $\log f(x) \sim -A_-|x|^{\gamma_{-}}$ as $x \downarrow a_- = - \infty$, where $A_-$ is some finite positive constant.
\end{enumerate}
\item At least one of the following three conditions is satisfied.
\begin{enumerate}
\item \label{p1}$1/f(x) + |f^{(1)}(x)| = \bigO(1)$ as $x \uparrow a_+$ (in which case $a_+$ has to be finite) .
\item \label{p2}$f(x)$ converges to $0$ or $\infty$ and $\log f(x)/\log (1-F(x))$ converges to $\beta_+ \in \mathbb{R}$ as $x \uparrow a_+$. Also $(1-F(x))/f(x) = \bigO(1)$ as $x \uparrow a_+$.
\item \label{p3}There exists $\gamma_{+} >0$ satisfying $\log f(x) \sim -A_+|x|^{\gamma_{+}}$ as $x \uparrow a_+= \infty$, where $A_+$ is some finite positive constant.
\end{enumerate}
\end{enumerate}
Under these conditions $\mathcal{P}$ satisfies Assumption~\ref{mu} for the pair $(\rho,g)$ with $p=1$ where $g$ is strictly increasing,  and
$\rho$ satisfies condition  \Cref{ass:F3} of Assumption~\ref{mu} with $c=0$. 
\end{lmm}


Let us now make a short comparative evaluation of the scopes and strengths of \Cref{genthm} and \Cref{genthm:g}. The conditions assumed on the input density $f$ in the statements of those two theorems, namely \Cref{ass:p} and \Cref{mu} respectively, are not exactly comparable; although it appears to be the case that \Cref{mu} is probably more restrictive. Nevertheless, both of them cover most standard densities in their scope. What we gain in \Cref{genthm:g} is less restrictions on the possible choices of $\theta_n$. To be precise, \Cref{genthm} requires $\theta_n = \bigO(\tau_n^{2+\lambda}k_n^{-1}\varsigma_{n,\lambda}^{-1-\lambda})$ whereas \Cref{genthm:g} requires $\theta_n = \bigO(\tau_n^2k_n^{-1/(1+\lambda)}\varsigma_{n,\lambda}^{-1})$. One can easily check that, for the case $\lambda=0$, the first condition is more restrictive provided that $(\tau_n/\varsigma_{n,\lambda})^{1+\lambda} = \bigO(k_n)$. This is usually the situation in most of the examples which we shall consider in this article. Particularly in the cases where only a fraction of the $H_{n,i}(\mathbf{x}^n;\omega)$'s are non-zero for fixed $\mathbf{x}^n$ and $\omega$, this phenomenon is more apparent since the total number of random inputs $k_n$ tends to be of a significantly higher order of magnitude compared to $\tau_n$; see \Cref{brwtight} for one such example. In order to gain this improvement in the choices of $\theta_n$ in \Cref{genthm:g}, we had to sacrifice a couple of things. Firstly, compared to \Cref{item:C2}, we now have to consider sub-Gamma vectors (in place of sub-Gaussian vectors) in \Cref{item:E2}. More importantly, in place of $d_{2,n}$ in \Cref{item:C3}, we have to tackle a more complicated random pseudo-metric $d_{g,n}$ in \Cref{item:E3} which turns out to be more challenging; we refer to \Cref{brwtight} again for one such example. This contrast is more stark when we compare the linearized versions of the two theorems, i.e., \Cref{genthm2} and \Cref{genthm:g2}, respectively. Whereas $d_{2,n}$ is reduced to a non-random pseudo-metric in \Cref{genthm2}, the pseudo-metric $d_{g,n}$ in \Cref{genthm:g2} is still random (if $g$ is not a linear map) and hence difficult to handle.

\begin{remark}{\label{choice2}}
Consider the case $d=p=1$. The condition \Cref{item:E3} is satisfied trivially when $g^{(1)}(Y) \geq \eta >0$ almost surely for some non-random $\eta$ where $Y \sim \mathcal{Q}_{\rho}$; since in this particular situation $d_{2,n} \leq \eta^{-1} d_{g,n}$, almost surely. One particular case is $g(x) =x$ for all $x \in \mathbb{R}$, which implies that if $\mathcal{P}=\mathcal{Q}_{\rho}$ for some $\rho$ satisfying conditions stated in \Cref{mu}, then we do not need to check \Cref{item:D1} at all. This particular case clearly involves Gaussian densities, the most common and well-studied choice of input distribution in most of the examples discussed later.
\end{remark}

\begin{proof}[Proof of Theorem~\ref{genthm:g}]
Recall that $g : \mathbb{R}^p \to \mathbb{R}^d$ is assumed to be twice continuously differentiable. We write $g=(g_1,\ldots,g_d)$ where $g_k$ is a real-valued function on $\mathbb{R}^p$ for each $k \in [d]$. Uniform boundedness of the derivatives of $g$ guarantee that 
$$ C_1 := \sup_{y \in \mathbb{R}^p} \|\nabla g (y)\|_{\mathrm{op}} < \infty, \; C_2 := \sup_{k \in [d]} \sup_{y \in \mathbb{R}^p} \|\nabla^2 g_k (y)\|_{\mathrm{op}} < \infty.$$
On the other hand, twice continuous differentiability ensures that for any $y_1,y_0 \in \mathbb{R}^p$, there exists $y^{(0)},y^{(1)}, \ldots,y^{(d)} \in \mathbb{R}^p$, lying on the line segment joining $y_0$ to $y_1$, such that 
\begin{align}{\label{MVT}}
g(y_1)-g(y_0) = \nabla g(y^{(0)}) (y_1-y_0),\; \text{ hence }  \|g(y_1)-g(y_0)\|_2 \leq C_1\|y_1-y_0\|_2,
\end{align} 
and 
\begin{align}{\label{MVT2}}
g(y_1)-g(y_0) &= ( g_k(y_1)-g_{k}(y_0))_{1 \leq k \leq d} \nonumber \\
&= \biggl( \nabla g_k(y_0) \cdot (y_1-y_0) + \dfrac{1}{2} (y_1-y_0)^{\top} \nabla^2g_k(y^{(k)})(y_1-y_0)\biggr)_{1 \leq k \leq d}.
\end{align}
With the preparatory works taken care of, we start with the i.i.d.~collection $\{Y_{0,n,i}=Y_{n,i} : i \in I_n\}$ from probability measure $\mathcal{Q}_{\rho}$ and a collection of independent standard $p$-dimensional Brownian motions on non-negative real line $\{W_{\cdot,n,i} : i \in I_n\}$; these two collection being mutually independent. As discussed earlier, we can construct the process $\{Y_{t,n,i} : t \geq 0\}$ satisfying 
$$  Y_{t,n,i} = Y_{0,n,i} - \int_{0}^t  \nabla \rho^{}(Y_{s,n,i}) \; ds + \sqrt{2}W_{t,n,i} \sim \mathcal{Q}_{\rho}, \; \forall \; t \geq 0, i \in I_n.$$ 
We define $X_{t,n,i} := g(Y_{t,n,i})$ for all $t,n,i$, noting that $X_{n,i}=X_{0,n,i}=g(Y_{0,n,i})=g(Y_{n,i})$. We shall use $\mathbf{X}_t^n$ and $\mathbf{Y}_t^n$ to denote the collections $\{ X_{t,n,i} : i \in I_n\}$ and  $\{ Y_{t,n,i} : i \in I_n\}$. Similar to the initial steps in the proof of \Cref{genthm}, we can write the following:
\begin{align}
\mathbb{E} \inf_{\omega \in \mathcal{S}_n} \psi_n( \mathbf{X}_0^n;\omega) &= \mathbb{E} \inf_{\omega \in \mathcal{S}_n} \psi_n( \mathbf{X}_t^n;\omega) \nonumber \\
& \leq \mathbb{E} \inf_{\omega \in \mathcal{N}_{n,\theta_n}} \psi_n( \mathbf{X}_t^n;\omega) \nonumber \\
& \leq \mathbb{E} \inf_{\omega \in \mathcal{N}_{n,\theta_n}} [\psi_n( \mathbf{X}_t^n;\omega) - \psi_n( \mathbf{X}_0^n;\omega)] + \mathbb{E} \inf_{\omega \in \mathcal{S}_n} \psi_n( \mathbf{X}_0^n;\omega) + \theta_n, \label{step1}
\end{align}
where \Cref{step1} follows from the definition of the (random) set $\mathcal{N}_{n,\theta_n}$. Define
$$ U_{t,n,i} :=  - \int_{0}^t \nabla \rho^{}(Y_{s,n,i}) \; ds, \; \forall \; t \geq 0, i \in I_n.$$
Note that
\begin{align}
\mathbb{E} \|U_{t,n,i}\|_1 &= \mathbb{E} \biggl \| \int_{0}^t  \nabla \rho(Y_{s,n,i}) \; ds \biggr \|_1 \notag \\
&\leq  \int_{0}^t \mathbb{E} \|\nabla \rho(Y_{s,n,i})\|_1 \; ds = t \mathbb{E} \|\nabla \rho(Y_{0,n,i})\|_1 = t\mu < \infty, \label{int}
\end{align}
where $\mu := \mathbb{E} \|\nabla \rho(Y_{n,i})\|_1$. Note that $\mu$ is finite by virtue of the assumption \Cref{ass:F2} in the theorem statement, which says that $\nabla \rho(Y_{n,i})$ is a sub-Gamma vector. Moreover, \Cref{lemsubg} yields that $U_{t,n,i}\in \mathscr{G}(32t^2\sigma^2+2^7t^2c^2, 8tc;p)$, since $\nabla \rho(Y_{s,n,i}) \in \mathscr{G}(\sigma^2,c;p).$

Similar to the proof of \Cref{genthm}, we now use \Cref{item:B1} to perform the following expansion in order to control the first term in the right hand side of \Cref{step1}.
\begin{align}
&\psi_n( \mathbf{X}_t^n;\omega) - \psi_n( \mathbf{X}_0^n;\omega) \nonumber \\
& \leq \sum_{i \in I_n} ( X_{t,n,i}-X_{0,n,i}) \cdot H_{n,i}( \mathbf{X}_0^n;\omega) + \sum_{j=1}^L R_{n,j}( \mathbf{X}_0^n,\mathbf{X}_t^n - \mathbf{X}_0^n; \omega)\widetilde{R}_{n,j}( \mathbf{X}_t^n). \label{expan2}
\end{align}
Applying \Cref{MVT} and \Cref{MVT2}, we obtain the following expansion for any $i \in I_n$. 
\begin{align}
X_{t,n,i}-X_{0,n,i} &=g(Y_{t,n,i})-g(Y_{0,n,i}) \nonumber \\
& = g(Y_{0,n,i}+\sqrt{2}W_{t,n,i})-g(Y_{0,n,i}) + g(Y_{t,n,i})-g(Y_{0,n,i}+\sqrt{2}W_{t,n,i}) \nonumber \\
& =  \sqrt{2} \nabla g( Y_{0,n,i}) W_{t,n,i} + ( W_{t,n,i}^{\top}\nabla^2 g_k(Y_{t,n,i}^{(k)} )W_{t,n,i} )_{k \in [d]} + \nabla g( Y_{t,n,i}^{(0)}) U_{t,n,i} \nonumber \\
& =  \sqrt{2} \nabla g( Y_{0,n,i}) W_{t,n,i} + V_{t,n,i} + \widetilde{V}_{t,n,i}. \label{expan}
\end{align}
where $Y_{t,n,i}^{(0)}, \ldots, Y_{t,n,i}^{(d)}$ are (random) points on the line joining $Y_{0,n,i}$ to $Y_{t,n,i}$ and 
$$V_{t,n,i} :=  (  W_{t,n,i}^{\top}\nabla^2 g_k(Y_{t,n,i}^{(k)} )W_{t,n,i} )_{1 \leq k \leq d}, \; \widetilde{V}_{t,n,i} := \nabla g( Y_{t,n,i}^{(0)}) U_{t,n,i}.$$ 
 Since $Y_{0,n,i}$ and $W_{t,n,i}$ are independent of each other, the first term in \Cref{expan} has mean zero. Uniform boundedness of $\nabla g$ and $\nabla^2 g$ also guarantee integrability of the random vectors $V_{t,n,i}$ and $\widetilde{V}_{t,n,i}$. Therefore, $ V_{t,n,i} + \widetilde{V}_{t,n,i}$ also has zero mean; this follows from the fact that $X_{t,n,i}$ and $X_{0,n,i}$ are identically distributed. We also have, $\|\nabla g(Y_{t,n,i}^{(0)})\|_{\mathrm{op}} \leq C_1$ and hence each entry of $\nabla g(Y_{t,n,i}^{(0)})$ is bounded by $C_1$. \Cref{lemsubg} and \Cref{int} now guarantee that $\widetilde{V}_{t,n,i} \in \mathscr{G}(C_3t^2(\sigma^2+c^2+\mu^2), C_3tc;d)$ where $C_3$ is some finite constant (depending on $p$ and $C_1$). Similarly, we can argue that $V_{t,n,i} \in \mathscr{G}(C_4t^2,C_4t;d)$  for some finite constant $C_4$ (depending on $p$ and $C_2$). Combining these two observations with the aid of \Cref{lemsubg}, we can come to the conclusion that $V_{t,n,i}+\widetilde{V}_{t,n,i}$ is a sub-Gamma vector with mean zero, variance proxy $C_5t^2$ and scale parameter $C_5t$, for some finite constant $C_5$ depending on $p,\sigma^2,c,\mu,C_1$ and $C_2$. 
 
Let us now introduce the notations 
 $$ \widetilde{\mathbf{W}}_t^n = (\widetilde{W}_{t,n,i})_{i \in I_n} := ( \sqrt{2}\nabla g(Y_{0,n,i})W_{t,n,i})_{i \in I_n}, \; \mathbf{V}_t^n := (V_{t,n,i})_{i \in I_n}, \; \; \widetilde{\mathbf{V}}_t^n := (\widetilde{V}_{t,n,i})_{i \in I_n}.$$
 Plugging  the expansion \Cref{expan} in the inequality in \Cref{expan2}, we obtain the following.
 \begin{align}
& \mathbb{E} \inf_{\omega \in \mathcal{N}_{n,\theta_n}} [\psi_n( \mathbf{X}_t^n;\omega) - \psi_n( \mathbf{X}_0^n;\omega)] \nonumber \\
 & \leq \mathbb{E} \inf_{\omega \in \mathcal{N}_{n,\theta_n}} \sum_{i \in I_n} \widetilde{W}_{t,n,i} \cdot H_{n,i}( \mathbf{X}^n_0;\omega ) + \mathbb{E} \sup_{\omega \in \mathcal{N}_{n,\theta_n}} \sum_{i \in I_n} ( V_{t,n,i}+\widetilde{V}_{t,n,i}) \cdot H_{n,i}( \mathbf{X}^n_0;\omega) \nonumber \\
 & \hspace{1.5 in} + \sum_{j=1}^L \mathbb{E}\sup_{\omega \in \mathcal{N}_{n,\theta_n}} R_{n,j}( \mathbf{X}_0^n,\widetilde{\mathbf{W}}_t^n+\mathbf{V}_t^n+\widetilde{\mathbf{V}}_t^n; \omega )\widetilde{R}_{n,j}( \mathbf{X}_t^n). \label{expan3}
 \end{align}
 The rest of the proof is now very similar to the proof of \Cref{genthm} and hence we shall skip some details to avoid unnecessary repetitions. To control the first term in the expression \Cref{expan3}, we note that conditional on $\mathbf{Y}_0^n$, the process 
 \[
 \biggl\{ \sum_{i \in I_n} \widetilde{W}_{t,n,i} \cdot H_{n,i}(\mathbf{X}^n_0;\omega) : \omega \in \mathcal{S}_n\biggr\}
 \]
  is Gaussian on the pseudo-metric space $(\mathcal{S}_n,\sqrt{2t}d_{g,n})$. 
 Applying \Cref{sudakov} (Sudakov minoration), we arrive at the following for any $\varepsilon^{\prime}>0$:
 \begin{align}
 \mathbb{E} \inf_{\omega \in \mathcal{N}_{n,\theta_n}} \sum_{i \in I_n} \widetilde{W}_{t,n,i} \cdot H_{n,i}( \mathbf{X}^n_0;\omega ) &\leq -A^*\sqrt{2t}\varepsilon^{\prime} \mathbb{E} \sqrt{\log P ( \mathcal{N}_{n,\theta_n},\sqrt{2t}d_{g,n}, \sqrt{2t}\varepsilon^{\prime})} \nonumber \\
 & = -A^*\sqrt{2t}\varepsilon^{\prime} \mathbb{E} \sqrt{\log P ( \mathcal{N}_{n,\theta_n},d_{g,n},\varepsilon^{\prime})}. \label{thm2:sud}
 \end{align}
To control the second term in the expression \Cref{expan3}, recall that $t^{-1}(V_{t,n,i}+\widetilde{V}_{t,n,i})$ is a mean-zero sub-Gamma vector with parameters $(C_5,C_5)$. Thus,
\begin{align}
&\mathbb{E} \sup_{\omega \in \mathcal{N}_{n,\theta_n}} \sum_{i \in I_n} ( V_{t,n,i}+\widetilde{V}_{t,n,i}) \cdot H_{n,i}( \mathbf{X}^n_0;\omega) \nonumber \\
 & \leq \mathbb{E} \biggl[ \biggl( \sum_{i \in I_n} \|  V_{t,n,i}+\widetilde{V}_{t,n,i} \|_2^{1+\lambda} \biggr)^{1/(1+\lambda)} \biggl( \sup_{ \omega \in \mathcal{N}_{n, \theta_n}} \sum_{i \in I_n} \| H_{n,i} (\mathbf{X}_0^n;\omega) \|_2^{(1+\lambda)/\lambda} \biggr)^{\lambda/(1+\lambda)}\biggr] \nonumber \\
 & \leq   \biggl( \mathbb{E} \sum_{i \in I_n} \|  V_{t,n,i}+\widetilde{V}_{t,n,i} \|_2^{1+\lambda} \biggr)^{1/(1+\lambda)} \biggl( \mathbb{E} \sup_{ \omega \in \mathcal{N}_{n, \theta_n}} \sum_{i \in I_n} \| H_{n,i} (\mathbf{X}_0^n;\omega) \|_2^{(1+\lambda)/\lambda} \biggr)^{\lambda/(1+\lambda)} \nonumber \\
 & \leq (C_6k_n t^{1+\lambda})^{1/(1+\lambda)} \varsigma_{n,\lambda} = C_7t k_n^{1/(1+\lambda)} \varsigma_{n,\lambda}, \label{thm2:sec}
\end{align} 
where $C_6$ and $C_7$ are finite constants depending on $C_5, \lambda$ and $d$. We have applied Lemmas~\ref{lemsubg:4} and \ref{item:E1} to obtain the first expression in \Cref{thm2:sec}. The argument also works for the case $\lambda =0$, similar to what we did in the proof of \Cref{genthm}.

Recall that $\widetilde{W}_{t,n,i} = \sqrt{2}\nabla g(Y_{0,n,i})W_{t,n,i}$. Conditioned on $\mathbf{Y}_0^n$, this is a Gaussian random vector with mean $0$ and covariance matrix $2t(\nabla g(Y_{0,n,i})(\nabla g(Y_{0,n,i}))^{\top}$. 
The uniform bound on the operator norm of $\nabla g$ guarantees that $\|\nabla g(Y_{0,n,i})(\nabla g(Y_{0,n,i}))^{\top} \|_{\mathrm{op}} \leq C_1^2$, and consequently $2t(\nabla g(Y_{0,n,i})(\nabla g(Y_{0,n,i}))^T \preceq 2tC_1^2I_d$. 
This shows that the conditional distribution of $\widetilde{W}_{t,n,i}$, and hence also the unconditional distribution, is sub-Gaussian with variance proxy $2tC_1^2$. Therefore, we can apply \Cref{lemsubg:2} to conclude that $\widetilde{W}_{t,n,i} + V_{t,n,i}+\widetilde{V}_{t,n,i}$  is a mean-zero sub-Gamma vector with variance proxy $(C_8t+C_8t^2)$ and scale parameter $C_8t$, for some finite constant $C_8 \geq 1$ not depending on $t$.  In particular, for any $t \in (0,1]$, the random vector $(2C_8)^{-1}t^{-1/2}(\widetilde{W}_{t,n,i} + V_{t,n,i}+\widetilde{V}_{t,n,i} )$ is mean-zero and belongs to the class $\mathscr{G}(1,1;d)$. Applying \Cref{item:B2} 
and \Cref{item:E2}, we conclude that for all $t \in (0,1]$,
\begin{align}
&\mathbb{E} \sup_{ \omega \in  \mathcal{N}_{n,\theta_n}} R_{n,j}( \mathbf{X}_0^n,\widetilde{\mathbf{W}}_t^n+ \mathbf{V}_t^n+\widetilde{\mathbf{V}}_t^n; \omega) \widetilde{R}_{n,j}(\mathbf{X}_t^n) \nonumber \\
& \leq (2C_8\sqrt{t})^{(1+\upsilon_j)} \mathbb{E} \sup_{ \omega \in  \mathcal{N}_{n,\theta_n}} R_{n,j}\biggl( \mathbf{X}_0^n,\dfrac{\widetilde{\mathbf{W}}_t^n+ \mathbf{V}_t^n+\widetilde{\mathbf{V}}_t^n}{2C_8\sqrt{t}}; \omega\biggr) \widetilde{R}_{n,j}(\mathbf{X}_t^n) \nonumber \\
& \leq  (2C_8\sqrt{t})^{(1+\upsilon_j)} \biggl[ \mathbb{E} \sup_{ \omega \in  \mathcal{N}_{n,\theta_n}} R_{n,j}\biggl( \mathbf{X}_0^n,\dfrac{\widetilde{\mathbf{W}}_t^n+ \mathbf{V}_t^n+\widetilde{\mathbf{V}}_t^n}{2C_8\sqrt{t}}; \omega\biggr)^{\tfrac{1+\gamma_j}{\gamma_j}} \biggr]^{\tfrac{\gamma_j}{1+\gamma_j}}  [ \mathbb{E} \widetilde{R}_{n,j}(\mathbf{X}_t^n)^{1+\gamma_j} ]^{\tfrac{1}{1+\gamma_j}} \nonumber \\
& =  (2C_8\sqrt{t})^{(1+\upsilon_j)} \biggl[ \mathbb{E} \sup_{ \omega \in  \mathcal{N}_{n,\theta_n}} R_{n,j}\biggl( \mathbf{X}_0^n,\dfrac{\widetilde{\mathbf{W}}_t^n+ \mathbf{V}_t^n+\widetilde{\mathbf{V}}_t^n}{2C_8\sqrt{t}}; \omega\biggr)^{\tfrac{1+\gamma_j}{\gamma_j}} \biggr]^{\tfrac{\gamma_j}{1+\gamma_j}}  [ \mathbb{E} \widetilde{R}_{n,j}(\mathbf{X}_0^n)^{1+\gamma_j} ]^{\tfrac{1}{1+\gamma_j}} \nonumber \\
& \leq C_{9}t^{(1+\upsilon_j)/2} \nu_{n,j}\widetilde{\nu}_{n,j},  \label{thm2:third}
\end{align}
for some finite constant $C_{9}$, independent of $t,j$ and $n$. 
We now combine \Cref{expan3} with \Cref{thm2:sud}, \Cref{thm2:sec} and \Cref{thm2:third} to arrive at the following for any $t \in (0,1]$ and $\varepsilon^{\prime} >0$.  
\begin{align}{\label{genthm2:finalstep1}}
A^*\sqrt{2t}\varepsilon^{\prime} \mathbb{E} \sqrt{\log P ( \mathcal{N}_{n,\theta_n},d_{g,n},\varepsilon^{\prime})} \leq \theta_n + C_7t k_n^{1/(1+\lambda)} \varsigma_{n,\lambda} + C_{9} \sum_{j=1}^L t^{(1+\upsilon_j)/2} \nu_{n,j}\widetilde{\nu}_{n,j}.
\end{align} 
Recall that we have assumed in the statement of the theorem that $\theta_n =\bigO(\tau_n)$ and hence $\theta_n \leq C_{10}\tau_n$ for some finite constant $C_{10}$. Fix $\varepsilon \in (0, \varepsilon_0)$ and $\delta >0$. Get $\kappa_{\varepsilon,\delta}$ using \Cref{item:E3} and plug in the following choices in \Cref{genthm2:finalstep1} : $\varepsilon^{\prime} = \tau_n \kappa_{\varepsilon,\delta}, t= \theta_n^2/(C_{10}^2\tau_n^2) \in (0,1]$. A few algebraic simplifications lead to the following upper bound.
\begin{align}{\label{last:proof}}
 \kappa_{\varepsilon,\delta} \mathbb{E} \sqrt{\log P ( \mathcal{N}_{n,\theta_n},d_{g,n},\tau_n \kappa_{\varepsilon,\delta})} \leq C_{11}\biggl( 1 + \dfrac{\theta_n k_n^{1/(1+\lambda)} \varsigma_{n,\lambda}}{\tau_n^2} + \sum_{j=1}^L \dfrac{\theta_n^{\upsilon_j}\nu_{n,j}\widetilde{\nu}_{n,j}}{\tau_n^{1+\upsilon_j}} \biggr) = \bigO(1),
\end{align}
 where the last equality follows from the growth conditions assumed on $\theta_n$ in \Cref{item:E4}. The above expression shows that $\{P(\mathcal{N}_{n, \theta_n}, d_{g,n}, \tau_n\kappa_{\varepsilon, \delta}) : n \geq 1\}$ is a tight sequence of random variables for any $\varepsilon \in (0, \varepsilon_0)$ and $ \delta >0$. Using \Cref{item:E3}, we now complete the proof.
\end{proof}

With a closer look at the proof of \Cref{genthm:g}, one can observe that we have not used the assumption of sub-Gamma-ness in its full strength. With a little more effort, \Cref{genthm:g} can be proved assuming only appropriate bounds on moments of all orders for the random vector $\nabla \rho(Y)$ when $Y \sim \mathcal{Q}_{\rho}$. Nevertheless, if we consider linear optimization problems, in particular those satisfying \Cref{ass:linear}, we can improve on \Cref{genthm:g2} using the full strength of sub-Gamma assumption.


\begin{thm}[Improved linearized version of \Cref{genthm:g}]{\label{genthm:g3}}
	Consider a sequence of random minimization problems $\{\mathscr{P}_n : n \geq 1\}$ with random inputs in $\mathcal{X}=\mathbb{R}^d$ and  satisfying Assumptions~\ref{unique},~\ref{ass:meas} and~\ref{ass:linear}. Suppose that the probability measure $\mathcal{P}$ (for the random inputs)   has a density $f$ with respect to Lebesgue measure on $\mathbb{R}^d$ satisfying Assumption~\ref{mu}  for the pair $(\rho,g)$. By virtue of this assumption, we can take the random inputs as $X_{n,i}=g(Y_{n,i})$, where $\{Y_{n,i} : i \in I_n\}$ is a collection of i.i.d.~random variables from the distribution $\mathcal{Q}_{\rho}$.  Further assume the following two conditions.
		\begin{enumerate} [leftmargin=*,label=(G\arabic*)]
			\item  \label{item:G1} Assume that there exists a sequence of positive real numbers $\{M_n : n \geq 1\}$ satisfying the following. If the input distribution $\mathcal{P}$ has zero mean and is sub-Gamma with variance proxy $1 $ and scale parameter $1$, then 
			$$ \infty > \mathbb{E} \biggl[ \inf_{\omega \in \mathcal{S}_n} \psi_n(\mathbf{X}^n; \omega) \biggr] = \mathbb{E} \biggl[ \inf_{\omega \in \mathcal{S}_n} \sum_{i \in I_n} X_{n,i}\cdot H_{n,i}(\omega) \biggr] \geq  - M_n, \; \forall \; n \geq 1.$$ 
			
			\item  \label{item:G2} Assume that the condition \Cref{item:F2} of Theorem \ref{genthm:g2} holds. Recall that $\left\{\tau_n : n \geq 1\right\}$ is the normalizing sequence such that $\tau_n^{-1}d_{g,n}$ is greater than (upto some constant) the metric $d_n$ with high probability, in the sense as described in \ref{item:F2}.
			\end{enumerate}
	Then  $\{P(\mathcal{N}_{n,\theta_n}, d_n, \varepsilon) : n \geq 1\}$ is a tight family of random variables for any $\varepsilon >0$, provided that $\theta_n = \bigO (\tau_n^2/M_n)$ as $n \to \infty$. 
\end{thm}

\begin{proof}
We shall complete the proof by altering certain estimates used in the proof of \Cref{genthm:g} in light of the assumption \Cref{item:G1}. Note that \Cref{step1} remains unaltered, whereas \Cref{expan3} simplifies to 
\begin{align}
& \mathbb{E} \inf_{\omega \in \mathcal{N}_{n,\theta_n}} [\psi_n( \mathbf{X}_t^n;\omega) - \psi_n( \mathbf{X}_0^n;\omega)] \nonumber \\
 & \leq \mathbb{E} \inf_{\omega \in \mathcal{N}_{n,\theta_n}} \sum_{i \in I_n} \widetilde{W}_{t,n,i} \cdot H_{n,i}(\omega) + \mathbb{E} \sup_{\omega \in \mathcal{N}_{n,\theta_n}} \sum_{i \in I_n} ( V_{t,n,i}+\widetilde{V}_{t,n,i}) \cdot H_{n,i}(\omega). \label{expan4}
 \end{align}
 Whereas the upper bound of the first term in the right hand side of \Cref{expan4}, as expressed in \Cref{thm2:sud}, doesn't need to be changed, we can improve on the upper bound of the second term in the right hand side of \Cref{expan4} compared to what we did in \Cref{thm2:sec}, with the aid of \Cref{item:G1}. Recall that $V_{t,n,i}+\widetilde{V}_{t,n,i}$ was a sub-Gamma vector with mean zero, variance proxy $C_5t^2$ and scale parameter $C_5t$. Without loss of generality, we can assume that $C_5 \geq 1$ and hence $-(V_{t,n,i}+\widetilde{V}_{t,n,i})/(C_5t)$ are i.i.d.~sub-Gamma vectors with mean zero, variance proxy $1$ and scale parameter $1$. Therefore, 
 \begin{align}
 \mathbb{E} \sup_{\omega \in \mathcal{N}_{n,\theta_n}} \sum_{i \in I_n} ( V_{t,n,i}+\widetilde{V}_{t,n,i}) \cdot H_{n,i}(\omega) &\leq  \mathbb{E} \sup_{\omega \in \mathcal{S}_n} \sum_{i \in I_n} ( V_{t,n,i}+\widetilde{V}_{t,n,i}) \cdot H_{n,i}(\omega) \nonumber \\
 & = - C_5t \mathbb{E} \inf_{\omega \in \mathcal{S}_n} \sum_{i \in I_n} \biggl( \frac{-V_{t,n,i}-\widetilde{V}_{t,n,i}}{C_5t}\biggr) \cdot H_{n,i}(\omega) \nonumber \\
 & \leq C_5tM_n. \label{step1:genthm:g3}
 \end{align}
 Combining the estimate in \Cref{step1:genthm:g3} with \Cref{thm2:sud}, \Cref{expan4} and \Cref{step1}, we obtain the following revised version of \Cref{genthm2:finalstep1}. For any $\varepsilon^{\prime} >0$, 
 \begin{align}{\label{genthmg3:finalstep1}}
 A^*\sqrt{2t}\varepsilon^{\prime} \mathbb{E} \sqrt{\log P ( \mathcal{N}_{n,\theta_n},d_{g,n},\varepsilon^{\prime})} \leq \theta_n + C_5tM_n.
 \end{align}
  Fix $\varepsilon \in (0, \varepsilon_0)$ and $\delta >0$. Get $\kappa_{\varepsilon,\delta}$ using \Cref{item:G2} and plug-in the following choices in \Cref{genthmg3:finalstep1} : $\varepsilon^{\prime} = \tau_n \kappa_{\varepsilon,\delta}, t= \theta_n^2/\tau_n^2$. Algebraic simplifications now lead to the following upper bound.
 \begin{align*}
  \kappa_{\varepsilon,\delta} \mathbb{E} \sqrt{\log P ( \mathcal{N}_{n,\theta_n},d_{g,n},\tau_n \kappa_{\varepsilon,\delta})} \leq C_{12}\biggl( 1 + \dfrac{\theta_n M_n} {\tau_n^2} \biggr) = \bigO(1),
 \end{align*}
  where the last equality follows from the growth conditions assumed on $\theta_n$. The above expression shows that $\{P(\mathcal{N}_{n, \theta_n}, d_{g,n}, \tau_n\kappa_{\varepsilon, \delta}) : n \geq 1\}$ is a tight sequence of random variables for any $\varepsilon \in (0, \varepsilon_0)$ and $ \delta >0$. Using \Cref{item:G2}, we now complete the proof.
   \end{proof}

A key step in applications of \Cref{genthm:g3} for optimization problems with linear objective function, i.e.,  satisfying \Cref{ass:linear}, is to verify \Cref{item:G1}. This type of growth condition is well-known in the literature for many well-studied problem like branching random walk (discussed in \Cref{brw}). Nevertheless,  with the help of the well-developed theory of sub-Gamma variables (see \Cref{subgamma}),  we can actually get a crude lower bound for $\mathbb{E} \inf_{\omega \in \mathcal{S}_n} \psi_n(\mathbf{X}^n; \omega)$, at least in situations with finite $\mathcal{S}_n$, which turns out to be optimal for most of the examples discussed in this article.

\begin{prop}{\label{subgmaxgrowth}}
Consider a random optimization problem satisfying Assumption~\ref{ass:linear}. Assume that 
$$ \sup_{\omega \in \mathcal{S}_n} \sum_{i \in I_n}\|H_{n,i}(\omega)\|_2^2 =: \zeta_{n,2}^2 < \infty \; \text{ and } \sup_{\omega \in \mathcal{S}_n} \sup_{i \in I_n}\|H_{n,i}(\omega)\|_{\infty} =: \zeta_{n,\infty} < \infty .$$
Then 
$$ \infty> \mathbb{E} \inf_{\omega \in \mathcal{S}_n} \psi_n(\mathbf{X}^n; \omega) \geq - \sqrt{2d\sigma^2 \zeta_{n,2}^2 \log \operatorname{card}(\mathcal{S}_n)} - dc\zeta_{n,\infty} \log \operatorname{card}(\mathcal{S}_n), \; \forall \; n \geq 1,$$
provided that the $X_{n,i}$'s are sub-Gamma random vectors with zero mean, variance proxy $\sigma^2$ and scale parameter $c$. In other words, we can take the absolute value of the right hand side above as a choice for $M_n$ in Theorem~\ref{genthm:g3} after plugging in $\sigma^2=c=1$.  
\end{prop}

\begin{proof}
Since the $X_{n,i}$'s are integrable, so is $\psi_n(\mathbf{X}^n, \omega)$ for any $\omega$; thus the infimum has expectation less than infinity. On the other hand,  since $X_{n,i} \in \mathscr{G}(\sigma^2,c;d)$ we can apply Lemma~\Cref{lemsubg:2} to conclude that  for all $i \in I_n$ and $\omega \in \mathcal{S}_n$,
$$ X_{n,i} \cdot  H_{n,i}(\omega) \in \mathscr{G} ( d\sigma^2\|H_{n,i}(\omega)\|_2^2, dc\|H_{n,i}(\omega)\|_{\infty}).$$
Since $X_{n,i}$'s are independent, a further application of Lemma~\Cref{lemsubg:2} guarantees that 
\begin{align*}
\psi_n( \mathbf{X}^n; \omega) = \sum_{i \in I_n} X_{n,i} \cdot H_{n,i}(\omega) &\in \mathscr{G} \biggl( d\sigma^2\sum_{i \in I_n}\|H_{n,i}(\omega)\|_2^2, dc \max_{i \in I_n} \|H_{n,i}(\omega)\|_{\infty}\biggr) \\
& \subseteq \mathscr{G} (d\sigma^2 \zeta_{n,2}^2, dc  \zeta_{n,\infty}).
\end{align*}
The $X_{n,i}$'s having zero mean implies the same for $\psi_n(\mathbf{X}^n, \omega)$ for all $\omega$; hence applying Proposition \Cref{subg:maxc}, we establish the required assertion. 
\end{proof}

\begin{remark}
An analogous computation can be performed even when $\mathcal{S}_n$ is not finite. Note that for $\omega_1, \omega_2 \in \mathcal{S}_n$, Lemma~\Cref{lemsubg:2} implies the following:
\begin{align*}
&\psi_n(\mathbf{X}^n;\omega_1) - \psi_n(\mathbf{X}^n;\omega_2) = \sum_{i \in I_n} X_{n,i} \cdot (H_{n,i}(\omega_1)-H_{n,i}(\omega_2) ) \\
& \in \mathscr{G} \biggl(  d\sigma^2\sum_{i \in I_n}\|H_{n,i}(\omega_1)-H_{n,i}(\omega_2)\|_2^2, dc \max_{i \in I_n} \|H_{n,i}(\omega_1)-H_{n,i}(\omega_2)\|_{\infty}\biggr) \\
& \subseteq \mathscr{G} ( d\sigma^2 d_{2,n}(\omega_1,\omega_2)^2 , dcd_{2,n}(\omega_1,\omega_2)).
\end{align*}
Thus, $\{\psi_n(\mathbf{X}^n;\omega) : \omega \in \mathcal{S}_n\}$ is a sub-Gamma process with parameters $(d\sigma^2,dc)$ on the pseudo-metric space $(\mathcal{S}_n,d_{2,n})$, in the sense of \Cref{subgamma:process}. We can therefore apply \Cref{dudley} to conclude that 
$$ \mathbb{E} \inf_{\omega \in \mathcal{S}_n} \psi_n(\mathbf{X}^n;\omega) \geq - 12 \int_{0}^{\infty} \biggl[ \sqrt{d}\sigma \sqrt{\log N(\mathcal{S}_n, d_{2,n},\varepsilon)} + dc \log N(\mathcal{S}_n, d_{2,n},\varepsilon) \biggr] \, d \varepsilon, $$
and hence we can take 
$$ M_n = 72d \int_{0}^{\infty}   \log N(\mathcal{S}_n, d_{2,n},\varepsilon) \,d \varepsilon,$$
while applying \Cref{genthm:g3}. We will not apply this bound in this article, since it results in a sub-optimal bound in the examples discussed here. Nevertheless, it can turn out to be useful in other examples. 
\end{remark}

Let us make a short comment on the improvement made by \Cref{genthm:g3} compared to the linearized version of the general result, i.e., \Cref{genthm:g2}. We restrict ourselves to the set-up described in \Cref{boundedH2}, i.e., 
$$ \sup_{ \omega \in  \mathcal{S}_n}  \sup_{i \in I_n}  \|H_{n,i}(  \omega ) \|_2 \leq \varsigma_n < \infty, \; \forall \; n \geq 1,$$
 and start by making the following basic observation:
\begin{align*}
\mathbb{E} [-\inf_{\omega \in \mathcal{S}_n} \psi_n(\mathbf{X}^n; \omega) ] &= \mathbb{E} \biggl[-\inf_{\omega \in \mathcal{S}_n} \sum_{i \in I_n} X_{n,i} \cdot H_{n,i}(\omega) \biggr]\\
&\geq  \mathbb{E} \biggl[- \sum_{i \in I_n} \|X_{n,i}\|_2 \varsigma_n  \biggr] = -k_n \varsigma_n \mathbb{E}_{\mathcal{P}} \|X\|_2.
\end{align*} 
Hence, we can always take $M_n = \bigO(k_n \varsigma_n)$ while applying \Cref{genthm:g3}. With this particular choice applied, we get back \Cref{genthm:g2} as can be observed with the help of \Cref{boundedH2}. In many cases, $\bigO(k_n\varsigma_n)$ will prove to be a cruder bound on $M_n$ (since it does not take into account sub-Gamma tails of input variables) and hence will yield sub-optimal choice for $\theta_n$; this, in particular, will be the situation for the examples discussed in \Cref{brw,skmodel,eigmatrix}. 

\begin{remark}{\label{future}}
We have exploited a smoothness condition on the objective function with respect to its random inputs in \Cref{ass:psi}. \Cref{ass:linear} being a special case of \Cref{ass:p}, all of our main theorems require the use of that smoothness criterion. This excludes a large class of random optimization problems where the function $\psi_n$ takes values in a discrete set. An interesting example of this kind is the problem of finding the largest independent set in a Erdos-Renyi random graph on a lattice. Our methods do not cover these types of situations; proving stability properties for this kind of problem could be an interesting avenue for further research. 
\end{remark}

\subsection{General results to prove condition~\ref{item:B} of Theorem \ref{strat}}

We now want to spend some time talking about techniques to prove \Cref{item:B}. As we have mentioned earlier, the technique to prove a statement like \Cref{item:B} varies from case to case, depending on the choice of perturbation blocks. Nevertheless, for linear objective functions we can write down a somewhat general theorem which encompasses a lot of examples that we are going to discuss later. The choice for the perturbation blocks we make here amounts to replacing a uniformly bounded number of  random inputs by i.i.d.~copies. In other words, we have $\mathcal{J}_n = \{J_{n,l} : l \in \mathcal{L}_n\}$ satisfying $\operatorname{card}(J_{n,l}) \leq J$ for all $l \in \mathcal{L}_n$, for some $J < \infty$. The scenario where we replace one input at a time, i.e. $J_{n,i}=\{i\}$ for all $i \in I_n = \mathcal{L}_n$, falls under its scope. 

\begin{prop}{\label{mostprop}}
Consider a sequence of random minimization problems $\{\mathscr{P}_n : n \geq 1\}$ with random inputs in $\mathcal{X}=\mathbb{R}^d$ with distribution $\mathcal{P}$, satisfying Assumptions~\ref{ass:linear} and~\ref{unique}. Moreover, assume that 
\begin{enumerate}
\item $  \sup_{i \in I_n} \sup_{\omega \in \mathcal{S}_n} \|H_{n,i}(\omega)\|_2  := \varsigma_n < \infty.$
\item $\{P(\mathcal{N}_{n,C\varsigma_n},d_n,\varepsilon) : n \geq 1 \}$ is a tight sequence for any $\varepsilon ,C>0$. 
\end{enumerate}
Then $\{\mathscr{P}_n : n \geq 1\}$ is stable under small perturbations with perturbation blocks $\mathcal{J}_n = \{J_{n,l} : l \in \mathcal{L}_n\}$ satisfying $\sup_{n \geq 1} \sup_{l \in \mathcal{L}_n} \operatorname{card}(J_{n,l}) =  J < \infty$, provided  that $\mathbb{E}_{\mathcal{P}}\|X\|_2 < \infty$.
\end{prop}

\begin{proof}
We shall establish that \Cref{strat} is applicable for the choice $\theta_n \equiv 1$. It will suffice to take the obvious choice $\omega^*_{n,l} = \widehat{\omega}_n(\mathbf{X}^n_l)$. Then the first assertion in \Cref{item:B} is trivial, and observe that for any $l \in \mathcal{L}_n$, 
\begin{align}
\mathbb{E}[\psi_n( \mathbf{X}^n; \widehat{\omega}_n(\mathbf{X}^n_l)) - \psi_{n, \mathrm{opt}}( \mathbf{X}^n)] &= \mathbb{E}[\psi_n( \mathbf{X}^n; \widehat{\omega}_n(\mathbf{X}^n_l)) - \psi_n( \mathbf{X}^n; \widehat{\omega}_n(\mathbf{X}^n))] \nonumber \\
& = \mathbb{E}[\psi_n( \mathbf{X}^n; \widehat{\omega}_n(\mathbf{X}^n_l)) - \psi_n( \mathbf{X}_l^n; \widehat{\omega}_n(\mathbf{X}_l^n)) ] \label{samedist} \\
& = \mathbb{E}  \sum_{i \in J_{n,l}} (X_{n,i} - X_{n,i}^{(l)} )\cdot H_{n,i} ( \widehat{\omega}_n(\mathbf{X}^n_l)) \nonumber \\
& \leq \mathbb{E}  \sum_{i \in J_{n,l}} (\|X_{n,i}\|_2 + \|X_{n,i}^{(l)}\|_2)\|H_{n,i} ( \widehat{\omega}_n(\mathbf{X}^n_l))\|_2 \nonumber \\
& \leq (2J\mathbb{E}_{\mathcal{P}}\|X\|_2 ) \varsigma_n, \nonumber 
\end{align}
where \Cref{samedist} follows from the fact that $\mathbf{X}^n \stackrel{d}{=} \mathbf{X}^n_l$, for any $l \in \mathcal{L}_n$. This completes the proof with the aid of the observation made in \Cref{rem:strateasy}. 
\end{proof}

If we can not apply \Cref{mostprop} directly to a particular problem, the technique used to prove \Cref{mostprop} are generally useful even in that scenario, especially for the problems satisfying \Cref{ass:linear}, as we shall see in \Cref{sec:lin}.

\section{Examples}
\label{sec:lin}

We shall first consider several well-known examples of optimization problems which fall under the scope of \Cref{ass:linear}. We shall demonstrate that for natural choices of perturbation blocks, solutions of all these problems are stable in the sense discussed in \Cref{stochopt}. 

\subsection{Branching random walk}
\label{brw}
Branching random walk is an important model in statistical physics and
probability. The basic model is very simple and intuitive. It starts with a particle at the
origin. The particle splits into a random number of particles following a specified progeny
distribution, given by a probability mass function $\{p_k : k \geq 0\}$, and each new particle undergoes a random displacement on $\mathbb{R}$. The new particles
form the first generation. Each particle in the first generation splits into a random number of
particles according to the same law and independently of the past  as well as of the other particles
in the same generation. Each new particle makes a random displacement from the position of its
parent following the same displacement distribution, independently from other particles. The new
particles form the second generation. This mechanism  goes on. This resulting system is called a
branching random walk (BRW).

It is clear that the particles in the system described above form a rooted Galton--Watson
tree if we forget about their positions. 
This Galton--Watson tree
will be denoted by $\mathbb{T}=(V,E)$, where $V$ is the set of vertices of the tree and $E$ is the
collection of edges. The collection of particles or vertices at the $n$-th generation will be
denoted by $D_n$, with $\operatorname{card}(D_n)=:Z_n$, whereas the collection of all edges with both end-vertices belonging to $n$-th generation or before is denoted by $E_n$.


We identify each edge $e_v$ of the Galton--Watson tree with its vertex $v$ away from the root; we then assign
a real-valued random variable $X_{e_v}$,  the displacement of the corresponding
particle. Our model implies that conditioned on the Galton-Watson tree $\mathbb{T}$,
$\{X_e: \, e\in E\}$ is a collection of i.i.d.~random variables. We shall denote the unique $m$-th generation ancestor of vertex/particle $v \in D_n$ by $A_m(v)$ for all $0 \leq m \leq n$. Because of the underlying tree
structure, for any two vertices $v_1$ and $v_2$, there is a unique geodesic path connecting them; we shall
denote the collection of all edges on this  path by $I(v_1 \leftrightarrow v_2)$.  The position of the
particle corresponding to the vertex $v \in V$ is therefore given  by
$$ S_v := \sum_{e \in I(\mathrm{root} \leftrightarrow v)} X_e.$$
The collection $\{S_v: v \in V \}$ is called the branching random walk (BRW) induced by
the tree $\mathbb{T}=(V,E)$ and the displacements $\{X_e : e \in E \}$. The main focus of our discussion will be the behavior of the minimum displacement,
$M_n := \min_{v \in D_n} S_v$; the minimum over an empty set is always defined to be $\infty$.  To express this problem in the form of the generic optimization problem introduced in \Cref{stochopt}, we shall condition on the underlying Galton--Watson tree $\mathbb{T}$, assuming that it does not become extinct. We shall assume the well-studied \textit{Kesten--Stigum condition}~\cite{Kesten}, i.e., the progeny mean $m=\sum_{k \geq 1} kp_k \in (1,\infty)$ and $\sum_{k \geq 1} kp_k \log k < \infty$. Under this assumption, the progeny mean is finite and the process is super-critical. This will ensure that the tree $\mathbb{T}$ survives with positive probability. Moreover, $Z_n/m^n$ converges almost surely to a finite random variable $W$ and $\mathbb{P}(W >0 \mid \mathbb{T} \text{ survives })=1$.   Conditioned on the tree $\mathbb{T}$, the random inputs for the optimization problems are the displacements associated with the edges; i.e., in the language of \Cref{stochopt}, we can take $I_n = E_n$ and $X_{n,e}=X_e$ for all $e \in E_n$ and $n \geq 1$. The random inputs take values in $\mathcal{X}=\mathbb{R}$, and the parameter space $\mathcal{S}_n$ can be thought of as $D_n$, the set of all particles in generation $n$. The obvious way to metrize this parameter space is to take $d_n(v_1,v_2)$ to be the shortest distance between the vertices $v_1$ and $v_2$, normalized by $n$; i.e.
$$ d_n(v_1,v_2) := \dfrac{1}{n} \operatorname{card}(I(v_1 \leftrightarrow v_2)) = \dfrac{1}{n} \operatorname{card}(I(\mathrm{root} \leftrightarrow v_1) \Delta I(\mathrm{root} \leftrightarrow v_2) ), \; \forall\; v_1, v_2 \in D_n.$$
Here $\Delta$ refers to the symmetric difference between two sets.
The objective function then can be expressed as 
\begin{align*}
\psi_n ( (X_{n,e})_{e \in E_n} ; v) :=  \sum_{e \in I(\mathrm{root} \leftrightarrow v)} X_{n,e}, \; \forall \; v \in D_n.
\end{align*}   
Clearly, this is in the form of \Cref{ass:linear} with $H_{n,e}(v) = \mathbbm{1}(e \in I(\text{root} \leftrightarrow v))$, for all $e \in E_n$.

The earliest works on branching random walks include \cite{Hammersley,Kingman,Biggins}, who established that the leftmost particle in $D_n$ has linear asymptotic velocity in $n$, provided that the displacement variables satisfy some exponential moment assumption. The result can be summarized in the following theorem. 

\begin{thm}[\cite{shi2015}]{\label{brw:thm}}
Consider BRW with progeny mean $m \in (1,\infty)$ and displacement distribution $\mathcal{P}$ satisfying
$$ \Psi(t) := \log m + \log \mathbb{E} \exp(-tX) \in \mathbb{R}, $$
for some $t>0$; here $X \sim \mathcal{P}$. Then almost surely on the set of non-extinction of the tree $\mathbb{T}$, we have
$$ \lim_{n \to \infty} \dfrac{1}{n} \min_{v \in D_n} S_v = - \inf_{s>0} \dfrac{\Psi(s)}{s} =: -\psi^* \in \mathbb{R}.$$
\end{thm}

\begin{remark}{\label{BRW:sol}}
The function $\Psi$ defined in the statement of \Cref{brw:thm} is differentiable in the open interval, on which it takes finite values. It is easy to see upon differentiation that the function $s \mapsto \Psi(s)/s$ is minimized on the positive real line at the solution of the equation $s\Psi^{(1)}(s)=\Psi(s)$, where $\Psi^{(1)}$ is the first order derivative of the function $\Psi$. The equation  $s\Psi^{(1)}(s)=\Psi(s)$ has a unique solution $s^*$ on the positive real line,  and thus $\Psi(s^*)=\psi^*s^*.$ The existence and uniqueness of the solution is ensured by the continuity and strong convexity of the function $s \mapsto \log \mathbb{E} \exp(-sX)$,
\end{remark}

We refer to the exposition by Shi~\cite{shi2015} for a detailed background on BRW, as well as a proof of a general version of \Cref{brw:thm}; see~\cite[Theorem 1.3]{shi2015}. Building upon \Cref{brw:thm} and applying \Cref{genthm:g3}, we can arrive at the following proposition.

\begin{prop}{\label{brwtight}}
Consider super-critical BRW  where the progeny distribution satisfies the Kesten--Stigum condition, and the displacement variables have common distribution $\mathcal{P}$ with finite mean and  density $f$  satisfying Assumption~\ref{mu} for a pair $(\rho,g)$ such that $g$ is strictly monotonic on $\mathbb{R}$. Moreover assume that $\mathcal{P}$ satisfies the hypothesis of Theorem~\ref{brw:thm}.  Then $\{P(\mathcal{N}_{n,\theta_n}, d_n, \varepsilon) : n \geq 1\}$ is a tight sequence almost surely on the non-extinction of the tree for any $\theta_n=\bigO(1)$.
\end{prop}

\begin{proof}
We emphasize that every statement made in this proof is conditioned on the tree $\mathbb{T}$, which does not go extinct. The  pseudo-metric $d_{2,n}$ in this situation is given as follows. For all $v_1,v_2 \in D_n$,
\begin{align*}
d_{2,n}(v_1,v_2) = \sqrt{\sum_{e \in E_n} ( H_{n,e}(v_1)-H_{n,e}(v_2) )^2} &= \sqrt{\operatorname{card}(I(\mathrm{root} \leftrightarrow v_1) \Delta I(\mathrm{root} \leftrightarrow v_1))} \\
&= \sqrt{n d_n(v_1,v_2)}.
\end{align*}
It is clear that $d_{2,n}$ is a proper metric on $D_n$. Existence of the density $f$ and integrability guarantees that \Cref{unique} and \Cref{ass:meas} are satisfied; see \Cref{uniqueremark} and \Cref{meanupper} for justification. 

In this problem, the assumptions made in \Cref{subgmaxgrowth} are satisfied since,
$$ \zeta_{n,2}^2 = \sup_{v \in D_n} \sum_{e \in E_n} \mathbbm{1} (e \in I(\mathrm{root} \leftrightarrow v) ) = n, \; \zeta_{n, \infty} =\sup_{v \in D_n} \sup_{e \in E_n} \mathbbm{1} (e \in I(\mathrm{root} \leftrightarrow v) )=1. $$
Therefore, if the displacement variables are mean-zero sub-Gamma with variance proxy and scale parameter both equal to $1$, 
we can apply \Cref{subgmaxgrowth} to conclude that  \Cref{item:G1} is satisfied with the choice $M_n = \sqrt{2n\log Z_n} +  \log Z_n.$ The Kesten--Stigum condition guarantees that $ (\log Z_n)/n \to m$ almost surely on non-extinction and hence $M_n \sim n$ as $n \to \infty$. 
It is now enough to prove that \Cref{item:G2} is satisfied with the choice $\tau_n = \bigO(\sqrt{n})$ to apply \Cref{genthm:g3} and conclude the proof. 

Without loss of generality we assume that $g$ is strictly increasing. According to the statement of \Cref{genthm:g3}, the displacement variables in this case are taken to be $\{X_e = g(Y_e) : e \in E\}$ where $\{Y_e : e \in E\}$ is an i.i.d.~collection from $\mathcal{Q}_{\rho}$, conditioned on the tree $\mathbb{T}$.  Since $\mathcal{Q}_{\rho}$ has positive density on the entire real line, monotonicity of $g$ guarantees that $\mathcal{P}$ has positive density on $(a_1,a_2) = g(\mathbb{R})$.  Fix $\varepsilon >0$ and apply \Cref{lem:brw}  to get hold of $b_{1,\varepsilon}, b_{2,\varepsilon} \in (a_1,a_2)$ satisfying 
$$ \liminf_{n \to \infty}  \dfrac{1}{(\varepsilon/2) n}   \inf_{v \in D_n} \sum_{e \in I(A_{\lfloor (1-\varepsilon/2)n \rfloor}(v) \leftrightarrow v)} \mathbbm{1}( b_{1,\varepsilon} \leq X_e \leq  b_{2,\varepsilon})  \geq \dfrac{1}{2}, \; \text{ almost surely on  } (\mathbb{T} \text{ survives}).$$
Let $a_{i,\varepsilon} := g^{-1}(b_{i,\varepsilon})$ for $i=1,2$. Since $g^{(1)}$ is strictly positive and continuous on $\mathbb{R}$, we have that $\eta_{\varepsilon}:=\inf_{a_{1,\varepsilon} \leq y \leq  a_{2,\varepsilon}} g^{(1)}(y) >0$. Now take any $v_1,v_2 \in D_n$ satisfying $d_{2,n}(v_1,v_2) >  \sqrt{\varepsilon n}$, i.e.,  $d_n(v_1,v_2) > \varepsilon$. Since the geodesic path from $v_1$ to $v_2$ needs at least $n\varepsilon$ many edges, the latest common ancestor of this two particles is in $l$-th generation for some $l \leq (1-\varepsilon/2)n$. Therefore
\begin{align*}
d_{g,n}(v_1,v_2)^2 & = \sum_{e \in E_n} ( H_{n,e}(v_1)-H_{n,e}(v_2))^2(g^{(1)}(Y_e))^2 \\
&\geq  \sum_{e \in I(A_{\lfloor (1-\varepsilon/2)n \rfloor}(v_1) \leftrightarrow v_1)} (g^{(1)}(Y_e))^2 \\
& \geq \sum_{e \in I(A_{\lfloor (1-\varepsilon/2)n \rfloor}(v_1) \leftrightarrow v_1)} \eta_{\varepsilon}^2 \mathbbm{1}(a_{1,\varepsilon} \leq Y_e \leq  a_{2,\varepsilon}).
\end{align*}
In other words,
\begin{align*}
\inf_{\substack{v_1,v_2 \in \mathcal{N}_{n, \theta_n} \\d_{n}(v_1,v_2) > \varepsilon }} \frac{d_{g,n}(v_1,v_2)}{\sqrt{n}} &\geq \inf_{\substack{v_1,v_2 \in D_n \\d_{n}(v_1,v_2) > \varepsilon }} \frac{d_{g,n}(v_1,v_2)}{\sqrt{n}} \\
& \geq \eta_{\varepsilon}\sqrt{ \dfrac{1}{ n}   \inf_{v \in D_n} \sum_{e \in I(A_{\lfloor (1-\varepsilon/2)n \rfloor}(v) \leftrightarrow v)} \mathbbm{1}( a_{1,\varepsilon} \leq Y_e \leq  a_{2,\varepsilon})} \\
 & = \eta_{\varepsilon}\sqrt{ \dfrac{1}{ n}   \inf_{v \in D_n} \sum_{e \in I(A_{\lfloor (1-\varepsilon/2)n \rfloor}(v) \leftrightarrow v)} \mathbbm{1}( b_{1,\varepsilon} \leq X_e \leq  b_{2,\varepsilon})}
\end{align*}
The limit inferior of the right hand side above is at least $\sqrt{\varepsilon} \eta_{\varepsilon}/2 >0$ almost surely on non-extinction of the tree $\mathbb{T}$ and hence the reciprocal of the left hand side is tight. As a consequence, for any $\delta >0$, there exists $\kappa^{\prime}_{\varepsilon,\delta}>0$ such that
\begin{equation}{\label{step:brwproof}}
\limsup_{n \to \infty} \mathbb{P} \biggl[\inf_{\substack{v_1,v_2 \in \mathcal{N}_{n, \theta_n} \\d_{n}(v_1,v_2) > \varepsilon }} d_{g,n}(v_1,v_2) < \kappa^{\prime}_{\varepsilon,\delta}\sqrt{n} \biggr] \leq \delta.
\end{equation}
 On the complement of the event inside the probability in \Cref{step:brwproof}, any $\varepsilon$-packing set of $\mathcal{N}_{n,\theta_n}$ with respect to metric $d_n$ is also a $\kappa^{\prime}_{\varepsilon,\delta}\varepsilon \sqrt{n}$-packing set with respect to metric $d_{g,n}$; hence $P(\mathcal{N}_{n,\theta_n},d_{g,n},\kappa^{\prime}_{\varepsilon,\delta}\varepsilon \sqrt{n}) \geq P(\mathcal{N}_{n,\theta_n},d_n,\varepsilon)$. This shows that \Cref{item:G2} is satisfied with $\tau_n=\sqrt{n}$, which completes the proof. 
\end{proof}

The choice of window-length in \Cref{brwtight} is optimal, in the sense that if we chose $\theta_n \to \infty$ as $n \to \infty$ then there will be too many near-optimal points. This conclusion can be argued using the weak convergence results for the left-most displacements for BRW. The following result was proved by \citet[Theorem 1.1]{aidekon}. 

\begin{thm}[\cite{aidekon}]{\label{brw:weak}}
Consider the set-up of Theorem~\ref{brw:thm}, and let $s^*$ be defined as in Remark~\ref{BRW:sol}. Moreover, assume that the progeny distribution has finite second moment (i.e., $\sum_{k \geq 0} k^2p_k < \infty$), and the displacement distribution $\mathcal{P}$ is continuous.  Under these conditions, there exists an almost surely finite random variable $L_{\infty}$, which is strictly positive on the set of non-extinction of the tree $\mathbb{T}$ and zero on the set of extinction, such that for all $x \in \mathbb{R}$, 
$$ \lim_{n \to \infty} \mathbb{P} \biggl( s^*M_n + n\psi^*s^* - \dfrac{3}{2}\log n \geq x\biggr) = \mathbb{E} [\exp(-L_{\infty}e^x)].$$
In other words, as $n \to \infty$, 
$$ s^*M_n + n\psi^*s^* - \dfrac{3}{2}\log n \stackrel{d}{\longrightarrow} - \mathrm{Gumbel} - \log L_{\infty}, $$
where $\mathrm{Gumbel}$ is a random variable, independent of $L_{\infty}$, having standard Gumbel distribution, i.e., having distribution function $\mathbb{P}(\mathrm{Gumbel} \leq x) = \exp(-e^{-x})$ for all $x \in \mathbb{R}$.  
\end{thm}

\begin{remark}
The random variable $L_{\infty}$ is connected to the limit of derivative martingales associated with a properly centered and normalized version of the BRW; see \cite{aidekon} and \cite{shi2015} for more details on derivative martingales which is a key instrument in the analysis of BRW. Also the conditions on the progeny and displacement distributions as stated can be weakened considerably, see \cite{aidekon} for the weaker conditions. We shall not use the full strength of \Cref{brw:weak}, rather the conclusion that conditioned on the survival of the tree, the distribution of the random variable $s^*M_n +n \psi^*s^* - \frac{3}{2}\log n$ converges weakly to a continuous distribution. This corollary will useful in proving the next result, which is about the optimality of the window-length in \Cref{brwtight}.  
\end{remark}

\begin{prop}
Consider the set-up of Theorem~\ref{brw:weak} and take any $\theta_n \to \infty$ as $n \to \infty$. Then for any $\varepsilon \in (0,2)$, the sequence of random variables $\{P(\mathcal{N}_{n,\theta_n}, d_n, \varepsilon) : n \geq 1\}$ is not  tight almost surely on the non-extinction of the tree $\mathbb{T}$. 
\end{prop}

\begin{proof}
Fix any $j,k \in \mathbb{N}$. Consider the $k$-th generation of the tree $\mathbb{T}$, and for any $v \in D_k$, let $\mathbb{T}_v$ be the sub-tree rooted at $v$. Let $M_{n,v}$ be the left-most displacement (from its root at vertex $v$) in the $n$-th generation of the sub-tree $\mathbb{T}_v$, and $u_{n,v}$ be the vertex where it is achieved. As usual, the minimum is achieved at a unique vertex almost surely due to continuity of the displacement distribution. Note that 
$$M_n = \min_{v \in D_k} (S_{u_v} + M_{n-k,v} ),$$ 
and let $v_n^* \in D_k$ denotes the vertex where the above minimum has been achieved. \Cref{brw:weak} guarantees the existence of an independent collection $\{(\mathrm{Gumbel}_v, L_{\infty,v}) : v \in D_k\}$ such that as $n \to \infty$, for all $v \in D_k$, 
\begin{equation}{\label{brw:weak2}}
 s^*M_{n,v} + n\psi^*s^* - \dfrac{3}{2}\log n \stackrel{d}{\longrightarrow} -\mathrm{Gumbel}_v - \log L_{\infty,v}, 
\end{equation}
where $\mathrm{Gumbel}_v$ and $L_{\infty,v}$ are independent for all $v \in D_k$. Let $D_k^{\prime} \subseteq D_k$ consists of those those vertices $v \in D_k$ for which $\mathbb{T}_v$ does not get extinct and hence $L_{\infty,v}>0$ almost surely. It is clear that $\mathbb{P}(v_n^* \in D_k^{\prime} \mid \mathbb{T} \text{ survives}) \to 1$ as $n \to \infty$, since $M_n \stackrel{a.s.}{\to} - \infty$ on the non-extinction of the tree $\mathbb{T}$. 

Now fix any $\varepsilon \in (0,2)$ and get $n_0$ such that $k < (1-\varepsilon/2)n_0$. For any $n \geq n_0$ and $v \neq v^{\prime} \in D_k$, we have $d_n(u_{v},u_{v^{\prime}}) = 2(n-k)/n > \varepsilon$, and hence,
$$ \{ P(\mathcal{N}_{n,\theta_n},d_n,\varepsilon) \geq \operatorname{card}(D_k^{\prime})\} \supseteq \{ v_n^* \in D_{k}^{\prime}, \; \max_{v \in D_k^{\prime}} (S_{u_v}+M_{n-k,v} ) \leq M_n + \theta_n \}.$$
From the weak convergence in \Cref{brw:weak2}, we have 
\begin{align}
&\max_{v \in D_k^{\prime}}(S_{u_v}+M_{n-k,v} ) - M_n = \max_{v \in D_k^{\prime}} (S_{u_v}+M_{n-k,v})  - \min_{v \in D_k^{\prime}} (S_{u_v}+M_{n-k,v}) \nonumber  \\
& \stackrel{d}{\longrightarrow} \max_{v \neq v^{\prime} \in D_k^{\prime}} \biggl( S_{u_v}-S_{u_{v^{\prime}}}-\mathrm{Gumbel}_v+\mathrm{Gumbel}_{v^{\prime}}-\log \dfrac{L_{\infty,v}}{L_{\infty,v^{\prime}}}\biggr). \label{brw:weak3}
\end{align} 
Since the random variable in \Cref{brw:weak3} is finite almost surely and $\theta_n \to \infty$, we can conclude that as $n \to \infty$,
$$ \mathbb{P} [  P(\mathcal{N}_{n,\theta_n},d_n,\varepsilon) \geq \operatorname{card}(D_k^{\prime})|\mathbb{T} \text{ survives}] \to 1.$$
The above holds true for any $k \geq 1$, and thus it is enough to show that conditionally on the survival of the tree $\mathbb{T}$, we have $\operatorname{card}(D_k^{\prime}) \stackrel{p}{\to} \infty$ as $k \to \infty$. But this is trivially true, since $ \operatorname{card}(D_k^{\prime}) \sim \mathrm{Binomial}( Z_k, p_{\mathrm{surv}})$, where $p_{\mathrm{surv}}:= \P(\mathbb{T} \text{ survives}) >0$ and $Z_k \stackrel{a.s.}{\to} \infty$ on the event $\{\mathbb{T} \text{ survives}\}$ by super-criticality of the underlying Galton--Watson tree. This completes the proof.  
\end{proof}

The intuitive choice for perturbation blocks in BRW amounts to change displacement associated with one edge at a time. Since $|H_{n,e}(v)| \leq 1$ for all $n,e,v$, we have 
\[
\varsigma_n := \sup_{e \in E_n} \sup_{v \in D_n} |H_{n,e}(v)| =1,
\]
and hence we can directly apply \Cref{mostprop} to arrive at the following stability result for BRW.

\begin{thm}{\label{brwstability}}
Consider super-critical BRW  where the progeny distribution satisfies the Kesten--Stigum condition, and the displacement variables have common distribution $\mathcal{P}$ with  finite mean and density $f$  satisfying Assumption~\ref{mu} for the pair $(\rho,g)$, such that $g$ is strictly monotonic on $\mathbb{R}$. Moreover assume that $\mathcal{P}$ satisfies the hypothesis of Theorem~\ref{brw:thm}. Then the optimization problem sequence $\{\mathscr{P}_n : n \geq 1 \}$, where $\mathscr{P}_n$ deals with the problem of  finding the left-most particle in $n$-th generation, is  stable under small perturbations with perturbation blocks $\mathcal{J}_n =\{J_{n,e} : e \in E_n\}$ where $J_{n,e} = \{e\}$ for all $e \in E_n$.
\end{thm}

In light of \Cref{exmu}, standard distributions like Uniform, Beta, Gamma, Gaussian, Gumbel are all valid choices for $\mathcal{P}$ in \Cref{brwstability}.

\begin{remark}{\label{brw"litstab}}
The vast literature on BRW focuses mostly on the asymptotic evaluation of its minimum displacements or the point process generated by properly centered displacements in a generation. As per our current knowledge, the result in \Cref{brwstability} is the only result regarding stability of optimal vertices in BRW upon perturbations of edge displacement variables.
\end{remark}

\subsection{The Sherrington--Kirkpatrick model} 
\label{skmodel}
In 1975, Sherrington and Kirkpatrick~\cite{skmodel} introduced a mean field model for a spin glass --- a disordered magnetic alloy that exhibits unusual magnetic behavior. Consider a system of $n$ spins $\boldsymbol{\sigma}=(\sigma_1, \ldots,\sigma_n)$ with $\sigma_i \in \{+1,-1\}$. The Sherrington--Kirkpatrick (SK) model is defined by the Hamiltonian
\begin{equation}{\label{skhamil}}
\mathcal{H}_n(\mathbf{X}^n;\boldsymbol{\sigma}) := - \sum_{i,j \in [n] : i<j} X_{n,(i,j)}\sigma_i\sigma_j, \; \forall \; \boldsymbol{\sigma} \in \{+1,-1\}^n,  
\end{equation}
where $\{X_{n,(i,j)} : 1 \leq i<j \leq n\}$ is collection of mean-zero i.i.d.~random variables, collectively called the \textit{disorder of the model}. The most studied and well-understood is the case when the disorder variables of the model are Gaussian. The Hamiltonian $\mathcal{H}_n$ is generally taken to to be normalized by $\sqrt{n}$ in the literature, but this will not be an issue in our analysis since we are only concerned about the optimal configuration. The configuration with lowest energy is known as the ground state of the system. In other words, the optimization problem of interest, $\mathscr{P}_n$, in this case is 
\begin{equation}{\label{skprob}}
\text{minimize } \;\;\mathcal{H}_n(\mathbf{X}^n;\boldsymbol{\sigma})=- \sum_{i<j} X_{n,(i,j)}\sigma_i\sigma_j, \; \text{over } \boldsymbol{\sigma} \in \{+1,-1\}^n.
\end{equation}
A  detailed account of the SK model can be found in \cite{talagrand} and \cite{panchenko}.  
To express the  optimization problem in \Cref{skprob} in the formulation of \Cref{stochopt}, let us start by observing that the Hamiltonian $\mathcal{H}_n$ is invariant under sign-flips of all the spins, i.e., $\mathcal{H}_n(\mathbf{X}^n;\boldsymbol{\sigma}) = \mathcal{H}_n(\mathbf{X}^n;-\boldsymbol{\sigma})$ for all $\boldsymbol{\sigma} \in \{+1,-1\}^n$. To get a unique solution corresponding to the lowest energy, we therefore consider the corresponding quotient space as the parameter space,
$$ \mathcal{S}_n := \{\boldsymbol{\sigma}=(\sigma_1, \ldots, \sigma_n) : \sigma_1=+1, \sigma_i \in \{+1,-1\}, \; \forall \; 2 \leq i \leq n \}.$$
We metrize this space with normalized Hamming distance, defined as 
$$ d_n (\boldsymbol{\sigma}, \boldsymbol{\sigma}^{\prime} ) = \dfrac{1}{n} \sum_{i=1}^n \mathbbm{1}(\sigma_i \neq \sigma_i^{\prime}), \; \forall \; \boldsymbol{\sigma} = (\sigma_1, \ldots, \sigma_n), \boldsymbol{\sigma}^{\prime}=(\sigma_1^{\prime}, \ldots, \sigma^{\prime}_n) \in \mathcal{S}_n.$$ 
The disorder variables are the random inputs of the problem; hence $I_n = \{(i,j) : 1 \leq i < j \leq n\}$ and $k_n=n(n-1)/2.$ The optimization problem in \Cref{skprob} clearly satisfies \Cref{ass:linear} with $H_{n,(i,j)}(\boldsymbol{\sigma}) = -\sigma_i\sigma_j$, for all $i,j,\boldsymbol{\sigma}$. We remind the reader that we have used the notation $\mathcal{H}_n$ to denote the Hamiltonian, instead of $H$ or $H_n$ which is generally used in the standard literature. We will now apply \Cref{genthm2} to get a tightness result for SK model optimization problem.

\begin{prop}{\label{sktight1}}
Consider SK model as defined in \Cref{skhamil} where the disorder variables have common distribution $\mathcal{P}$ with density $f$ having zero mean and  satisfying Assumption \ref{ass:p}. Then the collection of random variables  $\{P(\mathcal{N}_{n,\theta_n}, d_n, \varepsilon) : n \geq 1\}$ is  tight  for any $\theta_n=\bigO(1)$. 
\end{prop}

\begin{proof}
First, note that
\begin{equation}
d_{2,n}(\boldsymbol{\sigma},\boldsymbol{\sigma}^{\prime})^2   = \sum_{i<j} (\sigma_i\sigma_j-\sigma^{\prime}_i\sigma_j^{\prime})^2 =  4n^2 d_n(\boldsymbol{\sigma},\boldsymbol{\sigma}^{\prime})(1-d_n(\boldsymbol{\sigma},\boldsymbol{\sigma}^{\prime})), \; \forall \;\boldsymbol{\sigma},\boldsymbol{\sigma}^{\prime} \in \mathcal{S}_n. 
\end{equation}
Since $d_n(\boldsymbol{\sigma},\boldsymbol{\sigma}^{\prime}) \leq (n-1)/n$ for any 
$\boldsymbol{\sigma},\boldsymbol{\sigma}^{\prime} \in \mathcal{S}_n$, this shows  that $d_{2,n}$ is a proper 
metric on $\mathcal{S}_n$, i.e., $d_{2,n}(\boldsymbol{\sigma},\boldsymbol{\sigma}^{\prime}) =0 $ if and only if $\boldsymbol{\sigma}=\boldsymbol{\sigma}^{\prime}$. Existence of densities for the disorder variables with finite mean and finiteness of $\mathcal{S}_n$ guarantee that \Cref{unique} and \Cref{ass:meas} are satisfied; see \Cref{uniqueremark} and \Cref{meanupper}. The observation $|H_{n,(i,j)}(\boldsymbol{\sigma})| \leq 1$  
 for any $n,i,j,\boldsymbol{\sigma}$ guarantees that we can take $\lambda =0$ and $\varsigma_{n,0}=1$ to satisfy \Cref{item:D1}. 
Since $k_n=n(n-1)/2$, it is now enough to show that $\ref{item:D2}$ holds true with  $\tau_n=n$  to apply \Cref{genthm2} and complete the proof. 

Fix any $\varepsilon \in (0, 1/2)$, $\mathcal{B}\subseteq \mathcal{S}_n$ and start with an $\varepsilon$-cover of $\mathcal{B}$, say $\{\boldsymbol{\sigma}^1, \ldots,\boldsymbol{\sigma}^k\}$, with respect to the metric $d_{2,n}/(2n)$. Consider the set $\{\boldsymbol{\sigma}^i,\widetilde{\boldsymbol{\sigma}}^i : 1 \leq i \leq k\}$, where $\widetilde{\boldsymbol{\sigma}}^i$ is obtained by making the first coordinate of $-\boldsymbol{\sigma}^i$ to be $+1$. We shall show that the above set is an $\eta_{\varepsilon}$-cover of $\mathcal{B}$ with respect to metric $d_n$, where $\eta_{\varepsilon} \in (0,1/2)$ is the smallest solution of the quadratic equation $\eta(1-\eta)=\varepsilon^2$; the other solution being $1-\eta_{\varepsilon}$. Take any $\boldsymbol{\sigma} \in \mathcal{B}$ and get $i \in [k]$ such that $ d_{2,n}(\boldsymbol{\sigma},\boldsymbol{\sigma}^i)/(2n) \leq \varepsilon $ and hence $d_{n}(\boldsymbol{\sigma},\boldsymbol{\sigma}^i)(1-d_{n}(\boldsymbol{\sigma},\boldsymbol{\sigma}^i)) \leq \varepsilon^2$.  It is now easy to observe that we have either $d_{n}(\boldsymbol{\sigma},\boldsymbol{\sigma}^i) \leq \eta_{\varepsilon}$ or $d_{n}(\boldsymbol{\sigma},-\boldsymbol{\sigma}^i) \leq \eta_{\varepsilon}$. In the second case, we clearly have $d_{n}(\boldsymbol{\sigma},\widetilde{\boldsymbol{\sigma}}^i) \leq \eta_{\varepsilon}$, since both $\boldsymbol{\sigma}$ and $\widetilde{\boldsymbol{\sigma}}^i$ have their first co-ordinate to be $+1$. We, therefore, can conclude that 
$N(\mathcal{B},d_{n}, \eta_{\varepsilon}) \leq 2N(\mathcal{B},d_{2,n}, 2n\varepsilon).$ Combining this with \Cref{pnp} guarantees that $ P(\mathcal{B},d_n,2\eta_{\varepsilon}) \leq 2P(\mathcal{B},d_{2,n}, 2n\varepsilon).$
As $\varepsilon$ varies from $0$ to $1/2$, we have $\eta_{\varepsilon}$ to be strictly increasing from $0$ to $1/2$. Hence, for any $\varepsilon \in (0,1)$, we can write the following for any $\mathcal{B} \subseteq \mathcal{S}_n$.
\begin{equation}{\label{step0:sk}}
P(\mathcal{B},d_n,\varepsilon) \leq  2P(\mathcal{B},d_{2,n}, n\sqrt{\varepsilon(2-\varepsilon)})
\end{equation}
This shows that \Cref{item:D2} is satisfied for $\tau_n=n, \varepsilon_0=1, K_{\varepsilon}=2$ and $ \kappa_{\varepsilon} = \sqrt{\varepsilon(2-\varepsilon)}$. Recalling that here $k_n=n(n-1)/2$, we complete the proof. 
\end{proof}

Combining \Cref{sktight1} and \Cref{mostprop}, we can write our first stability result for SK model which states that \textit{the solution of the SK model is stable under the perturbation scheme which replaces one disorder variable at a time}.

\begin{thm}{\label{skthm1}}
Consider SK model as defined in \Cref{skhamil} where the disorder variables have common distribution $\mathcal{P}$ with density $f$ having zero mean and  satisfying Assumption~\ref{ass:p}. 
This optimization problem is stable under small perturbations with perturbation blocks $\mathcal{J}_n = \{J_{n,(i,j)} : (i,j) \in [n]^2, i< j\}$ where $J_{n,(i,j)} = \{(i,j)\}$ for all $1 \leq i<j \leq n$.
\end{thm}

Another natural  choice of perturbation blocks corresponds to replacing all the disorder variables associated with a particular spin by i.i.d.~copies. In other words, we consider the perturbation blocks $\mathcal{J}_n = \{J_{n,i} : i \in [n]\}$ where $J_{n,i}=\{(i,j) : j >i\} \cup \{(j,i) : i >j\}$. Unfortunately, \Cref{sktight1} is not strong enough to guarantee a stability result for this choice of perturbation blocks. Luckily, \Cref{genthm:g3} comes to our rescue to yield the following result.

\begin{prop}{\label{sktight2}}
Consider SK model as defined in \Cref{skhamil} where the disorder variables have common distribution $\mathcal{P}$ with density $f$ having zero mean and  satisfying \Cref{mu} for the pair $(\rho,g)$ such that $g^{(1)}(Y) \neq 0$ almost surely when $Y \sim \mathcal{Q}_{\rho}$. Then $\{P(\mathcal{N}_{n,\theta_n}, d_n, \varepsilon) : n \geq 1\}$ is a tight sequence for any $\theta_n=\bigO(\sqrt{n})$. 
\end{prop}

\begin{proof}
\Cref{mu} implies that we can take $X_{n,i}=g(Y_{n,i})$ where $Y_{n,i}$'s are i.i.d.~from the distribution $\mathcal{Q}_{\rho}$. As was the case for \Cref{sktight1}, existence of density $f$ for disorder variables and finiteness of $\mathcal{S}_n$ guarantee  that \Cref{unique} are \Cref{ass:meas} are satisfied. Moreover,
$$ \sum_{i <j} H_{n,(i,j)}(\boldsymbol{\sigma})^2 = n(n-1)/2, \; |  H_{n,(i,j)}(\boldsymbol{\sigma})| \leq 1, \; \forall \; \boldsymbol{\sigma} \in \mathcal{S}_n.$$
Thus the hypothesis of \Cref{subgmaxgrowth} is satisfied with $\zeta_{n,2}=\sqrt{n(n-1)/2}$ and $\zeta_{n,\infty}=1$. Applying \Cref{subgmaxgrowth} and observing that $\operatorname{card}(\mathcal{S}_n) =2^{n-1}$, we can deduce that \Cref{item:G1} is satisfied for the choice
\begin{equation}{\label{subgmax:sk}}
 M_n = \sqrt{ n(n-1)^2 \log 2} +   (n-1) \log 2 = \Theta(n^{3/2}), \; \forall \; \text{as } n \to \infty.
\end{equation}
Therefore, in light of \Cref{genthm:g3}, it is now enough to establish \Cref{item:G2} for any $\theta_n=\bigO(\sqrt{n})$ and for the choice $\tau_n=n$. 
Recall that 
$$ d_{g,n}(\boldsymbol{\sigma}, \boldsymbol{\sigma}^{\prime})^2 = \sum_{i<j} (g^{(1)}(Y_{n,(i,j)}) )^2 (\sigma_i\sigma_j - \sigma^{\prime}_i\sigma^{\prime}_j)^2, \; \forall \; \boldsymbol{\sigma}, \boldsymbol{\sigma}^{\prime} \in \mathcal{S}_n.$$
Fix $\varepsilon \in (0,1)$. 
Since $g^{(1)}(Y) \neq 0$ almost surely for $Y \sim \mathcal{Q}_{\rho}$, we can get a finite positive number  $\kappa$ such that $\mathbb{P}(|g^{(1)}(Y)| \leq  \kappa) \leq \varepsilon(2-\varepsilon)/2$ and set $F_n := \{(i,j) : 1 \leq i<j \leq n, |g^{(1)}(Y_{n,(i,j)})| \leq \kappa\}$. Take any $\boldsymbol{\sigma}, \boldsymbol{\sigma}^{\prime} \in \mathcal{S}_n$ with $d_{2,n}(\boldsymbol{\sigma}, \boldsymbol{\sigma}^{\prime} ) > \sqrt{\varepsilon(2-\varepsilon)} n$; hence
$$ \varepsilon(2-\varepsilon) n^2 < d_{2,n}(\boldsymbol{\sigma}, \boldsymbol{\sigma}^{\prime} )^2 = \sum_{i<j} (\sigma_i\sigma_j - \sigma^{\prime}_i\sigma^{\prime}_j)^2. $$
Since $(\sigma_i\sigma_j - \sigma^{\prime}_i\sigma^{\prime}_j)^2 \in \{0,4\}$, we can deduce that 
$$ \operatorname{card} ( \{(i,j) : i<j,(\sigma_i\sigma_j - \sigma^{\prime}_i\sigma^{\prime}_j)^2 = 4 \}) > \varepsilon(2-\varepsilon) n^2/4,$$
and hence
$$ \operatorname{card} ( \{(i,j) : i<j,(\sigma_i\sigma_j - \sigma^{\prime}_i\sigma^{\prime}_j)^2 = 4, |g^{(1)}(Y_{n,(i,j)})| > \kappa \}) > \varepsilon(2-\varepsilon) n^2/4 - \operatorname{card}(F_n).$$
It is then easy to see that 
$$ d_{g,n}(\boldsymbol{\sigma}, \boldsymbol{\sigma}^{\prime} )^2 \geq \kappa ( \varepsilon(2-\varepsilon) n^2 - 4\operatorname{card}(F_n)).$$
We can use the above estimate to write down the following.
\begin{align}
&\mathbb{P} \biggl[P(\mathcal{N}_{n,\theta_n}, d_n, \varepsilon) > 2P(\mathcal{N}_{n,\theta_n}, d_{g,n}, n\sqrt{\kappa\varepsilon(2-\varepsilon)}/2)  \biggr] \nonumber  \\
& \hspace{ 1 in} \leq \mathbb{P} \biggl[ P(\mathcal{N}_{n,\theta_n},d_{2,n}, n\sqrt{\varepsilon(2-\varepsilon)}) > P(\mathcal{N}_{n,\theta_n}, d_{g,n}, n\sqrt{\kappa\varepsilon(2-\varepsilon)}/2) \biggr] \label{step:sk} \\
& \hspace{ 1 in} \leq \mathbb{P} \biggl( \inf_{\substack{\boldsymbol{\sigma}, \boldsymbol{\sigma}^{\prime} \in \mathcal{S}_n \\ d_{2,n}(\boldsymbol{\sigma}, \boldsymbol{\sigma}^{\prime} ) > n \sqrt{\varepsilon(2-\varepsilon)} }} d_{g,n}(\boldsymbol{\sigma}, \boldsymbol{\sigma}^{\prime} )/n < \sqrt{\kappa\varepsilon(2-\varepsilon)}/2 \biggr) \nonumber \\ 
& \hspace{ 1 in} \leq \mathbb{P}( \kappa ( \varepsilon(2-\varepsilon) n^2 - 4\operatorname{card}(F_n)) < \kappa\varepsilon(2-\varepsilon) n^2/4) \nonumber \\
& \hspace{ 1 in} = \mathbb{P}( \operatorname{card}(F_n) > 3\varepsilon(2-\varepsilon)n^2/16), \label{binbound}
\end{align}
where \Cref{step:sk} follows from \Cref{step0:sk}. By law of large numbers, $\operatorname{card}(F_n)/n^2$ converges almost surely to $\mathbb{P}(|g^{(1)}(Y)| \leq \kappa)/2 < \varepsilon(2-\varepsilon)/4$ and hence the expression in \Cref{binbound} converges to $0$ exponentially fast. This establishes \Cref{item:G2} with the choice $\tau_n=n$ and completes the proof.
\end{proof}

\begin{thm}{\label{skthm2}}
Consider SK model as defined in \Cref{skhamil} where the disorder variables have common distribution $\mathcal{P}$ with density $f$ satisfying the assumptions made in the statement of Proposition~\ref{sktight2}. Further assume that $\mathcal{P}$ is also a sub-Gamma distribution.
Then the SK model optimization problem is stable under small perturbations with perturbation blocks $\mathcal{J}_n = \{J_{n,i} : i \in [n]\}$ where $J_{n,i}=\{(i,j) : j >i\} \cup \{(j,i) : i >j\}$ for all $i \in [n]$.
\end{thm}

\begin{proof}
We shall make a small abuse of notations by defining $X_{n,(i,j)}$ to be $X_{n,(j,i)}$, whenever $i>j$; same abuse of notations apply for other similar quantities also. For all $i \in [n]$, consider $\widehat{\boldsymbol{\sigma}}(\mathbf{X}^n_i)=(\widehat{\sigma}^{(i)}_1,\ldots,\widehat{\sigma}_n^{(i)})$ which is the optimal configuration for the input $\mathbf{X}^n_i$, where all the disorders associated  with the $i$-th spin has been replaced by i.i.d.~copies. $\widehat{\boldsymbol{\sigma}}(\mathbf{X}^n)=(\widehat{\sigma}_1,\ldots,\widehat{\sigma}_n)$ will be the optimal configuration for the input $\mathbf{X}^n$. Then 
\begin{align}
&\mathbb{E} [ \mathcal{H}_n(\mathbf{X}^n;\widehat{\boldsymbol{\sigma}}(\mathbf{X}^n_i)) -\mathcal{H}_{n,\mathrm{opt}}(\mathbf{X}^n)] \notag \\
& = 
\mathbb{E} [ \mathcal{H}_n(\mathbf{X}^n;\widehat{\boldsymbol{\sigma}}(\mathbf{X}^n_i)) -\mathcal{H}_n(\mathbf{X}^n_i;\widehat{\boldsymbol{\sigma}}(\mathbf{X}^n_i))] \nonumber \\ 
& =  \mathbb{E} \biggl[  \sum_{j \in [n] : j \neq i} X^{(i)}_{n,(i,j)} \widehat{\sigma}^{(i)}_i\widehat{\sigma}^{(i)}_j -\sum_{j \in [n] : j \neq i} X_{n,(i,j)}\widehat{\sigma}^{(i)}_i\widehat{\sigma}^{(i)}_j \biggr] \nonumber \\
& = \mathbb{E} \biggl[  \sum_{j \in [n] : j \neq i} X^{(i)}_{n,(i,j)} \widehat{\sigma}^{(i)}_i\widehat{\sigma}^{(i)}_j \biggr] = \mathbb{E} \biggl[  \sum_{j \in [n] : j \neq i} X_{n,(i,j)} \widehat{\sigma}_i\widehat{\sigma}_j \biggr] \label{mzero},
\end{align}
where the first equality in \Cref{mzero} follows from the fact that the disorder variables have mean zero and  $\widehat{\boldsymbol{\sigma}}(\mathbf{X}^n_i)$ is independent of the collection  $\{X_{n,(i,j)} : j \in [n], j \neq i\}$ for all $i \in [n]$.  
Symmetry between the spins in SK model Hamiltonian guarantees that for any $i \in [n]$,
\begin{align}
\mathbb{E} \biggl[  \sum_{j \in [n] : j \neq i} X_{n,(i,j)} \widehat{\sigma}_i\widehat{\sigma}_j \biggr] = \dfrac{1}{n}\mathbb{E} \biggl[  \sum_{i, j \in [n] : i \neq j} X_{n,(i,j)} \widehat{\sigma}_i\widehat{\sigma}_j \biggr] &= \dfrac{2}{n}\mathbb{E} \biggl[  \sum_{i <j} X_{n,(i,j)} \widehat{\sigma}_i\widehat{\sigma}_j \biggr] \nonumber \\
& = - \dfrac{2}{n}\mathbb{E} \biggl[ \inf_{\boldsymbol{\sigma} \in \mathcal{S}_n}  \mathcal{H}_n( \mathbf{X}^n; \boldsymbol{\sigma} )\biggr]. \label{eq2}
\end{align}
Since the distribution $\mathcal{P}$ is sub-Gamma, we can conclude from \Cref{subgmaxgrowth} that there exists $C(f) \in (0,\infty)$, depending on $f$, such that 
\begin{equation}{\label{eq1}}
 \inf_{\boldsymbol{\sigma} \in \mathcal{S}_n}  \mathcal{H}_n( \mathbf{X}^n; \boldsymbol{\sigma} ) \geq - C(f)n^{3/2}, \; \forall \; n \geq 2.
\end{equation}
See the arguments leading to \Cref{subgmax:sk} for more details on the argument behind \Cref{eq1}. Combining \Cref{mzero} with \Cref{eq2} and \Cref{eq1}, we can conclude that 
$$ \mathbb{E} [ \mathcal{H}_n(\mathbf{X}^n;\widehat{\boldsymbol{\sigma}}(\mathbf{X}^n_i)) -\mathcal{H}_{n, \mathrm{opt}}(\mathbf{X}^n)]  \leq C(f)\sqrt{n}, \; \forall \; i \in [n], \; n \geq 2.$$
Applying \Cref{sktight2} and \Cref{rem:strateasy}, we observe that the hypotheses of \Cref{strat} is satisfied for $\theta_n=\sqrt{n}$. This concludes the proof.
\end{proof}

\begin{remark}
If $g$ is strictly monotone in \Cref{skthm2}, then clearly $g^{(1)}(Y) \neq 0$ when $Y \sim \mathcal{Q}_{\rho}$. Thus all the examples described in \Cref{exmu} which has sub-Gamma tails, e.g., Uniform, Beta, Gamma, Gaussian and Gumbel density, satisfy the hypotheses for \Cref{skthm2}.   As it is mentioned earlier, the most well-studied choice for disordered distribution is Gaussian. Since the Gaussian distribution satisfies all the assumptions needed for the validity of \Cref{skthm1} and \Cref{skthm2}, the SK model with Gaussian inputs is stable under both types of perturbations discussed in this section.
\end{remark}

Before wrapping up this section, let us briefly discuss existing stability results for  the SK model and how they compare with ours. The most influential notion of stability (or rather lack of stability) in this context, formulated in terms of \textit{chaos}, was first established in \cite{chabook}. To understand the concept, first recall the definition of the SK model at nonzero temperature. At  inverse temperature $0\le \beta<\infty$, the model defines a random measure, called the \textit{Gibbs measure} and denoted by $\mathcal{G}_{\beta}$, on the configuration space $\left\{+1,-1\right\}^n$. The Gibbs measure being defined as 
$$ \mathcal{G}_{\beta} \left(\boldsymbol{\sigma} \right) = \dfrac{\exp\left(-\beta \mathcal{H}_n (\mathbf{X}^n;\boldsymbol{\sigma}) \right)}{\sum_{\boldsymbol{\sigma}^{\prime} \in \left\{+1,-1\right\}^n}\exp\left(-\beta \mathcal{H}_n (\mathbf{X}^n;\boldsymbol{\sigma}^{\prime}) \right)}, \; \forall \; \boldsymbol{\sigma} \in \left\{+1,-1\right\}^n.$$
The zero temperature case corresponds to $\beta=\infty$, when the Gibbs measure becomes supported on the unique pair of ground states. Suppose that the disorder variables are i.i.d.~standard Gaussian. Fix some small $\varepsilon >0$, and consider a perturbed disorder  $\{X_{n,(i,j)}^{(\varepsilon)} : i,j \in [n], i<j\}$ such that $\{(X_{n,(i,j)},X_{n,(i,j)}^{(\varepsilon)}) : i,j \in [n], i<j\}$ is a collection of i.i.d.~bivariate mean-zero Gaussian random vectors with equal variances and correlation $e^{-\varepsilon}$. Let $\mathcal{G}_{\beta}^{(\varepsilon)}$ be the Gibbs measure associated with this perturbed disorder. Conditioned on the disorder variables, we now take two independent samples  $\boldsymbol{\sigma} \sim \mathcal{G}_{\beta}$ and $\boldsymbol{\sigma}^{(\varepsilon)} \sim \mathcal{G}_{\beta}^{(\varepsilon)}$. If the energy landscape is not much altered due to the perturbation, we would expect the two configurations to be on average close to each other (or their negations). In other words, we shall expect their \textit{overlap} $R_{1,2}^{(\varepsilon)}$, as defined below, to be concentrated near $\left\{+1,-1\right\}$:
$$ R_{1,2}^{(\varepsilon)} := \sum_{i \in [n]} \sigma_i \sigma_i^{(\varepsilon)} = 1-2d_n( \boldsymbol{\sigma}, \boldsymbol{\sigma}^{(\varepsilon)} ).$$
It was established in \cite{chabook} that the situation is actually quite the opposite. The following theorem, which is a specialization of \cite[Theorem 1.11]{chabook}, formulates this rigorously.

\begin{thm}[\cite{chabook}]{\label{sk:chaos}}
Let $R_{1,2}^{(\varepsilon)}$ be defined as above. Then for any integer $k \geq 1$ and $\varepsilon \in (0,1)$, 
$$ \mathbb{E}[ ( R_{1,2}^{(\varepsilon)})^{2k} ] \leq C_1 \exp( -C_2 \varepsilon \log n),$$
where $C_1,C_2$ are finite positive constants depending only on $\beta$ and $k$. 
\end{thm}
From \Cref{sk:chaos}, one can conclude that even if the perturbation amount $\varepsilon$ is small but $\gg 1/\log n$, then the two configurations are almost orthogonal to each other. In other words, the energy landscape is greatly altered by a small perturbation and thus the ground state is \textit{chaotic}. \Cref{sk:chaos} leads us to the \textit{multiple valley property} (MVP) for the zero-temperature SK model, as proved in \cite[Theorem 1.16]{chabook}.

\begin{thm}[\cite{chabook}]{\label{sk:mvp}}
There exist $\gamma_n, \delta_n, \varepsilon_n \to 0$ and $K_n \to \infty$, such that with probability $1-\gamma_n$, there exists a set $\mathcal{N}_n \subseteq \mathcal{S}_n$ with  $\operatorname{card}(\mathcal{N}_n) \geq K_n$ such that $d_n ( \boldsymbol{\sigma}, \boldsymbol{\sigma}^{\prime}) \in ( 1/2-\varepsilon_n,1/2+\varepsilon_n )$ for all $\boldsymbol{\sigma} \neq \boldsymbol{\sigma}^{\prime} \in  \mathcal{N}_n$, and  
$$ \sup_{ \boldsymbol{\sigma} \in \mathcal{N}_{n}} \Bigg \rvert \dfrac{\mathcal{H}_n(\mathbf{X}^n;\boldsymbol{\sigma})}{ \inf_{ \boldsymbol{\sigma}^{\prime} \in \mathcal{S}_{n}} \mathcal{H}_n(\mathbf{X}^n;\boldsymbol{\sigma}^{\prime})} - 1\Bigg \rvert \leq \delta_n.$$
\end{thm}
In other words, there are infinitely many states which are significantly away from each other and have near-minimal energy (i.e., their energies are a small proportion away from the ground state energy). We refer to \cite[Section 4.3]{chabook} for the proofs and more details about the above two theorems.

At first glance, the conclusions of \Cref{sktight2} and \Cref{sk:mvp} seem to contradict each other. Upon a closer inspection, we observe that \Cref{sk:mvp} looks at near-ground states in a window of length $\bigO ( \rvert  \inf_{ \boldsymbol{\sigma}^{\prime} \in \mathcal{S}_{n}} \mathcal{H}_n(\mathbf{X}^n;\boldsymbol{\sigma}^{\prime})  \rvert  ) = \bigO(n^{3/2})$, as alluded by \Cref{eq1}; and concludes that there are infinitely many near-ground states in this window which are drastically different from each other. On the other hand, \Cref{sktight2} looks at near-ground states in a window of length $\bigO(\sqrt{n})$ and says that there are at most finitely many very different near-ground states in this shorter window. In the same spirit, we can compare the implication of \Cref{sk:chaos} to those of \Cref{skthm1} and \Cref{skthm2}. First of all, \Cref{sk:chaos} only deals with one perturbed system and compares its energy landscape with the original one, whereas \Cref{skthm1} and \Cref{skthm2} compare ground states of a large number of perturbed systems. We can measure the size of the perturbations in terms of the $L^2$-distance. In case of \Cref{sk:chaos} and small $\varepsilon >0$, this distance is 
\begin{align*}
\sum_{i<j} \mathbb{E} [( X_{n,(i,j)} - X_{n,(i,j)}^{\varepsilon})^2] &= n(n-1)(1-e^{-\varepsilon})  \approx n(n-1) \varepsilon \gg \dfrac{n^2}{\log n},
 \end{align*}
if $\varepsilon \gg 1/\log n$; whereas for \Cref{skthm1} and \Cref{skthm2} these distances are $2$ and $2(n-1)$ respectively. Thus, under smaller perturbations in \Cref{skthm1} and \Cref{skthm2}, the new perturbed ground states are generally close to the original one (in the sense that there are only a finite number of states which appear as ground states in the various perturbed systems), whereas larger perturbations in \Cref{sk:chaos} force the new energy landscape to be completely different.  Also, \Cref{sk:chaos} and \Cref{sk:mvp} require the disorders to be Gaussian. Although the methods applied in \cite{chabook} to prove these results might be extended to non-Gaussian cases, the scope remains unclear. In contrary, the results in this paper apply to almost all standard distributions. 


\subsection{The Edwards--Anderson model} 
The Edwards--Anderson Model is a model of spin glass where the spins are arranged on a $d$-dimensional lattice with only nearest neighbour interactions. Introduced by Edwards and Anderson in their seminal work~\cite{eamodel}, this model can be formally described as follows. Consider a system in a box $\Lambda_n \subset \mathbb{Z}^d$ made of interacting spins $\boldsymbol{\sigma} = (\sigma_i)_{i \in \Lambda_n}$ with $\sigma_i \in \{+1,-1\}$. The Edwards--Anderson (EA) model is defined by the Hamiltonian
\begin{equation}{\label{eahmaildef}}
\mathcal{H}_n( \mathbf{X}^n; \boldsymbol{\sigma}) = - \sum_{i,j \in \Lambda_n : \|i-j\|_1=1} X_{n,\{i,j\}}\sigma_i\sigma_j,
\end{equation}
where $\{X_{n,\{i,j\}}: i,j \in \Lambda_n , \|i-j\|_1=1\}$ are mean zero i.i.d.~random variables, the \textit{disorder} of the EA model. Note that the sum in the Hamiltonian is restricted to pairs of nearest neighbor sites, contrary to all pairs of interactions which are present in the Hamiltonian of SK model. This enforces the nature of the EA model to be heavily influenced by the geometry of the lattice box $\Lambda_n$, a phenomenon that is absent in SK model. We refer to \cite{contucci} for more detailed discussions on properties of EA model. Recent work on the near ground states for EA model at zero temperature can be found in \cite{chaea}. Like the SK model, the configuration with lowest energy yields the steady state of the system, in other words the optimization problem of interest, $\mathscr{P}_n$, in this case is 
\begin{equation}{\label{eaprob}}
\text{minimize } \;\;\mathcal{H}_n(\mathbf{X}^n;\boldsymbol{\sigma})=- \sum_{\|i-j\|_1=1} X_{n,\{i,j\}}\sigma_i\sigma_j, \; \text{over } \boldsymbol{\sigma} \in \{+1,-1\}^{\Lambda_n}.
\end{equation}
The domain of the spins, $\Lambda_n$, is generally taken to be boxes like $\{0,\ldots,n\}^d$ or $\{-n, \ldots, n\}^d$. We need not make such specific choices for $\Lambda_n$. All we shall assume about the sequence of boxes $\{\Lambda_n : n \geq 1\}$ are $(i)\, \Lambda_n$ is finite and connected (in the graph $\mathbb{Z}^d$), $(ii) \, \mathbf{0} \in \Lambda_n$ for all $n \geq 1$ and $(iii) \operatorname{card}(\Lambda_n) \to \infty$ as $n \to \infty$.

Many properties of the SK model are also present in EA model due to the obvious similarity of their Hamiltonians; invariance under $\boldsymbol{\sigma} \to - \boldsymbol{\sigma}$ is one of them. Like SK model, we therefore consider the parameter space
\begin{equation}{\label{eadndef}}
 \mathcal{S}_n := \{\boldsymbol{\sigma}=(\sigma_i)_{i \in \Lambda_n} : \sigma_{\mathbf{0}} =+1, \sigma_i \in \{+1,-1\}, \; \forall \; i \in \Lambda_n \setminus \{\mathbf{0} \} \}.
\end{equation}
Instead of the normalized Hamming distance, we metrize this space with the following distance:
\begin{align}{\label{defd:ea}}
d_n (\boldsymbol{\sigma}, \boldsymbol{\sigma}^{\prime} ) &= \biggl(\dfrac{1}{\operatorname{card}(\Lambda_n)} \sum_{i,j \in \Lambda_n : \|i-j\|_1=1}(\sigma_i\sigma_j-\sigma^{\prime}_i\sigma^{\prime}_j)^2 \biggr)^{1/2}, \notag\\
&\qquad \qquad \qquad \; \forall \; \boldsymbol{\sigma} = (\sigma_i)_{i \in \Lambda_n}, \boldsymbol{\sigma}^{\prime}=(\sigma_i^{\prime})_{i \in \Lambda_n} \in \mathcal{S}_n.
\end{align} 
This is the natural metric which arises from the Hamiltonian for EA model. The same can be said for SK model also albeit with a different scaling, see the metric $d_{2,n}$ for SK model. We were able to demonstrate in the proof of \Cref{sktight1} that $d_{2,n}(\boldsymbol{\sigma},\boldsymbol{\sigma}^{\prime})$ gives us enough information about the normalized Hamming distance between $ \boldsymbol{\sigma}$ and $\boldsymbol{\sigma}^{\prime}$; which enabled us to take the normalized Hamming distance as the metric $d_n$ on the parameter space for SK model. Unfortunately, similar kind of statement is not possible to make for EA model. For example consider $d=2$ and $\Lambda_n = \{-2n, \ldots,2n\}^2$.  Define $\boldsymbol{\sigma} \in \mathcal{S}_n$ as the configuration which assigns spin $+1$ to all the sites. Now for any $\varepsilon \in [0,2]$, define $\boldsymbol{\sigma}^{\varepsilon}$ in the following fashion.
$$ \sigma^{\varepsilon}_{(k,l)} = \begin{cases}
+1, & \text{ if $k$ is even and } 2n-\lfloor 2\varepsilon n\rfloor \leq l \leq 2n, \\
-1, & \text{ if $k$ is odd and } 2n-\lfloor 2\varepsilon n\rfloor \leq l \leq 2n,\\
+1, & \text{ if } 0 \leq k \leq 2n \text{ and } -2n \leq l < 2n-\lfloor 2\varepsilon n\rfloor,\\
-1, & \text{ if } -2n \leq k <0 \text{ and } -2n \leq l < 2n-\lfloor 2\varepsilon n\rfloor,
\end{cases}  
\forall \; k,l \in \{-2n,\ldots,2n\}.
$$
An easy computation yields that 
$$ \dfrac{1}{\operatorname{card}(\Lambda_n)} \sum_{i \in \Lambda_n} \mathbbm{1}(\sigma_i \neq \sigma^{\varepsilon}_i) \sim \dfrac{1}{2}\, \text{ and }  \, \dfrac{1}{\operatorname{card}(\Lambda_n)} \sum_{i,j \in \Lambda_n : \|i-j\|_1=1}(\sigma_i\sigma_j-\sigma^{\varepsilon}_i\sigma^{\varepsilon}_j)^2 \sim 2\varepsilon, \; \text{ as } n \to \infty.$$
Nevertheless, we need to argue that $d_n$, as defined in \Cref{defd:ea}, is a valid metric on $\mathcal{S}_n$ in the context of EA model. It is evident that $d_n$ is at least a pseudo-metric. Now observe that $d_n(\boldsymbol{\sigma},\boldsymbol{\sigma}^{\prime}) =0$ implies that $\sigma_i\sigma_j= \sigma_i^{\prime}\sigma_j^{\prime}$ for all $i,j \in \Lambda_n$ with $\|i-j|_1=1$. Since $\sigma_{\boldsymbol{0}} = \sigma^{\prime}_{\boldsymbol{0}} =1$ and $\Lambda_n$ is connected, we can conclude that $\boldsymbol{\sigma}=\boldsymbol{\sigma}^{\prime}$.

The disorder variables are again the random inputs; hence  $I_n = \{\{i,j \} : i,j \in \Lambda_n, \|i-j\|_1=1\}$. The connected-ness assumption on $\Lambda_n$ and the fact that each lattice point in $\mathbb{Z}^d$ has degree $2d$ guarantees that
\begin{equation}{\label{knbound}}
\operatorname{card}(\Lambda_n) -1 \leq \operatorname{card}(I_n)=:k_n \leq d\operatorname{card}(\Lambda_n).
\end{equation}
 The optimization problem in \Cref{eaprob}  satisfies \Cref{ass:linear} with $H_{n,\{i,j\}}(\boldsymbol{\sigma}) = -\sigma_i\sigma_j$, for all $i,j,\boldsymbol{\sigma}$ and \Cref{item:D1} is satisfied with $\lambda=0, \varsigma_{n,0} \equiv 1$.  We, therefore apply \Cref{genthm2} to get a tightness result for EA model optimization problem.

\begin{prop}{\label{eatight}}
Consider EA model as defined in \Cref{eahmaildef} where the disorder variables have common distribution $\mathcal{P}$ with density $f$ having zero mean and  satisfying \Cref{ass:p}. Then for any $\theta_n = \bigO(1)$, the sequence  $\{P(\mathcal{N}_{n,\theta_n}, d_n, \varepsilon) : n \geq 1\}$ is  tight. 
\end{prop}

\begin{proof}
The proof is almost identical to the proof of \Cref{sktight1}. 
Since 
\[
d_n = \frac{d_{2,n}}{\sqrt{\operatorname{card}(\Lambda_n)}}
\]
is a valid metric, \Cref{unique} and \Cref{ass:meas} are guaranteed by existence of density $f$ with finite mean and finiteness of $\Lambda_n$. The assertion is now obvious after noting that $\tau_n = (\operatorname{card}(\Lambda_n))^{1/2}$ satisfies \Cref{item:D2} and $k_n = \Theta(\operatorname{card}(\Lambda_n))$ by \Cref{knbound}.
\end{proof}

We shall consider two types of perturbations for EA model. In the first case, the disorder associated with one pair of neighboring sites is replaced by an i.i.d.~copy; whereas in the later case we change all the disorders associated with a particular site. For both of the schemes, we are changing uniformly bounded many random inputs (at most $2d$ many), and hence we can apply \Cref{eatight} and \Cref{mostprop} to get the following stability result.

\begin{thm}{\label{eathm}}
Consider EA model as defined in \Cref{eahmaildef} where the disorder variables have common distribution $\mathcal{P}$ with density $f$ having zero mean and  satisfying Assumption~\ref{ass:p}. 
This optimization problem is stable under small perturbations for both of the following choices of perturbation blocks.
\begin{enumerate}
\item $\mathcal{J}_n = \{J_{n,\{i,j\}} : i,j\in \Lambda_n, \|i-j\|_1=1 \}$, where $J_{n,\{i,j\}} = \{\{i,j \}\}$.
\item $\mathcal{J}_n = \{J_{n,i} : i \in \Lambda_n \}$, where $J_{n,i} = \{\{i,j \} : j \in \Lambda_n, \|j-i\|_1=1\}$.
\end{enumerate}
\end{thm}

As we did for the case of SK model in \Cref{skmodel}, we briefly discuss results on chaos in the context of EA model and compare them with our results. We again focus on the situation where the disorder variables are i.i.d.~standard Gaussian. The recent paper \cite{chaea} proves the chaotic nature of ground states in the EA model for two types of perturbations. Fix $\varepsilon \in (0,1)$ small and consider
\begin{enumerate}[label=(\Roman*)]
\item a perturbed disorder system $\{X_{n,\{i,j\}}^{(\varepsilon)}\}$ such that $\{(X_{n,\{i,j\}},X_{n,\{i,j\}}^{(\varepsilon)} ) \}$ is a collection of i.i.d.~bivariate mean-zero Gaussian random vectors with unit  variances and correlation $1-\varepsilon$, or 
\item replacing each disorder variable (simultaneously and independently) with probability $\varepsilon$. 
\end{enumerate}
We then consider the (almost surely unique) ground state of the original and the perturbed systems, denoted by $\boldsymbol{\sigma}$ and $\boldsymbol{\sigma}^{(\varepsilon)}$ respectively. Under chaotic nature of the ground state, we would expect that the overlap 
\[
R_{\varepsilon} := \frac{1}{\operatorname{card}(\Lambda_n)}\sum_{i \in \Lambda_n} \sigma_i \sigma_i^{(\varepsilon)}
\]
is typically small for some small $\varepsilon$. For the sake of simplicity, let us take $\Lambda_n = \left\{-n, \ldots,n\right\}^d$. Then \cite[Corollary 1]{chaea} shows that for any $\varepsilon >0$ and for both types of perturbations as described above, we have $\mathbb{E} R_{\varepsilon}^2 \leq Cn^{-d}\varepsilon^{-d}$, where $C$ is a finite constant depending only on dimension $d$. Thus, $R_{\varepsilon} \approx 0$ with high probability whenever $\varepsilon \gg n^{-1}$. This result does not contradict \Cref{eathm} for the reasons similar to the ones outlined at the end of \Cref{skmodel}. In particular, if $\varepsilon \gg n^{-1}$, the $L^2$ measure of perturbation( as discussed at the end of \Cref{skmodel}) is $\gg n^{d-1}$ for the setup in \cite{chaea}, whereas in \Cref{eathm} the same measure is $\bigO(1)$. For similar results on chaotic ground state for general $\Lambda_n$, one can look at \cite[Theorem 1]{chaea}.


\subsection{Graph optimization problems on weighted complete graphs}
\label{graphcom}

In this subsection we shall analyze a class of optimization problems regarding finite sub-graphs of an weighted complete graph. We shall focus on generic symmetric graph optimization problems where we take a complete graph on a fixed number of points, assign i.i.d.~weights to the edges and find out the shortest graph with these points as vertices among a fixed collection of pre-specified graphs. This can be formulated in the following form.
\begin{defn}[Symmetric graph optimization problem on weighted complete graphs]{\label{symgraph:def}}
A symmetric graph optimization problem on a weighted complete graph has following components:
\begin{enumerate}[label=(GrOpt:\Alph*)]
	\item \label{item:CG1} Number of input points $p_n \in \mathbb{N}$ with $p_n \uparrow \infty$ as $n \to \infty$.  Typically $p_n$ will be $\Theta(n)$. Let $E_n := \{\{i,j\} : i,j \in [p_n], i \neq j \}$ denote the edge set of the complete graph defined on the vertex set $V_n:=[p_n]$. 
	\item \label{item:CG2} A collection of edge-weights, $\{X_{n,e} : e \in E_n\}$ where $X_{n,e}$ is the cost associated with traversing the edge $e$.  
	\item \label{item:CG3} A collection of graphs $\mathcal{G}_n$ with vertex set $[p_n]=\{1,\ldots,p_n\}$. For any graph $G \in \mathcal{G}_n$, its edge set will be denoted by $E(G)$. We shall further assume following two conditions on the collection $\mathcal{G}_n$.
	\begin{enumerate}
	\item $\sup_{n \geq 1} \sup_{G \in \mathcal{G}_n} \operatorname{card}(E(G))/p_n  =: e_{\text{max}} < \infty.$ 
	\item For any $\pi$, permutation of the set $[p_n]$, and any $G \in \mathcal{G}_n$, we introduce the notation $\pi(G)$ to denote the graph with edge set $\{\{\pi(i),\pi(j)\} : \{i,j\} \in E(G) \}$. We assume the collection $\mathcal{G}_n$ to be invariant under permutation in the following sense:
	$$ G \in \mathcal{G}_n \Rightarrow \pi(G) \in \mathcal{G}_n, \; \forall \; \pi \in \operatorname{Symm}(p_n).$$
	\end{enumerate} 

\end{enumerate}
The graph optimization problem can be written as 
$$ \text{Minimize} \; \sum_{e \in E(G)} X_{n,e}, \;\text{ over } G \in \mathcal{G}_n.$$
\end{defn}

The simplest stochastic model of this problem assumes that $\{X_{n,e} : e \in E_n\}$ is an i.i.d.~collection from a probability measure $\mathcal{P}$ supported on $[0, \infty)$. The most well-studied examples in this situation are uniform  distribution on $(0,1)$ and exponential distributions. This simple stochastic model does not guarantee any kind of triangle inequality for the edge-weights. We could have managed to ensure validity of the triangle inequality if instead of choosing the edge-weights at random, we would have chosen the points at random from the Euclidean space and the edge-weights were taken to be Euclidean lengths of those edges. This would give rise to a completely different problem with vastly different properties and techniques to solve it. See \Cref{sec:nonlin} for analysis of such models in our context.

A large number of interesting problems fall into the category  described in \Cref{symgraph:def}; among them most well-understood examples include the traveling salesman problem (TSP), the minimum matching problem (MMP) and the minimum spanning tree (MST). 
For TSP, we have $p_n=n$ and we want to find the shortest Hamiltonian cycle  which passes through all of the $n$ points; i.e. $\mathcal{G}_n$ is collection of all Hamiltonian cycles. For MMP, we have $p_n=2n$ and we want to find the shortest perfect matching, i.e., $\mathcal{G}_n$ is the collection of all perfect matching on those $2n$ vertices. Similarly, for MST, we have $p_n=n$ and our goal is to find the shortest tree containing the $n$ points; in other words $\mathcal{G}_n$ is the collection of all spanning trees. It is easy to check that all of these examples satisfy the assumptions mentioned in \Cref{symgraph:def}. We refer interested readers to \cite{karp,frieze,friezetsp,berry} and references therein for detailed analyses of these particular problems, especially the limiting structure of the optimal graph and limiting behavior of the optimal length as the number of vertices becomes large. 

It is quite evident that the optimization problem described in \Cref{symgraph:def} satisfies \Cref{ass:linear} with $\mathcal{S}_n = \mathcal{G}_n$, $I_n = E_n, k_n = \operatorname{card}(I_n)=p_n(p_n-1)/2$ and $H_{n,e}(G) := \mathbbm{1}(e \in E(G))$ for all $e \in E_n, G \in \mathcal{G}_n$. Thus the objective function becomes
$$ \psi_n((x_{n,e})_{e \in E_n}; G) = \sum_{e \in E_n} x_{n,e}H_{n,e}(G) = \sum_{e \in E(G)} x_{n,e}.$$
Moreover, \Cref{item:D1} is satisfied for $\lambda=0$ and $\varsigma_{n,0} \equiv 1$. We endow the space $\mathcal{G}_n$ with metric $d_n$ which will denote the normalized Hamming distance, i.e.
$$ d_n(G_1,G_2) := \dfrac{1}{p_n} \sum_{e \in E_n} \rvert \mathbbm{1}(e \in E(G_1)) - \mathbbm{1}(e \in E(G_2))  \rvert = \dfrac{\operatorname{card}( E(G_1) \Delta E(G_2) )}{p_n}, \; \forall \; G_1, G_2 \in \mathcal{G}_n.$$
The choice of the normalizing constant is made in such a way that $0 \leq d_n \leq 2e_{\text{max}}$, as follows from \Cref{item:CG3}. With these notations having been introduced, we apply \Cref{genthm2} to get a tightness result for near optimal solution of symmetric graph optimization problem on weighted complete graph.

\begin{prop}{\label{tightgraphopt}}
Consider the optimization problem in Definition~\ref{symgraph:def} where the edge-weight variables are i.i.d.~with common distribution $\mathcal{P}$ with density $f$ supported on $[0, \infty)$, having finite mean and  satisfying \Cref{ass:p}. Then for any $\varepsilon >0$, the sequence $\{P(\mathcal{N}_{n,\theta_n}, d_n, \varepsilon) : n \geq 1\}$ is  tight if $\theta_n=\bigO(1/p_n)$. 
\end{prop} 

\begin{proof}
For this problem, the canonical pseudo-metric $d_{2,n}$ can be evaluated as follows: for any $G_1, G_2 \in \mathcal{G}_n$,
\begin{equation}{\label{d2n&dn}}
 d_{2,n}(G_1,G_2)^2 = \sum_{e \in E_n} ( \mathbbm{1}(e \in E(G_1))-\mathbbm{1}(e \in E(G_2)))^2 =  p_nd_n(G_1,G_2).
\end{equation}
It shows that $d_{2,n}$ is a proper metric. Therefore, existence of densities, non-negativity and integrability for the edge-weight variables guarantee that \Cref{unique} and \Cref{ass:meas} are satisfied. 
We have already discussed that \Cref{item:D1} is satisfied for $\lambda=0$ and $\varsigma_{n,0} \equiv 1$. On the other hand, the expression in \Cref{d2n&dn} shows that  $\ref{item:D2}$ is satisfied with $\tau_n = \sqrt{p_n}$. Since $k_n=p_n(p_n-1)/2 \sim p_n^2/2$, applying \Cref{genthm2}, we can now complete the proof.
\end{proof}

\Cref{tightgraphopt} does not require the ``invariance under permutation" assumption we made about $\mathcal{G}_n$ in \Cref{item:CG3}. Under that assumption, $(\pi,G) \mapsto \pi(G)$ defines a group action on the set $\mathcal{G}_n$ with respect to the group $\operatorname{Symm}(p_n)$ and hence for any $\pi \in \operatorname{Symm}(p_n)$ we have 
\begin{equation}{\label{same}}
\mathcal{G}_n = \{G : G \in \mathcal{G}_n\} = \{\pi(G) : G \in \mathcal{G}_n\}.
\end{equation}
The following lemma is a by-product of this observation.

\begin{lmm}{\label{probbound}}
Consider the optimization problem  in Definition~\ref{symgraph:def} where the edge-weight variables are i.i.d.~with common distribution $\mathcal{P}$ having density $f$. Let $\widehat{G}_n(\mathbf{X}^n)$ is the optimal graph for edge-weight input $\mathbf{X}^n = (X_{n,e})_{e \in E_n}$. Then for any $e \in E_n$, we have $\mathbb{P} ( e \in E(\widehat{G}_n(\mathbf{X}^n))) \leq 2e_{\text{max}}/(p_n-1).$
\end{lmm}

\begin{proof}
Existence of the density $f$ guarantees that $\widehat{G}_n(\mathbf{X}^n)$ is uniquely defined. Since edge-weights are independent and identically distributed, it is easy to observe that for any $\pi \in \operatorname{Symm}(p_n)$,
$$ \{ \psi_n(\mathbf{X}^n; G) = \sum_{e \in E(G)} X_{n,e} : G \in \mathcal{G}_n\} \stackrel{d}{=} \biggl \{ \psi_n(\mathbf{X}^n; \pi(G)) = \sum_{e \in E(\pi(G))} X_{n,e} : G \in \mathcal{G}_n\biggr\},$$
and hence by \Cref{same} we obtain for any $G^* \in \mathcal{G}_n$,
\begin{align}
\mathbb{P} (\widehat{G}_n(\mathbf{X}^n)  = G^* ) &= \mathbb{P} ( \psi_n(\mathbf{X}^n; G^*) \leq \psi_n(\mathbf{X}^n; G), \; \forall \; G \in \mathcal{G}_n) \nonumber \\
 &= \mathbb{P} ( \psi_n(\mathbf{X}^n; \pi(G^*)) \leq \psi_n(\mathbf{X}^n; \pi(G)), \; \forall \; G \in \mathcal{G}_n) \nonumber \\
 & = \mathbb{P} ( \psi_n(\mathbf{X}^n; \pi(G^*)) \leq \psi_n(\mathbf{X}^n; G), \; \forall \; G \in \mathcal{G}_n) = \mathbb{P} (\widehat{G}_n(\mathbf{X}^n)  = \pi(G^*) ). \label{same2}
\end{align} 
Fix $e \in E_n$ and for any $e^{\prime} \in E_n$, let $\pi_{e,e^{\prime}}$ be a permutation on $[p_n]$ which takes endpoints of $e$ to endpoints of $e^{\prime}$; hence $e \in E(G)$ if and only if $e^{\prime} \in E(\pi_{e,e^{\prime}}(G))$, for any $G \in \mathcal{G}_n$. Combining this observation with \Cref{same2} and the fact that for any fixed $\pi$ the map $G \mapsto \pi(G)$ is a bijection, we obtain the following:
\begin{align*}
\mathbb{P} ( e \in E(\widehat{G}_n(\mathbf{X}^n))) &= \sum_{G : E(G) \ni e} \mathbb{P} (\widehat{G}_n(\mathbf{X}^n)   = G )  = \sum_{G : E(G) \ni e} \mathbb{P} (\widehat{G}_n(\mathbf{X}^n)  = \pi_{e,e^{\prime}}(G) ) \\
& = \sum_{G : E(\pi_{e,e^{\prime}}(G)) \ni e^{\prime}} \mathbb{P} (\widehat{G}_n(\mathbf{X}^n)  = \pi_{e,e^{\prime}}(G) ) \\
& = \sum_{G^{\prime} : E(G^{\prime}) \ni e^{\prime}} \mathbb{P} (\widehat{G}_n(\mathbf{X}^n)  =G^{\prime} ) = \mathbb{P} ( e^{\prime} \in E(\widehat{G}_n(\mathbf{X}^n))).
\end{align*}
Therefore,
\begin{align*}
\operatorname{card}(E_n)\mathbb{P} ( e \in E(\widehat{G}_n(\mathbf{X}^n))) = \sum_{e^{\prime} \in E_n} \mathbb{P} ( e^{\prime} \in E(\widehat{G}_n(\mathbf{X}^n))) &= \mathbb{E} \sum_{e^{\prime} \in E_n} \mathbbm{1}( e^{\prime} \in E(\widehat{G}_n(\mathbf{X}^n)))\\
& = \mathbb{E} \operatorname{card}( E(\widehat{G}_n(\mathbf{X}^n))) \leq e_{\text{max}} p_n.
\end{align*}
This completes the proof.
\end{proof}

We now consider the perturbation scheme which replaces one edge-weight variable at a time by an i.i.d.~copy. We get the following stability result for the symmetric graph optimization problem with the aid of \Cref{probbound}.

\begin{thm}{\label{graphopt:thm}}
Consider the optimization problem in Definition~\ref{symgraph:def} where the edge-weight variables are i.i.d.~with common distribution $\mathcal{P}$ with density $f$ supported on $[0, \infty)$, having finite mean and  satisfying \Cref{ass:p}. This optimization problem is stable under small perturbations with perturbation blocks $\mathcal{J}_n = \{J_{n,e} : e \in E_n\}$ where $J_{n,e} = \{e\}$ for all $e \in E_n$. 
\end{thm}

\begin{proof}
Recall the notation $\mathbf{X}^n = \{X_{n,e} : e \in E_n\}$ which denote the collection of original edge-weight variables and for any $e^{\prime} \in E_n$, we denote by $\mathbf{X}^n_{e^{\prime}} = \{X^{(e^{\prime})}_{n,e} : e \in E_n\} $ the edge-weight variable collection which is obtained from $\mathbf{X}^n$ after replacing $X_{n,e^{\prime}}$ by an i.i.d.~copy $X_{n,e^{\prime}}^{(e^{\prime})}$. Our goal is to prove \Cref{most} and then apply \Cref{strat} to conclude the proof.
Observe that for any $e^{\prime} \in E_n$,  
\begin{align}
\mathbb{E} [\psi_n(\mathbf{X}^n;\widehat{G}_n(\mathbf{X}^n_{e^{\prime}})) - \psi_{n,\mathrm{opt}}(\mathbf{X}^n) ]&
=\mathbb{E} [\psi_n(\mathbf{X}^n;\widehat{G}_n(\mathbf{X}^n_{e^{\prime}})) - \psi_n(\mathbf{X}^n_{e^{\prime}};\widehat{G}_n(\mathbf{X}^n_{e^{\prime}})) ] \nonumber \\
& =\mathbb{E} \sum_{e \in E_n} (  X_{n,e} \mathbbm{1}( e \in \widehat{G}_n(\mathbf{X}^n_{e^{\prime}})) - X_{n,e}^{(e^{\prime})} \mathbbm{1}( e \in \widehat{G}_n(\mathbf{X}^n_{e^{\prime}})) ) \nonumber \\
& = \mathbb{E} [ ( X_{n,e^{\prime}} - X_{n,e^{\prime}}^{e^{\prime}}) \mathbbm{1}(e^{\prime} \in \widehat{G}_n(\mathbf{X}^n_{e^{\prime}})) ] \nonumber \\
& \leq \mathbb{E}[  X_{n,e^{\prime}} \mathbbm{1}(e^{\prime} \in  \widehat{G}_n(\mathbf{X}^n_{e^{\prime}})) ] \label{eq4} \\
& = \mathbb{E} X_{n,e^{\prime}} \mathbb{P}( e^{\prime} \in \widehat{G}_n(\mathbf{X}^n_{e^{\prime}})) \leq  \dfrac{2e_{\text{max}}\mathbb{E}_{\mathcal{P}}X}{(p_n-1)} = \bigO(1/p_n), \nonumber
\end{align}
where \Cref{eq4} follows from the observation that $X_{n,e^{\prime}} \perp\!\!\!\perp \mathbf{X}^n_{e^{\prime}}$; which implies that $X_{n,e^{\prime}} \perp\!\!\!\perp \widehat{G}_n(\mathbf{X}^n_{e^{\prime}})$. The final inequality follows from \Cref{probbound} and the fact that $\widehat{G}_n(\mathbf{X}^n_{e^{\prime}}) \stackrel{d}{=}\widehat{G}_n(\mathbf{X}^n).$ We can now complete the proof by applying \Cref{tightgraphopt}, \Cref{rem:strateasy} and \Cref{strat} for the choice $\theta_n=1/p_n$. 
\end{proof}
 
 \subsection{The random assignment problem}
 \label{ram}
 Consider the task of choosing an assignment of $n$ jobs to $n$ machines in order to minimize the total cost of performing the $n$ jobs. The basic input of the problem is $\{X_{n,(i,j)} : 1 \leq i ,j \leq n\}$, where $X_{n,(i,j)}$ is the cost of performing job $i$ on machine $j$. The \textit{assignment problem} is to determine an assignment that minimizes total cost, i.e.,
 \begin{equation}
 \label{ram:model}
 \text{Minimize } \; \sum_{i=1}^n X_{n,(i,\pi(i))}, \; \text{ over } \pi \in \operatorname{Symm}(n).
 \end{equation} 
 This problem can also viewed as the problem of finding the shortest perfect matching in a bipartite graph with $2n$ vertices with $X_{n,(i,j)}$ denoting the cost associated with traversing the edge connecting the $i$-th vertex of first group to $j$-th vertex of the second group. The simplest stochastic model we shall be considering here assumes that the cost variables $\{X_{n,(i,j)} : i,j \in [n]\}$ are i.i.d.~random variables from probability measure $\mathcal{P}$ supported on $[0, \infty)$ with $\mathcal{P}=$ Uniform distribution on $[0,1]$ and $\mathcal{P}=$ Exponential distribution with rate $1$ being the most well-studied examples. We refer to~\cite{coppersmith} and \cite{steele} for detailed history and description of these analyses. A pioneering work is due to Aldous~\cite{aldous}, who gave a rigorous proof of the limiting value of the expected optimal cost and limit distribution of costs and their rank orders in the optimal assignment. The following theorem summarizes the results on the expected value of the optimal cost.
 
 \begin{thm}[\cite{aldous, balaji, wastlund}]{\label{parisi}}
 Consider the random assignment problem as defined in \Cref{ram:model} with cost distribution $\mathcal{P}$ being Exponential distribution with rate $1$. Then
 \begin{equation}{\label{parisiconj}}
 \mathbb{E} \left[\min_{\pi \in \operatorname{Symm}(n)} \sum_{i=1}^n X_{n,(i,\pi(i))} \right]= \sum_{i=1}^n i^{-2}.
 \end{equation}
Hence, 
$$ \lim_{n \to \infty} \left[ \mathbb{E} \min_{\pi \in \operatorname{Symm}(n)} \sum_{i=1}^n X_{n,(i,\pi(i))} \right] = \sum_{i=1}^{\infty} i^{-2} = \zeta(2) = \dfrac{\pi^2}{6}.$$
 \end{thm}
 
The exact formula in \Cref{parisiconj} was conjectured by \citet{parisi} and later proved, along with some related variants of this conjecture, by \citet{wastlund} and \citet*{balaji}. In his seminal paper, \citet{aldous} proved the $\zeta(2)$ limit, without proving \textit{Parisi conjecture}, and also argued that the limit only depends on the cost distribution through its density at $0$ (assuming it exists). 
 
 
 This optimization problem bears many similarities, both in structure and analysis, to the problem discussed in \Cref{graphcom}. We shall therefore omit those details in this section. It is easy to observe that the optimization problem described in \Cref{ram:model} satisfies \Cref{ass:linear} with $\mathcal{S}_n$ being $\operatorname{Symm}(n)$; $I_n$ = $\{(i,j) : i,j \in [n]\}, k_n = \operatorname{card}(I_n)=n^2$ and $H_{n,(i,j)}(\pi) := \mathbbm{1}(\pi(i)=j)$ for all $i,j \in [n], \pi \in \operatorname{Symm}(n)$. The objective function becomes
 $$ \psi_n((x_{n,(i,j)})_{i,j \in [n]}; \pi) = \sum_{(i,j) \in I_n} x_{n,(i,j)}H_{n,(i,j)}(\pi) = \sum_{i=1}^n x_{n,(i,\pi(i))}.$$
 We endow the space $\mathcal{S}_n$ with metric $d_n$ that will denote the normalized Hamming distance on $\operatorname{Symm}(n)$, i.e.,
 $$ d_n(\pi_1,\pi_2) := \dfrac{1}{n} \sum_{i=1}^n \mathbbm{1}(\pi_1(i) \neq \pi_2(i) ), \; \forall \; \pi_1, \pi_2 \in \operatorname{Symm}(n).$$
 With these notations having been introduced, we apply \Cref{genthm2} to get a tightness result for near optimal solution of random assignment problem.
 
 \begin{prop}{\label{tightram}}
 Consider the random assignment problem defined in equation~\eqref{ram:model}, where the cost variables have common distribution $\mathcal{P}$ with density $f$ supported on $[0, \infty)$, having finite mean and  satisfying Assumption~\ref{ass:p}. Then for any $\varepsilon >0$, the sequence $\{P(\mathcal{N}_{n,\theta_n}, d_n, \varepsilon) : n \geq 1\}$ is  tight if $\theta_n=\bigO(1/n)$. 
 \end{prop} 
 
 \begin{proof}
 The proof is very similar to the proof of \Cref{tightgraphopt}. First of all, \Cref{item:D1} is satisfied with $\lambda=0, \varsigma_{n,0} \equiv 1$. Since  
 for any $\pi_1, \pi_2 \in \operatorname{Symm}(n)$,
 $$ d_{2,n}(\pi_1,\pi_2)^2 = \sum_{i,j \in [n]} ( \mathbbm{1}(\pi_1(i)=j)-\mathbbm{1}(\pi_2(i)=j))^2 = 2 \sum_{i=1}^n \mathbbm{1}( \pi_1(i) \neq \pi_2(i)) = 2nd_n(\pi_1,\pi_2),$$
 \Cref{item:D2} is  satisfied for $\tau_n=\sqrt{2n}$. Recalling that $k_n= n^2$ here and applying \Cref{genthm2}, we can now complete the proof.
 \end{proof}

\begin{remark}
 In light of \Cref{tightram}, it is instructive to discuss the AEU result proved by \citet{aldous}. We restrict our attention to standard exponential distribution for cost variables (which is not strictly necessary, as suggested in \cite{aldous}) and let $\pi_n \in \mathrm{Symm}(n)$ be the optimal assignment. The AEU result, as described in \cite[Theorem 4]{aldous}, states that for any $\delta \in (0,1)$ there exists $\varepsilon >0$ such that for any other (random) assignment $\mu_n \in \mathrm{Symm}(n)$ satisfying $\mathbb{E} d_n(\pi_n,\mu_n) \geq \delta$ for all $n \geq 1$, we have
 $$ \liminf_{n \to \infty} \mathbb{E} \sum_{i=1}^n X_{n,(i,\mu_n(i))} \geq \dfrac{\pi^2}{6} + \varepsilon.$$
 In other words, any (random) assignment which is a positive distance away from the optimal one, has the expected cost associated with it to be at least some positive margin greater than the minimum cost; thus the optimal assignment is effectively unique. This result suggests that for any $\varepsilon >0$ and $\theta_n = o(1)$, the sequence of random variables $P(\mathcal{N}_{n,\theta_n},d_n,\varepsilon)$ converges weakly to $1$ as $n \to \infty$. While, this is not a direct corollary of \cite[Theorem 4]{aldous}, this assertion might be proved using the same techniques as used in \cite{aldous}. This assertion is stronger than our conclusion in \Cref{tightram}, although \Cref{tightram} is enough for our purpose of proving stability.
 \end{remark}
 
 
 \begin{thm}{\label{ramthm}}
 Consider random assignment problem as defined in equation~\eqref{ram:model}, where the cost variables have common distribution $\mathcal{P}$ with density $f$ supported on $[0, \infty)$, having finite mean and  satisfying \Cref{ass:p}. 
 This optimization problem is stable under small perturbations with perturbation blocks $\mathcal{J}_n = \{J_{n,(i,j)} : (i,j) \in [n]^2\}$ where $J_{n,(i,j)} = \{(i,j)\}$ for all $(i,j) \in [n]^2$.
 \end{thm}
 
 \begin{proof}
 Recall the notation $\mathbf{X}^n = \{X_{n,(i,j)} : i,j \in [n]\}$ which denote the collection of original cost variables and for any $k,l \in [n]$, we denote by $\mathbf{X}^n_{(k,l)} =\{X^{(k,l)}_{n,(i,j)} : i,j \in [n]\} $ the cost variable collection which is obtained from $\mathbf{X}^n$ after replacing $X_{n,(k,l)}$ by an i.i.d.~copy $X_{n,(k,l)}^{(k,l)}$. 
 A calculation similar to the one done in the proof of \Cref{graphopt:thm} establishes that
 \begin{align*}
 \mathbb{E} [\psi_n(\mathbf{X}^n;\widehat{\pi}_n(\mathbf{X}^n_{(k,l)})) - \psi_{n,\mathrm{opt}}(\mathbf{X}^n)]&
 =\mathbb{E} [\psi_n(\mathbf{X}^n;\widehat{\pi}_n(\mathbf{X}^n_{(k,l)})) - \psi_n(\mathbf{X}^n_{(k,l)};\widehat{\pi}_n(\mathbf{X}^n_{(k,l)})) ] \\
 & \leq \mathbb{E} [  X_{n,(k,l)} \mathbbm{1}( \widehat{\pi}_n(\mathbf{X}^n_{(k,l)})(k)=l) ] \\
 & = \mathbb{E} X_{n,(k,l)} \mathbb{P}( \widehat{\pi}_n(\mathbf{X}^n_{(k,l)})(k)=l) = (\mathbb{E}_{\mathcal{P}}X)/n.
 \end{align*}
 The last equality follows from the fact that, the cost variables being i.i.d., the symmetry between different jobs and different machines guarantee that $\widehat{\pi}_n(\mathbf{X}^n)$ is uniformly distributed on $\operatorname{Symm}(n)$; therefore $ \widehat{\pi}_n(\mathbf{X}^n)(k) \sim \text{Uniform}([n])$. We can now complete the proof by applying \Cref{tightram}, \Cref{rem:strateasy} and \Cref{strat} with the choice $\theta_n=1/n$. 
 \end{proof}

\subsection{Wigner matrices}
\label{eigmatrix}

Consider an  $n \times n$ real symmetric random matrix $\mathbf{M}_n = (X_{n,ij})_{i,j \in [n]}$ with $X_{n,ij}$ denoting the entry in $i$-th row and $j$-th column. Since it is  symmetric, we can guarantee that all of its eigenvalues are real and the problem of finding the smallest eigenvalue can be re-formulated as 
\begin{equation}{\label{eign:def}}
\text{ Minimize } \psi_n(\mathbf{M}_n; \boldsymbol{v}) := \boldsymbol{v}^{\top} \mathbf{M}_n\boldsymbol{v} = \sum_{i,j \in [n]} X_{n,ij}v_iv_j \;\; \text{over } \boldsymbol{v}=(v_1, \ldots,v_n) \in \mathscr{S}^{n-1},
\end{equation}
where $\mathscr{S}^{n-1}$ denotes the unit sphere in $\mathbb{R}^n$. A basic stochastic model for the entries in $M_n$ can be taken as follows: $\{X_{n,ij}:i,j \in [n], i<j\}$ and $\{X_{n,ii} : i \in [n]\}$ are two independent and identically distributed collections of mean-zero stochastic variables. Such a matrix is called a \textit{Wigner matrix} and if the entries are Gaussian then it is called \textit{Gaussian Wigner matrix}. In standard literature, the matrices are generally scaled by $\sqrt{n}$, but we shall not concern ourselves with that since we will be considering only near optimal vectors. The asymptotics for the empirical distribution of the eigenvalues of Wigner matrices were derived by Wigner~\cite{wigner} and can be considered as the starting point of random matrix theory. For a detailed account of  results on asymptotic behavior of the smallest eigenvalues of Wigner matrices, we refer to \cite{anderson} and references therein.

We shall assume that the entire collection  $\{X_{n,ij}:i,j \in [n], i \leq j\}$  is independent and identically distributed. Indeed the analysis of the general case is very similar and the results which we shall write in this section still apply. Nevertheless, since the general case does not fall into the premise presented in \Cref{stochopt} directly, we shall contend ourselves with this restricted scenario. 

The objective function expressed in \Cref{eign:def} is very similar in structure to the objective function of SK model in \Cref{skprob} and hence most of the computations follow in the exactly same fashion; in order to avoid repetition we shall omit many details in this section. We start by observing that $\boldsymbol{v}^{\top} \mathbf{M}_n\boldsymbol{v}$ is invariant under reflection through the origin,  i.e., $\boldsymbol{v}^{\top} \mathbf{M}_n\boldsymbol{v} = (-\boldsymbol{v})^{\top} \mathbf{M}_n(-\boldsymbol{v})$. To get a unique solution almost surely we therefore consider the following parameter space where we only keep those vectors whose first  co-ordinate is non-negative.
$$ \mathcal{S}_n = \{\boldsymbol{v}=(v_1, \ldots,v_n) : \boldsymbol{v} \in \mathscr{S}^{n-1},  v_1 \geq 0 \}.$$
We take $d_n$ to be the Euclidean distance. The upper triangular entries of the matrix are the random inputs of the problem; hence $I_n = \{(i,j) : 1 \leq i \leq j \leq n\}$ and $k_n=n(n+1)/2.$ The optimization problem in \Cref{eign:def} clearly satisfies \Cref{ass:linear} with 
$$ H_{n,(i,j)}(\boldsymbol{v}) = \begin{cases}
2v_iv_j, & \text{ if } i< j,\\
v_i^2, & \text{ if } i=j,
\end{cases}$$
 for all $1 \leq i \leq j\leq n$ and $\boldsymbol{v}$. We try to apply \Cref{genthm:g3} to get a tightness result for this optimization problem. Towards achieving that goal, we need a nice estimate for the expectation of the smallest eigenvalue of $\mathbf{M}_n$ so that the condition in \Cref{item:G1} is satisfied. This is obtained with the help of following result. For any $n \times n$ matrix $A_n$, we shall denote its eigenvalues by the following notation : $\lambda_1(A_n) \leq \cdots \leq \lambda_n(A_n)$.
 
 \begin{prop}{\label{eigen:expecbound}}
 Consider the random Wigner matrix $\textbf{M}_n = (X_{n,ij})_{i,j \in [n]}$ with upper triangular entries being i.i.d.~with common distribution $\mathcal{P}$ having mean zero and variance $\sigma^2 < \infty$. Further assume that there exists a constant $\beta \in (0, \infty)$ such that $\mathbb{E}_{\mathcal{P}} |X|^k \leq \sigma^{k}k^{\beta k}$, for all positive integers $k$. Then there exists another constant $C(\beta)$, depending on $\beta$ only, such that  
 $$ \mathbb{E} \lambda_1(\mathbf{M}_n) \geq -C(\beta) \sigma \sqrt{n}, \; \forall \; n \geq 1.$$ 
 Moreover, $n^{-1/2}\lambda_1(\mathbf{M}_n)$ converges in probability to $-2\sigma$ as $n \to \infty$. 
 \end{prop}
 
 \begin{proof}
Without loss of generality we can assume that $\sigma =1$; otherwise we can work with the Wigner ensemble $\sigma^{-1}\mathbf{M}_n$.  The proof of the part regarding convergence in probability can be found in \cite[Lemma 2.1.23]{anderson}. We shall follow the same technique as that proof to arrive at the lower bound on the expectation. From \cite[Equation 2.1.32]{anderson}, we can conclude the existence of a constant $c_1=c_1(\beta)>0$ such that for all $k < n$, we have
 \begin{equation}{\label{eigenbound1}}
 \dfrac{1}{n} \mathbb{E} \biggl[ \sum_{i=1}^n \biggl( \dfrac{\lambda_i(\mathbf{M}_n)}{\sqrt{n}} \biggr)^{2k}\biggr] \leq 4^k \sum_{i=0}^k \biggl( \dfrac{k^{c_1}}{n} \biggr)^i.
 \end{equation} 
Take $k = k_n = (\log n)^2$ in \Cref{eigenbound1}. Since $k_n,k_n^{c_1} < n/2$ for all $n \geq N=N(\beta) \in \mathbb{N}$, we can conclude that
\begin{align*}
\mathbb{E} (\lambda_1(\mathbf{M}_n)^{2k_n}) \leq \mathbb{E} \biggl[ \sum_{i=1}^n (\lambda_i(\mathbf{M}_n))^{2k_n}\biggr] \leq  4^{k_n} n^{k_n+1}\sum_{i=0}^{k_n} \biggl( \dfrac{k_n^{c_1}}{n} \biggr)^i \leq 4^{k_n} n^{k_n+1}\sum_{i=0}^{k_n} 2^{-i} \leq 2n^{k_n+1} 4^{k_n}, 
\end{align*}
for all $n \geq N$. Therefore, by Jensen's inequality, 
$$ \mathbb{E} \big \rvert  \lambda_1(\mathbf{M}_n) \big \rvert \leq [ \mathbb{E} (\lambda_1(\mathbf{M}_n)^{2k_n}) ]^{1/2k_n} \leq 2\sqrt{n} (2n)^{1/2k_n}. $$
Since $\log (2n)^{1/2k_n} = (\log 2n)/(2(\log n)^2) \to 0$, as $n \to \infty$, we conclude that 
$$ \limsup_{n \to \infty}  [ -n^{-1/2}\mathbb{E}  \lambda_1(\mathbf{M}_n) ] \leq \limsup_{n \to \infty} n^{-1/2}\mathbb{E} \big \rvert  \lambda_1(\mathbf{M}_n) \big \rvert \leq 2.$$ 
This completes the proof.
 \end{proof}
   
\begin{remark}{\label{eigen:remark}}
From Proposition~\cref{subg:moments} and Stirling's approximation, one can observe that any mean-zero sub-Gamma variables with parameters $(1,1)$ satisfies the hypotheses of \Cref{eigen:expecbound} and hence the choice $M_n = C \sqrt{n}$ satisfies \Cref{item:G1} for some finite constant $C$.  
\end{remark}

With the aid of \Cref{genthm:g3} and \Cref{eigen:remark}, we can arrive at the tightness result for the optimization problem described in \Cref{eign:def}.

\begin{prop}{\label{eigtight1}}
Consider the optimization problem defined in equation~\eqref{eign:def}, where the upper triangular entries are i.i.d.~with common distribution $\mathcal{P}$ with density $f$ having zero mean and  satisfying Assumption~\ref{mu} for the pair $(\rho,g)$ with $g^{(1)}$ bounded away from $0$. Then for any $\varepsilon >0$, the sequence $\{P(\mathcal{N}_{n,\theta_n}, d_n, \varepsilon) : n \geq 1\}$ is  tight whenever $\theta_n=\bigO(1/\sqrt{n})$. 
\end{prop}

\begin{proof}
Thinking of $\mathbf{M}_n$ as a $\mathbb{R}^{n(n+1)}$-dimensional random vector with density with respect to Lebesgue measure, we can argue that its $n$ eigenvalues have joint density with respect to Lebesgue measure on $\mathbb{R}^n$ and hence the smallest eigenvalue has dimension of its eigenspace equal to $1$ almost surely. Moreover, it is well-known that, under the assumptions of existence of a density for $\mathbf{M}_n$, the probability of the eigenspace corresponding to the smallest eigenvalue being contained in a $(n-1)$-dimensional subspace of $\mathbb{R}^n$ is $0$ and hence with probability $1$ we have a unique eigenvector (corresponding to the smallest eigenvalue) in the space $\mathcal{S}_n$; this shows that \Cref{unique} is satisfied.  Integrability of the matrix entries now guarantees that \Cref{ass:meas} is satisfied. 
We have already argued that \Cref{item:G1} is satisfied with $M_n = C\sqrt{n}$ for some finite constant $C$, see \Cref{eigen:remark}. Thus it is enough to establish that \Cref{item:G2} holds true with the choice $\tau_n \equiv 1$ to complete the proof. 

The condition on $g$ guarantees existence of $\eta >0$ such that $|g^{(1)}(y)| \geq \eta$ for all $y \in \mathbb{R}$. A small computation yields the following for any $\boldsymbol{u}=(u_1, \ldots,u_n),\boldsymbol{v}=(v_1, \ldots,v_n) \in \mathscr{S}^{n-1}$: 
\begin{align}
d_{g,n}(\boldsymbol{u},\boldsymbol{v})^2 &\geq \eta^2\sum_{i <j} 4(u_iu_j-v_iv_j)^2 + \eta^2 \sum_{i=1}^n (u_i^2 - v_i^2)^2 \notag\\
& \geq \eta^2\sum_{i,j \in [n]}(u_iu_j-v_iv_j)^2  = 2\eta^2 d_{n}(\boldsymbol{u},\boldsymbol{v})^2 ( 1-d_{n}(\boldsymbol{u},\boldsymbol{v})^2/4). \label{fr1}
\end{align}
Using arguments similar to the one employed to prove \Cref{sktight1}, one can deduce from \Cref{fr1} that \Cref{item:G2} is satisfied for  $\tau_n \equiv 1$. This completes the proof.
\end{proof}

\Cref{eigtight1} is rather unsatisfactory in its scope since the assumption that $g^{(1)}$ is bounded away from zero severely restricts the possible choices for input distributions. Although it permits standard distributions like Gaussian and Laplace, it cannot be applied to Wigner ensembles with entry distributions having at least one bounded tail. With more effort the result can indeed be extended to situations where $g$ needs to be only strictly monotonic. Nevertheless, the rationale behind writing a weaker result like \Cref{eigtight1} is to demonstrate the applicability of the general technique derived in this article in the case of Wigner matrices to prove a stability result. We shall use some techniques of matrix computations to arrive at the following stronger version of \Cref{eigtight1}.

\begin{prop}{\label{eigtight2}}
Consider the random Wigner matrix $\mathbf{M}_n = (X_{n,ij})_{i,j \in [n]}$ with upper triangular entries being i.i.d.~with common distribution $\mathcal{P}$ having density $f$,  mean zero and satisfying the following condition. There exists a constant $\beta \in (0, \infty)$ such that $\mathbb{E}_{\mathcal{P}} |X|^k \leq \sigma^{k}k^{\beta k}$, for all positive integers $k$. 
Then for any $\varepsilon >0$, the sequence $\{P(\mathcal{N}_{n,\theta_n}, d_n, \varepsilon) : n \geq 1\}$ is  tight provided $\theta_n=\bigO(n^{-1/2})$.
\end{prop}

\begin{proof}
We start with the spectral decomposition of $\mathbf{M}_n$; $\mathbf{M}_n = \sum_{i=1}^n \lambda_i(\mathbf{M}_n)\boldsymbol{p}_i(M_n)\boldsymbol{p}_i(M_n)^{\top}$, where  $\boldsymbol{p}_i(M_n)$'s are corresponding orthonormal eigenvectors. 
Fix $T \geq 2$. It is clear that for any $\boldsymbol{v}=(v_1, \ldots,v_n) \in \mathscr{S}^{n-1}$, we can write the following:
\begin{align*}
&\boldsymbol{v}^{\top} \mathbf{M}_n \boldsymbol{v} - \inf_{\boldsymbol{u} \in \mathscr{S}^{n-1}} \boldsymbol{u}^{\top} \mathbf{M}_n \boldsymbol{u} \leq \theta_n \iff \sum_{i=1}^n \lambda_i(\mathbf{M}_n) (\boldsymbol{v} \cdot \boldsymbol{p}_i )^2 - \lambda_1(\mathbf{M}_n) \leq \theta_n \\
& \Rightarrow  \lambda_T(\mathbf{M}_n) \sum_{i=T}^n  (\boldsymbol{v} \cdot \boldsymbol{p}_i )^2 + \sum_{i=1}^{T-1} \lambda_i(\mathbf{M}_n) (\boldsymbol{v} \cdot \boldsymbol{p}_i )^2 - \lambda_1(\mathbf{M}_n) \leq \theta_n \\
& \Rightarrow  \lambda_T(\mathbf{M}_n) \biggl( 1- \sum_{i=1}^{T-1}  (\boldsymbol{v} \cdot \boldsymbol{p}_i)^2 \biggr) + \sum_{i=1}^{T-1} \lambda_i(\mathbf{M}_n) (\boldsymbol{v} \cdot \boldsymbol{p}_i )^2  - \lambda_1(\mathbf{M}_n) \leq \theta_n \\
& \Rightarrow \sum_{j=1}^{T-1} ( \lambda_{j+1}(\mathbf{M}_n) - \lambda_{j}(\mathbf{M}_n) ) \biggl(1 -  \sum_{i=1}^{j}  (\boldsymbol{v} \cdot \boldsymbol{p}_i )^2  \biggr) \leq \theta_n,
\end{align*}
and hence
$$ \mathcal{N}_{n,\theta_n} \subseteq\biggl \{\boldsymbol{v} \in \mathscr{S}^{n-1} : \sum_{j=1}^{T-1} ( \lambda_{j+1}(\mathbf{M}_n) - \lambda_{j}(\mathbf{M}_n) ) \biggl(1 -  \sum_{i=1}^{j}  (\boldsymbol{v} \cdot \boldsymbol{p}_i )^2  \biggr) \leq \theta_n\biggr\} =: \mathcal{N}^{\prime}_{n,T,\theta_n}.$$
Since $d_n(\boldsymbol{u},\boldsymbol{v})^2 = \|\boldsymbol{u}-\boldsymbol{v}\|_2^2 = \sum_{i=1}^n (\boldsymbol{u} \cdot \boldsymbol{p}_i-\boldsymbol{v} \cdot \boldsymbol{p}_i)^2$, for any $\boldsymbol{u}, \boldsymbol{v} \in \mathbb{R}^n$, we conclude that the metric space $(\mathcal{N}^{\prime}_{n,T,\theta_n},d_n)$ is isometric to the metric space 
$$\mathcal{N}_{n,T,\theta_n}^*:=\biggl\{\boldsymbol{x}=(x_1, \ldots,x_n) \in \mathscr{S}^{n-1} : \sum_{j=1}^{T-1} ( \lambda_{j+1}(\mathbf{M}_n) - \lambda_{j}(\mathbf{M}_n)) \biggl(1 -  \sum_{i=1}^{j}  x_i^2 \biggr) \leq \theta_n\biggr\},$$
 equipped with the Euclidean distance. Thus it is enough to focus our attention on the random set $\mathcal{N}_{n,T,\theta_n}^*$ for a suitably chosen $T$.  Fix $\varepsilon >0$ and consider the event $A_{n,T} := \{\lambda_{T}(\mathbf{M}_n) - \lambda_{1}(\mathbf{M}_n) > \varepsilon^{-2}\theta_n \}$. On the event $A_{n,T}$, any $\boldsymbol{x} = (x_1, \ldots,x_n) \in \mathcal{N}^*_{n,T,\theta_n}$ satisfies 
 \begin{align*}
 \theta_n \geq  \sum_{j=1}^{T-1} ( \lambda_{j+1}(\mathbf{M}_n) - \lambda_{j}(\mathbf{M}_n) ) \biggl(1 -  \sum_{i=1}^{j}  x_i^2 \biggr) &\geq  \sum_{j=1}^{T-1} ( \lambda_{j+1}(\mathbf{M}_n) - \lambda_{j}(\mathbf{M}_n) ) \biggl(1 -  \sum_{i=1}^{T-1}  x_i^2 \biggr) \\
 & = ( \lambda_{T}(\mathbf{M}_n) - \lambda_{1}(\mathbf{M}_n) ) \biggl(1 -  \sum_{i=1}^{T-1}  x_i^2 \biggr) \\
 & \geq \varepsilon^{-2}\theta_n \biggl(1 -  \sum_{i=1}^{T-1}  x_i^2 \biggr).
 \end{align*}  
 Therefore, on the event $A_{n,T}$, we have 
 \begin{equation}{\label{eigenstep1}}
 \mathcal{N}^*_{n,T,\theta_n} \subseteq \biggl\{\boldsymbol{x} \in \mathscr{S}^{n-1} : \sum_{j=T}^{n}   x_j^2 \leq \varepsilon^2 \biggr\} \subseteq B_{T-1}(\mathbf{0},1) \times B_{n-T+1}(\mathbf{0},\varepsilon).
 \end{equation}
 Here $B_d(\mathbf{0},r)$ is the closed ball of radius $r$ around $\mathbf{0}$ in $\mathbb{R}^d$. It is easy to see that 
 \begin{align}{\label{eigenstep2}}
 N ( B_{T-1}(\mathbf{0},1) \times B_{n-T+1}(\mathbf{0},\varepsilon), d_n,2\varepsilon) &\leq N ( B_{T-1}(\mathbf{0},1), d_{T-1},\varepsilon) N (  B_{n-T+1}(\mathbf{0},\varepsilon), d_{n-T+1},\varepsilon) \nonumber \\
 & = N ( B_{T-1}(\mathbf{0},1), d_{T-1},\varepsilon) \leq \biggl(1+\dfrac{2}{\varepsilon} \biggr)^{T-1},
 \end{align}   
 where the last equality follows from \cite[Corollary 4.2.13]{hdp}. Combining \Cref{eigenstep1} and \Cref{eigenstep2}, with the aid of \Cref{pnp}, we can guarantee that $P(\mathcal{N}_{n,T,\theta_n}^*, d_n, 4\varepsilon) \leq (1+2/\varepsilon)^{T-1}$,  on the event $A_{n,T}$. Therefore,
 \begin{align*}
  \limsup_{n \to \infty} \mathbb{P}( P(\mathcal{N}_{n,\theta_n}, d_n, 4\varepsilon) > (1+2/\varepsilon)^{T-1}) & \leq \limsup_{n \to \infty} \mathbb{P} (P(\mathcal{N}_{n,T,\theta_n}^*, d_n, 4\varepsilon) > (1+2/\varepsilon)^{T-1} ) \\
  & \leq \limsup_{n \to \infty} \mathbb{P} ( A_{n,T}^c),
  \end{align*}
 and hence, to conclude the proof, it is enough to show that $\lim_{T \to \infty} \limsup_{n \to \infty} \mathbb{P} ( A_{n,T}^c) =0$. 
 
 We shall make use of the comparison inequality from \Cref{interlace}. Let $\mathbf{M}_n^{(k)}$ be the random matrix consisting of only the first $k$ rows and columns of $\mathbf{M}_n$. By \Cref{interlace} applied to the symmetric matrix $\mathbf{M}_n$, we have $\lambda_1(\mathbf{M}_n^{(n-k+1)}) \leq \lambda_k(\mathbf{M}_n)$ for all $1 \leq k \leq n$ with equality for the case $k=1$ and hence 
 \begin{align}{\label{gap4}}
 	\sqrt{n}(\lambda_T(\mathbf{M}_n) - \lambda_1(\mathbf{M}_n) )& \geq  \sqrt{n}(\lambda_1(\mathbf{M}_n^{(n-T+1)}) - \lambda_1(\mathbf{M}_n^{(n)}) )  \nonumber \\
 	& \geq  \sum_{k=n-T+2}^{n} \sqrt{k}(\lambda_1(\mathbf{M}_n^{(k-1)}) - \lambda_1(\mathbf{M}_n^{(k)})),
 \end{align} 
where the last inequality in \Cref{gap4} is true since $\lambda_1(\mathbf{M}_n^{(k)})$ is non-increasing in $k$ by \Cref{interlace}.
 We shall now use \Cref{gapestimate} to estimate the difference  $\lambda_1(\mathbf{M}_n^{(k-1)}) - \lambda_1(\mathbf{M}_n^{(k)}) $ for any $n-T+2\leq k \leq n$. Let $\mathbf{X}^{(n,k)} = (X_{n,ki})_{1 \leq i \leq k-1}$ be the vector formed by the first $(k-1)$ entries of the $k$-th row of $\mathbf{M}_n^{(k)}$ and $\boldsymbol{u}_{n,k-1}$ be the orthonormal eigenvector associated with eigenvalue $\lambda_1(\mathbf{M}_n^{(k-1)})$ for the matrix  $\mathbf{M}_n^{(k-1)}$; the vector $\boldsymbol{u}_{n,k-1}$  is unique almost surely due to existence of densities for the matrix entries. Applying \Cref{gapestimate} to the matrix $-\mathbf{M}_n^{(k)}$, we obtain for any $t \in (0,1)$,
 \begin{equation}{\label{gap1}}
 \lambda_1(\mathbf{M}_n^{(k-1)}) - \lambda_1(\mathbf{M}_n^{(k)}) \geq 2\sqrt{t(1-t)}  \rvert (\boldsymbol{u}_{n,k-1})^{\top}\mathbf{X}^{(n,k)}\rvert + t \lambda_1(\mathbf{M}_n^{(k-1)})  - t X_{n,kk}.  
 \end{equation}
Now fix $\upsilon >0$ small enough such that the assertion in \Cref{xuconv} holds true with $\delta = \delta(\upsilon)>0$. Take any $0 < \eta < \min(\delta^2/(25\sigma^2), 3/4)$ and set $t=\eta/k$. Plugging-in this particular choice in \Cref{gap1}, we have, on the event $0>\lambda_1(\mathbf{M}_n^{(k-1)})> -3\sigma\sqrt{k}$, 
 \begin{equation}{\label{gap}}
 	\sqrt{k}(\lambda_1(\mathbf{M}_n^{(k-1)}) - \lambda_1(\mathbf{M}_n^{(k)}) )\geq \sqrt{\eta}\rvert (\boldsymbol{u}_{n,k-1})^{\top}\mathbf{X}^{(n,k)}\rvert -3\sigma \eta - \dfrac{\eta X_{n,kk}}{\sqrt{k}}. 
 \end{equation} 
Let $\Omega_{n,k} := \{ \sqrt{k}(\lambda_1(\mathbf{M}_n^{(k-1)}) - \lambda_1(\mathbf{M}_n^{(k)}) )\geq \sigma \eta\}$.  The estimate in \Cref{gap} yields the following on the event $ 0>\lambda_1(\mathbf{M}_n^{(k-1)})> -3\sigma\sqrt{k}$:
 \begin{align*}
\mathbb{P} ( \Omega_{n,k} \; \rvert \;\mathbf{M}_n^{(k-1)}) & \geq \mathbb{P} ( \rvert (\boldsymbol{u}_{n,k-1})^{\top}\mathbf{X}^{(n,k)}\rvert \geq 5\sigma \sqrt{\eta}, |X_{n,kk}|  \leq \sigma \sqrt{k} \; \rvert \;\mathbf{M}_n^{(k-1)}) \\
&  \geq \mathbb{P} (  \rvert (\boldsymbol{u}_{n,k-1})^{\top}\mathbf{X}^{(n,k)} \rvert \geq \delta \;\rvert \;\mathbf{M}_n^{(k-1)}) \mathbb{P}( |X_{n,kk}|  \leq \sigma \sqrt{k}) \geq \upsilon/2,
 \end{align*}
for all $k \geq \sqrt{n}$ and $n$ large enough. The last inequality is an application of \Cref{xuconv}, since $\mathbf{X}^{(n,k)}$ is independent of $\mathbf{M}_n^{(k-1)}$, whereas $\boldsymbol{u}_{n,k-1}$ is measurable with respect to $\mathbf{M}_n^{(k-1)}$.

Note that the sequence $\{\mathbbm{1}_{\Omega_{n,k}} - \mathbb{P} ( \Omega_{n,k} \; \rvert \;\mathbf{M}_n^{(k-1)}) : n-T+2 \leq k \leq n\}$ is a martingale difference sequence with respect to the filtration $\{\mathcal{F}_{n,k} : n-T+2 \leq k \leq n\}$ and with absolute values uniformly bounded above by $1$. Here $\mathcal{F}_{n,k}$ is the $\sigma$-algebra generated by the first $k$ rows and columns of $\mathbf{M}_n$. By Azuma's inequality, we have for any $\kappa >0$, 
\begin{equation}{\label{gap2}}
\mathbb{P} \biggl( \sum_{k=n-T+2}^n \mathbbm{1}_{\Omega_{n,k}} - \sum_{k=n-T+2}^n \mathbb{P} ( \Omega_{n,k} \; \rvert \;\mathbf{M}_n^{(k-1)})  \leq -\kappa \biggr) \leq \exp\biggl(-\dfrac{\kappa^2}{2(T-1)} \biggr). 	
\end{equation}
On the other hand, \Cref{gap} tells us that, provided that $n$ is large enough and $T \leq \sqrt{n}$, 
\begin{align}
	\sum_{k=n-T+2}^n \mathbb{P} ( \Omega_{n,k} \; \rvert \;\mathbf{M}_n^{(k-1)}) &\geq \dfrac{\upsilon}{2}\sum_{k=n-T+2}^n \mathbbm{1}( 0>\lambda_1(\mathbf{M}_n^{(k-1)})\notag \\
	&> -3\sigma\sqrt{k})  \geq \dfrac{\upsilon(T-1)}{2} \mathbbm{1}_{B_{n,T}}, \label{gap3}
\end{align}
where $B_{n,T}:= \{ 0>\lambda_1(\mathbf{M}_n^{(n-T+1)}),  \lambda_1(\mathbf{M}_n^{(n)})>-3\sigma\sqrt{n-T+1}\}$. The last inequality in \Cref{gap3} is again a consequence of non-increasingness of  $\lambda_1(\mathbf{M}_n^{(k)})$ in $k$. For any fixed $T$, by \Cref{eigen:expecbound}, we have both $n^{-1/2}\lambda_1(\mathbf{M}_n^{(n-T+1)})$ and $n^{-1/2}\lambda_1(\mathbf{M}_n^{(n)})$ converging to $-2\sigma$ in probability and hence $\mathbb{P}(B_{n,T}) \to 1$ as $n \to \infty$. Set $\kappa = \upsilon(T-1)/4$ in \Cref{gap2} and we have the following:
\begin{align*}
	&\mathbb{P}\biggl( \sum_{k=n-T+2}^n \mathbbm{1}_{\Omega_{n,k}} <  \dfrac{\upsilon(T-1)}{4}\biggr) \\
	& \leq \mathbb{P} \biggl( \sum_{k=n-T+2}^n \mathbbm{1}_{\Omega_{n,k}} - \sum_{k=n-T+2}^n \mathbb{P}( \Omega_{n,k} \; \rvert \;\mathbf{M}_n^{(k-1)})  \leq - \dfrac{\upsilon(T-1)}{4}\biggr)  \\
	&\qquad + \mathbb{P} \biggl( \sum_{k=n-T+2}^n \mathbb{P} ( \Omega_{n,k} \; \rvert \;\mathbf{M}_n^{(k-1)}) < \upsilon(T-1)/2\biggr) \\
	& \leq \exp(-\upsilon^{2}(T-1)/32) + \mathbb{P}(B_{n,T}^c).
\end{align*}
On the other hand, \Cref{gap4} implies that $ \sqrt{n}(\lambda_T(\mathbf{M}_n)-\lambda_1(\mathbf{M}_n) )  \geq \sigma \eta \sum_{k=n-T+2}^n \mathbbm{1}_{\Omega_{n,k}}$ and hence for all $n$ large enough and $T \leq \sqrt{n}$,
\begin{align*}
	\mathbb{P} \biggl( \sqrt{n}(\lambda_T(\mathbf{M}_n)-\lambda_1(\mathbf{M}_n)) < \dfrac{\sigma \eta \upsilon (T-1)}{4} \biggr) \leq \exp(-\upsilon^{2}(T-1)/32) + \mathbb{P}(B_{n,T}^c).
\end{align*}
Now recall that $A_{n,T} =\{ \lambda_T(\mathbf{M}_n)-\lambda_1(\mathbf{M}_n) \geq \varepsilon^{-2}\theta_n\}$. Since $\theta_n =\bigO(n^{-1/2})$, get $C < \infty$ such that $\theta_n \leq Cn^{-1/2}$ for all $n \geq 1$. Therefore, for all $T$ such that $\sigma \eta \upsilon (T-1)/4 > C\varepsilon^{-2}$, we have 
$$ \mathbb{P}(A_{n,T}^c) \leq \mathbb{P}( \sqrt{n}\left(\lambda_T(\mathbf{M}_n)-\lambda_1(\mathbf{M}_n) ) \leq C\varepsilon^{-2} \right) \leq  \exp(-\upsilon^{2}(T-1)/32) + \mathbb{P}(B_{n,T}^c).$$
Taking first $n \to \infty$ and then $T \to \infty$, we conclude our proof.
\end{proof}

\begin{remark}{\label{eigen:improve}}
Under some more assumptions on the distribution $\mathcal{P}$, one can actually prove a weak convergence result for the gaps between  eigenvalues at the edge of the spectrum: $$\{ n^{1/6}(\lambda_{i+1}(\mathbf{M}_n)-\lambda_i(\mathbf{M}_n) ) : 1 \leq i \leq k \}$$
 converges to a jointly continuous distribution for any $k \geq 1$. One example of such situation can be found in Soshnikov~\cite{soshnikov}, who proved the joint convergence for the eigenvalues at the edge of the spectrum, from which the aforementioned weak convergence for the gaps can be obtained, under the additional assumption that the distribution $\mathcal{P}$ is symmetric. In such situations, one can modify the proof of \Cref{eigtight2} to obtain a stronger result. In particular, one can estimate $\mathbb{P}(A_{n,T}^c)$, as defined in the proof of \Cref{eigtight2}, with the joint convergence result and complete the proof with the improved assumption that $\theta_n = \bigO(n^{-1/6})$. Nevertheless, such a stronger version of \Cref{eigtight2} will not be required in our analysis.
\end{remark}

Applying \Cref{eigtight2} , we can write our first stability result for Wigner matrices which states that \textit{the eigenvector corresponding to the smallest eigenvalue is stable under the perturbation scheme which replaces one matrix entry at a time}.

\begin{thm}{\label{eigenthm1}}
Consider a Wigner matrix  where the upper triangular entries are i.i.d.~with common distribution $\mathcal{P}$ with density $f$, having zero mean and satisfying the hypotheses of Proposition~\ref{eigtight2}.
The optimization problem in \Cref{eign:def} is then stable under small perturbations with perturbation blocks $\mathcal{J}_n = \{J_{n,(i,j)} : (i,j) \in [n]^2, i \leq  j\}$ where $J_{n,(i,j)} = \{(i,j)\}$ for all $1 \leq i \leq j \leq n$.
\end{thm}

\begin{proof}
We have already shown that \Cref{unique} is satisfied in \Cref{eigtight1}. The rest of the proof goes following exactly the same arguments used in the proof of \Cref{mostprop} except for the last step. If $\widehat{\boldsymbol{v}}(\mathbf{M}_n)=(\widehat{v}_1, \ldots, \widehat{v}_n)$ denotes the eigenvector corresponding to the smallest eigenvalue for $\mathbf{M}_n$ and for any $k \leq l \in [n]$,  $\widehat{\boldsymbol{v}}(\mathbf{M}_n^{(k,l)})=(\widehat{v}_1^{(k,l)}, \ldots, \widehat{v}_n^{(k,l)})$ does the same for $\mathbf{M}_n^{(k,l)}$ where $\mathbf{M}_n^{(k,l)}$ is the matrix obtained from $\mathbf{M}_n$ after replacing $X_{n,(k,l)}$ (and hence $X_{n,(l,k)})$ by i.i.d.~copy $X_{n,(k,l)}^{(k,l)}$, then  
\begin{align*}
\mathbb{E} [\psi_n(\mathbf{M}_n;\widehat{\boldsymbol{v}}(\mathbf{M}_n^{(k,l)})) - \psi_{n, \mathrm{opt}}(\mathbf{M}_n) ] &= \mathbb{E} [\widehat{\boldsymbol{v}}(\mathbf{M}_n^{(k,l)})^{\top} \mathbf{M}_n\widehat{\boldsymbol{v}}(\mathbf{M}_n^{(k,l)}) -  \widehat{\boldsymbol{v}}(\mathbf{M}_n^{(k,l)})^{\top} \mathbf{M}_n^{(k,l)} \widehat{\boldsymbol{v}}(M_n^{(k,l)})] \\
& = \mathbb{E} [( 1 + \mathbbm{1}(k \neq l))( X_{n,kl} - X_{n,kl}^{(k,l)})\widehat{v}^{(k,l)}_k \widehat{v}^{(k,l)}_l ] \\
& \leq 2 \sqrt{2 \operatorname{Var}_{\mathcal{P}}(X) \mathbb{E}\widehat{v}_1^2\widehat{v}_2^2}.
\end{align*}
Since the distribution of the matrix is invariant under permutation of row and columns, we can conclude that $\{\widehat{v}_1, \ldots, \widehat{v}_n\}$ is an exchangeable collection and hence
\begin{align*}
1= \mathbb{E} \biggl[ \biggl(\sum_{i=1}^n \widehat{v}_i^2\biggr)^2\biggr] = \mathbb{E} \biggl[ \sum_{i,j=1}^n \widehat{v}_i^2  \widehat{v}_j^2\biggr] = n(n-1)\mathbb{E}\widehat{v}_1^2\widehat{v}_2^2 + n \mathbb{E}\widehat{v}_1^4 \geq n(n-1)\mathbb{E}\widehat{v}_1^2\widehat{v}_2^2,
\end{align*}
and thus 
$$ \max_{k \leq l} \mathbb{E} [\psi_n(\mathbf{M}_n;\widehat{\boldsymbol{v}}^{(k,l)}) - \psi_{n, \mathrm{opt}}(\mathbf{M}_n) ]  =\bigO(1/n).$$
We can now finish the proof by applying \Cref{eigtight2} and \Cref{rem:strateasy} with $\theta_n=1/n.$
\end{proof}

Another natural choice of perturbation blocks corresponds to replacing all entries in a particular row (and hence its corresponding column) by i.i.d.~entries. In other words, we consider the perturbation blocks $\mathcal{J}_n = \{J_{n,i} : i \in [n]\}$ where $J_{n,i}=\{(i,j) : j \geq i \} \cup \{(j,i) : j \leq i\}$. 

\begin{thm}{\label{eigthm2}}
Consider a Wigner matrix  where the upper triangular entries are i.i.d.~with common distribution $\mathcal{P}$ with density $f$ having zero mean and satisfying the hypotheses of Proposition~\ref{eigtight2}.
The optimization problem in \Cref{eign:def} is then stable under small perturbations with perturbation blocks 
$\mathcal{J}_n = \{J_{n,i} : i \in [n]\}$ where $J_{n,i}=\{(i,j) : j \geq i \} \cup \{(j,i) : j \leq i \}$, for all $i \in [n]$. 
\end{thm}

\begin{proof}
  This is similar to the proof of \Cref{skthm2}.
For all $i \in [n]$, let $\widehat{\boldsymbol{v}}(\mathbf{M}^n_i)=(\widehat{v}^{(i)}_1,\ldots,\widehat{v}_n^{(i)})$ be the eigenvector corresponding to the smallest eigenvalue of $\mathbf{M}_n^{(i)}=(X_{n,kl}^{(i)})_{k,l \in [n]}$, which is obtained by replacing the upper triangular entries of $i$-th row and column of $\mathbf{M}_n$ by i.i.d.~copies. $\widehat{\boldsymbol{v}}(\mathbf{M}_n)=(\widehat{v}_1,\ldots,\widehat{v}_n)$ is the same for matrix $\mathbf{M}_n$. Then for any $i \in [n] $, 
\begin{align}
\mathbb{E} [ \psi_n(\mathbf{M}_n;\widehat{\boldsymbol{v}}(\mathbf{M}_n^{(i)})) -\psi_{n,\mathrm{opt}}(\mathbf{M}_n)] & = \mathbb{E} [ \psi_n(\mathbf{M}_n;\widehat{\boldsymbol{v}}(M_n^{(i)})) - \psi_n(\mathbf{M}_n^{(i)};\widehat{\boldsymbol{v}}(\mathbf{M}_n^{(i)}))]
 \nonumber \\ 
& =  \mathbb{E} \sum_{j \in [n] } (   X_{n,ij} - X^{(i)}_{n,ij} ) \widehat{v}^{(i)}_i\widehat{v}^{(i)}_j (1+\mathbbm{1}(i \neq j)) \nonumber \\
& = - \mathbb{E} \sum_{j \in [n] } X^{(i)}_{n,ij} \widehat{v}^{(i)}_i\widehat{v}^{(i)}_j (1+\mathbbm{1}(i \neq j))  \label{mzero1} \\
& = -\mathbb{E} \sum_{j \in [n] } X_{n,ij} \widehat{v}_i\widehat{v}_j (1+\mathbbm{1}(i \neq j)) \label{eq5},
\end{align}
where the equality in \Cref{mzero1} again follows from the fact that the matrix entries have mean zero and  $\widehat{\boldsymbol{v}}(\mathbf{M}_n^{(i)})$ is independent of the collection  $ \{X_{n,ij} : j \in [n]\}$ for all $i \in [n]$.  
It is again obvious that the distribution of the random quantity in \Cref{eq5} does not depend upon $i$ and hence,
\begin{align}
\mathbb{E} \biggl[ \sum_{j \in [n] } X_{n,ij} \widehat{v}_i\widehat{v}_j (1+\mathbbm{1}(i \neq j)) \biggr] &= \dfrac{1}{n}\mathbb{E} \biggl[ \sum_{i \in [n]} \sum_{j \in [n] } X_{n,ij} \widehat{v}_i\widehat{v}_j (1+\mathbbm{1}(i \neq j)) \biggr] \nonumber \\
 &= \dfrac{1}{n}\mathbb{E} \biggl(2\psi_{n, \mathrm{opt}}(\mathbf{M}_n) - \sum_{i \in [n]} X_{n,ii} \widehat{v}_i^2 \biggr) \nonumber \\
& \geq   \dfrac{2}{n}\mathbb{E} [ \lambda_1(\mathbf{M}_n)] - \dfrac{1}{n} \mathbb{E}( \max_{i \in [n]} X_{n,ii}).
\end{align}
By \Cref{eigen:expecbound} we have $\mathbb{E} [ \lambda_1(\mathbf{M}_n)] \geq - \bigO(\sqrt{n})$, 
whereas 
$$ \mathbb{E}( \max_{i \in [n]} X_{n,ii}) \leq [\mathbb{E} \max_{i \in [n]} |X_{n,ii}|^4]^{1/4} \leq  \biggl[\mathbb{E} \sum_{i \in [n]} |X_{n,ii}|^4 \biggr]^{1/4} \leq (n \sigma^44^{4\beta} )^{1/4} = \bigO(n^{1/4}). $$
 Combining these estimates we obtain
$$ \max_{i \in [n]} \mathbb{E} [ \psi_n(\mathbf{M}_n;\widehat{\boldsymbol{v}}^{(i)}) -\psi_{n,\mathrm{opt}}(\mathbf{M}_n) ] = \bigO(n^{-1/2}). $$
Applying \Cref{eigtight2} and \Cref{rem:strateasy}, we observe that the hypotheses of \Cref{strat} is satisfied for $\theta_n=n^{-1/2}$. This concludes the proof.
\end{proof}

\begin{remark}
As we mentioned in \Cref{eigen:improve}, we can prove the result in \Cref{eigtight2} with $\theta_n = \bigO(n^{-1/6})$ under some additional constraints on the entry distribution. Since, the proof of \Cref{eigenthm1} and \Cref{eigthm2} requires only the assertion in \Cref{eigtight2} with $\theta_n = \bigO(\sqrt{n})$, it suggests that under additional assumptions on matrix entries, one might be able to replicate the conclusion of \Cref{eigthm2} for even larger perturbations blocks. One such example could be replacing multiple rows (and columns) with independent copies, where the number of such perturbed rows grows to infinity with the matrix dimension $n$. This can be an interesting topic for further inspection.
\end{remark}

We conclude this section with a short description of  chaotic properties of eigenvectors of random matrices. \citet{chaea} describes this results in the setting of \textit{Gaussian Unitary Ensembles} (GUE), which are Hermitian random matrices with independent entries on and above the diagonal,  where the diagonal entries are standard Gaussian variables and off-diagonal entries are standard complex Gaussian variables (i.e., real and imaginary parts are i.i.d.~$N(0,1/2)$). As suggested in the beginning of this section, this setting can be easily accommodated in our setup with little effort. \citet{chaea} proved the following result:  For any $\varepsilon \in (0,1)$, take an $n \times n$ GUE matrix, denoted by $\mathbf{M}_n$, and perturb it slightly to get 
\[
\mathbf{M}_n^{(\varepsilon)} := (1-\varepsilon) \mathbf{M}_n + \sqrt{2\varepsilon - \varepsilon^2}\mathbf{M}_n^{\prime},
\]
where $\mathbf{M}_n^{\prime}$ is an independent copy of $\mathbf{M}_n$. Note that, the $(i,j)$-th entries of $\mathbf{M}_n$ and $\mathbf{M}_n^{(\varepsilon)}$ have correlation $1-\varepsilon$, for $i \neq j$.  If $\widehat{\boldsymbol{v}}$ and $\widehat{\boldsymbol{v}}^{(\varepsilon)}$ are the unit eigenvectors corresponding to the smallest eigenvalues for $\mathbf{M}_n$ and $\mathbf{M}_n^{(\varepsilon)}$ respectively, then 
$$ \mathbb{E} \big \rvert \widehat{\boldsymbol{v}} \cdot \widehat{\boldsymbol{v}}^{(\varepsilon)}  \big \rvert^2 \leq C \varepsilon^{-1}n^{-1/3},$$
where $C$ is some universal constant. See \cite[Theorem 3.8]{chaea} for the statement and  proof of the above result. This shows that whenever $\varepsilon \gg n^{-1/3}$, the eigenvectors are nearly orthogonal with high probability. This result does not contradict \Cref{eigenthm1} or \Cref{eigthm2} for the reasons similar to the ones outlined at the end of \Cref{skmodel}. In particular, if $\varepsilon \gg n^{-1/3}$, the $L^2$ measure of perturbation, as discussed at the end of \Cref{skmodel}, is $\gg n^{5/3}$ for the GUE perturbations described above; whereas in \Cref{eigenthm1} and \Cref{eigthm2} the same measures are $\bigO(1)$ and $\bigO(n)$, respectively. Some other recent works on this approach were carried out by \cite{georges,georges2}, where they considered perturbative expansion of the empirical spectral distribution and co-ordinates of eigenvectors, respectively, for large Hermitian matrices perturbed by an independent random matrix with small operator norm.



\subsection{Wishart matrices}
We shall consider another random matrix model here and shall inspect the stability of the optimization problem associated with its largest eigenvalue. This is the first example where we shall encounter an objective function $\psi_n$ which is not linear in its arguments, i.e., it does not satisfy \Cref{ass:linear}. Let $\{m_n : n \geq 1\}$ be a sequence of positive integers. Consider an  $m_n \times n$ real random matrix $\mathbf{M}_n = (X_{n,ij})_{i \in [m_n], j \in [n]}$. The problem of finding the largest singular value of $\mathbf{M}_n$ or the largest eigenvalue of $\mathbf{Y}_n:=\mathbf{M}_n^{\top}\mathbf{M}_n$ can be reformulated as 
\begin{align}{\label{wish:def}}
\text{ Minimize } \psi_n(\mathbf{M}_n; \boldsymbol{v}) := -\|\mathbf{M}_n\boldsymbol{v}\|_2^2 &=- \sum_{i,j \in [n]} \sum_{k \in [m_n]}  X_{n,ki}X_{n,kj}v_iv_j \notag \\
 &\qquad \qquad \qquad \text{over } \boldsymbol{v}=(v_1, \ldots,v_n) \in \mathscr{S}^{n-1}.
\end{align}
The stochastic model for the entries in $\mathbf{M}_n$ will be taken as follows: $\{X_{n,ij}:i \in [m_n], j \in [n]\}$ is an independent and identically distributed collection of mean-zero random variables. The  matrix $\mathbf{Y}_n$ is called a \textit{Wishart matrix}, although the name is generally applied to the situation when the matrix entries of $\mathbf{M}_n$ are Gaussian variables.  The asymptotics for the empirical distribution of the eigenvalues of Wishart matrices, under the assumption that $m_n/n =\alpha +o(1)$ for some $\alpha \in [1,\infty)$, is well-known; the limiting law is the famous \textit{Mar\v{c}enko--Pastur law}, as stated in the following theorem.

\begin{thm}[\cite{marchenko}]{\label{marcenko}}
Consider an $m_n \times n$ real random matrix $\mathbf{M}_n$ with i.i.d.~entries of mean zero, variance $1$ and finite moments of all order. Let $\lambda_1(\mathbf{Y}_n) \leq \cdots \leq \lambda_n(\mathbf{Y}_n)$ be the sorted eigenvalues of $\mathbf{Y}_n = \mathbf{M}_n^{\top}\mathbf{M}_n$. Then the empirical measure $n^{-1}\sum_{i=1}^n \delta_{\lambda_i(\mathbf{Y}_n)/n}$ converges to $\mathcal{L}_{\alpha}$ weakly, in probability as $n \to \infty$. Here $\alpha = \lim_{n \to \infty} m_n/n \in [1,\infty)$ and $\mathcal{L}_{\alpha}$ is the probability measure supported on $((1-\sqrt{\alpha})^2,(1+\sqrt{\alpha})^2)$ with density 
$$ f_{\alpha} (x) := \dfrac{\sqrt{(x-(1-\sqrt{\alpha})^2)((1+\sqrt{\alpha})^2)-x)}}{2\pi x}, \; \forall \; x \in  ((1-\sqrt{\alpha})^2,(1+\sqrt{\alpha})^2).$$
\end{thm}


Similar to \Cref{eigmatrix}, we consider the parameter space $ \mathcal{S}_n = \{\boldsymbol{v}=(v_1, \ldots,v_n) : \boldsymbol{v} \in \mathscr{S}^{n-1},  v_1 \geq 0 \},$ and equip it with the Euclidean distance, $d_n$. The  entries of the matrix $\mathbf{M}_n$  are the random inputs of the problem; hence $I_n = \{ij : i \in [m_n], j \in [n]\}$ and $k_n=nm_n.$ We start our analysis by trying to show an expansion like \Cref{ass:psi} for this problem. Take $\mathbf{x}^n =(x_{ij})_{i \in [m_n], j \in [n]}$ and $\mathbf{c}^n =(c_{ij})_{i \in [m_n], j \in [n]}$; denote the corresponding matrices by $M(\mathbf{x}^n)$ and $M(\mathbf{c}^n)$ respectively. For any $\boldsymbol{v} \in \mathscr{S}^{n-1}$, the map $\mathbf{x}^n \mapsto \psi_n(M(\mathbf{x}^n); \boldsymbol{v}) = -\|M(\mathbf{x}^n)\boldsymbol{v}\|_2^2$ is jointly concave; hence 
\begin{align*}
\psi_n(M(\mathbf{x}^n)+M(\mathbf{c}^n); \boldsymbol{v})- \psi_n(M(\mathbf{x}^n); \boldsymbol{v}) \leq \sum_{i,j} c_{ij} \,\, \partial_{ij} \psi_n (M(\mathbf{x}^n)); \boldsymbol{v}),
\end{align*}
where $\partial_{ij}$ refers to the partial derivative of $\psi_n$ with respect to the $(i,j)$-th entry of the matrix input, i.e., $x_{ij}$. Computing the derivatives and plugging them in the above inequality yields the following:
 \begin{align*}
 \psi_n(M(\mathbf{x}^n)+M(\mathbf{c}^n); \boldsymbol{v})- \psi_n(M(\mathbf{x}^n); \boldsymbol{v}) \leq -\sum_{i \in [m_n], j \in [n]} 2c_{ij} \biggl( \sum_{k=1}^n x_{ik}v_jv_k\biggr),
 \end{align*}
and hence \Cref{item:B1} is satisfied for $L=0$, 
\begin{equation}{\label{wishart:H}}
H_{n,ij}(\mathbf{x}^n; \boldsymbol{v}) := -2 v_j \sum_{k=1}^n x_{ik}v_k, \;\forall \; i \in [m_n], \; j \in [n].
\end{equation}
We now want to apply \Cref{genthm:g} to prove a tightness result, but before that we need an analogue of \Cref{eigen:expecbound}.

 \begin{prop}{\label{eigen:expecbound2}}
\begin{enumerate}[label={(\alph*)}]
\item Assume that $m_n =\bigO(n^\gamma)$ for some $\gamma \in (0, \infty)$.  Consider the random  matrix $\mathbf{M}_n = (X_{n,ij} )_{i \in [m_n], j \in [n]}$ with  entries being i.i.d.~with common distribution $\mathcal{P}$ having mean zero and variance $\sigma^2 < \infty$. Further assume that there exists a constant $\beta \in (0, \infty)$ such that $\mathbb{E}_{\mathcal{P}} |X|^k \leq \sigma^{k}k^{\beta k}$, for all positive integers $k$. Then there exists another constant $\widetilde{C}(\beta)$, depending on $\beta$ only, such that  
 $$ \mathbb{E} \lambda_n(\mathbf{M}_n^{\top}\mathbf{M}_n) \leq \widetilde{C}(\beta) \sigma^2 (n+m_n), \; \forall \; n \geq 1.$$ 
 \item If we only assume that the distribution $\mathcal{P}$ has finite fourth moment, then we still have  
 $$\mathbb{P}( \lambda_n(\mathbf{M}_n^{\top}\mathbf{M}_n) < \varepsilon_n \max(m_n,n) ) =o(1),$$
  for any $\varepsilon_n=o(1)$. Here we don't assume any growth rate on $m_n$. 
\end{enumerate}

 \end{prop}

\begin{prop}{\label{wishtight1}}
Consider the optimization problem defined in \Cref{wish:def} where the  entries are i.i.d.~with common distribution $\mathcal{P}$ with density $f$ having zero mean, finite moments of all order and  satisfying Assumption~\ref{mu} for the pair $(\rho,g)$ with $g^{(1)}$ bounded away from $0$. Also assume that $m_n/n =\alpha +o(1)$ as $n \to \infty$ for some $\alpha \in [1,\infty)$.  Then for any $\varepsilon >0$, the sequence $\{P(\mathcal{N}_{n,\theta_n}, d_n, \varepsilon) : n \geq 1\}$ is  tight whenever $\theta_n=\bigO(1)$. 
\end{prop}

\begin{proof}
Our strategy is to apply \Cref{genthm:g}, but with a slight modification. The conditions in \Cref{unique} and \Cref{ass:meas} are satisfied for reasons similar to the one discussed in the proof of \Cref{eigtight1}. 
We start by checking the validity of the condition \Cref{item:E3}. Fix $\varepsilon >0$ and consider an $\varepsilon$-net $\mathcal{N}_{n,\theta_n,\varepsilon}$  of $\mathcal{N}_{n,\theta_n}$, with respect to the random pseudo-metric $d_{g,n}$.  Now take $\boldsymbol{u}=(u_1, \ldots, u_n) \in \mathcal{N}_{n,\theta_n}$ and get $\boldsymbol{v}^* \in \mathcal{N}_{n,\theta_n,\varepsilon}$ such that $d_{g,n}(\boldsymbol{u},\boldsymbol{v}^*) \leq \varepsilon$. Introduce the notation $(\mathbf{M}_{n})_{i \cdot}$ to denote the $i$-th row of the matrix $\mathbf{M}_n = \mathbf{M}(\mathbf{X}^n)$ (as a column vector) and note that $H_{n,ij}(\mathbf{X}^n;\boldsymbol{v}) =  -2v_j (\mathbf{M}_{n})_{i \cdot}^{\top} \boldsymbol{v} = -2 (\mathbf{M}_{n})_{i \cdot}^{\top} \boldsymbol{v}\boldsymbol{v}^{\top}\mathbf{e}_j$, where $\mathbf{e}_j$ is the $j$-th co-ordinate vector. Since $|g^(1)| \geq \eta >0$ , 
\begin{align}{\label{stepwish1}}
\varepsilon^2 \geq d_{g,n}(\boldsymbol{v}^*,\boldsymbol{u})^2 & \geq \eta^2 \sum_{i,j} ( H_{n,ij}(\mathbf{X}^n;\boldsymbol{v}^*)  - H_{n,ij}(\mathbf{X}^n;\boldsymbol{u}) )^2 \nonumber \\
& = 4\eta^2  \sum_{i,j} ( (\mathbf{M}_{n})_{i \cdot}^{\top} \boldsymbol{v}^*\boldsymbol{v}^{*\top}\mathbf{e}_j - (\mathbf{M}_{n})_{i \cdot}^{\top} \boldsymbol{u}\boldsymbol{u}^{\top}\mathbf{e}_j)^2 \nonumber \\
& = 4\eta^2  \sum_{j=1}^n \sum_{i=1}^{m_n} \mathbf{e}_j^{\top} (\boldsymbol{v}^*\boldsymbol{v}^{*\top} - \boldsymbol{u}\boldsymbol{u}^{\top} ) (\mathbf{M}_{n})_{i \cdot}(\mathbf{M}_{n})_{i \cdot}^{\top} (\boldsymbol{v}^*\boldsymbol{v}^{*\top} - \boldsymbol{u}\boldsymbol{u}^{\top} ) \mathbf{e}_j \nonumber \\
& = 4\eta^2 \operatorname{Tr} ( (\boldsymbol{v}^*\boldsymbol{v}^{*\top} - \boldsymbol{u}\boldsymbol{u}^{\top} ) \mathbf{Y}_n (\boldsymbol{v}^*\boldsymbol{v}^{*\top} - \boldsymbol{u}\boldsymbol{u}^{\top} )) \nonumber \\
& = 4\eta^2 ( \boldsymbol{v}^{*\top} \mathbf{Y}_n \boldsymbol{v}^* + \boldsymbol{u}^{\top} \mathbf{Y}_n \boldsymbol{u} - 2 \boldsymbol{v}^{*\top} \mathbf{Y}_n \boldsymbol{u} \boldsymbol{v}^{*\top}\boldsymbol{u} ) \nonumber \\
& \geq 4\eta^2 ( \boldsymbol{v}^{*\top} \mathbf{Y}_n \boldsymbol{v}^* + \boldsymbol{u}^{\top} \mathbf{Y}_n \boldsymbol{u} - 2 |\boldsymbol{v}^{*\top} \mathbf{Y}_n \boldsymbol{u}| ).
\end{align}
Now, first consider the case when $\boldsymbol{v}^{*\top} \mathbf{Y}_n \boldsymbol{u} \geq 0$. Let $\mathbf{Y}_n = \Gamma_{n} \Lambda_n \Gamma_n^{\top}$ be the spectral decomposition of $\mathbf{Y}_n$, where $\Lambda_n = \operatorname{diag}(\lambda_1(\mathbf{Y}_n), \ldots, \lambda_n(\mathbf{Y}_n))$. Setting $\boldsymbol{x}=(x_1,\ldots,x_n) = \Gamma_n^{\top} \boldsymbol{v}^*$ and $\boldsymbol{y}=(y_1,\ldots,y_n) = \Gamma_n^{\top} \boldsymbol{u}$, we can rewrite \Cref{stepwish1} as 
$$ \sum_{i=1}^n \lambda_i(\mathbf{Y}_n)x_i^2 + \sum_{i=1}^n \lambda_i(\mathbf{Y}_n)y_i^2 - 2 \sum_{i=1}^n \lambda_i(\mathbf{Y}_n)x_iy_i \leq \dfrac{\varepsilon^2}{4\eta^2}.$$
By definition, $\boldsymbol{x},\boldsymbol{y}$ are both unit vectors. Pick any $m \in [n-1]$ and note that, since $\boldsymbol{v}^* \in \mathcal{N}_{n,\theta_n}$, we have 
\begin{align*}
 \lambda_n(\mathbf{Y}_n) - \theta_n \leq \boldsymbol{v}^{*\top} \mathbf{Y}_n \boldsymbol{v}^* = \sum_{i=1}^n \lambda_i(\mathbf{Y}_n)x_i^2 &\leq \lambda_n(\mathbf{Y}_n) \sum_{i >m} x_i^2 + \lambda_m(\mathbf{Y}_n) \sum_{i \leq  m} x_i^2 \\
 & \leq \lambda_n(\mathbf{Y}_n) \biggl(1-\sum_{i \leq m} x_i^2 \biggr) + \lambda_m(\mathbf{Y}_n) \sum_{i \leq  m} x_i^2,
\end{align*} 
and hence $\sum_{i \leq m} x_i^2 \leq \theta_n/( \lambda_n(\mathbf{Y}_n)- \lambda_m(\mathbf{Y}_n))$. Similarly, one can also show that  $\sum_{i \leq m} y_i^2 \leq \theta_n/( \lambda_n(\mathbf{Y}_n)- \lambda_m(\mathbf{Y}_n)).$ Hence,
\begin{align*}
d_n(\boldsymbol{v}^*,\boldsymbol{u})^2 = \|\boldsymbol{x}-\boldsymbol{y}\|_2^2 = \sum_{i=1}^n (x_i-y_i)^2 &\leq \sum_{i>m} (x_i-y_i)^2 + 2\sum_{i \leq m} x_i^2 + 2\sum_{i \leq m} y_i^2 \\
& \leq \dfrac{ \sum_{i=1}^n \lambda_i(\mathbf{Y}_n) (x_i-y_i)^2 }{\lambda_m(\mathbf{Y}_n)}+  \dfrac{4\theta_n}{ \lambda_n(\mathbf{Y}_n)- \lambda_m(\mathbf{Y}_n)} \\
& \leq \dfrac{\varepsilon^2 }{4\eta^2 \lambda_m(\mathbf{Y}_n)}+  \dfrac{4\theta_n}{ \lambda_n(\mathbf{Y}_n)- \lambda_m(\mathbf{Y}_n)}.
\end{align*}
If $\boldsymbol{v}^{*\top} \mathbf{Y}_n \boldsymbol{u} \leq 0$, we can perform similar computations with $\boldsymbol{v}^*$ replaced by $-\boldsymbol{v}^*$. Thus 
$$ 2N(\mathcal{N}_{n,\theta_n}, d_{g,n},\varepsilon) \geq N \biggl(\mathcal{N}_{n,\theta_n}, d_{n},\sqrt{\dfrac{\varepsilon^2 }{4\eta^2 \lambda_m(\mathbf{Y}_n)}+  \dfrac{4\theta_n}{ \lambda_n(\mathbf{Y}_n)- \lambda_m(\mathbf{Y}_n)}} \biggr),$$
or equivalently for any $\varepsilon >0$, we have $2N (\mathcal{N}_{n,\theta_n}, d_{g,n}, \eta \varepsilon^2 \sqrt{n}  ) \geq N (\mathcal{N}_{n,\theta_n}, d_{n}, \varepsilon ),$ provided there exists $m<n$ such that $\lambda_m(\mathbf{Y}_n) >\varepsilon n$ and $16\theta_n < 3\varepsilon^2 (\lambda_n(\mathbf{Y}_n)- \lambda_m(\mathbf{Y}_n))$.  Hence, by \Cref{pnp}, 
\begin{align*}
&\mathbb{P} ( 2P (\mathcal{N}_{n,\theta_n}, d_{g,n}, \eta \varepsilon^2 \sqrt{n} ) < P (\mathcal{N}_{n,\theta_n}, d_{n}, 2\varepsilon ) ) \nonumber \\ & \hspace{1 in} \leq \mathbb{P} \biggl( \nexists \, m \in [n-1]\, \text{ such that } \varepsilon n < \lambda_m(\mathbf{Y}_n) < \lambda_n(\mathbf{Y}_n) -\dfrac{ 16\theta_n}{3\varepsilon^2} \biggr) = o(1),
\end{align*}
 where the last convergence for small enough $\varepsilon >0$  follows from \Cref{marcenko}. This shows that \Cref{item:E3} is satisfied with $\tau_n = \sqrt{n}$.

There is nothing to check for \Cref{item:E2} here, since $L=0$. We shall not directly apply \Cref{item:E1}, rather than modify the proof of \Cref{genthm:g} in places where \Cref{item:E1} was used to get a better estimate. Recall that \Cref{item:E1} was used only in deriving the estimate \Cref{thm2:sec}. In our case, the first expression in the series of inequalities that led to \Cref{thm2:sec} becomes 
\begin{align*}
\mathbb{E} \sup_{ \boldsymbol{v} \in \mathcal{N}_{n,\theta_n}} -2\sum_{i,j} (V_{t,n,ij}+\widetilde{V}_{t,n,ij})v_j (\mathbf{M}_n)_{i\cdot}^{\top}\boldsymbol{v} & = \mathbb{E} \sup_{ \boldsymbol{v} \in \mathcal{N}_{n,\theta_n}} -2\sum_{i} (\mathbf{V}_{t,n}+\mathbf{\widetilde{V}}_{t,n})_{i\cdot}^{\top}\boldsymbol{v} (\mathbf{M}_n)_{i\cdot}^{\top}\boldsymbol{v} \\
&= \mathbb{E} \sup_{ \boldsymbol{v} \in \mathcal{N}_{n,\theta_n}} -2 \boldsymbol{v}^{\top}(\mathbf{V}_{t,n}+\mathbf{\widetilde{V}}_{t,n})^{\top} \mathbf{M}_n \boldsymbol{v} \\
& \leq 4 \mathbb{E} \sup_{ \boldsymbol{v} \in \mathcal{N}_{n,\theta_n}} \| (\mathbf{V}_{t,n}+\mathbf{\widetilde{V}}_{t,n})\boldsymbol{v}  \|_2 \|\mathbf{M}_n \boldsymbol{v}  \|_2 \\
& \leq \sqrt{\mathbf{E} \lambda_n((\mathbf{V}_{t,n}+\mathbf{\widetilde{V}}_{t,n})^{\top}(\mathbf{V}_{t,n}+\mathbf{\widetilde{V}}_{t,n}))}\sqrt{\mathbb{E}\lambda_n(\mathbf{Y}_n)} \\
& = t\bigO(n),
\end{align*} 
where the last equality follows from \Cref{eigen:expecbound2} and the fact that $t^{-1}(V_{t,n,ij}+\widetilde{V}_{t,n,ij})$ is sub-Gamma with parameters $\bigO(1)$. Plugging-in this estimate in \Cref{last:proof}, we observe that the conclusion of \Cref{genthm:g} is true provided $\theta_n = \bigO(1)$. This completes the proof.
\end{proof}

As was the case with Wigner matrices, we shall now prove a tightness result directly, without using the general results developed in this article, to improve upon the scope of \Cref{wishtight1}.

\begin{prop}{\label{wishtight2}}
Consider the optimization problem defined in \Cref{wish:def} where the  entries are i.i.d.~with common distribution $\mathcal{P}$ with density $f$ having zero mean and finite moment of order four. Then for any $\varepsilon >0$, the sequence $\{P(\mathcal{N}_{n,\theta_n}, d_n, \varepsilon) : n \geq 1\}$ is  tight whenever $\theta_n=\bigO(1)$. 
\end{prop}

\begin{proof}
The initial computation of the  proof goes in the same way as the proof of \Cref{eigtight2}. For any $T \geq 2$, we define $A_{n,T} := \{ \lambda_n(\mathbf{Y}_n) - \lambda_{n-T+1}(\mathbf{Y}_n)  > \varepsilon^{-2}\theta_n \}$ and can conclude that on the event $A_{n,T}$ we have $P ( \mathcal{N}_{n,\theta_n},d_n,\varepsilon ) \leq (1+2/\varepsilon)^{T-1}$. We, therefore, again concentrate on bounding the probability of $A_{n,T}^c$.  

 We shall again make use of the comparison inequality as stated in \Cref{interlace}. Let $\mathbf{Y}_n^{(k)}$ is the random matrix obtained by considering only the first $k$ rows and columns of $\mathbf{Y}_n$. By \Cref{interlace} applied to the symmetric matrix $\mathbf{Y}_n$, we have $\lambda_k(\mathbf{Y}_n^{(k)}) \leq \lambda_k(\mathbf{Y}_n)$ for all $1 \leq k \leq n$ with equality for the case $k=n$ and hence 
 \begin{align}
 	\lambda_n(\mathbf{Y}_n) - \lambda_{n-T+1}(\mathbf{Y}_n)  &\geq  \lambda_n(\mathbf{Y}_n^{(n)}) - \lambda_{n-T+1}(\mathbf{Y}_n^{(n-T+1)})  \notag \\
	&\geq  \sum_{k=n-T+2}^{n} (\lambda_k(\mathbf{Y}_n^{(k)}) - \lambda_{k-1}(\mathbf{Y}_n^{(k-1)}) ), \label{gapw4}
 \end{align} 
where the last inequality in \Cref{gapw4} is true since $\lambda_k(\mathbf{Y}_n^{(k)})$ is non-decreasing in $k$ by \Cref{interlace}.
 We shall now use \Cref{gapestimate} to estimate the difference  $\lambda_k(\mathbf{Y}_n^{(k)}) - \lambda_{k-1}(\mathbf{Y}_n^{(k-1)}) $ for any $n-T+2\leq k \leq n$. Let $\mathbf{M}_n^{\leq k}$ denotes the matrix formed by first $k$ many columns of $\mathbf{M}_n$ and observe that $\mathbf{Y}_n^{(k)} = (\mathbf{M}_n^{\leq k} )^{\top} \mathbf{M}_n^{\leq k}$ for all $k$. Moreover, if we write $ (\mathbf{M}_n)_{\cdot k}$ to denote the $k$-th row of $\mathbf{M}_n$, then 
 \begin{align*}
 \mathbf{Y}_n^{(k)} = (\mathbf{M}_n^{\leq k})^{\top} \mathbf{M}_n^{\leq k} &= \begin{bmatrix}
 (\mathbf{M}_n^{\leq k-1} )^{\top} \\
 (\mathbf{M}_n)_{\cdot k}^{\top}
 \end{bmatrix} \begin{bmatrix}
 \mathbf{M}_n^{\leq k-1} & (\mathbf{M}_n)_{\cdot k}
 \end{bmatrix} \\
 &= \begin{bmatrix}
 \mathbf{Y}_n^{(k-1)} & (\mathbf{M}_n^{\leq k-1} )^{\top}(\mathbf{M}_n)_{\cdot k} \\
 (\mathbf{M}_n)_{\cdot k}^{\top}\mathbf{M}_n^{\leq k-1} & (\mathbf{M}_n)_{\cdot k}^{\top}(\mathbf{M}_n)_{\cdot k}
 \end{bmatrix}.
 \end{align*}
   Since $m_n \geq n$, with probability one the matrix $\mathbf{M}_{n}^{\leq k}$ has full column rank (by existence of densities for its entries).  Consider the singular value decomposition 
   \begin{equation}{\label{gapw6}}
\mathbf{M}_{n}^{\leq k-1} = \sum_{i=1}^{k-1} \sigma_{i,n,k-1}\boldsymbol{u}_{i,n,k-1}\boldsymbol{v}_{i,n,k-1}^{\top},   
   \end{equation} 
   where the singular values are ordered such that $0 \leq |\sigma_{1,n,k-1}| \leq \cdots \leq |\sigma_{k-1,n,k-1}|.$ The eigenvalues of $\widetilde{\mathbf{Y}}_n^{(k-1)}$ are then $\lambda_i(\widetilde{\mathbf{Y}}_n^{(k-1)}) = \sigma_{i,n,k-1}^2$, with eigenvector $\boldsymbol{v}_{i,n,k-1}$, for all $i \leq k-1$. Applying \Cref{gapestimate}, we have for any $t \in (0,1)$, 
   \begin{align}{\label{gapw5}}
   \lambda_k(\mathbf{Y}_n^{(k)}) - \lambda_{k-1}(\mathbf{Y}_n^{(k-1)}) &\geq 2\sqrt{t(1-t)} \bigg \rvert \boldsymbol{v}_{k-1,n,k-1}^{\top}(\mathbf{M}_n^{\leq k-1} )^{\top}(\mathbf{M}_n)_{\cdot k}\bigg \rvert - t \lambda_{k-1}(\mathbf{Y}_n^{(k-1)}) \nonumber \\
   & \hspace{2.5 in}  + t (\mathbf{M}_n)_{\cdot k}^{\top}(\mathbf{M}_n)_{\cdot k}.
   \end{align}
   By the singular value decomposition in \Cref{gapw6}, we have 
   \begin{align*}
    \boldsymbol{v}_{k-1,n,k-1}^{\top}(\mathbf{M}_n^{\leq k-1})^{\top}(\mathbf{M}_n)_{\cdot k} &= \sum_{i=1}^{k-1} \sigma_{i,n,k-1}\boldsymbol{v}_{k-1,n,k-1}^{\top} \boldsymbol{v}_{i,n,k-1}\boldsymbol{u}_{i,n,k-1}^{\top}(\mathbf{M}_n)_{\cdot k} \\
    & = \sigma_{k-1,n,k-1}\boldsymbol{u}_{k-1,n,k-1}^{\top}(\mathbf{M}_n)_{\cdot k}, 
   \end{align*}
   and hence the estimate in \Cref{gapw5} simplifies to 
   \begin{equation}{\label{gapw7}}
   \lambda_k(\mathbf{Y}_n^{(k)}) - \lambda_{k-1}(\mathbf{Y}_n^{(k-1)}) \geq 2\sqrt{t(1-t)\lambda_{k-1}(\mathbf{Y}_n^{(k-1)})} \rvert \boldsymbol{u}_{k-1,n,k-1}^{\top}(\mathbf{M}_n)_{\cdot k} \rvert - t \lambda_{k-1}(\mathbf{Y}_n^{(k-1)}).
   \end{equation}
 
The remaining part of the proof is almost identical to the one in \Cref{eigtight2}. Fix $\upsilon >0$ small enough such that the assertion in \Cref{xuconv} holds true with $\delta = \delta(\upsilon)>0$. Take any $0 < \eta < \min(\delta^2/4)$ and set $t=\eta ( \lambda_{k-1}(\mathbf{Y}_n^{(k-1)}))^{-1}$. Plugging-in this particular choice in \Cref{gap1}, we have, on the event $\lambda_{k-1}(\mathbf{Y}_n^{(k-1)}) > 4\eta/3$, 
 \begin{equation}{\label{gapw}}
 	 \lambda_k(\mathbf{Y}_n^{(k)}) - \lambda_{k-1}(\mathbf{Y}_n^{(k-1)}) \geq \sqrt{\eta} \big \rvert \boldsymbol{u}_{k-1,n,k-1}^{\top}(\mathbf{M}_n)_{\cdot k} \big \rvert  -\eta. 
 \end{equation} 
Let $\widetilde{\Omega}_{n,k} := \{ \lambda_k(\mathbf{Y}_n^{(k)}) - \lambda_{k-1}(\mathbf{Y}_n^{(k-1)}) \geq  \eta\}$.  The estimate in \Cref{gapw} yields the following on the event $\lambda_{k-1}(\mathbf{Y}_n^{(k-1)}) > 4\eta/3$:
 \begin{align*}
\mathbb{P} ( \widetilde{\Omega}_{n,k} \; \rvert \;\mathbf{M}_n^{ \leq k-1}) & \geq \mathbb{P} (\rvert \boldsymbol{u}_{k-1,n,k-1}^{\top}(\mathbf{M}_n)_{\cdot k}\rvert \geq 2\sqrt{\eta} \; \rvert \;\mathbf{M}_n^{ \leq k-1}) \\
&  \geq \mathbb{P} ( \rvert \boldsymbol{u}_{k-1,n,k-1}^{\top}(\mathbf{M}_n)_{\cdot k} \rvert \geq \delta \; \rvert \;\mathbf{M}_n^{ \leq k-1})  \geq \upsilon,
 \end{align*}
for all $k,n$. The last inequality is an application of \Cref{xuconv}, since $\mathbf{M}_n)_{\cdot k}$ is independent of $\mathbf{M}_n^{ \leq k-1}$, whereas $\boldsymbol{u}_{k-1,n,k-1}$ is measurable with respect to $\mathbf{M}_n^{\leq k-1}$. Note that the sequence of random variables $\{\mathbbm{1}_{\widetilde{\Omega}_{n,k}} - \mathbb{P} ( \widetilde{\Omega}_{n,k} \; \rvert \;\mathbf{M}_n^{ \leq k-1}) : n-T+2 \leq k \leq n\}$ is an Martingale difference sequence with respect to the filtration $\{\mathcal{F}^{\prime}_{n,k} : n-T+2 \leq k \leq n\}$ and with absolute values uniformly bounded above by $1$. Here $\mathcal{F}^{\prime}_{n,k}$ is the $\sigma$-algebra generated by the first $k$ columns of $\mathbf{M}_n$. By Azuma's inequality, we have for any $\kappa >0$, 
\begin{equation}{\label{gapw2}}
\mathbb{P} \biggl( \sum_{k=n-T+2}^n \mathbbm{1}_{\widetilde{\Omega}_{n,k}} - \sum_{k=n-T+2}^n \mathbb{P} ( \widetilde{\Omega}_{n,k} \; \rvert \;\mathbf{M}_n^{ \leq k-1})  \leq -\kappa \biggr) \leq \exp\biggl(-\dfrac{\kappa^2}{2(T-1)} \biggr). 	
\end{equation}
On the other hand, \Cref{gapw} tells us that, 
\begin{align}{\label{gapw3}}
	\sum_{k=n-T+2}^n \mathbb{P} ( \widetilde{\Omega}_{n,k} \; \rvert \;\mathbf{M}_n^{ \leq k-1}) &\geq \upsilon\sum_{k=n-T+2}^n \mathbbm{1}( \lambda_{k-1}(\mathbf{Y}_n^{(k-1)}) > 4\eta/3)  \geq \upsilon(T-1) \mathbbm{1}_{B_{n,T}}, 
\end{align}
where $B_{n,T}:= \{ \lambda_{n-T+1}(\mathbf{Y}_n^{(n-T+1)})< 4\eta/3 \}$. The last inequality in \Cref{gapw3} is again a consequence of non-decreasingness of  $\lambda_k(\mathbf{Y}_n^{(k)})$ in $k$. 
Set $\kappa = \upsilon(T-1)/2$ in \Cref{gapw2} to get
\begin{align*}
	\mathbb{P}\biggl( \sum_{k=n-T+2}^n \mathbbm{1}_{\widetilde{\Omega}_{n,k}} <  \dfrac{\upsilon(T-1)}{2}\biggr) 
	& \leq \exp(-\upsilon^{2}(T-1)/8) + \mathbb{P}(B_{n,T}^c).
\end{align*}
By \Cref{gapw4}, $ \lambda_n(\mathbf{Y}_n)-\lambda_{n-T+1}(\mathbf{M}_n)  \geq \eta  \sum_{k=n-T+2}^n \mathbbm{1}_{\widetilde{\Omega}_{n,k}}$ and hence 
\begin{align*}
	\mathbb{P} \biggl( \lambda_n(\mathbf{Y}_n)-\lambda_{n-T+1}(\mathbf{M}_n)  < \dfrac{ \eta \upsilon (T-1)}{2} \biggr) \leq \exp(-\upsilon^{2}(T-1)/8) + \mathbb{P}(B_{n,T}^c).
\end{align*}
Since $\theta_n =\bigO(1)$, we can conclude that $\lim_{T \to \infty} \limsup_{n \to \infty} \mathbb{P}(A_{n,T}^c)=0$, provided we can show that $\lim_{T \to \infty} \limsup_{n \to \infty} \mathbb{P}(B_{n,T}^c)=0$. But this is obvious from \Cref{eigen:expecbound2}.
\end{proof}

\begin{remark}
Let us make a short comparison between the scope of \Cref{wishtight1} and \Cref{wishtight2}. In \Cref{wishtight2}, we did away with any assumption on the row size $m_n$ and we only assumed finite moments upto order $4$ for matrix entries. In comparison, \Cref{wishtight1} needed linear growth for row size $m_n$ and at least finite moments of all order for the matrix entries (and possibly more restriction on its tail due to the nature of \Cref{mu}). Nevertheless, our general framework derived in \Cref{genthm:g} allowed us to prove tightness assuming $\theta_n = \bigO(1)$, which will be enough for our examples. 
\end{remark}

\begin{remark}{\label{eigen:improve2}}
Under some more assumptions on the distribution $\mathcal{P}$, one can again prove an weak convergence result for the gaps between  eigenvalues at the edge of the spectrum : $$\{ n^{-1/3}(\lambda_{n-i+1}(\mathbf{Y}_n)-\lambda_{n-i}(\mathbf{Y}_n) ) : 1 \leq i \leq k\}$$
 converges to a jointly continuous distribution for any $k \geq 1$. One such situation occurs when the distribution $\mathcal{P}$ is Gaussian, see \cite{anderson} and references within for more details of such results. In such scenarios, one can prove the tightness with the improved assumption that $\theta_n = \bigO(n^{1/3})$. Nevertheless, such a stronger version of \Cref{eigtight2} will not be required in our analysis.
\end{remark}

The following two theorems are the analogues of \Cref{eigenthm1} and \Cref{eigthm2} in context of Wishart matrices. 

\begin{thm}{\label{wishthm1}}
Consider the setup of Proposition~\ref{wishtight2}. The optimization problem in \Cref{wish:def} is then stable under small perturbations with perturbation blocks $\mathcal{J}_n = \{J_{n,(i,j)} : (i,j) \in [n]^2, i \leq  j\}$ where $J_{n,(i,j)} = \{(i,j)\}$ for all $1 \leq i \leq j \leq n$.
\end{thm}

\begin{proof}
Suppose that $\widehat{\boldsymbol{v}}=(\widehat{v}_1, \ldots, \widehat{v}_n)$ denotes the eigenvector corresponding to the largest eigenvalue for $\mathbf{M}_n^{\top}\mathbf{M}_n = \mathbf{Y}_n$ and for any $k \leq l \in [n]$,  $\widehat{\boldsymbol{v}}^{(k,l)}=(\widehat{v}_1^{(k,l)}, \ldots, \widehat{v}_n^{(k,l)})$ does the same for the matrix  $(\mathbf{M}_n^{(k,l)})^{\top}\mathbf{M}_n^{(k,l)}$ where $\mathbf{M}_n^{(k,l)}$ is the matrix obtained from $\mathbf{M}_n$ after replacing $X_{n,(k,l)}$ by an i.i.d.~copy $X_{n,kl}^{(k,l)}$. Then we have 
\begin{align}
\psi_n(\mathbf{M}_n^{(k,l)};\widehat{\boldsymbol{v}}) - \psi_{n, \mathrm{opt}}(\mathbf{M}_n) 
&=  \psi_n(\mathbf{M}_n^{(k,l)};\widehat{\boldsymbol{v}}) - \psi_n(\mathbf{M}_n;\widehat{\boldsymbol{v}}) \nonumber  \\
&= -2\sum_{j \neq l} (X_{n,kl}^{(k,l)}-X_{n,kl})X_{n,kj}\widehat{v}_j \widehat{v}_l -((X_{n,kl}^{(k,l)})^2 -(X_{n,kl})^2 )\widehat{v}_l^2. \label{exp1}
\end{align} 
Note that, $\left(\mathbf{M}_n^{(k,l)},\widehat{\boldsymbol{v}}\right) \stackrel{d}{=} \left( \mathbf{M}_n, \widehat{\boldsymbol{v}}^{(k,l)}\right)$, we have 
$$ \mathbb{E} \left[\psi_n(\mathbf{M}_n;\widehat{\boldsymbol{v}}^{(k,l)}) - \psi_{n, \mathrm{opt}}(\mathbf{M}_n)  \right]  = \mathbb{E} \left[ \psi_n(\mathbf{M}_n^{(k,l)};\widehat{\boldsymbol{v}}) - \psi_{n, \mathrm{opt}}(\mathbf{M}_n) \right].$$
Combining \Cref{exp1} and the fact that  random variable $X_{n,kl}^{(k,l)}$ has zero mean and is independent of $\widehat{\boldsymbol{v}}$, we have 
\begin{align}
\mathbb{E} [\psi_n(\mathbf{M}_n;\widehat{\boldsymbol{v}}^{(k,l)}) - \psi_{n, \mathrm{opt}}(\mathbf{M}_n) ] 
& =  2\sum_{j \neq l} \mathbb{E}  X_{n,kl}X_{n,kj}\widehat{v}_j \widehat{v}_l + \mathbb{E} ((X_{n,kl})^2 -(X_{n,kl}^{(k,l)})^2 ) \widehat{v}_l^2 \nonumber \\
& \leq 2\sum_{j} \mathbb{E}  X_{n,kl}X_{n,kj}\widehat{v}_j \widehat{v}_l \label{gapwish3} \\
& \leq 2(n-1) \mathbb{E}_{\mathcal{P}}X^2 \sqrt{ \mathbb{E}(\widehat{v}_1^2\widehat{v}_2^2)} + 2 \sqrt{\operatorname{Var}_{\mathcal{P}}(X^4)\mathbb{E}(\widehat{v}_1^4)}, \nonumber,
\end{align}
where we have used the exchangeability of the collection   $\{\widehat{v}_1, \ldots, \widehat{v}_n\}$, which also implies that $\mathbb{E}\widehat{v}_1^2\widehat{v}_2^2 \leq 1/(n(n-1))$ and $\mathbb{E}\widehat{v}_1^4 \leq \mathbb{E}\widehat{v}_1^2 = 1/n$. 
Thus 
\begin{equation}{\label{gapwish1}}
\max_{k \leq l} \mathbb{E} [\psi_n(\mathbf{M}_n;\widehat{\boldsymbol{v}}^{(k,l)}) - \psi_{n, \mathrm{opt}}(\mathbf{M}_n) ]  =\bigO(1).
\end{equation}
We can now finish the proof by applying \Cref{wishtight2} and \Cref{rem:strateasy} with $\theta_n=1.$
\end{proof}


\begin{thm}{\label{wishthm2}}
Consider the set-up of part (a) in Proposition~\ref{eigen:expecbound2} and assume that $m_n = \bigO(n)$. Then the optimization problem in \Cref{eign:def} is then stable under small perturbations with perturbation blocks 
$\mathcal{J}_n = \{J_{n,i} : i \in [n]\}$ where $J_{n,i}=\{(k,i) : k \in [m_n] \}$, for all $i \in [n]$. 
\end{thm}

\begin{proof}
For all $i \in [n]$, let $\widehat{\boldsymbol{v}}^{(i)}=(\widehat{v}^{(i)}_1,\ldots,\widehat{v}_n^{(i)})$ be the eigenvector corresponding to the largest eigenvalue of $(\mathbf{M}_n^{(i)})^{\top}\mathbf{M}_n^{(i)}$, where  $\mathbf{M}_n^{(i)}=(X_{n,kl}^{(i)})_{k,l \in [n]}$, which is obtained by replacing the  entries of $i$-th  column of $\mathbf{M}_n$ by i.i.d.~copies. $\widehat{\boldsymbol{v}}=(\widehat{v}_1,\ldots,\widehat{v}_n)$ is the same for matrix $\mathbf{M}_n$. Then for any $i \in [n] $, using similar arguments as in \Cref{eigthm2}, we arrive at  
\begin{align}
\mathbb{E} [ \psi_n(\mathbf{M}_n;\widehat{\boldsymbol{v}}^{(i)}) -\psi_{n,\mathrm{opt}}(\mathbf{M}_n) ] 
& =  \mathbb{E} \sum_{j \neq i } \sum_{k}  ( X^{(i)}_{n,ki} - X_{n,ki} )X_{n,kj} \widehat{v}^{(i)}_i\widehat{v}^{(i)}_j \nonumber \\
& \hspace{1 in } + \mathbb{E} \sum_{k}  ( (X^{(i)}_{n,ki})^2 - (X_{n,ki})^2) (\widehat{v}^{(i)}_i)^2 \nonumber \\
& \leq   \mathbb{E} \sum_{j \neq i } \sum_{k}   X^{(i)}_{n,ki} X_{n,kj} \widehat{v}^{(i)}_i\widehat{v}^{(i)}_j + \mathbb{E} \sum_{k}  (X^{(i)}_{n,ki})^2 (\widehat{v}^{(i)}_i)^2  \nonumber \\
& =  \mathbb{E} \sum_{j,k }  X^{(i)}_{n,ki} X_{n,kj} \widehat{v}^{(i)}_i\widehat{v}^{(i)}_j =  \mathbb{E} \sum_{j,k } X_{n,ki} X_{n,kj}\widehat{v}_i\widehat{v}_j.
\end{align}
It is again obvious that the distribution of the random quantity in \Cref{eq5} does not depend upon $i$ and hence,
\begin{align}
\mathbb{E} \biggl[ \sum_{j,k} X_{n,ki} X_{n,kj}  \widehat{v}_i\widehat{v}_j  \biggr] &= \dfrac{1}{n}\mathbb{E} \biggl[ \sum_{i ,j,k}  X_{n,ki}X_{n,kj} \widehat{v}_i\widehat{v}_j  \biggr] = \dfrac{1}{n}\mathbb{E} \psi_{n, \mathrm{opt}}(\mathbf{M}_n) = \bigO(1),
\end{align}
where we have applied \Cref{eigen:expecbound2}. Combining these estimates we obtain
\begin{equation}{\label{gapwish2}}
\max_{i \in [n]} \mathbb{E} [ \psi_n(\mathbf{M}_n;\widehat{\boldsymbol{v}}^{(i)}) -\psi_{n,\mathrm{opt}}(\mathbf{M}_n)] = \bigO(1).
\end{equation} 
Applying \Cref{wishtight2} and \Cref{rem:strateasy}, we observe that the hypotheses of \Cref{strat} is satisfied for $\theta_n=1$. This concludes the proof.
\end{proof}

\begin{remark}
One might be surprised to see that we got an $\bigO(1)$ upper bound in both \Cref{gapwish1} and~\Cref{gapwish2}. In fact, one would expect that the effect of replacing an entire column would be greater than the effect of replacing a single entry. This is indeed true and we can deduce a tighter upper bound in \Cref{wishtight1}. To see this, we start from \Cref{gapwish3} and then using exchangeability of the rows and columns of $\mathbf{M}_n$ we can write 
\begin{align}
&\mathbb{E} [\psi_n(\mathbf{M}_n;\widehat{\boldsymbol{v}}^{(k,l)}) - \psi_{n, \mathrm{opt}}(\mathbf{M}_n) ]  \leq  2\sum_{j} \mathbb{E}  X_{n,kl}X_{n,kj}\widehat{v}_j \widehat{v}_l \nonumber \\
& =  \dfrac{2}{n}  \mathbb{E}  \sum_{i,j \in [n]} X_{n,ki}X_{n,kj}\widehat{v}_i \widehat{v}_j  =  \dfrac{2}{nm_n}  \mathbb{E}  \sum_{i,j \in [n]} \sum_{k^{\prime} \in [m_n]} X_{n,k^{\prime}i}X_{n,k^{\prime}j}\widehat{v}_i \widehat{v}_j \nonumber \\
& \leq   \dfrac{2}{nm_n} \mathbb{E} \lambda_n \left( \mathbf{M}_n^{\top}\mathbf{M}_n\right) = \bigO \left(\dfrac{n+m_n}{nm_n} \right) = \bigO \left(1/n + 1/m_n \right). \label{gapwish4}
\end{align}
Provided that $m_n \to \infty$, this upper bound improves on \Cref{gapwish1}. Moreover, if $m_n/n$ is bounded away from $0$, then the upper bound in \Cref{gapwish4} becomes $\bigO(1/n)$. One such example is when we assume $m_n \geq n$ to make sure that the matrix $\mathbf{Y}_n=\mathbf{M}_n^{\top}\mathbf{M}_n$ is non-singular almost surely. On the other hand, if $m_n = \bigO(1)$, then \Cref{gapwish4} also results in an $\bigO(1)$ upper bound. This is quite reasonable since replacing each column here amounts to replacing only a finite number of entries.
\end{remark}





\subsection{Graph optimization on Euclidean spaces}\label{sec:nonlin}
In this section we shall analyze a class of optimization problems regarding finite graphs on Euclidean spaces. They are our quintessential examples of application of \Cref{genthm} and \Cref{genthm:g} with their full strength. We shall focus on generic graph optimization problems where we take a specific number of points, independent and identically distributed, from an Euclidean space and find out the shortest graph with these points as vertices among a fixed collection of pre-specified graphs. This can be formulated in the following form.
\begin{defn}[Graph Optimization Problem on Euclidean spaces]{\label{graph:def}}
A generic graph optimization problem on Euclidean space has following components.
\begin{enumerate}[label=(EuOpt\Alph*)]
	\item \label{item:1} Number of input points $p_n \in \mathbb{N}$ with $p_n \uparrow \infty$ as $n \to \infty$.  Typically $p_n$ will be $\Theta(n)$.
	\item \label{item:2} A collection of graphs with vertex set $[p_n]=\{1,\ldots,p_n\}$. We shall denote this collection by $\mathcal{G}_n$. For any graph $G \in \mathcal{G}_n$, its edge set will be denoted by $E(G)$. We shall further assume that 
	$$\sup_{n \geq 1} \sup_{G \in \mathcal{G}_n} \max_{i \in [p_n]} d_i(G) =: D < \infty,$$
	where $d_i(G)$ is the degree of the vertex $i$ in graph $G$. 
	\item A probability measure $\mathcal{P}$ on $\mathbb{R}^d$ which will be the distribution of the inputs and a \textit{power weight} parameter $ q \in [1,d)$.
\end{enumerate}
Get $X_{n,1},\ldots,X_{n,p_n} \stackrel{i.i.d.}{\sim} \mathcal{P}$. The graph optimization problem can be written as 
$$ \text{Minimize} \; \sum_{i,j \in [p_n] : \{i,j\} \in E(G)} \| X_{n,i} - X_{n,j}\|^{q}_2, \;\text{ over } G \in \mathcal{G}_n.$$
\end{defn}

The case with $q=1$ is considered the standard case where we intend to find out the graph with shortest Euclidean length. A vast number of interesting problems fall into the category as described in \Cref{graph:def}. Among them we shall particularly consider in detail the traveling salesman problem (TSP) and the minimum spanning tree (MST). 
These examples will be discussed separately in detail later in this section. Nevertheless, the assumptions made in \Cref{graph:def} are by themselves strong enough to guarantee a statement  like \Cref{item:A} for generic graph optimization problems as demonstrated via the following theorem.

\begin{assumption}{\label{graphtight:ass}}
Fix $q \in [1,d)$. Assume that $\mathcal{P}$ is a probability measure on $\mathbb{R}^d$ satisfying
\begin{enumerate}
\item $\mathcal{P}$ has a density $f$(with respect to the Lebesgue measure on $\mathbb{R}^d$), which is bounded, almost everywhere continuous and $\mathbb{E}_{\mathcal{P}}\|X\|_2^q < \infty$.
\item Assume that 
\begin{equation}{\label{rate}}
 \mathbb{E} \biggl[ \inf_{G \in \mathcal{G}_n} \sum_{i,j \in [p_n] : \{i,j\} \in E(G)} \| X_{i} - X_{j}\|^{q}_2  \biggr] = \bigO( p_n^{(d-q)/d}), \text{ as } n \to \infty,
\end{equation}
where $X_1,X_2,\ldots \stackrel{i.i.d.}{\sim} \mathcal{P}$. 
\item $\mathcal{P}$ satisfies Assumption~\ref{mu} for the pair $(\rho,g)$ where $\nabla g \neq 0$ almost everywhere. Moreover, if $q \in [1,1+2/d)$, then this assumption can be replaced by Assumption~\ref{mu}.  This includes the standard case of $q=1$. 
\end{enumerate}
\end{assumption}

\begin{thm}{\label{nearoptgraph}}
Consider a generic graph optimization problem as in Definition~\ref{graph:def}. Assume that $\mathcal{P}$ satisfies Assumption~\ref{graphtight:ass}.  Take  $X_1,X_2,\ldots \stackrel{i.i.d.}{\sim} \mathcal{P}$. Under these assumptions, $\{P(\mathcal{N}_{n,\theta_n},d_n,\varepsilon) : n \geq 1\}$ is tight for any $\varepsilon >0$ if $\theta_n = \bigO(p_n^{-q/d})$.
Here $d_n$ is the metric on $\mathcal{G}_n$ defined as follows.
\begin{align*}
d_n(G_1,G_2) &:= \dfrac{1}{p_n} \sum_{i \neq j \in [p_n]} \rvert \mathbbm{1}( \{i,j\} \in E(G_1)) - \mathbbm{1}( \{i,j\} \in E(G_2))\rvert \\
& = \dfrac{1}{p_n}\operatorname{card}(E(G_1)\Delta E(G_2) ), \forall \; G_1,G_2 \in \mathcal{G}_n,
\end{align*}
where $\Delta$ refers to symmetric difference between two sets.
\end{thm}

\begin{remark}
The uniform bounded degree assumption made in \Cref{item:2} guarantees that for all $G \in \mathcal{G}_n$, we have $\operatorname{card}(E(G)) \leq Dp_n$ . This motivates the choice of the scaling in the definition of $d_n$ since this ensures that $0 \leq d_n \leq 2D$. 
\end{remark}

\begin{remark}
As can be seen from \Cref{exmu} and \Cref{condrho}, the scope of \Cref{nearoptgraph} contains a plethora of standard examples like multivariate versions of Gaussian, Gamma distributions, product of uniform distributions, Dirichlet distributions, etc.
\end{remark}

\begin{proof}[Proof of Theorem \ref{nearoptgraph}]
We shall show that the assumptions made in the statement of \Cref{genthm:g} are satisfied in this case for $q \in [1,d)$, whereas for $q \in [1,2)$ they also satisfy the assumptions in \Cref{genthm}. To connect the current setup with the notations in \Cref{genthm} and \Cref{genthm:g}, note that we have $\mathcal{S}_n= \mathcal{G}_n$, $I_n = [p_n]$ and 
$$ \psi_{n}( ( x_{i})_{i \in [p_n]}; G) = \sum_{i,j \in [p_n] : \{i,j\} \in E(G)} \| x_{i} - x_{j}\|^{q}_2, \; \forall \; G \in \mathcal{G}_n.$$
We shall denote by $D$ the uniform bound on the degree of the graphs in $\mathcal{G}_n$, i.e. $d_i(G) \leq D < \infty$, for all $i \in [p_n], G \in \mathcal{G}_n$ and $n \geq 1$. It is enough to consider $X_{n,i}=X_i$, for all $i \in [p_n]$ and $n \geq 1$ where $X_1,X_2,\ldots \stackrel{i.i.d.}{\sim} f$. For the sake of brevity, for any $\mathbf{x}^n =(x_i)_{i \in [p_n]}$ and any $i,j \in [p_n]$, we shall use the notation $\Delta x_{ij}$ to denote $x_i-x_j$.  Throughout the proof we shall use the following crude bound many times without reference:
\begin{equation}{\label{crude}}
\biggl( \sum_{j=1}^m a_j \biggr)^r \leq m^r \sum_{j=1}^m a_j^r, \; \text{ for any } a_1,\ldots,a_m,r \geq 0, \; m \in \mathbb{N}.
\end{equation}
Starting the main arguments behind the proof, first of all note that \Cref{unique} is clearly satisfied since $\psi_n((X_{n,i})_{i \in [p_n]}; G)$'s are distinct with probability $1$; this is because $X_i$'s have density $f$ with respect to Lebesgue measure and $\mathcal{G}_n$ is finite. On the other hand,  
$$ 0 \leq \inf_{G \in \mathcal{G}_n} \psi_n((X_{i})_{i \in [p_n]}; G) \leq 2^{q}D \sum_{i \in [p_n]} \|X_i\|^{q}_2,$$
where the right hand side is integrable by assumption. This proves the  assertion required to satisfy \Cref{ass:meas}.  %
In order to show the validity of \Cref{ass:psi}, we observe that since the function 
$$\mathbf{x}^n = (x_{i})_{i \in [p_n]} \mapsto \psi_n((x_{i})_{i \in [p_n]}; G)$$
is jointly convex for any $G \in \mathcal{G}_n$ (a consequence of $q \geq 1$), we can apply a first order approximation as follows: for (almost) any $\mathbf{x}^n=(x_i)_{i \in [p_n]}, \mathbf{c}^n=(c_i )_{i \in [p_n]} \in (\mathbb{R}^d)^{[p_n]}$,
\begin{equation}{\label{tspstep31}}
\psi_n(\mathbf{x}^n+\mathbf{c}^n; G)-\psi_n(\mathbf{x}^n; G) \leq \sum_{i \in [p_n]} c_i \cdot \partial_i \psi_n(\mathbf{x}^n+\mathbf{c}^n; G),
\end{equation}
where $\partial_i$ refers to gradient with respect to the $i$-th input. Computing the gradient and plugging it into (\ref{tspstep31}), we obtain the following: 
\begin{align}{\label{tspstep311}}
\psi_n(\mathbf{x}^n+\mathbf{c}^n; G)-\psi_n(\mathbf{x}^n; G) & \leq q \sum_{i \in [p_n]} c_i \cdot \biggl[ \sum_{\substack{j : \{i,j \} \in E(G)}} \dfrac{\Delta x_{ij}+\Delta c_{ij}}{\big\|\Delta x_{ij}+\Delta c_{ij}\big\|^{2-q}_2}\biggr] \nonumber \\
& = q \sum_{\{i,j\} : \{i,j\} \in E(G)} \dfrac{\Delta c_{ij} \cdot (\Delta x_{ij}+\Delta c_{ij} )}{\big\|\Delta x_{ij}+\Delta c_{ij}\big\|^{2-q}_2}.
\end{align}
The equality in \Cref{tspstep311} is derived by changing the sum over vertices to a sum over edges and observing that each edge $\left\{i,j\right\} \in E(G)$ contributes a summand of the form 
$$ qc_i \cdot \dfrac{\Delta x_{ij}+\Delta c_{ij}}{\big\|\Delta x_{ij}+\Delta c_{ij}\big\|^{2-q}_2} + q c_j \cdot \dfrac{\Delta x_{ji}+\Delta c_{ji}}{\big\|\Delta x_{ji}+\Delta c_{ji}\big\|^{2-q}_2} = q \dfrac{\Delta c_{ij} \cdot (\Delta x_{ij}+\Delta c_{ij} )}{\big\|\Delta x_{ij}+\Delta c_{ij}\big\|^{2-q}_2},$$
to the sum.

Before diving further into the analysis of the term in the right hand side of \Cref{tspstep311}, where we try to prove that it permits an expansion satisfying \Cref{item:C1}, we introduce two generic terms which will be encountered repeatedly throughout the process of establishing this aforementioned expansion. For any $\mathbf{x}^n=(x_i )_{i \in [p_n]}, \mathbf{c}^n=(c_i)_{i \in [p_n]} \in (\mathbb{R}^d)^{[p_n]}$; $r_1 \in [0,\infty), r_2 \in [0,q)$ and $r_3,r \in [0,d)$, define the following for any $G \in \mathcal{G}_n$:
\begin{equation}{\label{termdef1}}
T( \mathbf{x}^n,\mathbf{c}^n; G;r_1,r_2,r_3) := \sum_{i \in [p_n]} \dfrac{\sum_{ j: \{i,j\} \in E(G)} \|\Delta c_{ij} \|_2^{r_1} \;\|\Delta x_{ij} \|_2^{r_2}}{\min_{j \in [p_n]: j \neq i} \|\Delta x_{ij} + \Delta c_{ij} \|_2^{r_3}},
\end{equation}
\begin{equation}{\label{termdef11}}
Q( \mathbf{x}^n;r) := \sum_{i \in [p_n]} \dfrac{1}{\min_{j \in [p_n]: j \neq i} \|\Delta x_{ij} \|_2^{r}}.
\end{equation}
Note that, 
\begin{equation}{\label{TScomp}}
T( \mathbf{x}^n,\mathbf{c}^n; G;0,0,r) \leq D Q( \mathbf{x}^n+\mathbf{c}^n;r),
\end{equation} 
where $D$ is the maximum possible degree for any vertex in the graph $G$, as warranted by the assumption in \Cref{item:2}. For any $u \in (0,\infty)$, H\"older's inequality, aided by \Cref{crude}, yields the following bound which will also be instrumental in our computations. Omitting the reference to $\mathbf{x}^n, \mathbf{c}^n$ and $G$ in the notation $T( \mathbf{x}^n,\mathbf{c}^n; G; r_1,r_2,r_3 )$ and $Q(\mathbf{x}^n;r)$ for the sake of brevity, note that
\begin{align}
T(r_1,r_2,r_3)&\leq \biggl[ \sum_{i \in [p_n]} \biggl(\sum_{j  : \{i,j\} \in E(G)}  \|\Delta c_{ij} \|_2^{r_1} \;\|\Delta x_{ij} \|_2^{r_2} \biggr)^{1+u} \biggr]^{1/(1+u)} \nonumber \\
& \hspace{ 1.5 in} \biggl[ \sum_{i \in [p_n]} \dfrac{1}{\min_{j \in [p_n]: j \neq i} \|\Delta x_{ij} + \Delta c_{ij} \|_2^{r_3(1+u)/u}} \biggr]^{u/(1+u)} \nonumber \\
& \leq D \biggl[ \sum_{i \in [p_n]} \sum_{j  : \{i,j\} \in E(G)}  \|\Delta c_{ij} \|_2^{r_1(1+u)} \;\|\Delta x_{ij} \|_2^{r_2(1+u)} \biggr]^{1/(1+u)} \nonumber \\
& \hspace{ 1.5 in} \biggl[ \sum_{i \in [p_n]} \dfrac{1}{\min_{j \in [p_n]: j \neq i} \|\Delta x_{ij} + \Delta c_{ij} \|_2^{r_3(1+u)/u}} \biggr]^{u/(1+u)} \nonumber \\
& \leq D \biggl[ T (r_1(1+u),r_2(1+u),0 )\biggr]^{\frac{1}{1+u}}\biggl[ T\biggl(0,0,\frac{r_3(1+u)}{u}\biggr)\biggr]^{\frac{u}{1+u}}. \label{termdef2}
\end{align}
Of course we assume $u$ satisfies $r_2(1+u)\in [0,q)$ and $r_3(1+u)/u \in [0,d)$ in \Cref{termdef2}.

	Coming back to our original computations, we consider the cases $q \in [1,2)$ and $q \in [2,d)$ separately. For the case $1 \leq q <2$,
	we apply \Cref{binombound2} to obtain the following for any $i \neq j \in [p_n]$:
	\begin{align}{\label{tspstep32}}
	 & \Bigg \|	 \dfrac{\Delta x_{ij}+\Delta c_{ij}}{\big\|\Delta x_{ij}+\Delta c_{ij} \big\|_2^{2-q}} - \dfrac{\Delta x_{ij}}{\|\Delta x_{ij}\|_2^{2-q}}  \Bigg \|_2 \nonumber \\
	   & \leq \|\Delta x_{ij}\|_2 \Bigg \rvert\dfrac{1}{\|\Delta x_{ij}+\Delta c_{ij}\|_2^{2-q}} - \dfrac{1}{\|\Delta x_{ij}\|_2^{2-q}} \Bigg \rvert + \dfrac{\|\Delta c_{ij}\|_2}{\|\Delta x_{ij}+\Delta c_{ij}\|^{2-q}_2} \nonumber \\
	& = \|\Delta x_{ij}\|_2 \dfrac{\big \rvert \|\Delta x_{ij}+\Delta c_{ij}\|^{2-q}_2 - \|\Delta x_{ij}\|_2^{2-q} \big \rvert}{\|\Delta x_{ij}+\Delta c_{ij}\|^{2-q}_2 \|\Delta x_{ij}\|_2^{2-q}}  + \dfrac{\|\Delta c_{ij}\|_2}{\|\Delta x_{ij}+\Delta c_{ij}\|^{2-q}_2} \nonumber \\
		& \leq \dfrac{(2-q)\|\Delta x_{ij}\|_2^{q-1}\|\Delta c_{ij}\|_2}{\|\Delta x_{ij}+\Delta c_{ij}\|_2^{2-q}}(   \|\Delta x_{ij} \|_2^{1-q} +  \|\Delta x_{ij} + \Delta c_{ij}\|_2^{1-q}  ) + \dfrac{\|\Delta c_{ij}\|_2}{\|\Delta x_{ij}+\Delta c_{ij}\|_2^{2-q}} \nonumber \\
		& \leq \dfrac{ \|\Delta c_{ij}\|_2 \; \|\Delta x_{ij}\|_2^{q-1}}{\|\Delta x_{ij}+\Delta c_{ij}\|_2} + \dfrac{2\|\Delta c_{ij}\|_2}{\|\Delta x_{ij}+\Delta c_{ij}\|_2^{2-q}}.
	\end{align}
The inequality in \Cref{tspstep32}, after plugging it  in \Cref{tspstep311}, yields the following.
\begin{align}{\label{choice5}}
& \psi_n(\mathbf{x}^n+\mathbf{c}^n; G)-\psi_n(\mathbf{x}^n; G) \nonumber  \\
& \leq q\sum_{\{i,j\} \in E(G)}  \dfrac{\Delta c_{ij} \cdot \Delta x_{ij}}{\|\Delta x_{ij}\|_2^{2-q}} +  q \sum_{\{i,j\} \in E(G)} \dfrac{ \|\Delta c_{ij}\|_2^2 \; \|\Delta x_{ij}\|_2^{q-1}}{\|\Delta x_{ij} + \Delta c_{ij}\|_2}  + 2q \sum_{\{i,j\} \in E(G)} \dfrac{ \|\Delta c_{ij}\|_2^2 }{\|\Delta x_{ij} + \Delta c_{ij}\|_2^{2-q}} \nonumber \\
 &  \leq q \sum_{i \in [p_n]} c_i \cdot  \biggl[ \sum_{\substack{j : \{i,j \} \in E(G)}} \dfrac{\Delta x_{ij}}{\|\Delta x_{ij}\|_2^{2-q}}\biggr] +  \dfrac{q}{2}  \sum_{i \in [p_n]} \dfrac{\sum_{j:\{i,j\} \in E(G)}  \|\Delta c_{ij}\|_2^2 \; \|\Delta x_{ij}\|_2^{q-1}}{\min_{j \in [p_n] : j \neq i} \|\Delta x_{ij} + \Delta c_{ij}\|_2}  \nonumber  \\
& \hspace{3 in} + q  \sum_{i \in [p_n]} \dfrac{\sum_{j:\{i,j\} \in E(G)}  \|\Delta c_{ij}\|_2^2 }{\min_{j \in [p_n] : j \neq i} \|\Delta x_{ij} + \Delta c_{ij}\|_2^{2-q}}  \nonumber \\
& = q \sum_{i \in [p_n]} c_i \cdot  \biggl[ \sum_{\substack{j : \{i,j \} \in E(G)}} \dfrac{\Delta x_{ij}}{\|\Delta x_{ij}\|_2^{2-q}}\biggr]  + \dfrac{q}{2}T(2,q-1,1) + qT(2,0,2-q) \nonumber \\
& \leq q \sum_{i \in [p_n]} c_i \cdot  \biggl[ \sum_{\substack{j : \{i,j \} \in E(G)}} \dfrac{\Delta x_{ij}}{\|\Delta x_{ij}\|_2^{2-q}}\biggr] \notag \\
&\qquad \qquad + \dfrac{qD}{2} \biggl[T (2(1+s),(q-1)(1+s),0)\biggr]^{\frac{1}{1+s}}  \biggl[T \biggl(0,0,\dfrac{1+s}{s}\biggr) \biggr]^{\frac{s}{1+s}} \nonumber  \\
& \hspace{ 2.5 in}  + qD \biggl[T(6,0,0)\biggr]^{1/3} \biggl[T\biggl(0,0,\dfrac{3(2-q)}{2}\biggr)\biggr]^{2/3},
\end{align}
for any $s \in (1,1/(q-1))$, which guarantees that $(q-1)(1+s)<q$ while $(1+s)/s<2 \leq d$.  The last inequality in \Cref{choice5} follows from \Cref{termdef2} applied with $u=s$ and $u=2$ respectively for the last two terms. We now clearly have the expansion required to satisfy \Cref{item:B1} with $L=2$, 
\begin{equation}{\label{choice1}}
 H_{n,i}(\mathbf{x}^n;G) = \sum_{\substack{j : \{i,j \} \in E(G)}} \dfrac{\Delta x_{ij}}{\|\Delta x_{ij}\|_2^{2-q}}, \forall\; i \in [p_n], 
\end{equation}
\begin{align}{\label{choice6}}
R_{n,1}=\dfrac{qD}{2} \biggl[T (2(1+s),(q-1)(1+s),0)\biggr]^{\frac{1}{1+s}}, \; \; R_{n,2} = qD \biggl[T(6,0,0)\biggr]^{1/3},
\end{align}
and
\begin{equation}{\label{choice7}}
\widetilde{R}_{n,1} =  \biggl[Q \biggl(\dfrac{1+s}{s}\biggr) \biggr]^{\frac{s}{1+s}}, \; \; \widetilde{R}_{n,2} =  \biggl[Q \biggl(\dfrac{3(2-q)}{2}\biggr) \biggr]^{3/2}.
\end{equation}
Here \Cref{item:B2} is satisfied for the choice $\upsilon_1 = \upsilon_2= 1$ and it is clear that \Cref{item:B2prime} also holds true.

We now concentrate on the case $q \in [2,d)$.  In this particular scenario, we can employ \Cref{binombound2} 
to arrive at the following for any $i \neq j \in [p_n]$:
\begin{align}{\label{alphageq2}}
&   \|	 (\Delta x_{ij} + \Delta c_{ij})\|\Delta x_{ij} + \Delta c_{ij} \|_2^{q-2} - (\Delta x_{ij})\|\Delta x_{ij}\|_2^{q-2}  \|_2 \nonumber \\   
& \leq \|\Delta x_{ij}\|_2  \rvert	\|\Delta x_{ij} + \Delta c_{ij} \|_2^{q-2} - \|\Delta x_{ij}\|_2^{q-2}  \rvert +  \|\Delta c_{ij}\|_2  	\|\Delta x_{ij} + \Delta c_{ij} \|_2^{q-2} \nonumber \\
& \leq \|\Delta x_{ij}\|_2 \biggl(\sum_{k=1}^{\lfloor q \rfloor -1} (q-2)^k \|\Delta c_{ij}\|_2^k \; \|\Delta x_{ij}\|_2^{q-2-k} \notag\\
&\qquad \qquad \qquad \qquad + (q-2)^{\lfloor q \rfloor -1} \|\Delta c_{ij}\|_2^{\lfloor q \rfloor -1} \|\Delta x_{ij} + \Delta c_{ij}\|^{q-\lfloor q \rfloor -1}\;  \biggr) \nonumber \\
& \hspace{3.5 in} + 2^{q-2} \|\Delta c_{ij}\|_2 (\|\Delta x_{ij}\|_2^{q-2} + \|\Delta c_{ij}\|_2^{q-2})  \nonumber \\
& \leq d^d \sum_{k=1}^{\lfloor q \rfloor -1} \|\Delta c_{ij}\|_2^k \; \|\Delta x_{ij}\|_2^{q-1-k} + d^d \|\Delta c_{ij}\|_2^{\lfloor q \rfloor -1} \|\Delta x_{ij}\|_2 \; \|\Delta x_{ij} + \Delta c_{ij}\|^{q-\lfloor q \rfloor -1} \nonumber \\
& \hspace{3 in} + 2^{q-2} \|\Delta c_{ij}\|_2 \|\Delta x_{ij}\|_2^{q-2} + 2^{q-2}\|\Delta c_{ij}\|_2^{q-1} \nonumber \\
&   \leq 2d^d \sum_{k=1}^{\lfloor q \rfloor -1} \|\Delta c_{ij}\|_2^k \; \|\Delta x_{ij}\|_2^{q-1-k} + d^d \dfrac{\|\Delta c_{ij}\|_2^{\lfloor q \rfloor -1} \|\Delta x_{ij}\|_2}{ \|\Delta x_{ij} + \Delta c_{ij}\|^{\lfloor q \rfloor +1-q}}  + 2^{q-2}\|\Delta c_{ij}\|_2^{q-1}.
\end{align}
Plugging in the upper bound from \Cref{alphageq2} in the inequality \Cref{tspstep311}, we obtain the following:
\begin{align}
& \psi_n(\mathbf{x}^n+\mathbf{c}^n; G)-\psi_n(\mathbf{x}^n; G) \nonumber \\
 &\leq q\sum_{i \in [p_n]} c_i \cdot  \biggl[ \sum_{\substack{j : \{i,j \} \in E(G)}} (\Delta x_{ij})\|\Delta x_{ij}\|_2^{q-2}\biggr] + qd^d \sum_{k=1}^{\lfloor q \rfloor -1} T ( k+1,q-1-k,0) \nonumber \\
 & \hspace{ 2.5 in}  + \dfrac{qd^d}{2} T ( \lfloor q \rfloor, 1, \lfloor q \rfloor+ 1-q) + q2^{q-3}T(q,0,0) \nonumber \\
 & \leq q\sum_{i \in [p_n]} c_i \cdot  \biggl[ \sum_{\substack{j : \{i,j \} \in E(G)}} (\Delta x_{ij})\|\Delta x_{ij}\|_2^{q-2}\biggr] + qd^d \sum_{k=1}^{\lfloor q \rfloor -1} T ( k+1, q-1-k,0) \nonumber \\
  & \hspace{ 0.5 in}  + \dfrac{qDd^d}{2} \biggl[ T ( 2\lfloor q \rfloor, 2,0)\biggr]^{1/2} \biggl[ T ( 0, 0, 2(\lfloor q \rfloor +1-q) )\biggr]^{1/2} + q2^{q-3}T(q,0,0). \label{ineqabove}
\end{align}
The inequality in \Cref{ineqabove}, aided by \Cref{TScomp},  shows that we have the expansion required to satisfy \Cref{item:B1} with the following choices : $L=\lfloor q \rfloor +1$, 
\begin{equation}{\label{choice8}}
 H_{n,i}(\mathbf{x}^n;G) = q\sum_{\substack{j : \{i,j \} \in E(G)}} (\Delta x_{ij} ) \|\Delta x_{ij}\|_2^{q-2}, \forall\; i \in [p_n].
\end{equation}
For all $1 \leq k \leq \lfloor q \rfloor -1$, we have 
\begin{equation}{\label{choice9}}
R_{n,k} = qd^d  T ( k+1,q-1-k,0),  
\end{equation}
and 
\begin{equation}{\label{choice10}}
R_{n,\lfloor q \rfloor} = \dfrac{qDd^d}{2} \biggl[ T ( 2\lfloor q \rfloor, 2,0 )\biggr]^{1/2}, \; \;  R_{n,\lfloor q \rfloor + 1} = q2^{q-3}T(q,0,0).
\end{equation}
Correspondingly, we can take $ \widetilde{R}_{n,k} \equiv 1 $, for all $k \neq \lfloor q \rfloor$, while 
\begin{equation}{\label{choice11new2}}
\widetilde{R}_{n,\lfloor q \rfloor}= \; \biggl[Q( 2(\lfloor q \rfloor +1-q))\biggr]^{1/2}.
\end{equation}
One can  check that \Cref{item:B2} and \cref{item:B2prime} are satisfied for the choices $\upsilon_k = k$ for $1 \leq k \leq \lfloor q \rfloor-1$, $\upsilon_{\lfloor q \rfloor} =\lfloor q \rfloor -1$ and  $\upsilon_{\lfloor q \rfloor+1}=q-1$. 




We shall now proceed  prove the assumptions in the statement of \Cref{genthm},  one by one.

\begin{itemize}


\item \textit{Proof of \Cref{item:C1} and \Cref{item:E1}:}

As we can see from the previous computations, for all $q \in [1,d)$, we have 
$$ H_{n,i}(\mathbf{x}^n;G) = q\sum_{\substack{j : \{i,j \} \in E(G)}} (\Delta x_{ij} ) \|\Delta x_{ij}\|_2^{q-2}, \forall\; i \in [p_n].$$
Therefore, 
\begin{align*}
&\sup_{G \in \mathcal{N}_{n,\theta_n}} \sum_{i \in [p_n]} \bigg \|H_{n,i}(\mathbf{X}^n;G) \bigg \|_2^{q/(q-1)} \\
& \leq q^{q/(q-1)}\sup_{G \in \mathcal{N}_{n,\theta_n}} \sum_{i \in [p_n]} \bigg \| \sum_{\substack{j : \{i,j \} \in E(G)}} (\Delta X_{ij} ) \|\Delta X_{ij}\|_2^{q-2} \bigg \|_2^{q/(q-1)} \\
& \leq (qD)^{q/(q-1)}\sup_{G \in \mathcal{N}_{n,\theta_n}} \sum_{i \in [p_n]}  \sum_{\substack{j : \{i,j\} \in E(G)}} \|\Delta X_{ij}\|_2^{q} \\
& = 2  (qD)^{q/(q-1)} \sup_{G \in \mathcal{N}_{n,\theta_n}} \psi_n ( \mathbf{X}^n; G ) \\
& \leq 2  (qD)^{q/(q-1)} \biggl( \inf_{G \in \mathcal{G}_{n}} \psi_n (\mathbf{X}^n; G) + \theta_n\biggr),
\end{align*}
which leads us to the following by virtue of the  assumption \Cref{rate} in the hypothesis of \Cref{nearoptgraph}. 
\begin{equation}{\label{nearoptgraph:E1C1}}
 \biggl[ \mathbb{E} \sup_{G \in \mathcal{N}_{n,\theta_n}} \sum_{i \in [p_n]}  \|H_{n,i}(\mathbf{X}^n;G) \|_2^{q/(q-1)} \biggr]^{(q-1)/q} = \bigO( p_n^{(d-q)(q-1)/qd}).
\end{equation}
The growth bound derived in \Cref{nearoptgraph:E1C1} shows that the assumptions in \Cref{item:C1} and \Cref{item:E1} are satisfied with $\lambda = q-1$ and $\varsigma_{n,\lambda} = p_n^{(d-q)(q-1)/qd}.$

\item \textit{Proof of \Cref{item:C2} and \Cref{item:E2}:}
Since \Cref{item:E2} is a stronger statement than \Cref{item:C2}, it is enough to only prove the former. Recall that $\mathbf{X}^n=(X_{i})_{i \in [p_n]}$, where $X_{i} \stackrel{i.i.d.}{\sim} f$; and 
take $\mathbf{Z}^n=(Z_{i})_{i \in [p_n]}$ to be a collection of sub-Gamma vectors with variance proxy $1$ and scale parameter $1$ such that $\{(X_i,Z_i)\}_{i \in [p_n]}$ is an i.i.d.~collection.  
Looking at the expressions of the functions $R_{n,k}$'s that we derived earlier, it is evident that, they are of the form 
\begin{equation}{\label{form}}
R_{n,k}( \mathbf{x}^n, \mathbf{c}^n;G ) = \mathrm{constant}(d,q,D)\biggl[T ( \mathbf{x}^n,\mathbf{c}^n;G;r_1,r_2,0 ) \biggr]^{t/(1+t)}
\end{equation}
for some $r_1 \in [0, \infty)$, $r_2 \in [0,q)$ and $t \in (0,\infty]$, where $\infty/(1+\infty):=1$; $\mathrm{constant}(d,q,D)$ is a constant depending on $d,q,D$ and the particular  optimization problem we are concerned with.  Therefore,
in order to prove \Cref{item:E2}, our attention should be focused on deriving upper bound to the following term.
$$ \mathbb{E} \bigg[\sup_{G \in \mathcal{N}_{n,\theta_n}} T (\mathbf{X}^n, \mathbf{Z}^n; G; r_1,r_2,0) \bigg].$$


Using H\"older's inequality, we obtain the following for ant $G \in \mathcal{N}_{n,\theta_n}$:
\begin{align}{\label{chistep1}}
&T(\mathbf{X}^n,\mathbf{Z}^n; G; r_1,r_2,0) =  \sum_{i \in [p_n]}  \sum_{\substack{j : \{i,j \} \in E(G)}}\|\Delta Z_{ij}\|_2^{r_1} \; \|\Delta X_{ij}\|_2^{r_2} \nonumber \\
& \leq \biggl( \sum_{i \in [p_n]}  \sum_{\substack{j : \{i,j\} \in E(G)}} \|\Delta Z_{ij}\|_2^{\frac{qr_1}{q-r_2}} \biggr)^{(q-r_2)/q}  \biggl( \sum_{i \in [p_n]}  \sum_{\substack{j : \{i,j \} \in E(G)}} \|\Delta X_{ij}\|_2^q\biggr)^{r_2/q} \nonumber \\
& \leq 2^{r_1}\biggl[ \sum_{i \in [p_n]}  \sum_{\substack{j : \{i,j \} \in E(G)}} \biggl(\|Z_i\|_2^{\frac{qr_1}{q-r_2}} + \|Z_j\|_2^{\frac{qr_1}{q-r_2}} \biggr) \biggr]^{(q-r_2)/q}  (\psi_n(\mathbf{X}^n;G) )^{r_2/q} \nonumber \\
& \leq 2^{r_1} (2D)^{(q-r_2)/q}\biggl[ \sum_{i \in [p_n]}  \|Z_i\|_2^{\frac{qr_1}{q-r_2}} \biggr]^{(q-r_2)/q}  (\psi_n(\mathbf{X}^n;G) )^{r_2/q}.
\end{align}
Hence, we arrive at the following bound:
\begin{align}
&\sup_{G \in \mathcal{N}_{n,\theta_n}} T(\mathbf{X}^n,\mathbf{Z}^n; G, r_1,r_2,0) \notag \\
& \leq 
2^{r_1} (2D)^{\frac{q-r_2}{q}}\biggl[ \sum_{i \in [p_n]}  \|Z_i\|_2^{\frac{qr_1}{q-r_2}} \biggr]^{\frac{q-r_2}{q}}  \biggl( \inf_{G \in \mathcal{G}_n} \psi_n(\mathbf{X}^n;G) + \theta_n\biggr)^{\frac{r_2}{q}}. \label{chistep2}
\end{align}
We concentrate on the two terms in the right hand side of \Cref{chistep2} separately. 
By Lemma~\Cref{lemsubg:4}, we have that 
\begin{equation}{\label{chistep3}}
\mathbb{E} \biggl[ \sum_{i \in [p_n]}  \|Z_i\|_2^{\frac{qr_1}{q-r_2}} \biggr] = \bigO (p_n),
\end{equation}
whereas  \Cref{rate} guarantees that 
\begin{equation}{\label{chistep4}}
\mathbb{E} \biggl[\inf_{G \in \mathcal{G}_n} \psi_n(\mathbf{X}^n;G) + \theta_n \biggr] = \bigO (p_n^{(d-q)/d})+ \bigO (p_n^{-q/d}) = \bigO(p_n^{(d-q)/d}).
\end{equation}
Since $r_2 \in [0,q)$, we combine \Cref{chistep2,chistep3,chistep4} with the aid of H\"older's inequality to obtain the following.
\begin{equation}{\label{chistep5}}
\mathbb{E} \biggl[ \sup_{G \in \mathcal{N}_{n,\theta_n}} T (\mathbf{X}^n, \mathbf{Z}^n; G; r_1,r_2,0) \biggr] = \bigO ( p_n^{1-r_2/d}),
\end{equation}
 Therefore, for any $R_{n,k}$ satisfying \Cref{form}, we have the following upper bound.
\begin{align}{\label{chistep6}}
\biggl[\mathbb{E} \biggl( \sup_{G \in \mathcal{N}_{n,\theta_n}} R_{n,k} (\mathbf{X}^n, \mathbf{W}^n; G)^{\frac{1+t}{t} } \biggr)\biggr]^{\frac{t}{1+t}} = \bigO (  p_n^{\frac{t(d-r_2)}{(1+t)d}}).
\end{align}
Similarly, looking at the expressions of the functions $\widetilde{R}_{n,k}$'s that we derived earlier, they are of the form 
\begin{equation}{\label{form2}}
\widetilde{R}_{n,k}( \mathbf{x}^n) = \biggl[Q ( \mathbf{x}^n,r)\biggl]^{1/(1+t)}
\end{equation}
for some $r\in [0,d)$ and $t \in (0,\infty]$, where $1/(1+\infty):=0$.
Thus we need to find upper bound to the following term : $\mathbb{E} Q(\mathbf{X}^n;r)$.
Applying \Cref{mindist}, we have the following.
\begin{equation}{\label{tildeestim}}
\mathbb{E} Q ( \mathbf{X}^n;r ) = p_n \mathbb{E} \biggl[\biggl( \min_{j \in [p_n] : j \neq 1} \|\Delta X_{ij}\|_2\biggr)^{-r}\biggr] =\bigO (p_n^{1+r/d} ).
\end{equation}
 Applying the above computation for the case described in \Cref{form2}, we obtain the following.
\begin{align}{\label{chistep7}}
[\mathbb{E}( \widetilde{R}_{n,k}( \mathbf{X}^n)^{t+1}) ]^{1/(t+1)} =  \biggl[\mathbb{E} Q( \mathbf{X}^n;r)\biggr]^{1/(t+1)} = \bigO ( p_n^{\frac{d+r}{d(1+t)}}).
\end{align}
One can check directly that the above computation is also valid for $t=\infty$.  Applying \Cref{chistep6} and \Cref{chistep7} for the quantities mentioned in \Cref{choice6}, \Cref{choice7}, \Cref{choice9}, \Cref{choice10} and \Cref{choice11new2}, we can easily check that the assumption \Cref{item:E2} is satisfied in each case.  The corresponding values of the sequences $\nu_{n,k}$, $\widetilde{\nu}_{n,k}$ and parameters $\gamma_k$ has been summarized in \Cref{tab2}. Note that, we have listed the value of $L$ to be equal to $2$ when $q \in [1,2)$. For the standard case, which corresponds to $q=1$, one can deduce from the last four columns of \Cref{tab2} that for the particular choice of $s=2$, the rows corresponding to the index values $k=1$ and $k=2$ become identical; hence if we choose $s=2$, we can satisfy \Cref{ass:psi} and the assumptions in \cref{item:C1} and \Cref{item:C2} (or similarly, the assumptions in \Cref{item:E1} and \Cref{item:E2}) with $L=2$. This matches with the discussion we had earlier in \Cref{exam:tsp} and its continuations in \Cref{stochopt} and \Cref{sec:thm1}. 
\begin{table}
\begin{center}
\renewcommand*{\arraystretch}{1.5}
\caption{List of  $\nu_{n,k}, \widetilde{\nu}_{n,k}, \upsilon_k$ and $\gamma_k$} \label{tab2}
\begin{tabular}{ccccccc}
\toprule
$q$ & $L$ &  \texttt{index}($k$) & $\nu_{n,k}$ & $\widetilde{\nu}_{n,k}$ & $\upsilon_k$ & $\gamma_k$ \\
\midrule
$[1,2)$ & $2$ & $1$ & $p_n^{\frac{1}{1+s}-\frac{q-1}{d}}$ & $p_n^{\frac{s}{1+s}+\frac{1}{d}}$ & $1$ & $\frac{1}{s}$ \\ 
& & $2$ & $p_n^{\frac{1}{3}}$ & $p_n^{\frac{2}{3}+\frac{2-q}{d}}$ & $1$ & $\frac{1}{2}$\\
$[2,d)$ & $\lfloor q \rfloor +1 $ & $\{1, \lfloor q \rfloor -1\}$ & $p_n^{1-\frac{q-k-1}{d}}$ & $1$ &$k$ & $\infty$  \\
 &  & $\lfloor q \rfloor$ & $p_n^{\frac{1}{2}-\frac{1}{d}}$ & $p_n^{\frac{1}{2}+\frac{\lfloor q \rfloor +1 -q}{d}}$& $\lfloor q \rfloor -1$ & $1$  \\
 & & $ \lfloor q \rfloor + 1$ & $p_n$  & $1$ & $q-1$ & $\infty$ \\
 \bottomrule
\end{tabular}
\end{center}
\end{table}

	\item \textit{Proof of \Cref{item:C3} and \Cref{item:E3}:}  
It will be enough just to prove \Cref{item:E3}, where we shall take $X_i = g(Y_i)$ with $Y_1,Y_2, \ldots \stackrel{i.i.d.}{\sim} \mathcal{P}^*$; here $\mathcal{P}^*$ is a probability measure on $\mathbb{R}^p$ with density $f^*$ which is continuous almost everywhere. $\mathcal{P}$ will still denote the law of $X_1$, with density $f$ which is continuous almost everywhere by hypothesis of this theorem. Note that, \Cref{item:C3} can be obtained by just taking $g$ to be the identity map and $f^*=f$.  For any $G_1,G_2 \in \mathcal{G}_n$, we have
	\begin{align}
		&d_{g,n}(G_1,G_2)^2 \notag\\
		&=  \sum_{i \in [p_n]} \biggl \| (\nabla g(Y_i))^{\top} \biggl( \sum_{\substack{j : \{i,j \} \in E(G_1)}} \dfrac{\Delta X_{ij}}{\|\Delta X_{ij}\|_2^{2-q}} - \sum_{\substack{k : \{i,k \} \in E(G_2)}} \dfrac{\Delta X_{ik}}{\|\Delta X_{ik}\|_2^{2-q}} \biggr)\biggr \|_2^2, \label{tspsq}
	\end{align} 
We shall  show that the (random) metric  $p_n^{-1/2+(q-1)/d}d_{g,n}$ is greater than or equal to the metric $d_n$ (upto some constant) with high probability. This observation will lead us to a statement like \Cref{item:E3}. In order to reach our target we now define a class of events which will be instrumental in our proof. 
Let $\mathcal{D}$ be the set of all $2D$-tuples, where each co-ordinate is from the set $\{-1,0,1\}$, with at least one non-zero co-ordinate. For any $\underline{\beta}=(\beta_1, \ldots,\beta_{2D}) \in \mathcal{D}$, define the following events,
	\begin{align}{\label{type1}}
		\Lambda^{(\underline{\beta})}_{n,M,i, \alpha} &:= \bigcup_{\substack{j_1,\ldots,j_{2D} \in [p_n] \\ i,j_1, \ldots,j_{2D} \text{ are distinct }}} \biggl\{  \|X_{j_l}-X_i\|_2 \leq Mp_n^{-1/d} \text{ for } l\in [2D], \text{ and } \nonumber \\
		& \qquad \qquad \biggl \|(\nabla g(Y_i))^{\top} \biggl( \sum_{l=1}^{2D} \dfrac{\beta_l(X_i-X_{j_l})}{\|X_i-X_{j_l}\|_2^{2-q}} \biggr)\biggr \|_2 \leq \alpha p_n^{(1-q)/d} \biggr\},
	\end{align}
	for $i \in [p_n]$ and $M, \alpha \in (0, \infty)$. We shall show that the above events have very small probability  for large enough $n$ and small enough $\alpha$. 
	To start, fix $\underline{\beta}\in \mathcal{D}$ and note that the events $\Lambda^{(\underline{\beta})}_{n,M,i, \alpha}$'s are exchangeable for different $i \in [p_n]$, while
	\begin{align*}
		&\mathbb{P} ( \Lambda^{( \underline{\beta})}_{n,M,p_n, \alpha}) \\
		 &= \int_{\mathbb{R}^p} f^*(y)\biggl[ \sum_{m \geq 0} \mathbb{P}\biggl( \exists \; j_1, \ldots, j_{2D} \in [p_n-1] \text{ distinct, such that } \\
		 &\qquad \qquad \qquad \max_{l=1}^{2D} \|g(y)-X_{j_l}\|_2 \leq Mp_n^{-1/d} \text{ and }  \\
		 &\qquad \qquad \qquad  \biggl \|(\nabla g(y))^{\top}  \biggl( \sum_{l=1}^{2D} \dfrac{\beta_l(g(y)-X_{j_l})}{\|g(y)-X_{j_l}\|_2^{2-q}} \biggr)\biggr \|_2 \leq \alpha p_n^{(1-q)/d}, \text{ and }\\
		&\qquad \qquad \qquad  \sum_{i=1}^{p_n-1}\mathbbm{1}(\|g(y)-X_i\|\leq Mp_n^{-1/d}) = m \biggr) \biggr] \, dy \\
		&= \int_{\mathbb{R}^p} f^*(y) \biggl\{ \sum_{m \geq 0} \mathbb{P}[ \mathrm{Binomial}( p_n-1, \mathcal{P}(B_d(g(y),Mp_n^{-1/d})))=m] \\
	& \qquad  \cdot \mathbb{P}\biggl( \exists \; j_1, \ldots, j_{2D} \in [m] \text{ distinct s.t. }\biggl \| (\nabla g(y))^{\top} \biggl(\sum_{l=1}^{2D} \dfrac{\beta_l(g(y)-X_{j_l})}{\|g(y)-X_{j_l} \|_2^{2-q}} \biggr) \biggr \|_2     \\
		&\qquad  \qquad \qquad \leq \alpha p_n^{(1-q)/d}  \biggl \rvert \; X_i \in B_d(g(y),Mp_n^{-1/d}), \;\forall \; i \in [m] \biggr) \biggr\} \, dy.
	\end{align*}
	Our basic intuition dictates that if we consider a finite i.i.d.~sample taken from $f$, conditioned on lying in a very small ball around $g(y)$, then the points will be independent and approximately uniformly distributed on that small ball. To make this precise, fix $y \in \mathbb{R}^p$ such that $f(g(y))>0$ and $f$ is continuous at $g(y)$; this happens almost everywhere on the set $f^*(y)>0$ since $X_1=g(Y_1)$ has a density $f$ (when $Y_1$ has density $f^*$) and set of discontinuities of $f$ has Lebesgue measure zero. Note that, for such $y$, we have  as $n \to \infty$, 
	$$ \dfrac{\mathcal{P}(B_d(g(y),Mp_n^{-1/d}))}{\operatorname{Vol}(B_d(g(y),Mp_n^{-1/d}))} =  \dfrac{1}{\operatorname{Vol}(B_d(g(y),Mp_n^{-1/d}))} \int_{B_d(g(y),Mp_n^{-1/d})} f(z) \, dz \longrightarrow f(g(y)),$$
	or equivalently, $p_n\mathcal{P}(B_d(g(y),Mp_n^{-1/d})) \to M^dC_df(g(y))$ as $n \to \infty$. Here $C_d$ is the volume of unit ball around $\mathbf{0}$ in $\mathbb{R}^d$. Since $f(g(y))>0$ and $f$ is continuous at $g(y)$, we have $f$ to be  bounded above and bounded below by a positive real number in a neighborhood around $g(y)$. Let ${V}_1, {V}_2,\ldots \stackrel{iid}{\sim} \text{Uniform}(B_d(\mathbf{0},1)).$ Using Pinsker's inequality and chain rule for Kullback--Leibler divergence, we can obtain the following fo any $m \geq 1$: 
	\begin{align}
		&\| \mathcal{L}( \{X_i\}_{i \in [m]}  \; \rvert \; X_i \in B_d(g(y),Mp_n^{-1/d}), \;\forall \; i \in [m] ), \mathcal{L} ( \{g(y) + Mp_n^{-1/d}{V}_i\}_{i \in [m]}) \|_{TV}^2 \nonumber \\
		& \leq \dfrac{1}{2} D_{KL} (\mathcal{L} ( \{g(y) + Mp_n^{-1/d}{V}_i\}_{i \in [m]}),\mathcal{L}( \{X_i\}_{i \in [m]} \; \rvert\; X_i \in B_d(g(y),Mp_n^{-1/d}), \;\forall \; i \in [m] )) \nonumber \\
		& = \dfrac{m}{2} D_{KL} (\mathcal{L} ( g(y) + Mp_n^{-1/d}{V}_1),\mathcal{L}( X_1  \; \rvert \; X_1 \in B_d(g(y),Mp_n^{-1/d}) )) \nonumber \\
		&= \dfrac{m}{2} \int_{B_d(g(y),Mp_n^{-1/d})}  \dfrac{1}{C_dM^dp_n^{-1}} \log  \dfrac{\mathcal{P}(B_d(g(y),Mp_n^{-1/d}))}{f(z)C_dM^dp_n^{-1}}  dz \nonumber \\
		&= \dfrac{m}{2} \int_{B_d(\mathbf{0},M)}  \dfrac{1}{C_dM^d} \log  \dfrac{\mathcal{P}(B_d(g(y),Mp_n^{-1/d}))}{f(g(y)+zp_n^{-1/d})C_dM^dp_n^{-1}}  dz \nonumber \\
		&= \dfrac{m}{2} \biggl[ \log \dfrac{p_n\mathcal{P}(B_d(g(y),Mp_n^{-1/d}))}{M^dC_df(g(y))} \notag \\
		&\qquad \qquad \qquad - \int_{B_d(\mathbf{0},M)}  \dfrac{1}{C_dM^d} \log  \dfrac{f(g(y)+zp_n^{-1/d})}{f(g(y))}  dz \biggr] 
	 \longrightarrow 0,\label{tv}
	\end{align}
	as $n \to \infty$. Therefore, for all $m \geq 0$, as $n \to \infty$
	\begin{align}{\label{tvimplication}}
	& \mathbb{P}\biggl( \exists \; \{j_l\}_{l=1}^{2D} \subseteq [m] \text{ distinct s.t. }\biggl \| (\nabla g(y))^{\top}  \biggl( \sum_{l=1}^{2D} \dfrac{\beta_l(g(y)-X_{j_l})}{\|g(y)-X_{j_l}\|_2^{2-q}} \biggr) \biggr \|_2 \leq \alpha p_n^{(1-q)/d} \nonumber \\
			& \qquad \qquad  \biggl \rvert\; X_i \in B_d(g(y),Mp_n^{-1/d}), \;\forall \; i \in [m] \biggr) \nonumber \\
			& = \mathbb{P}\biggl( \exists \; \{j_l\}_{l=1}^{2D} \subseteq [m] \text{ distinct, such that }\notag\\
			&\qquad \qquad \qquad \biggl \|(\nabla g(y))^{\top}  \biggl( \sum_{l=1}^{2D} \dfrac{\beta_l V_{j_l}}{\|V_{j_l}\|_2^{2-q}} \biggr)\biggr \|_2  \leq \alpha M^{1-q} \biggr) + o(1).
	\end{align}
	
	On the other hand, $\mathrm{Binomial}( p_n-1, \mathcal{P}(B_d(g(y),Mp_n^{-1/d})))$ distribution converges in law, and hence in total variation distance, to $\mathrm{Poisson}(M^dC_df(g(y)))$. Applying the dominated convergence theorem, we conclude that as $n \to \infty$, 
	\begin{align*}
		\mathbb{P} ( \Lambda^{(\underline{\beta})}_{n,M,1, \alpha}) &\longrightarrow \int_{\mathbb{R}^p} f^*(y) \biggl[ \sum_{m \geq 0} \mathbb{P}[ \mathrm{Poisson}( M^dC_df(g(y)))=m]  \\
		& \qquad \qquad \cdot \mathbb{P} \biggl( \exists \; j_1,\ldots,j_{2D} \in [m] \text{ distinct, such that } \\
		&\qquad \qquad \biggl \| (\nabla g(y))^{\top}  \biggl( \sum_{l=1}^{2D} \dfrac{\beta_l V_{j_l}}{\|V_{j_l}\|_2^{2-q}} \biggr) \biggr\|_2 \leq \alpha M^{1-q} \biggr)\biggr]dy.
	\end{align*}
For notational convenience, introduce the following notation for all $m \geq 0, \; M,\alpha >0$ and $y \in \mathbb{R}^p$. 
	\begin{align*} 
	p_{\underline{\beta}}(y, m,M, \alpha) &:= \mathbb{P} \biggl( \exists \; j_1, \ldots,j_{2D} \in [m] \text{ distinct, such that } \\
	&\qquad \biggl \| (\nabla g(y))^{\top}  \biggl( \sum_{l=1}^{2D} \dfrac{\beta_l V_{j_l}}{\|V_{j_l}\|_2^{2-q}} \biggr)\biggr\|_2 \leq \alpha M^{1-q} \biggr).
	\end{align*}
	Since the random vector $ \sum_{l=1}^{2D} \beta_l V_{j_l}\|V_{j_l} \|_2^{q-2}$ has density with respect to Lebesgue measure, we can conclude that 
	$$ \mathbb{P} \biggl[ (\nabla g(y))^{\top}  \biggl( \sum_{l=1}^{2D} \dfrac{\beta_l V_{j_l}}{\|V_{j_l}\|_2^{2-q}} \biggr)=\mathbf{0} \biggr]=0,$$
	provided that $\nabla g(y) \neq 0$. Hence, for such $y$, we have $p_{\underline{\beta}}(y,m,M,\alpha) \downarrow 0$ as $\alpha \downarrow 0$, it is non-decreasing in $m$ for all fixed $\alpha,M$; and as $n \to \infty$, 
	\begin{align} {\label{g1}}
		\mathbb{P} ( \Lambda^{(\underline{\beta})}_{n,M,1, \alpha}) 
		& \longrightarrow \int_{\mathbb{R}^p} f^*(y) \mathbb{E} [ p_{\underline{\beta}}(y,\mathrm{Poisson}( M^dC_df(g(y))),M, \alpha) ]\, dy  \nonumber \\
		& =\mathbb{E} [ p_{\underline{\beta}}(Y_1,\mathrm{Poisson}( M^dC_df(X_1)),M, \alpha) ]  =: P_{\underline{\beta}}(M,\alpha).
	\end{align}
Since we have assumed that $\nabla g \neq 0$ almost everywhere, we have $P_{\underline{\beta}}(M,\alpha)$ also goes to $0$ as $\alpha \downarrow 0$ for any fixed $M$.
	On the other hand.
	\begin{align*}
		& \mathbb{P} ( \Lambda^{(\underline{\beta})}_{n,M,1, \alpha} \cap \Lambda^{(\underline{\beta})}_{n,M,2, \alpha}  ) \\ &\leq  \mathbb{P}( \Lambda^{(\underline{\beta})}_{n,M,1, \alpha} \cap \Lambda^{(\underline{\beta})}_{n,M,2, \alpha} , \|X_1 - X_2\|_2 > 2Mp_n^{-1/d} ) +   \mathbb{P} ( \|X_1 - X_2\|_2 \leq  2Mp_n^{-1/d}  ) \\
		& \leq  \mathbb{P} ( \Lambda^{(\underline{\beta})}_{n,M,1, \alpha} \cap \Lambda^{(\underline{\beta})}_{n,M,2, \alpha} , \|X_1 - X_2\|_2 > 2Mp_n^{-1/d}  ) +  o(1), 
	\end{align*}
	where the last inequality is true since $p_n \to \infty$ as $n \to \infty$ and $X_1-X_2$ has density with respect to Lebesgue measure on $\mathbb{R}^d$. 
	For the first term in the right hand side of the above inequality, we have 
	\begin{align*}
		&\mathbb{P} ( \Lambda^{(\underline{\beta})}_{n,M,1, \alpha}  \cap \Lambda^{(\underline{\beta})}_{n,M,2, \alpha} , \|X_1 - X_2\|_2 > 2Mp_n^{-1/d} ) \\ 
		&= \int_{\mathbb{R}^p} \int_{\mathbb{R}^p} f^*(y_1) f^*(y_2)\mathbbm{1}( \|g(y_1) - g(y_2)\|_2 >  2Mp_n^{-1/d} ) \\
		& \qquad \qquad  \cdot \biggl[ \sum_{m_1,m_2 \geq 0} \mathbb{P}[ \mathrm{Multinomial}( p_n-2; \mathcal{P}(B_d(g(y_1),Mp_n^{-1/d})), \\
		&\qquad \qquad \qquad \qquad \qquad \qquad \mathcal{P}(B_d(g(y_2),Mp_n^{-1/d}))) =(m_1,m_2)] \\
			& \qquad \qquad \qquad \cdot\mathbb{P}\biggl( \exists \; j_1, \ldots, j_{2D} \in [m_1] \text{ distinct, such that }\\
			&\qquad \qquad \qquad \qquad \biggl \| (\nabla g(y_1))^{\top} \biggl(\sum_{l=1}^{2D} \dfrac{\beta_l(g(y_1)-X_{j_l})}{\|g(y_1)-X_{j_l} \|_2^{2-q}} \biggr) \biggr \|_2  \leq \alpha p_n^{(1-q)/d}    \\
				& \qquad \qquad \qquad \qquad \qquad  \biggl \rvert \; X_i \in B_d(g(y_1),Mp_n^{-1/d}), \;\forall \; i \in [m_1] \biggr) \\
	& \qquad \qquad \qquad \cdot \mathbb{P}\biggl( \exists \; j_1, \ldots, j_{2D} \in [m_2] \text{ distinct, such that }\\
	&\qquad \qquad \qquad \qquad  \biggl \| (\nabla g(y_2))^{\top} \biggl(\sum_{l=1}^{2D} \dfrac{\beta_l(g(y_2)-X_{j_l})}{\|g(y_2)-X_{j_l} \|_2^{2-q}} \biggr) \biggr \|_2  \leq \alpha p_n^{(1-q)/d}    \\
					&\qquad \qquad \qquad \qquad   \qquad   \biggl \rvert\; X_i \in B_d(g(y_2),Mp_n^{-1/d}), \;\forall \; i \in [m_2] \biggr)	\biggr]  \, dy_1 \, dy_2.
	\end{align*}
	Applying (\ref{tv}) and the fact that as $n \to \infty$, 
	\begin{align*}
	& \mathrm{Multinomial}( p_n-2; \mathcal{P}(B_d(g(y_1),Mp_n^{-1/d})), \mathcal{P}(B_d(g(y_2),Mp_n^{-1/d}))) \\ &\qquad \qquad \stackrel{d}{\longrightarrow}\mathrm{Poisson}(M^dC_df(g(y_1))) \otimes \mathrm{Poisson}(M^dC_df(g(y_2))),
	\end{align*}
	 almost everywhere on $\{f^*(y_1)>0,f^*(y_2)>0\}$,  we conclude that 
	 \begin{align*}
	 &\mathbb{P} ( \Lambda^{(\underline{\beta})}_{n,M,1, \alpha}  \cap \Lambda^{(\underline{\beta})}_{n,M,2, \alpha} , \|X_1 - X_2\|_2 > 2Mp_n^{-1/d} )  \\
	 & \longrightarrow  \int_{\mathbb{R}^p} \int_{\mathbb{R}^p} f^*(y_1) f^*(y_2)\mathbbm{1}( g(y_1) \neq g(y_2) ) \mathbb{E} [ p_{\underline{\beta}}(y_1,\mathrm{Poisson}( M^dC_df(g(y_1))),M, \alpha) ] \\
	 & \hspace{2 in} \cdot \mathbb{E} [ p_{\underline{\beta}}(y_2,\mathrm{Poisson}( M^dC_df(g(y_2))),M, \alpha) ]\, dy_1\,dy_2 \\
	 &  = \mathbb{E} [ p_{\underline{\beta}}(Y_1,\mathrm{Poisson}( M^dC_df(X_1)),M, \alpha)  p_{\underline{\beta}}(Y_2,\mathrm{Poisson}( M^dC_df(X_2)),M, \alpha); X_1 \neq X_2] \\
	 & = P_{\underline{\beta}}(M,\alpha)^2,
	 \end{align*}
where the last equality follows from the fact that $\mathbb{P}(X_1 =X_2)=0$, both random vectors having density and being independent. Therefore,	 
	$$ \mathbb{P} ( \Lambda^{(\underline{\beta})}_{n,M,1, \alpha} \cap \Lambda^{(\underline{\beta})}_{n,M,2, \alpha}   ) \leq   (\mathbb{P} ( \Lambda^{(\underline{\beta})}_{n,M,1, \alpha}  ))^2 + o(1)= P_{\underline{\beta}}(M,\alpha)^2 + o(1).$$
	As a consequence,
	\begin{align*}
		&\operatorname{Var}\biggl( \dfrac{1}{p_n}\sum_{i=1}^{p_n} \mathbbm{1}_{\Lambda^{(\underline{\beta})}_{n,M,i, \alpha}}\biggr) \\
		 & = \dfrac{1}{p_n^2}[ p_n\mathbb{P}(\Lambda^{(\underline{\beta})}_{n,M,1, \alpha}) (1-\mathbb{P} (\Lambda^{(\underline{\beta})}_{n,M,1, \alpha}))+ p_n(p_n-1)\mathbb{P} ( \Lambda^{(\underline{\beta})}_{n,M,1, \alpha} \cap \Lambda^{(\underline{\beta})}_{n,M,2, \alpha}   ) \\
		& \hspace{3 in} -p_n(p_n-1) (\mathbb{P}( \Lambda^{(\underline{\beta})}_{n,M,1, \alpha}  ))^2 ] \\
		& \leq o(p_n^{-1}) + o(1) = o(1),
	\end{align*}
	and hence as $n \to \infty$, 
	\begin{equation}{\label{conv1}}
		\dfrac{1}{p_n}\sum_{i=1}^n \mathbbm{1}_{\Lambda^{(\underline{\beta})}_{n,M,i, \alpha}} \stackrel{p}{\longrightarrow} P_{\underline{\beta}}(M,\alpha), \; \; \forall \; \underline{\beta} \in \mathcal{D}.
	\end{equation} 
	

	We shall need to define another set of events to facilitate our proof. Working towards that direction, fix $\varepsilon > 0$ and let $\mathcal{T}_{n,\varepsilon}$ be an optimal $\varepsilon$-packing set of $\mathcal{N}_{n,\theta_n}$ with respect tot he metric $d_n$. To stress the pedantic point, this choice can be done in a measurable fashion (since $\mathcal{G}_n$ is finite). Let $\widetilde{G}_{1,n}, \widetilde{G}_{2,n}$ be chosen in such a way that 
	$$ d_{g,n}( \widetilde{G}_{1,n}, \widetilde{G}_{2,n}) = \min_{G_1,G_2 \in \mathcal{T}_{n,\varepsilon} : G_1 \neq G_2} d_{g,n}(G_1,G_2).$$
	We again mention this choice can be done in a measurable way. If $\mathcal{T}_{n,\varepsilon}$ is a singleton set, then we define $\widetilde{G}_{1,n}= \widetilde{G}_{2,n}$
 and equal to the only element in $\mathcal{T}_{n,\varepsilon}$. Now consider the following event:
 $$ \Gamma_{n,M,i} := \bigcup_{j \in [p_n]} \{ \{i,j\} \in E(\widetilde{G}_{1,n})\cup E(\widetilde{G}_{2,n}) \text{ and } \|\Delta X_{ij}\|_2 > Mp_n^{-1/d}\}, \; \forall \; i \in [p_n].$$
 	Clearly,
 	$$ M^qp_n^{-q/d}\sum_{i=1}^{p_n} \mathbbm{1}_{\Gamma_{n,M,i} }   \leq 2 ( \psi_n (\mathbf{X}^n; \widetilde{G}_{1,n})+ \psi_n(\mathbf{X}^n; \widetilde{G}_{1,n})) \leq 4\inf_{G \in \mathcal{G}_n} \psi_n( \mathbf{X}^n;G)+4 \theta_n.$$
 	Taking expectations on both sides of the above inequality and applying \Cref{rate}, we conclude that 
 	\begin{equation}{\label{secondevent}}
 	\mathbb{E} \biggl[ \dfrac{1}{p_n} \sum_{i=1}^{p_n} \mathbbm{1}_{\Gamma_{n,M,i} } \biggr] \leq \dfrac{C^*(d,q,f)}{M^q}, \; \forall \; n \geq 1,
 	\end{equation}
 	where $C^*(d,q,f)$ is a finite constant depending on $d,q,f$ only (along with the specific optimization problem at hand). 
To understand why the events analyzed above are important in our proof, we recall that our goal is to compare the (random) pseudo-metric $d_{g,n}$ to the metric $d_n$. Towards that goal, letting $N(i,G)$ to be the set of neighbors of the vertex $i$ in graph $G$, we observe that 
	\begin{align}
		&d_{g,n}(\widetilde{G}_{1,n},\widetilde{G}_{2,n})^2 \notag\\
		&=  \sum_{i \in [p_n]}  \biggl \| ( \nabla g(Y_i) )^{\top} \biggl(\sum_{\substack{ \{i,j \} \in E(\widetilde{G}_{1,n})}} \dfrac{\Delta X_{ij}}{\|\Delta X_{ij}\|_2^{2-q}} - \sum_{\substack{ \{i,k \} \in E(\widetilde{G}_{2,n})}} \dfrac{\Delta X_{ik}}{\|\Delta X_{ik}\|_2^{2-q}} \biggr)\biggr \|_2^2 \nonumber \\
	 & \geq  \sum_{i=1}^{p_n} \alpha^2p_n^{2(1-q)/d} \mathbbm{1}_{(N(i,\widetilde{G}_{1,n}) \neq N(i,\widetilde{G}_{2,n}) )} \mathbbm{1}_{\Gamma_{n,M,i}^c } \prod_{\underline{\beta} \in \mathcal{D}} \mathbbm{1}_{(\Lambda^{(\underline{\beta})}_{n,M,i,\alpha})^c} \nonumber \\
		& \geq \alpha^2 p_n^{2(1-q)/d}\sum_{i=1}^{p_n} \biggl(\mathbbm{1}_{(N(i,\widetilde{G}_{1,n}) \neq N(i,\widetilde{G}_{2,n}) )} -  \mathbbm{1}_{\Gamma_{n,M,i}}-\sum_{\underline{\beta} \in \mathcal{D}} \mathbbm{1}_{\Lambda^{(\underline{\beta})}_{n,M,i,\alpha}} \biggr). \label{comp1}
	\end{align}
	Since, for any $G_1,G_2 \in \mathcal{G}_n$, 
	$$ \sum_{i \in [p_n]} \mathbbm{1}_{( N(i,G_1) \neq N(i,G_2))} \geq  \dfrac{2 \operatorname{card}( E(G_1) \Delta E(G_2))}{D} = \dfrac{2p_n}{D} d_n(G_1,G_2),$$
	we obtain the following comparison of $d_{g,n}$ and $d_n$ from \Cref{comp1}, provided that  $\mathcal{T}_{n,\varepsilon}$ is not a singleton set:
	\begin{align}{\label{comp2}}
	& p_n^{-1+2(q-1)/d} d_{g,n}(\widetilde{G}_{1,n},\widetilde{G}_{2,n})^2 \nonumber \\ & \hspace{ 0.2 in} \geq \dfrac{2\alpha^2}{D} d_{n}(\widetilde{G}_{1,n},\widetilde{G}_{2,n}) - \dfrac{\alpha^2}{p_n}\sum_{i \in [p_n]}\mathbbm{1}_{\Gamma_{n,M,i}}  - \alpha^2\sum_{\underline{\beta} \in \mathcal{D}} \dfrac{1}{p_n} \sum_{i \in [p_n]}  \mathbbm{1}_{\Lambda^{(\underline{\beta})}_{n,M,i,\alpha}} \nonumber \\
	&\hspace{ 0.2 in} \geq \dfrac{2\alpha^2 \varepsilon}{D}  - \dfrac{\alpha^2}{p_n}\sum_{i \in [p_n]}\mathbbm{1}_{\Gamma_{n,M,i}}  -  \alpha^2\sum_{\underline{\beta} \in \mathcal{D}} \dfrac{1}{p_n} \sum_{i \in [p_n]}  \mathbbm{1}_{\Lambda^{(\underline{\beta})}_{n,M,i,\alpha}}.	
	\end{align}

Set $\tau_n = p_n^{1/2-(q-1)/d}$ and fix $\delta >0$. Choose $M_{\varepsilon,\delta}$ large enough such that $2DC^*(d,q,f) \leq M_{\varepsilon,\delta}^q\varepsilon^2\delta$ and then choose $\alpha_{\varepsilon,\delta} >0$ small enough such that $\sum_{\underline{\beta} \in \mathcal{D}} P_{\underline{\beta}}(M_{\varepsilon,\delta},\alpha_{\varepsilon,\delta}) < \varepsilon^2/(2D).$ The second choice is possible due to the comment made after \Cref{g1}. In the following computation we drop $\varepsilon,\delta$ from the subscript of $ M_{\varepsilon,\delta},  \alpha_{\varepsilon,\delta}$ for the sake of brevity of notation.
\begin{align}{\label{es1}}
	&\mathbb{P} (P(\mathcal{N}_{n,\theta_n}, d_{g,n},\tau_n\alpha\sqrt{\varepsilon}/\sqrt{D} ) < P(\mathcal{N}_{n,\theta_n}, d_n,\varepsilon )) \nonumber \\
	& \hspace{0.5 in} \leq \mathbb{P} ( \mathcal{T}_{n,\varepsilon} \text{ is not a } \tau_n\alpha\sqrt{\varepsilon}/\sqrt{D}-\text{packing set for } \mathcal{N}_{n,\theta_n} \text{w.r.t. }d_{g,n}) \nonumber \\
	&  \hspace{0.5 in} \leq \mathbb{P} \biggl(  \min_{G_1,G_2 \in \mathcal{T}_{n,\varepsilon} : G_1 \neq G_2} d_{g,n}(G_1,G_2) \leq \tau_n\alpha\sqrt{\varepsilon}/\sqrt{D} \biggr) \nonumber \\
	& \hspace{0.5 in} = \mathbb{P} \biggl(  \tau_n^{-2} d_{g,n}(\widetilde{G}_{1,n},\widetilde{G}_{2,n})^2 \leq \dfrac{\alpha^2\varepsilon^2}{D} \biggr) \nonumber \\
	& \hspace{0.5 in} \leq \mathbb{P} \biggl( \dfrac{1}{p_n}\sum_{i \in [p_n]}\mathbbm{1}_{\Gamma_{n,M,i}} \geq \dfrac{\varepsilon^2}{2D}\biggr) + \mathbb{P} \biggl( \sum_{\underline{\beta} \in \mathcal{D}} \dfrac{1}{p_n} \sum_{i \in [p_n]}  \mathbbm{1}_{\Lambda^{(\underline{\beta})}_{n,M,i,\alpha}} \geq \dfrac{\varepsilon^2}{2D}\biggr) \nonumber \\
	& \hspace{0.5 in} \leq \dfrac{2DC^*(d,q,f)}{M^q\varepsilon^2} + \mathbb{P} \biggl( \sum_{\underline{\beta} \in \mathcal{D}} \dfrac{1}{p_n} \sum_{i \in [p_n]}  \mathbbm{1}_{\Lambda^{(\underline{\beta})}_{n,M,i,\alpha}} \geq \dfrac{\varepsilon^2}{2D}\biggr),
	\end{align} 
	where the last inequality follows from \Cref{secondevent}. The first term in the last expression of \Cref{es1} is at most $\delta$ while the second term converges to $0$ as $n \to \infty$; see \Cref{conv1}. Therefore,
\begin{align}{\label{es2}}
		\limsup_{n \to \infty} \mathbb{P} (P(\mathcal{N}_{n,\theta_n}, d_{g,n},\tau_n\alpha_{\varepsilon,\delta}\sqrt{\varepsilon}/\sqrt{D}) < P(\mathcal{N}_{n,\theta_n}, d_n,\varepsilon ) ) \leq \delta.
	\end{align}

Therefore \Cref{item:E3} and \Cref{item:C3} holds true with $\tau_n = p_n^{1/2-(q-1)/d}$, $\varepsilon_0 =1$, $K_{\varepsilon,\delta} \equiv 1$ and $\kappa_{\varepsilon,\delta} = \alpha_{\varepsilon,\delta}\sqrt{\varepsilon}/\sqrt{D}.$
\end{itemize}

Recall that we showed \Cref{item:E1} and \Cref{item:C1} are satisfied for the choice $\lambda=q-1$ and $\varsigma_{n,\lambda}= p_n^{(d-q)(q-1)/qd}$. We have just shown that \Cref{item:E3} and \Cref{item:C3} are satisfied with $\tau_n= p_n^{1/2-(q-1)/d}$. First consider the case when $\mathcal{P}$ satisfies \Cref{mu} for the pair $(\rho,g)$ where $\nabla g \neq 0$ almost everywhere. In this case, plugging-in values stated above and those obtained from \Cref{tab2} in \Cref{item:E4}, we can conclude using \Cref{genthm:g} that the assertion of this theorem is true provided 
$$ \theta_n = \bigO \biggl( \min \biggl\{ p^{\frac{1}{2}-\frac{q-1}{d}}_n, p_n^{- \frac{q-1}{d}}, p_n^{-\frac{q}{d}}\biggr\}\biggr), \; \mathrm{if} \;\; q \in [1,2),$$
and 
$$ \theta_n = \bigO \biggl( \min \biggl\{ p_n^{\frac{1}{2}-\frac{q-1}{d}}, p_n^{- \frac{q-1}{d}}, \min_{k=1}^{\lfloor q \rfloor -1} p_n^{\frac{1}{2}-\frac{1}{2k}-\frac{q}{d}}, p_n^{\frac{1}{2}-\frac{1}{2(q-1)}-\frac{q}{d}}\biggr\}\biggr), \; \mathrm{if} \; \; q \in [2,d).$$
All of the above are satisfied if and only if $\theta_n = \bigO(p_n^{-q/d})$, which completes the proof in this case.

Now consider the case when $\mathcal{P}$ satisfies \Cref{ass:p} and $q \in [1,1+ 2/d)$. In this case, plugging-in the values  in \Cref{item:C4}, we can conclude using \Cref{genthm} that the assertion of this theorem is true provided 
$$ \theta_n = \bigO \biggl( \min \biggl\{ p_n^{- \frac{q-1}{d}- \frac{q-1}{2}}, p_n^{-\frac{q}{d}}\biggr\}\biggr),$$ 
which will be satisfied if and only if $\theta_n = \bigO(p_n^{-q/d})$. This completes the proof.
\end{proof}

We shall now separately concentrate on two important  examples, namely \textit{Minimum Spanning Tree }(MST) and \textit{Traveling Salesman problem} (TSP), which fall under the scope of \Cref{graph:def} and prove that solutions to these problems are ``stable" under small perturbations as defined in this article. Another important problem which falls under the scope of \Cref{nearoptgraph} is the \textit{Minimum Matching Problem} (MMP). Indeed, one can easily see that MMP satisfies the conditions of \Cref{nearoptgraph} and hence tightness of the near-optimal solution sets follows. But proving  a result like \Cref{item:B}, and hence stability, seems quite tricky for MMP and remains a challenging problem for further research.

\subsection{Minimum spanning tree}{\label{mst}}
The minimum spanning tree (MST) problem on Euclidean spaces, a very well-studied problem in computational geometry, finds the shortest (with respect to Euclidean length) spanning tree with the given set of points as vertex set. This problem can be identified as a graph optimization problem, as defined in \Cref{graph:def}, for the choice $p_n=n$ and  $\mathcal{G}_n$ being collection of all trees on $n$ vertices. At a first glance, it is not clear why MST problem satisfies the uniform bound on the vertex degrees assumed in \Cref{item:2}. In reality, whenever two edges of a Euclidean minimum spanning tree meet at a vertex, they must form an angle of $60^{\circ}$ or more. This is because, for two edges forming a sharper angle, one of the two edges can be replaced by the third and shorter edge of the triangle they form and thus forming a shorter spanning tree. It is a well-known fact that maximum number of edges that can concur to a point in $\mathbb{R}^d$ so that all of those pairs of edges have an angle of at least $60^{\circ}$  between them  is same as the \textit{kissing number} in $\mathbb{R}^d$, which is the maximum number of non-overlapping unit spheres that can be arranged in $\mathbb{R}^d$ such that they each touch a common unit sphere. Let $\kappa(d) \in \mathbb{N}$ denotes the kissing number in $\mathbb{R}^d$. Some examples of particular values include $\kappa(2)=6,\kappa(3)=12$ and $\kappa(4)=24$. Since the optimal spanning tree has vertex degree bounded by $\kappa(d)$,  we shall identify MST problem on $\mathbb{R}^d$ as a graph optimization problem, as defined in \Cref{graph:def}, for the choice $p_n=n$ and  $\mathcal{G}_n$ being collection of all trees on $n$ vertices with vertex degrees bounded by $\kappa(d)$, i.e. $D=\kappa(d)< \infty$. 

In this section $L_{\mathrm{MST}}(x_1, \ldots,x_n;q)$ will denote the $q$-power length of the shortest spanning tree through the set of points $\{x_1, \ldots,x_n\}$. In other words,
$$ L_{\mathrm{MST}}(x_1, \ldots,x_n;q) := \min_{G \in \mathcal{G}_n} \sum_{\{i,j\} : \{i,j\} \in E(G)} \big \| x_i - x_j \big \|_2^q,$$
where $\mathcal{G}_n$ is as described above. The natural and most well-studied case corresponds to the situation where the power weight parameter $q=1$.

It is interesting to ask, and crucial in out application,``at what rate the typical length of the $q$-power weight spanning tree grow as $n \to \infty$?" We refer to the following theorem which appears in \cite[Chapter 7]{yukich} to answer the question posed above.

\begin{assumption}{\label{suff}}
Assume that $\mathcal{P}$ is a probability measure on $\mathbb{R}^d$ satisfying the following assumptions. 
\begin{enumerate}
\item $\mathcal{P}$ is absolutely continuous (with respect to the Lebesgue measure on $\mathbb{R}^d)$  with density $f$.
\item We let $A_k$ denote the annular shell in $\mathbb{R}^d$ with inner and outer radius $2^{k}$ and $2^{k+1}$ respectively, and we set
$$ a_{k,q}(f) := 2^{kdq/(d-q)} \int_{A_k} f({x})\, d {x}, \; \forall \; k \geq 0.$$
Assume that  the density $f$ satisfies 
\begin{equation*}
\sum_{k \geq 0} (a_{k,q}(f))^{(d-q)/d} < \infty.
\end{equation*}
\end{enumerate}
\end{assumption}

\begin{thm}[See \cite{yukich}]{\label{yukich:general:mst}}
Fix $d \geq 2$. Suppose that $\{{X}_i \}_{i \geq 1}$ is an i.i.d. sample from a probability measure $\mathcal{P}$ satisfying Assumption~\ref{suff}. 
Then
$$ n^{-(d-q)/d}L_{\mathrm{MST}}({X}_1, \ldots, {X}_n;q) \stackrel{a.s.}{\longrightarrow} \beta_{\mathrm{MST}}(d,q) \int_{\mathbb{R}^d} f({x})^{(d-q)/d}\, d {x},$$
for some universal finite constant $\beta_{\mathrm{MST}}(d,q)$.  Moreover, $\mathbb{E} L_{\mathrm{MST}}({X}_1, \ldots, {X}_n;q) = \bigO(n^{(d-q)/d})$. 
\end{thm}

\begin{remark}{\label{rhee1}}
Under \Cref{suff}, we have the following upon application of H\"older's inequality:
\begin{align*}
\int_{A_k} f({x})^{(d-q)/d}\, d {x} & \leq \biggl(\int_{A_k} f({x})\, d {x} \biggr)^{(d-q)/d} \biggl(\int_{A_k} \, d {x} \biggr)^{q/d} \\
& \leq (2^{-kdq/(d-q)}a_{k,q}(f) )^{(d-q)/d} ( C_d2^{kd}(2^d-1))^{q/d} \\
& = ( C_d(2^d-1))^{q/d}(a_{k,q}(f))^{(d-q)/d},
\end{align*}
for all $k \geq 0$. Similarly,
$$ \int_{B_d(\mathbf{0},1)} f({x})^{(d-q)/d}\, d {x}  \leq \biggl(\int_{B_d(\mathbf{0},1)} f({x})\, d {x} \biggr)^{(d-q)/d} \biggl(\int_{B_d(\mathbf{0},1)} \, d {x} \biggr)^{q/d} \leq C_d^{q/d}.$$
Therefore, we can write
\begin{align*}
\int_{\mathbb{R}^d} f({x})^{(d-q)/d}\, d {x} &= \int_{B_d(\mathbf{0},1)} f({x})^{(d-q)/d}\, d {x} + \sum_{k \geq 0} \int_{A_k} f({x})^{(d-q)/d}\, d {x} \\
& \leq C_d^{q/d} + ( C_d(2^d-1))^{q/d} \sum_{k \geq 0} (a_{k,q}(f))^{(d-q)/d} < \infty,
\end{align*}
demonstrating that the almost sure limit in \Cref{yukich:general:mst} is indeed finite under the assumptions.
\end{remark}

\begin{remark}{\label{general:proof}}
We refer to \cite[Section 7.3]{yukich} for discussion and detailed proof of \Cref{yukich:general:mst}. Although the book doesn't explicitly prove the upper bound on the growth rate of the  expectations in the situations where $f$ has unbounded support, a proof can be obtained by following the arguments used to prove almost sure convergence. A short outline of the proof is given in \Cref{proofout}.
\end{remark}


\begin{remark}{\label{suff2}}
\begin{enumerate}[label=(\roman*)]
\item It is easy to see that \Cref{suff} implies $\int_{\mathbb{R}^d} \|x\|_2^q f(x)\, dx < \infty$. Indeed,
\begin{align*}
\int_{\mathbb{R}^d} \|x\|_2^q f(x)\, dx \leq 1 + \sum_{k \geq 0} 2^{q(k+1)} \int_{A_k} f(x)\, dx & \leq 1 + \sum_{k \geq 0} 2^{q(k+1)} \biggl(\int_{A_k} f(x)\, dx \biggr)^{(d-q)/d} \\
& = 1 + 2^q \sum_{k \geq 0} (a_{k,q}(f))^{(d-q)/d} < \infty.
\end{align*}
\item 
On the other hand, an easy to check moment condition which is sufficient for \Cref{suff} can be given as follows : $\int_{\mathbb{R}^d} \|x\|_2^r f(x)\, dx < \infty$, for some $r>dq/(d-q)$. To see why this is true, get $\varepsilon >0$ such that $r=(1+\varepsilon)dq/(d-q)$ and note that 
\begin{align*}
\sum_{k \geq 0} (a_{k,q}(f))^{(d-q)/d} & = \sum_{k \geq 0} 2^{-\varepsilon k q}\biggl(\int_{A_k} f({x})\, d {x} \biggr)^{(d-q)/d} 2^{(1+\varepsilon) k q} \\
& \leq \sum_{k \geq 0} 2^{-\varepsilon k q}\biggl(\int_{A_k} \big\|x \big\|_2^{(1+\varepsilon)dq/(d-q)}f({x})\, d {x} \biggr)^{(d-q)/d} \\
& \leq \sum_{k \geq 0} 2^{-\varepsilon k q}\biggl(\int_{A_k} \big\|x \big\|_2^{r}f({x})\, d {x} \biggr)^{(d-q)/d} \\
& \leq \biggl(\sum_{k \geq 0} 2^{-\varepsilon k d} \biggr)^{q/d} \biggl(  \sum_{k \geq 0} \int_{A_k} \|x \|_2^{r}f({x})\, d {x} \biggr)^{(d-q)/d} \\
& \leq (1-2^{-\varepsilon d} )^{-q/d} \biggl(  \int_{\mathbb{R}^d} \big\|x \big\|_2^{r}f({x})\, d {x} \biggr)^{(d-q)/d} < \infty. 
\end{align*}
This argument is an extension of the computation presented in \cite[pp.~85]{yukich} for the case $q=1$. 
\end{enumerate}
\end{remark}

\begin{remark}
\Cref{yukich:general:mst} is a special case of similar kind of results for a more general type of functionals defined on finite subsets of $\mathbb{R}^d$, called smooth and sub-additive \textit{Euclidean functionals}. The length of shortest MST as discussed in this section, the length of shortest traveling salesman tour  as discussed in \Cref{tsp} and the length of shortest matching are all examples of smooth and sub-additive Euclidean functionals. We refer to \cite[Chapter 3]{steele} and \cite[Chapter 3,4,7]{yukich} for detailed discussion and asymptotic results on this topic. The outline in \Cref{proofout} for upper bound on the growth rate of expectation of length of the $q$-power-weighted shortest MST can trivially be extended to other smooth and sub-additive Euclidean functionals like length of the shortest traveling salesman tour.
\end{remark}

We want to write a stability result for MST problem as defined in \Cref{stable}. Note that $\mathcal{G}_n$, the collection of all spanning trees on $n$ points with degree bounded bu $\kappa(d)$, is equipped with the metric $d_n$ where $d_n(G_1,G_2)$ is the number of edges belonging to exactly one of the graphs $G_1$ and $G_2$, normalized by $n$.
We make the obvious choice of the perturbation blocks that amounts to replacing one random input by an i.i.d. copy. In other words, $\mathcal{J}_n = \{J_{n,i} : i \in [n]\}$ where $J_{n,i} = \{i\}$ for all $i \in [n]$. The following theorem shows that MST problem is stable under small perturbations for the above mentioned choice of perturbation blocks.

\begin{thm}{\label{thm:mst}}
Fix $d \geq 2$ and $q \in [1,d)$. Consider the minimum Spanning Tree problem where the input distribution $\mathcal{P}$ satisfies Assumption~\ref{graphtight:ass} and Assumption~\ref{suff}. Then MST problem is stable under small perturbations with perturbation blocks $\mathcal{J}_n = \{J_{n,i} : i \in [n]\}$ where $J_{n,i} = \{i\}$ for all $i \in [n]$.
\end{thm}

\begin{proof}
\Cref{nearoptgraph}, \Cref{rhee1} and \Cref{suff2} again guarantee that $\{P(\mathcal{N}_{n,\theta_n}), d_n, \varepsilon\}$ is a tight sequence for any $\theta_n = \bigO(n^{-q/d})$. All that remains is to establish a statement like \Cref{most} for MST.

Recall that for all $l \in [n]$, $\mathbf{X}^{n,l} = (X_{n,i}^{(l)} )_{i \in [n]}$ is an independent copy of $\mathbf{X}^n = (X_{n,i})_{i \in [n]}$ and $\mathbf{X}^n_{l}$ is obtained through replacing $X_{n,l}$ by $X_{n,l}^{(l)}$ in $\mathbf{X}^n$. Set $\widehat{T}_n(\mathbf{x}^n;q)$ to be the optimal $q$-power weight MST  for the points $\mathbf{x}^n = (x_1, \ldots,x_n)$. We define $T^*_{n,l}$, a random spanning tree depending on $(\mathbf{X}^n,\mathbf{X}^n_l)$, through the following algorithm for any 
$l \in [n]$.  This $T^*_{n,l}$ will be a ``sister" tree of $\widehat{T}_n(\mathbf{X}^n_l;q)$ (i.e. close to $\widehat{T}_n(\mathbf{X}^n_l;q)$ in metric $d_n$), but will typically  be a near-optimal tree for the input $\mathbf{X}^n$.
$l \in [n]$.  
\begin{enumerate}
\item Let $l_1, \ldots,l_{d(l)}$ are the neighbors of $l$ in $\widehat{T}_n(\mathbf{X}_l^n;q)$. Delete the edges $\{l,l_i\}$ for all $i=1,\ldots,d(l)$.
\item Add the edges $\{l_{i-1},l_i\}$ for all $i=1, \ldots,d(l)-1$, provided $d(l)>1$. This new graph is a tree through $n-1$ vertices.
\item Find $k \in [n] \setminus \{l\}$ such that $k \in \arg \min_{j \in [n]: j \neq l} \|X_{n,j}-X_{n,l}\|_2.$  
\item Add the edge $\{k,l\}$ to the modified tree to get $T^*_{n,l}$.
\end{enumerate}
It is easy to see that $d_n(\widehat{T}_n(\mathbf{X}^n_l;q),T^*_{n,l}) \leq 2d(l)/n \leq 2\kappa(d)/n$, whereas triangle inequality implies that 
\begin{align*}
\|X_{n,l_{i+1}}-X_{n,l_i}\|_2^q &\leq (\|X_{n,l_{i+1}}-X_{n,l}^{(l)}\|_2 + \|X_{n,l_{i}}-X_{n,l}^{(l)}\|_2 )^q \\
& \leq 2^{q-1} ( \|X_{n,l_{i+1}}-X_{n,l}^{(l)}\|_2^q + \|X_{n,l_{i}}-X_{n,l}^{(l)}\|_2^q),
\end{align*} 
for all $i=1,\ldots,d(l)$. Here $l_{d(l)+1}$ is defined to be $l_1$. Summing the above inequality over $i$, we get   
$$ 2^{q}\sum_{i=1}^{d(l)} \|X_{n,l_i}-X_{n,l}^{(l)}\|^q_2 \geq  \sum_{i=1}^{d(l)} \|X_{n,l_{i+1}}-X_{n,l_i}\|^q_2.$$ 
Therefore,
\begin{align*}
&\psi_n (\mathbf{X}^n; T^*_{n,l}) = \psi_n ( \mathbf{X}_l^n; \widehat{T}_{n}(\mathbf{X}^n_l;q)) - \sum_{i=1}^{d(l)} \|X_{n,l_i}-X_{n,l}^{(l)}\|_2^q + \sum_{i=1}^{d(l)-1} \|X_{n,l_{i+1}}-X_{n,l_i}\|_2^q \\
& \qquad + \|X_{n,k}-X_{n,l}\|_2^q \\
& \leq \psi_n ( \mathbf{X}_l^n; \widehat{T}_{n}(\mathbf{X}^n_l;q))  + (2^q-1)\sum_{i=1}^{d(l)} \|X_{n,l_i}-X_{n,l}^{(l)}\|^q_2  + \min_{j \in [n]: j \neq l} \|X_{n,j}-X_{n,l}\|^q_2.
\end{align*}
Taking expectations on both sides, we get
\begin{align*}
\mathbb{E} [ \psi_n \left(\mathbf{X}^n; T^*_{n,l}\right) - \psi_{n, \mathrm{opt}}(\mathbf{X}^n) ] &= \mathbb{E} [ \psi_n (\mathbf{X}^n; T^*_{n,l}) - \psi_n ( \mathbf{X}_l^n; \widehat{T}_{n}(\mathbf{X}^n_l;q))] \\
& \leq (2^q-1)\mathbb{E}\sum_{i=1}^{d(l)} \|X_{n,l_i}-X_{n,l}^{(l)}\|^q_2 + \mathbb{E} \min_{j \in [n]: j \neq l} \|X_{n,j}-X_{n,l}\|^q_2 \\
& = (2^q-1)\mathbb{E} \sum_{\substack{i \in [n] \\ \{i,l\} \in E(\widehat{T}_n(\mathbf{X}^n;q))}} \|X_{n,i}-X_{n,l}\|^q_2 \\
& \hspace{ 2 in} + \mathbb{E} \min_{j \in [n]: j \neq l} \|X_{n,j}-X_{n,l}\|^q_2 \\
& = \dfrac{2^q-1}{n} \mathbb{E} \sum_{j \in [n]} \sum_{\substack{i \in [n] \\ \{i,j\} \in E(\widehat{T}_n(\mathbf{X}^n;q))}} \|X_{n,i}-X_{n,j}\|^q_2 \\
& \hspace{ 2 in}+ \mathbb{E} \min_{j \in [n]: j \neq l} \|X_{n,j}-X_{n,l}\|^q_2 \\
& = \dfrac{2(2^q-1)}{n} \mathbb{E} L_{\mathrm{MST}}(X_1, \ldots, X_n;q) + \mathbb{E} \min_{j \in [n]: j \neq l} \|X_{n,j}-X_{n,l}\|^q_2 \\
&= \bigO(n^{-q/d}),
\end{align*}
where the last line follows from \Cref{yukich:general:mst} and \Cref{mindistmean}. We can now complete the proof by applying  \Cref{rem:strateasy} and \Cref{strat} for the choice $\theta_n=n^{-q/d}$.
\end{proof}

\begin{remark}
To be completely precise, the ``sister" tree in the proof of \Cref{thm:mst}, $T^*_{n,l}$, might have some vertices with degree $\kappa(d)+1$ and hence not lie in $\mathcal{G}_n$, the collection of trees on $n$ vertices with maximum degree $\kappa(d)$. This creates a problem in direct application of \Cref{strat}. Nevertheless, the ``sister" tree $T^*_{n,l}$ has maximum degree $\kappa(d)+1$ and hence we could have just defined the optimization problem regarding MST with parameter space $\mathcal{G}_n$ being  the collection of trees on $n$ vertices with maximum degree at most $\kappa(d)+1$. Under this new definition, \Cref{nearoptgraph} would still be valid and hence so is the proof of \Cref{thm:mst}.
\end{remark}

\begin{remark}{\label{mst:aeu}}
\citet*{aldous:mst} proved the following result: Suppose we have $n$ i.i.d.~uniformly distributed points $\mathbf{X}^n =(X_1, \ldots, X_n)$ on $[0,1]^d$ and $T_n^*$ is their minimum spanning tree with respect to the Euclidean distance (i.e., $q=1$). For any $\delta >0$, consider the trees which are at least $\delta$-distance away from the optimal tree and observe how much their lengths exceed the length of the MST. To be precise, define
$$ \varepsilon_n(\delta) := \inf \left\{ \dfrac{\psi_n(\mathbf{X}^n;T) - \psi_n(\mathbf{X}^n;T_n^*)}{n^{(d-1)/d}} \; : \; d_n(T,T_n^*) \geq \delta\right\}.$$
Then
$$ 0 < \liminf_{\delta \downarrow 0} \delta^{-2} \liminf_{n \to \infty} \mathbb{E} \varepsilon_n(\delta) \leq  \limsup_{\delta \downarrow 0} \delta^{-2} \limsup_{n \to \infty} \mathbb{E} \varepsilon_n(\delta) < \infty.$$
In other words, for any small $\delta >0$, if any tree is $\delta$-distance away from the optimal tree, then its length is likely to exceed the length of the MST by $\delta^2n^{(d-1)/d}$, upto some finite positive constant. The MST is said to have \textit{scaling exponent} $2$ since for small $\delta$ and large $n$, we have $\varepsilon_n(\delta) \sim \delta^2$, upto some positive constant. This shows that if we look only at the trees whose length exceeds the length of the MST by $\lito(n^{(d-1)/d})$, then they are asymptotically indistinguishable from the MST. This is a much stronger statement than \Cref{nearoptgraph} when applied for the MST with $q=1$. \citet{aldous:mst} also suggested that their result is very robust under model details. In fact, we can replace the uniform distribution on $[0,1]^d$ by any density supported on $[0,1]^d$ and bounded away from zero. Nevertheless, our result in \Cref{nearoptgraph} is strong enough to prove stability as shown in \Cref{thm:mst}. It also allows us to consider the cases with $q >1$ and $f$ having unbounded support, both of which might be out of scope of the results in \cite{aldous:mst}. 
\end{remark}


\subsection{Traveling salesman problem}{\label{tsp}}
The traveling salesman problem (TSP) on Euclidean spaces asks the following question : ``Given a list of points in an Euclidean space, what is the shortest possible route that visits each point exactly once and returns to the original point?" This problem can be identified as a graph optimization problem, as defined in \Cref{graph:def}, for the choice $p_n=n$ and  $\mathcal{G}_n$ being collection of all Hamiltonian cycles on $n$ vertices. In this situation, every vertex has degree $2$ in every graph $G \in \mathcal{G}_n$, yielding $D=2$. The natural and most well-studied case corresponds to the situation where the power weight parameter $q=1$.  In this section $L_{\mathrm{TSP}}(x_1, \ldots,x_n;q)$ will denote the $q$-power length of the shortest Hamiltonian tour through the set of points $\{x_1, \ldots,x_n \}$. In other words,
$$ L_{\mathrm{TSP}}(x_1, \ldots,x_n;q) := \min_{G \in \mathcal{G}_n} \sum_{\{i,j\} : \{i,j\} \in E(G)} \big \| x_i - x_j \big \|_2^q,$$
where $\mathcal{G}_n$ is collection of all Hamiltonian cycles on $n$ vertices. As mentioned in the previous section, we have the following growth rate for the typical length of TSP tour.



\begin{thm}[\cite{bjh, rhee} ]{\label{yukich:general:tsp}}
Fix $d \geq 2$. Suppose that $\{{X}_i \}_{i \geq 1}$ is an i.i.d. sample from a probability measure $\mathcal{P}$ satisfying Assumption~\ref{suff}. 
Then
$$ n^{-(d-q)/d}L_{\mathrm{TSP}}({X}_1, \ldots, {X}_n;q) \stackrel{a.s.}{\longrightarrow} \beta_{\mathrm{TSP}}(d,q) \int_{\mathbb{R}^d} f({x})^{(d-q)/d}\, d {x},$$
for some universal finite constant $\beta_{\mathrm{TSP}}(d,q)$.  Moreover, $\mathbb{E} L_{\mathrm{TSP}}({X}_1, \ldots, {X}_n;q) = \bigO(n^{(d-q)/d})$. 
\end{thm}

\begin{remark}
\Cref{yukich:general:tsp} was first proved in the context of $\mathcal{P}=\mathrm{Uniform}([0,1]^d)$ and $q=1$ by Beardwood et.al.~\cite{bjh}. Rhee~\cite{rhee} later provided a sufficient condition under which the statement is also true for $\mathcal{P}$ with unbounded support in the case of $q=1$. This sufficient condition is in effect the special case of the condition in \Cref{suff} for $q=1$.  Rhee~\cite{rhee} also showed that \Cref{suff} is the best that one can do in terms of conditions on $a_{k,q}(f)$. In particular, at least for the case $q=1$, for any sequence of positive real numbers $\{a_k\}_{k \geq 1}$ with $\sum_{k \geq 1} a_k^{(d-1)/d}= \infty$, there is a density $f$ on $\mathbb{R}^d$ such that $a_{k,1}(f) \leq a_k$ for all $k$, yet for which
$$  \int_{\mathbb{R}^d} f({x})^{(d-1)/d}\, d {x} < \infty \;\; \text{and} \;\;  L_{\mathrm{TSP}}({X}_1, \ldots, {X}_n;1)/n^{(d-1)/d} \stackrel{a.s.}{\longrightarrow} \infty.$$
\end{remark}
We shall now write a stability result for TSP corresponding to the perturbation blocks which replaces one point at a time by an i.i.d.~copy. We shall only consider the case $q=1$ since this is the standard (and most interesting) case while the applicability of triangle inequality facilitate the exposition to be a little clearer. The technique (and construction) in the proof of the following theorem also works for $q\in (1,d)$ albeit with a bit more effort.


\begin{thm}{\label{thm:tsp}}
Fix $d \geq 2$.  Consider the Traveling Salesman Problem where the input distribution $\mathcal{P}$ has density $f$ satisfying \Cref{graphtight:ass} and \Cref{suff} for the case $q=1$. Then TSP with $q=1$ is stable under small perturbations with perturbation blocks $\mathcal{J}_n = \{J_{n,i} : i \in [n]\}$ where $J_{n,i} = \{i\}$ for all $i \in [n]$.
\end{thm}

\begin{proof}
\Cref{nearoptgraph}, \Cref{rhee1} and \Cref{suff2} guarantee that $\{P(\mathcal{N}_{n,\theta_n}), d_n, \varepsilon\}$ is a tight sequence for any $\theta_n = \bigO(n^{-1/d})$. All that remains is to establish a statement like \Cref{most} for TSP.

 Set $\widehat{G}_n(\mathbf{x}^n)$ to be the optimal Hamiltonian cycle for the points $\mathbf{x}^n = (x_1, \ldots,x_n)$. We define $G^*_{n,l}$, a random Hamiltonian cycle depending on $(\mathbf{X}^n,\mathbf{X}^n_l)$, through the following algorithm for any 
$l \in [n]$. 
\begin{enumerate}
\item If $l_1 \neq l_2$ are the neighbors of $l$ in $\widehat{G}_n(\mathbf{X}^n_l)$, then delete the edges $\{l_1,l\}$, $\{l_2,l\}$ and add the edge $\{l_1,l_2\}$. Note that the modified graph is still a Hamiltonian cycle, albeit with $n-1$ points since the point $X_{n,l}^{(l)}$ is not in the cycle.
\item Find $k \in [n] \setminus \{l\}$ such that $k \in \arg \min_{j \in [n]: j \neq l} \|X_{n,j}-X_{n,l}\|_2.$  
\item Let $k_1$ be one of the neighbors of $k$ in the modified graph after the first step which is also a neighbor of $k$ in the original graph $\widehat{G}_n(\mathbf{X}^n_l)$. It is always possible to find such neighbor of $k$ since $k \neq l$. To the modified graph,  add the edges $\{k,l\}$, $\{k_1,l\}$ and delete the edge $\{k,k_1\}$. Define this new cycle to be $G^*_{n,l}$. Note that it is a Hamiltonian cycle through all of the $n$ points of $\mathbf{X}^n$.
\end{enumerate}
It is easy to see that $d_n(\widehat{G}_n(\mathbf{X}^n_l),G^*_{n,l}) \leq 6/n$, whereas we have the following simple observation using triangle inequality.
\begin{align*}
\psi_n (\mathbf{X}^n; G^*_{n,l}) &= \psi_n ( \mathbf{X}_l^n; \widehat{G}_{n}(\mathbf{X}^n_l)) - \|X_{n,l}^{(l)}-X_{n,l_1}\|_2-\|X_{n,l}^{(l)}-X_{n,l_2}\|_2 + \|X_{n,l_1}-X_{n,l_2}\|_2\\
& \hspace{1.5 in}+ \|X_{n,k}-X_{n,l}\|_2+\|X_{n,k_1}-X_{n,l}\|_2-\|X_{n,k_1}-X_{n,k}\|_2 \\
& \leq  \psi_n ( \mathbf{X}_l^n; \widehat{G}_{n}(\mathbf{X}^n_l)) + 2\|X_{n,k}-X_{n,l}\|_2 \\
& \leq  \psi_n ( \mathbf{X}_l^n; \widehat{G}_{n}(\mathbf{X}^n_l)) + 2\min_{j \in [n]: j \neq l} \|X_{n,j}-X_{n,l}\|_2.
\end{align*}
Taking expectations on both sides, we get
\begin{align*}
\mathbb{E} [ \psi_n (\mathbf{X}^n; G^*_{n,l}) - \psi_{n, \mathrm{opt}}(\mathbf{X}^n)] &= \mathbb{E} [ \psi_n (\mathbf{X}^n; G^*_{n,l}) - \psi_n ( \mathbf{X}_l^n; \widehat{G}_{n}(\mathbf{X}^n_l))] \\
& \leq 2 \mathbb{E} \min_{j \in [n]: j \neq l} \|X_{n,j}-X_{n,l}\|_2 = \bigO(n^{-1/d}),
\end{align*}
where the last line follows from \Cref{mindistmean}. We can now complete the proof by applying  \Cref{rem:strateasy} and \Cref{strat} for the choice $\theta_n=n^{-1/d}$.
\end{proof}

\begin{remark}
In their work, \citet{aldous:mst} mentioned that there is fairly good evidence to suggest that TSP satisfies similar theorem as mentioned in \Cref{mst:aeu} with \textit{scaling exponent} $3$. In other words, for small $\delta>0$ and large $n$, we have $\varepsilon_n(\delta) \sim \delta^3$, up to some positive constant. Nevertheless, the proof of such a result remains elusive till now. Our result in \Cref{nearoptgraph} proves a weaker version of this kind of phenomenon while being strong enough to prove stability. 
\end{remark}

\section{Appendix}{}{\label{appnd}}

\subsection{Sub-Gamma variables}{\label{subgamma}}
Here we try to present a short survey of sub-Gamma variables and related results that are useful in our application. Basic definitions and initial results are borrowed from \cite[Section 2.4,2.5]{boucheron}. The following properties of sub-Gamma variables are instrumental in our analysis.

\begin{prop}{\label{subg}}

\begin{enumerate}[label={(\alph*)}, ref={\theprop~(\alph*)}]
\item \label{subg:scale} If $X \in \mathscr{G}(\sigma^2, c)$, then $tX \in \mathscr{G}(t^2\sigma^2, |t|c)$, for any $t \in \mathbb{R}$. This fact is obvious from \Cref{subGamma}.
\item \label{subg:tail} If $X \in \mathscr{G}(\sigma^2,c)$, then for any $t >0$,
$$ \mathbb{P}(X - \mathbb{E}X > \sqrt{2\sigma^2t} + ct ), \mathbb{P}(X - \mathbb{E}X < -\sqrt{2\sigma^2t} - ct ) \leq e^{-t}.$$
Conversely, if the above tail bound holds true for some $\sigma^2,c \geq 0$, then $X \in \mathscr{G}(4(8\sigma^2+16c^2),8c)$. 
\item \label{subg:moments} If $X \in \mathscr{G}(\sigma^2,c)$, then for every integer $q \geq 1$, we have
$$ \mathbb{E}[ ( X-\mathbb{E}X)^{2q}] \leq 8^qq!\sigma^{2q} + 4^{2q}(2q)!c^{2q}.$$
Conversely, if for some $A \in [0,\infty)$ and $B \in [0, \infty)$,
$$ \mathbb{E} [( X-\mathbb{E}X)^{2q}] \leq q!A^{2q} + (2q)!B^{2q}, \; \forall \; q \in \{1,2,\ldots\},$$
then $X \in \mathscr{G}( 4A^2+8B^2,2B)$. A proof of this characterization can be found in \cite[Theorem 2.3]{boucheron}.

\item \label{subg:maxc} If $X_i \in \mathscr{G}(\sigma^2,c)$, for all $1 \leq i \leq m$, then 
$$ \mathbb{E} \max_{i=1}^m (X_i-\mathbb{E}X_i ) \leq \sqrt{2\sigma^2\log m} + c \log m.$$ We refer to \cite[Corollary 2.6]{boucheron} for the proof of the above estimate.
\end{enumerate}
\end{prop}

\begin{lmm}{\label{lemsubg}}

\begin{enumerate}[label={(\alph*)}, ref={\thelemma~(\alph*)}]
\item \label{lemsubg:1} If $X \in \mathscr{G}(\sigma^2,c)$ with $\mathbb{E}|X| =: \mu < \infty$ and $|Y| \leq M$ almost surely, then 
$$XY \in \mathscr{G}( 2^8M^2(8\sigma^2 \vee \mu^2) + 2^7M^2c^2, 8Mc ).$$
\item \label{lemsubg:2} If $X_i \in \mathscr{G}(\sigma^2_i, c_i)$ for all $1 \leq i \leq m$ and $t_1, \ldots, t_m \in \mathbb{R}$, then $$\sum_{i=1}^m t_iX_i \in \mathscr{G}\biggl( m\sum_{i=1}^m t_i^2\sigma_i^2, m\max_{i=1}^m |t_i|c_i\biggr).$$
Moreover, if $X_i$'s are mutually independent, the above result can be improved to  
$$\sum_{i=1}^m t_iX_i \in \mathscr{G}\biggl( \sum_{i=1}^m t_i^2\sigma_i^2, \max_{i=1}^m |t_i|c_i\biggr).$$
\item \label{lemsubg:3} Let $\{X_t : 0 \leq t \leq T\}$ be a real-valued stochastic process with Lebesgue-measurable sample paths and $\int_{0}^T\mathbb{E}|X_t|\, dt< \infty.$ 
Further assume that $X_t \in \mathscr{G}(\sigma^2,c)$, for all $0 \leq t \leq T$. Then
$$ \int_{0}^T X_t \, dt \in \mathscr{G}( 32T^2\sigma^2 + 2^7T^2c^2, 8Tc).$$
\item \label{lemsubg:4} If $X:=(X_1,\ldots,X_d)$ is a sub-Gamma vector with parameters $(\sigma^2,c)$, then for any $k \geq 0$, we have
$$ \mathbb{E} \|X \|_2^k \leq 2^k d ( \sigma^k \left( 8(k+1)\right)^{k/2} + 8^kc^k(k+1)^k).$$
\item \label{lemsubg:5} If $X_i \in \mathscr{G}(\sigma^2,c)$, for all $1 \leq i \leq m$, then for any $\alpha >0$, 
$$ \mathbb{E} \max_{i=1}^m |X_i-\mathbb{E}X_i|^{\alpha} \leq \bigO ((\sqrt{2\sigma^2\log m} + c \log m )^{\alpha} ), \; \text{ as } m \to \infty.$$
\end{enumerate}
\end{lmm}

\begin{proof}
\begin{enumerate}[label=(\alph*)]

\item Take any integer $q \geq 1$ and note that 
\begin{align*}
&\mathbb{E} [( XY-\mathbb{E}XY)^{2q}] \leq 4^{2q-1}[ \mathbb{E}Y^{2q} [( X-\mathbb{E}X)^{2q}] + ( \mathbb{E}X)^{2q}\mathbb{E}[(Y-\mathbb{E}Y )^{2q}] \\
& \qquad+ ( \mathbb{E}X)^{2q}( \mathbb{E}Y)^{2q}  + ( \mathbb{E}XY)^{2q}] \\
& \leq 4^{2q-1}[M^{2q}8^qq!\sigma^{2q} + (4M)^{2q}(2q)!c^{2q} + (2\mu M )^{2q} + (\mu M )^{2q} + (\mu M )^{2q} ]\\
& \leq q!8^q(4M \sigma)^{2q} + (4\mu M)^{2q} + (2q)!(4Mc)^{2q} \\
& \leq 2q! ( (8\sqrt{2}M\sigma) \vee 4\mu M)^{2q} + (2q)!(4Mc)^{2q} \\
&\leq  q! ( (16\sqrt{2}M\sigma) \vee 8\mu M)^{2q} + (2q)!(4Mc)^{2q}.
\end{align*} 
Applying Proposition~\Cref{subg:moments}, we conclude the proof.

\item Setting $Y_i=X_i-\mathbb{E}X_i$, we apply H\"older's inequality for any $|\lambda| < ( m \max_{i=1}^m |c_it_i|)^{-1}$,
\begin{align*}
\log \mathbb{E} \exp \biggl(\lambda \sum_{i=1}^m t_iY_i\biggr) \leq \log \prod_{i=1}^m [\mathbb{E} e^{\lambda mt_iY_i}]^{1/m} &= \dfrac{1}{m} \sum_{i=1}^m \log \mathbb{E} \exp ( \lambda mt_iY_i) \\
& \leq \dfrac{1}{m}\sum_{i=1}^m \dfrac{\lambda^2m^2t_i^2\sigma_i^2}{2( 1-mc_i|\lambda t_i|)} \\
& \leq  \dfrac{\lambda^2m \sum_{i=1}^m t_i^2\sigma_i^2}{2( 1-m  |\lambda| \max_{i=1}^m |t_ic_i|)},
\end{align*}
proving the first assertion. For the second assertion, note that for independent $X_i$'s, we have that for any $|\lambda| <   (  \max_{i=1}^m |c_it_i|)^{-1}$,
\begin{align*}
\log \mathbb{E} \exp \biggl(\lambda \sum_{i=1}^m t_iY_i\biggr) &= \sum_{i=1}^m \log \mathbb{E} \exp( \lambda t_iY_i ) \\
&\leq \sum_{i=1}^m \dfrac{\lambda^2t_i^2\sigma_i^2}{2( 1-c_i|\lambda t_i|)} \leq \dfrac{\lambda^2 \sum_{i=1}^m t_i^2\sigma_i^2}{2( 1-  |\lambda| \max_{i=1}^m |t_ic_i|)}.
\end{align*}
\item Without loss of generality we can assume that $\mathbb{E}X_t=0$ for all $t \in [0,T]$. For any integer $q \geq 1$, applying Proposition~\Cref{subg:moments}, Fubini's theorem and H\"older's inequality we can write the following:
\begin{align*}
\mathbb{E} \biggl( \int_{0}^T X_t \, dt\biggr)^{2q} &= \mathbb{E} \int_{[0,T]^{2q}} \biggl(\prod_{i=1}^{2q} X_{t_i} \biggr)\, \prod_{i=1}^{2q} dt_i = \int_{[0,T]^{2q}} \mathbb{E}\biggl(\prod_{i=1}^{2q} X_{t_i} \biggr)\, \prod_{i=1}^{2q} dt_i \\
& \leq  \int_{[0,T]^{2q}} \prod_{i=1}^{2q} (\mathbb{E}X_{t_i}^{2q})^{1/(2q)} \,  \prod_{i=1}^{2q} dt_i\\
& \leq  \int_{[0,T]^{2q}} \prod_{i=1}^{2q} (8^qq!\sigma^{2q}+4^{2q}(2q)!c^{2q} )^{1/(2q)} \,  \prod_{i=1}^{2q} dt_i \\
& = q!8^q(T\sigma)^{2q} + (2q)!(4cT)^{2q}.
\end{align*}
The assertion follows directly from here.

\item This is again a straight-forward application of Proposition~\Cref{subg:moments}. Let $m \in \mathbb{N}$  be such that $2m-2 < k \leq 2m$. By Jensen's inequality, we have
\begin{align*}
(\mathbb{E} \|X \|_2^k)^{1/k} \leq (\mathbb{E} \|X \|_2^{2m})^{1/(2m)} &= \biggl[\mathbb{E} \biggl(\sum_{i=1}^d X_i^2 \biggr)^{m} \biggr]^{1/(2m)} \\
 &\leq \biggl[ d^m \sum_{i=1}^d \mathbb{E} X_i^{2m}\biggr]^{1/(2m)} \\
& \leq d (8^{m}m! \sigma^{2m}+4^{2m}(2m)!c^{2m} )^{1/(2m)} \\
& \leq d (8^{m}m^m \sigma^{2m}+4^{2m}(2m)^{2m}c^{2m} )^{1/(2m)} \\
& \leq d ( \sigma \sqrt{8m} + 8cm ) \leq d ( \sigma \sqrt{8(k+1)} + 8c(k+1) ),
\end{align*}
since $m \leq k/2+1$. The upper bound computed above implies that
$$ \mathbb{E} \|X \|_2^k \leq 2^k d ( \sigma^k ( 8(k+1))^{k/2} + 8^kc^k(k+1)^k).$$

\item Without loss of generality we can assume $\mathbb{E}X_i =0$ for all $i \in [m]$. Using the tail bound in Proposition~\Cref{subg:tail}, we can get for $m \geq 2$, 
\begin{align*}
\mathbb{E} \biggl[ \dfrac{\max_{i=1}^m |X_i|^{\alpha}}{( \sqrt{2\sigma^2 \log m + c \log m})^{\alpha}}\biggr] & = \int_{0}^{\infty} \alpha t^{\alpha -1} \mathbb{P} \biggl[ \max_{i=1}^m |X_i| \geq t \sqrt{2\sigma^2 \log m + c \log m} \biggr] \, dt \\
& =  1 + \int_{1}^{\infty} \alpha t^{\alpha -1} \mathbb{P} \biggl[ \max_{i=1}^m |X_i| \geq t \sqrt{2\sigma^2 \log m + c \log m} \biggr] \, dt \\
& \leq 1 + \int_{1}^{\infty} \alpha t^{\alpha -1} \mathbb{P} \biggl[ \max_{i=1}^m |X_i| \geq \sqrt{2t\sigma^2 \log m + c t\log m} \biggr] \, dt \\
& \leq 1 + m \int_{1}^{\infty} \alpha t^{\alpha -1} \mathbb{P} [  |X_1| \geq \sqrt{2t\sigma^2 \log m + c t\log m} ] \, dt \\
& \leq 1 + 2m \int_{1}^{\infty} \alpha t^{\alpha -1} e^{-t \log m} \, dt \\
& = 1 + 2 \int_{1}^{\infty} \alpha t^{\alpha -1}m^{1-t} \, dt \leq 1 + 2 \int_{1}^{\infty} \alpha t^{\alpha -1}2 ^{1-t} \, dt < \infty,
\end{align*}
which completes the proof.
\end{enumerate}
\end{proof}

\begin{defn}{\label{subgamma:process}}
Let $(\mathcal{T}, d_{\mathcal{T}})$ be a metric space. A stochastic process $\{X_t : t \in \mathcal{T}\}$ is said to be sub-Gamma with parameters $(\sigma^2,c)$ if 
$$ X_t- X_s \in \mathscr{G} ( \sigma^2d_{\mathcal{T}}(t,s)^2, cd_{\mathcal{T}}(t,s)), \; \; \forall \; s,t \in \mathcal{T}.$$
\end{defn}

The following theorem is an analogue of Dudley's integral inequality (see \cite[Theorem 8.1.3]{hdp}) for sub-Gamma processes.

\begin{thm}{\label{dudley}}
Let $\{X_t : t \in \mathcal{T}\}$ be a separable mean-zero sub-Gamma process on the metric space $(\mathcal{T}, d_{\mathcal{T}})$ with parameters  $(\sigma^2,c)$. Then 
$$ \mathbb{E} \sup_{t \in \mathcal{T}} X_t \leq 12 \int_{0}^{\operatorname{diam}(\mathcal{T})}  ( \sigma  \sqrt{\log N(\mathcal{T},d_{\mathcal{T}},\varepsilon) } + c  \log N(\mathcal{T},d_{\mathcal{T}},\varepsilon) ) \, d\varepsilon.$$
\end{thm}

\begin{proof}
This proof closely follows the proof of the corresponding result for sub-Gaussian processes as in \cite[Theorem 8.1.4]{hdp}. Separability of the process guarantees that we can assume $\mathcal{T}$ is finite without loss of generality. Set $\varepsilon_k = 2^{-k}$ for all $k \in \mathbb{Z}$ and let $\mathcal{T}_k$ to be an optimum $\varepsilon_k$-net of $\mathcal{T}$, i.e. $|\mathcal{T}_k| = N_k =: N(\mathcal{T},d_{\mathcal{T}},2^{-k})$. Since $\mathcal{T}$ is finite, without loss of generality we can choose these nets in such a way that $\mathcal{T}_k=\{t_0\}$ for all $k \leq \kappa$ and $\mathcal{T}_k=\mathcal{T}$ for all $k \geq K$; where $\kappa < K \in \mathbb{Z}$. For a point $t \in \mathcal{T}$, let $\pi_k(t)$ be a closest point of $t$ in $\mathcal{T}_k$. Observe that $\pi_k(t)=t$ for $k \geq K$. Therefore, for any $t \in \mathcal{T}$,
\begin{equation}{\label{series}}
 X_t-X_{t_0} = ( X_{\pi_{\kappa}(t)} - X_{t_0} ) + \sum_{k=\kappa}^{K-1} ( X_{\pi_{k+1}(t)} - X_{\pi_k(t)} ) + ( X_t-  X_{\pi_K(t)}). 
\end{equation}
The first and last term in the right hand side of \Cref{series} being zero, we can conclude the following.
\begin{equation}{\label{expec}}
\mathbb{E} \sup_{t \in \mathcal{T}} X_t = \mathbb{E} \sup_{t \in \mathcal{T}} (X_t-X_{t_0} ) \leq \sum_{k=\kappa}^{K-1}  \mathbb{E} \sup_{t \in \mathcal{T}}( X_{\pi_{k+1}(t)} - X_{\pi_k(t)} ).
\end{equation}
By construction, 
$$d_{\mathcal{T}}(\pi_{k+1}(t),\pi_k(t) ) \leq d_{\mathcal{T}}(\pi_{k+1}(t),t ) + d_{\mathcal{T}}(t,\pi_k(t) ) \leq \varepsilon_{k+1}+\varepsilon_k \leq 32^{-k-1}, $$
and hence $X_{\pi_{k+1}(t)} - X_{\pi_k(t)} \in \mathscr{G}( 9 \times 2^{-2k-2}\sigma^2,3 \times 2^{-k-1}c).$ Also, observe that $(\pi_{k+1}(t),\pi_k(t) )$ can take at most $N_kN_{k+1} \leq N_{k+1}^2$ many values. Applying Proposition~\Cref{subg:maxc} therefore yields the following.
\begin{align*}
\mathbb{E} \sup_{t \in \mathcal{T}} X_t &\leq \sum_{k=\kappa}^{K-1}  \biggl[\sqrt{18 \times 2^{-2k-2}\sigma^2 \log N_{k+1}^2} + 3 \times 2^{-k-1}c \log N_{k+1}^2 \biggr] \\
& = \sum_{k=\kappa+1}^{K}  \biggl[6 \sigma  \times 2^{-k}\sqrt{\log N_{k}} + 6c \times 2^{-k} \log N_k \biggr] \\
&\leq 12 \sum_{k \in \mathbb{Z}} \int_{2^{-k-1}}^{2^{-k}} \biggl[ \sigma  \sqrt{\log N(\mathcal{T},d_{\mathcal{T}},2^{-k}) } + c  \log N(\mathcal{T},d_{\mathcal{T}},2^{-k}) \biggr] \, d\varepsilon \\
& \leq 12 \int_{0}^{\infty}  \biggl[ \sigma  \sqrt{\log N(\mathcal{T},d_{\mathcal{T}},\varepsilon) } + c  \log N(\mathcal{T},d_{\mathcal{T}},\varepsilon) \biggr] \, d\varepsilon.
\end{align*}
The integral clearly can be restricted to the region $[0,\operatorname{diam}(\mathcal{T})]$. This completes the proof.
\end{proof}

\subsection{Supplementary materials for Section~\ref{intro}}

\begin{lmm}{\label{argmincont}}
	Let $\mathcal{X}$ and $\mathcal{Y}$ be two topological spaces with $\mathcal{Y}$ being Hausdorff and  compact. Let $\psi : \mathcal{X} \times \mathcal{Y} \to \mathbb{R}$ be a continuous function such that for all $x \in \mathcal{X}$, there exists a unique $\varphi(x) \in \mathcal{Y}$ satisfying $\psi(x,\varphi(x)) = \inf_{y \in \mathcal{Y}} \psi(x,y)$. Then the map $\varphi : \mathcal{X} \to \mathcal{Y}$ is continuous. 
\end{lmm}

\begin{proof}
	By the closed graph theorem, it is enough to show that the graph of $\varphi$, defined as the set  $\operatorname{Gr}(\varphi) := \{(x,\varphi(x)) : x \in \mathcal{X}\}$, is closed in $\mathcal{X} \times \mathcal{Y}$; or equivalently $\operatorname{Gr}(\varphi)^c$ is open. Fix $(x,y) \in \operatorname{Gr}(\varphi)^c$ implying $y \neq \varphi(x)$, and get $z \in (\psi(x,\varphi(x)),\psi(x,y))$. Since $\psi$ is continuous, we conclude that $U:= \psi^{-1}((z,\infty))$ is an open set in $\mathcal{X} \times \mathcal{Y}$ containing $(x,y)$. On the other hand, the function $ x^{\prime} \mapsto \psi(x^{\prime},\varphi(x))$ being continuous, we can also claim that  $V:= \{x^{\prime} : \psi(x^{\prime},\varphi(x)) < z\}$ is an open set in $\mathcal{X}$ containing $x$. For any $x^{\prime} \in V$, we have $\psi(x^{\prime},\varphi(x^{\prime})) \leq \psi(x^{\prime},\varphi(x)) < z$ and hence $(x^{\prime},\varphi(x^{\prime})) \notin U$. Building upon these observations, we claim that $W:= U \cap (V \times \mathcal{Y})$ is an open set in $\mathcal{X} \times \mathcal{Y}$ containing $(x,y)$ but is disjoint from $\operatorname{Gr}(\varphi)$. This claim completes the proof. To see why the claim is true, note that openness of $W$ and the fact that it contains $(x,y)$ is clearly evident. If $(x^{\prime},\varphi(x^{\prime})) \in W$, then by definition $x^{\prime} \in V$ and $(x^{\prime},\varphi(x^{\prime})) \in U$, which contradicts our previous observation.
\end{proof}

\begin{proof}[Proof of Lemma \ref{f1}]
	This proof is basically a straightforward computation. Denoting the standard $d$-dimensional Gaussian density by $\varphi(\cdot;d)$, we observe that for any $s >0$ and $w \in \mathbb{R}^d$, the straight line joining $x$ to $x+sw$ lies entirely in $\operatorname{int}(S)$ for any $x \in S \setminus S_{s\|w\|_2}$. Therefore for $p \geq 0$ we have,
	\begin{align}
		&\mathbb{E} \biggl[\|W\|_2^p \biggl(1-\dfrac{f(X+sW)}{f(X)} \biggr)_+\biggr] \nonumber \\ &= \int_{\mathbb{R}^d} \int_{\mathbb{R}^d} \|w\|_2^p ( f(x)-f(x+sw))_{+} \varphi(w;d) \, dx \,  dw \nonumber\\
		& = \int_{\mathbb{R}^d} \int_{S \setminus S_{s\|w\|_2}} \|w\|_2^p ( f(x)-f(x+sw))_{+} \varphi(w;d) \, dx \,  dw \nonumber \\
		& \qquad  + \int_{\mathbb{R}^d} \int_{S_{s\|w\|_2}} \|w\|_2^p ( f(x)-f(x+sw))_{+} \varphi(w;d) \, dx \,  dw \nonumber  \\
		& \leq  \int_{\mathbb{R}^d} \int_{S \setminus S_{s\|w\|_2}} \|w\|_2^p \biggl(  \int_{0}^s w \cdot \nabla f(x+tw) \, dt\biggr)_{-} \varphi(w;d) \, dx \,  dw \nonumber \\
				& \qquad  +  \int_{\mathbb{R}^d} \int_{S_{s\|w\|_2}} \|w\|_2^p f(x) \varphi(w;d) \, dx \,  dw\nonumber \\
		& \leq \int_{0}^s \int_{\mathbb{R}^d} \int_{\mathbb{R}^d} \|w\|_2^{p+1}  \|\nabla f(x+tw)\|_2 \, \varphi(w;d) \, dx \,  dw\, dt +  C(f)\int_{\mathbb{R}^d}  \|w\|_2^ps \varphi(w;d) \,  dw\nonumber \\
		&= \int_{0}^s \int_{\mathbb{R}^d} \int_{\mathbb{R}^d} \|w\|_2^{p+1}  \|\nabla f(y)\|_2 \, \varphi(w;d) \, dy \,  dw\, dt +  C(f)\int_{\mathbb{R}^d}  \|w\|_2^ps \varphi(w;d) \,  dw \nonumber \\
		&= s \mathbb{E}\|W\|_2^{p+1} \int_{\mathbb{R}^d} \|\nabla f(x)\|_2 \, dx + sC(f)\mathbb{E}\|W\|_2^p, \nonumber
	\end{align}
	which completes the proof.
\end{proof}

\begin{proof}[Proof of Lemma~\ref{f2}]
Under the hypothesis of \Cref{f2}, $\mathcal{P}$ has density $f$ with respect to Lebesgue measure given by $f(x)=f_1(x_1)f_2(x_2)$ for all $x=(x_1,x_2) \in \mathbb{R}^d$. The proof now follows from the simple observation that 
$$ (1-uv)_{+} \leq (1-u)_+ + (1-v)_{+}, \; \forall \; u,v \in \mathbb{R},$$
which implies for $p \geq 0$,
\begin{align*}
\mathbb{E}\biggl[\|W\|_2^p \biggl( 1- \dfrac{f(X+sW)}{f(X)}\biggr)_{+} \biggr] &\leq \sum_{i=1}^2 \mathbb{E}\biggl[\|W\|_2^p \biggl( 1- \dfrac{f_i(X_i+sW_i)}{f_i(X_i)}\biggr)_{+}\biggr] \\
& \leq \sum_{i=1}^2 \mathbb{E}\biggl[( \|W_1\|_2+\|W_2\|_2)^p \biggl( 1- \dfrac{f_i(X_i+sW_i)}{f_i(X_i)}\biggr)_{+} \biggr]\\
&  \leq 2^{p \wedge 1}\sum_{i,j=1}^2 \mathbb{E}\biggl[\|W_j\|_2^p \biggl( 1- \dfrac{f_i(X_i+sW_i)}{f_i(X_i)}\biggr)_{+}\biggr] = \bigO(s),
\end{align*}
where $X=(X_1,X_2)$ and $W=(W_1,W_2)$. For the second assertion it is enough to take $a \in \mathbb{R}$, $b>0$ and to show that 
$$ \mathbb{E}\biggl[ \|W\|_2^p \biggl(1- \dfrac{f_{a,b}(a+bX+sW)}{f_{a,b}(a+bX)}\biggr)_{+}\biggr] = \bigO(s), \; \text{ as } s \downarrow 0,$$
where $f_{a,b}$ is the density of $a+bX$ and hence $f_{a,b}(\cdot) = b^{-1}f((\cdot-a)/b)$. Therefore,
\begin{align*}
	\mathbb{E} \biggl[\|W\|_2^p \biggl(1- \dfrac{f_{a,b}(a+bX+sW)}{f_{a,b}(a+bX)}\biggr)_{+}\biggr] = \mathbb{E} \biggl[\|W\|_2^p \biggl(1- \dfrac{f(X+b^{-1}sW)}{f(X)}\biggr)_{+}\biggr] = \bigO(s),
\end{align*}
completing the proof.
\end{proof}

\begin{proof}[Proof of Lemma~\ref{condrho}]
Fix $\alpha \in (0, \gamma_{-} \wedge \gamma_{+} \wedge 1)$ and consider $\rho(x) = |x|^{\alpha} \tanh^2(x)$ for all $x \in \mathbb{R}$. It is easy to observe that $\rho$ is differentiable on $\mathbb{R}$ with uniformly bounded derivative and hence satisfies the conditions \Cref{ass:F1,ass:F2,ass:F3}. $F_{\rho}$ and $f_{\rho}$ will respectively denote the distribution function and density associated with the probability measure $\mathcal{Q}_{\rho}$. It is now enough to show that $h := F^{\leftarrow} \circ F_{\rho}$ is twice continuously differentiable with uniformly bounded first and second order derivatives. 

We start by observing that for any $x \in  \mathbb{R}$, 
\begin{equation}{\label{h}}
 h^{(1)}(x) = \dfrac{f_{\rho}(x)}{f(h(x))},
\end{equation}
whereas 
\begin{equation}{\label{h2}}
 h^{(2)}(x) =  \dfrac{f_{\rho}^{(1)}(x)}{f(h(x))} - \dfrac{f_{\rho}(x)f^{(1)}(h(x))h^{(1)}(x)}{(f(h(x)))^2} = \dfrac{f^{(1)}(h(x))}{f(h(x))} \biggl[ \dfrac{f^{(1)}_{\rho}(x)}{f^{(1)}(h(x))} - \biggl(\dfrac{f_{\rho}(x)}{f(h(x))} \biggr)^2\biggr].
\end{equation}
Since $h$ takes value in $(a_-,a_+)$ and $f$ is strictly positive and continuously differentiable on $(a_-,a_+)$, we can guarantee that $h$ is twice continuously differentiable on any compact subset of the real line; for $h^{(2)}$ we need to consider the first expression rather than the second in \Cref{h2} to validate this argument. Thus it is enough to show that both $h^{(1)},h^{(2)}$ remains bounded as $|x| \to \infty$. We shall write down the proof for the case $x \to -\infty$, the other case is exactly similar. The proof actually shows that $h^{(1)}(x),h^{(2)}(x) \to 0$ as $x \to -\infty$. The case which satisfies the condition in (\Cref{n1}) is trivial  from the expressions in \Cref{h} and \Cref{h2} since $h(x) \downarrow a_-$ as $x \downarrow - \infty$. We therefore consider the remaining cases only.

Suppose that we have already established the fact that $h^{(1)}(x) \to 0$ as $x \to -\infty$. Since $f(h(x))$ converges to either $0$ or $\infty$ as $x \downarrow -\infty$, we apply l'H\^{o}spital's rule to conclude that 
$$ 0 = \lim_{x \to -\infty} h^{(1)}(x) = \lim_{x \to -\infty}  \dfrac{f_{\rho}(x)}{f(h(x))} =  \lim_{x \to -\infty}  \dfrac{f_{\rho}^{(1)}(x)}{f^{(1)}(h(x))h^{(1)}(x)}, $$
and 
\begin{align*} 
\beta_-&= \lim_{y \to a_-} \dfrac{\log f(y)}{\log F(y)} = \lim_{y \to a_-} \dfrac{f^{(1)}(y)/f(y)}{f(y)/F(y)}\\
&\qquad \qquad \qquad  \Rightarrow \dfrac{f^{(1)}(h(x))}{f(h(x))} \sim \beta_-\dfrac{f(h(x))}{F(h(x))} = \beta_-\dfrac{f(h(x))}{F_{\rho}(x)}, \; \text{ as } x \downarrow - \infty.
\end{align*}
Though we only have the convergence for $\log f(y)/\log F(y)$ in case (\Cref{n2}), we shall show that this is indeed true with $\beta_-=1$ for the case satisfying (\Cref{n3}). Combining above two observations we can write that
\begin{align*}
h^{(2)}(x) &= \dfrac{f^{(1)}(h(x))}{f(h(x))} \dfrac{f_{\rho}(x)}{f(h(x))} \biggl[\dfrac{f_{\rho}^{(1)}(x)}{f^{(1)}(h(x))h^{(1)}(x)} - h^{(1)}(x) \biggr] \\
&= \beta_- \dfrac{f_{\rho}(x)}{F_{\rho}(x)} o(1) = o(1), \text{ as } x \to - \infty,
\end{align*}
since application of  l'H\^{o}spital's rule also shows that $f_{\rho}(x) \sim \alpha |x|^{\alpha-1}F_{\rho}(x)$ as $x \to -\infty$. Hence we only need to show that $h^{(1)}(x)$ converges to $0$.

Since $f_{\rho}(x)/F(h(x)) = f_{\rho}(x)/F_{\rho}(x) = o(1)$, the assertion is trivial when $F(y)/f(y) =\bigO(1)$ as $y \downarrow a_-$ . Otherwise we have $a_- = -\infty$ and $\log f(y) \sim -A_-|y|^{\gamma_{-}}$ as $y \to - \infty$. Using L'H\^{o}spital's rule again, we obtain $ f^{(1)}(y) \sim -A_-\gamma_{-}f(y)|y|^{\gamma_{-}-1}$, and thus 
\[
f(y) \sim A_-\gamma_{-}F(y)|y|^{\gamma_{-}-1}
\]
as $y \downarrow - \infty$. This also implies  that $\log F(y) \sim \log f(y) \sim -A_-|y|^{\gamma_{-}}$ as $y \to - \infty$. We already know that $\log F_{\rho}(x) \sim -|x|^{\alpha}$ as $x \to - \infty$ from the definition of $\rho$ and l'H\^{o}spital's rule. Combining these asymptotics with the fact that $F_{\rho}(x)=F(h(x))$, we conclude the following. 
$$ A_-|h(x)|^{\gamma_{-}} \sim |x|^{\alpha}, \; \text{ as } x \downarrow -\infty.$$
On the other hand,
\begin{align*}
h^{(1)}(x) =  \dfrac{f_{\rho}(x)}{f(h(x))} \sim \dfrac{\alpha|x|^{\alpha-1}F_{\rho}(x)}{A_-\gamma_{-}F(h(x))|h(x)|^{\gamma_{-}-1}} &= \dfrac{\alpha}{A_-\gamma_{-}}\dfrac{|x|^{\alpha}}{|h(x)|^{\gamma_-}} \dfrac{|h(x)|}{|x|} \\
& \sim \dfrac{\alpha}{\gamma_-} A_-^{-1/\gamma_-} |x|^{\alpha/\gamma_- -1} = o(1) , \; \text{ as }  x \to - \infty,
\end{align*}
since $\alpha < \gamma_-$. This completes the proof.
\end{proof}

\subsection{Supplementary Materials for Section \ref{sec:lin}}

\begin{lmm}{\label{lem:brw}}
Consider the set-up for Theorem~\ref{brw:thm}. Assume that probability distribution $\mathcal{P}$ has finite mean and is  supported on the interval $(a_1,a_2)$ with $-\infty \leq a_1<a_2 \leq \infty$. Then 
$$ - \inf_{s >0} \dfrac{\Psi(s)}{s} > a_1.$$
Moreover, for any given $\varepsilon, \delta >0$ there exists $b_1,b_2 \in (a_1,a_2)$ such that, conditioned on the survival of the tree, 
$$ \liminf_{n \to \infty}  \dfrac{1}{\varepsilon n}   \inf_{v \in D_n} \sum_{e \in I(A_{\lfloor (1-\varepsilon)n \rfloor}(v) \leftrightarrow v)} \mathbbm{1}( b_1 \leq X_e \leq b_2)   \geq 1-\delta, \; \text{almost surely}.$$
The second assertion doesn't need finite mean or exponential moment assumption on $\mathcal{P}$.
\end{lmm}

\begin{proof}
For the first assertion, it is enough to show that there exists $t>0$ with $\psi(t) < -a_1t$. If $a_1=-\infty$, this is trivial from the hypothesis of \Cref{brw:thm}; hence we assume $a_1 \in \mathbb{R}$. Since $\mathcal{P}$ is supported on $(a_1, \infty)$, we have $X-a_1 >0$ with probability $1$ where $X \sim \mathcal{P}$ and therefore $\mathbb{E}\exp(-s(X-a_1)) \to 0$ as $s \to \infty$. In other words,
\begin{align*}
 \Psi(s) + a_1s &= \log m + a_1s + \log \mathbb{E}\exp(-sX) \\
 &= \log m +  \log \mathbb{E}\exp(-s(X-a_1)) \to -\infty, \; \text{ as } s \to \infty,
 \end{align*}
and hence we can get $t>0$ such that $\Psi(t) < -a_1t.$ 

For the second assertion, we first consider the case when $\mathcal{P}$ is standard Gaussian. Consider the modified BRW on the same tree $\mathbb{T}$, but with different displacements defined as $X_e^{\prime} = -|X_e| \sim \mathcal{P}^{\prime},$ for all $e \in E$;
 $\{S_v^{\prime} : v \in V\}$ will be the new vertex positions. The modified displacement distribution $\mathcal{P}^{\prime}$ clearly satisfies assumptions of \Cref{brw:thm} and hence, almost surely on non-extinction of $\mathbb{T}$, we have
$$ 0 \geq \lim_{n \to \infty} \dfrac{1}{n}\min_{v \in D_n} S_v^{\prime} = \gamma^{\prime} := - \inf_{s >0} \dfrac{\log m + \log \mathbb{E} \exp(-sX^{\prime})}{s} > -\infty,$$
where $X^{\prime} \sim \mathcal{P}^{\prime}$. Fix $\varepsilon, \delta >0$ and let $u_n^*$ be the vertex in $\lfloor (1-\varepsilon)n \rfloor$-th generation which has a descendant in the $n$-th generation with largest negative displacement among all the particles in $D_n$ from their ancestor in $\lfloor (1-\varepsilon)n \rfloor$-th generation in the modified BRW. In other words,
$$ v_n^* \in \arg \inf_{v \in D_n} ( S_v^{\prime} - S^{\prime}_{A_{\lfloor (1-\varepsilon)n \rfloor}(v)}),$$
and $u_n^* := A_{\lfloor (1-\varepsilon)n \rfloor}(v_n^*)$. Using independence of displacement variables, it is easy to see that $u_n^* \sim \text{Uniform}(D_{\lfloor (1-\varepsilon)n \rfloor})$, conditioned on the tree $\mathbb{T}$, and hence $S_{u_n^*}^{\prime} \sim \mathcal{P}^{\prime \otimes \lfloor (1-\varepsilon)n \rfloor}.$ Strong law of large numbers guarantees that $S^{\prime}_{u_n^*}/n$ converges almost surely (on the non-extinction of the tree) to $\mu^{\prime}(1-\varepsilon)$ (where $\mu^{\prime}$ is the mean of $\mathcal{P}^{\prime}$) whereas the following simple observation is quite self-evident.
$$ \inf_{v \in D_n} S_v^{\prime} \leq S_{v_n^*}^{\prime} = S_{u_n^*}^{\prime} + \inf_{v \in D_n} ( S_v^{\prime} - S^{\prime}_{A_{\lfloor (1-\varepsilon)n \rfloor}(v)}).$$
We can therefore conclude that
\begin{equation}{\label{liminf}}
 \liminf_{n \to \infty} \dfrac{1}{n} \inf_{v \in D_n} ( S^{\prime}_v - S^{\prime}_{A_{\lfloor (1-\varepsilon)n \rfloor}(v)}) \geq \gamma^{\prime} - \mu^{\prime}(1-\varepsilon), \; \text{ almost surely on } (\mathbb{T}\text{ survives}).
\end{equation}
It is easy to see that for any $b > 0 $, we have $X_e^{\prime} \leq -b+b\mathbbm{1}( X^{\prime}_e > -b)$ and hence
\begin{align*}
 \dfrac{1}{\varepsilon n}  \inf_{v \in D_n} \sum_{e \in I(A_{\lfloor (1-\varepsilon)n \rfloor}(v) \leftrightarrow v)} \mathbbm{1}( X_e^{\prime} > -b) &\geq \dfrac{1}{\varepsilon n}  \inf_{v \in D_n} \sum_{e \in I(A_{\lfloor (1-\varepsilon)n \rfloor}(v) \leftrightarrow v)} \dfrac{X_e^{\prime}+b}{b} \\
 &= \dfrac{1}{\varepsilon n}  \inf_{v \in D_n} \dfrac{ S_v^{\prime} - S^{\prime}_{A_{\lfloor (1-\varepsilon)n \rfloor}(v)}+(n-\lfloor (1-\varepsilon)n \rfloor) b}{b}.
\end{align*}
Taking $\liminf$ on both sides in the above equation and applying \Cref{liminf}, we can conclude that, almost surely on non-extinction of the tree,
$$ \liminf_{n \to \infty}  \dfrac{1}{\varepsilon n}  \inf_{v \in D_n} \sum_{e \in I(A_{\lfloor (1-\varepsilon)n \rfloor}(v) \leftrightarrow v)} \mathbbm{1}( |X_e| < b) \geq \dfrac{\gamma^{\prime} - \mu^{\prime}(1-\varepsilon) + \varepsilon b}{\varepsilon b} = 1- \dfrac{ \mu^{\prime}(1-\varepsilon)-\gamma^{\prime}}{\varepsilon b}.$$
Since the right hand side converges to $1$ as $b$ goes to infinity, we can get large enough $b$ such that the right hand side is at least $1-\delta$.

For the general case, suppose $F$ denotes the distribution function of $\mathcal{P}$ which is supported on $(a_1,a_2)$ and $\Phi$ is the standard Gaussian distribution function. We have $\{X_e : e \in E\} \stackrel{d}{=} \{F^{\leftarrow} \circ \Phi(W_e) : e \in E\}$, where $F^{\leftarrow}$ is the (left-continuous)  inverse of $F$ and $\{W_e : e\in E\}$ is an i.i.d.~collection of standard Gaussian variables conditioned on the tree $\mathbb{T}$. Since $F^{\leftarrow} \circ \Phi$ is non-decreasing,  we can use the previous part and conclude that for any $\varepsilon, \delta>0$, there exists $b \in \mathbb{R}$ satisfying the following inequality almost surely on the non-extinction of the tree.
\begin{align*}
 &\liminf_{n \to \infty}  \dfrac{1}{\varepsilon n}  \inf_{v \in D_n} \sum_{e \in I(A_{\lfloor (1-\varepsilon)n \rfloor}(v) \leftrightarrow v)} \mathbbm{1}( F^{\leftarrow} \circ \Phi (-b) \leq X_e \leq F^{\leftarrow} \circ \Phi(b)) \\
 & \hspace{2 in} \geq \liminf_{n \to \infty}  \dfrac{1}{\varepsilon n}  \inf_{v \in D_n} \sum_{e \in I(A_{\lfloor (1-\varepsilon)n \rfloor}(v) \leftrightarrow v)} \mathbbm{1}( |W_e| < b) \geq 1-\delta.
\end{align*}
Since, $F^{\leftarrow} \circ \Phi(\mathbb{R}) \subseteq (a_1,a_2)$, we can complete the proof with the choices $b_1=F^{\leftarrow} \circ \Phi(-b) $ and $b_2=F^{\leftarrow} \circ \Phi(b) .$
\end{proof}

\begin{lmm}{\label{quadmeanvar}}
Let $\mathbf{X}=(X_1,\ldots,X_n)$ is a vector of independent random variables with mean zero, variance $\sigma^2$ and $\mathbb{E}X_i^4 = \nu$ for all $i\in [n]$. Let $A$ be an $n \times n$ real symmetric matrix. Then $\mathbb{E} \mathbf{X}^{\top}A \mathbf{X} = \sigma^2\operatorname{Tr}(A)$ and 
$$ \operatorname{Var} (\mathbf{X}^{\top}A \mathbf{X} ) = (\nu-3\sigma^4) \operatorname{Tr}(A_{\mathrm{diag}}^2) + 2\sigma^4 \operatorname{Tr}(A^2),$$
where $ A_{\mathrm{diag}}$ is the diagonal matrix of $A$.
\end{lmm}

\begin{proof}
Let $A=(a_{ij})_{i,j \in [n]}$. Then 
\begin{align*}
\mathbb{E} \mathbf{X}^{\top}A \mathbf{X} = \mathbb{E} \operatorname{Tr}(A \mathbf{X}\mathbf{X}^{\top}) = \operatorname{Tr} (A \,\mathbb{E}\mathbf{X}\mathbf{X}^{\top}) = \operatorname{Tr} (\sigma^2 A \, \mathrm{Id}_n ) = \sigma^2\operatorname{Tr}(A), 
\end{align*}
whereas
\begin{align*}
\mathbb{E} (\mathbf{X}^{\top}A \mathbf{X})^2 &= \sum_{i,j,k,l=1}^n a_{ij}a_{kl}\mathbb{E}( X_iX_jX_kX_l) \\
&= \sum_{i,j,k,l=1}^n a_{ij}a_{kl} ( \nu\mathbbm{1}_{i=j=k=l}+ \sigma^4\mathbbm{1}_{i=j\neq k=l}+\sigma^4\mathbbm{1}_{i=k \neq j =l}+\sigma^4\mathbbm{1}_{i=l \neq j=k}) \\
& = \nu \sum_{i=1}^n a_{ii}^2 + \sigma^4 \sum_{i=1}^n \biggl[ \sum_{k \neq i} a_{ii}a_{kk} + 2 \sum_{j \neq i} a_{ij}^2\biggr] \\
& = \nu \sum_{i=1}^n a_{ii}^2 + \sigma^4 \sum_{i=1}^n \biggl[ \sum_{k=1}^n  a_{ii}a_{kk} - a_{ii}^2+ 2 \sum_{j=1}^n a_{ij}^2 - 2a_{ii}^2\biggr] \\
& = (\nu-3\sigma^4) \sum_{i=1}^n a_{ii}^2 + \sigma^4 \biggl( \sum_{i=1}^n a_{ii} \biggr)^2 + 2 \sigma^4 \sum_{i,j=1}^n a_{ij}^2 \\
&= (\nu-3\sigma^4) \operatorname{Tr}(A_{\mathrm{diag}}^2) + \sigma^4 ( \operatorname{Tr}(A))^2 +  2\sigma^4 \operatorname{Tr}(A^2).
\end{align*}
This proves the expression for variance.
\end{proof}

\begin{lmm}{\label{xuconv}}
Let $X_1,X_2, \ldots$ be a sequence of i.i.d.~random variables with mean zero, variance $\sigma^2 \in (0, \infty)$ and $\mathbb{E}X_1^4 = \nu < \infty$.  Then there exists $c>0$ such that such that 
$$ a_n := \inf_{(u_1,\ldots,u_n) \in \mathbb{R}^n : \sum_{i=1}^n u_i^2} \mathbb{E} \biggl \rvert \sum_{i=1}^n u_iX_i\biggr \rvert \geq c,\; \forall \; n \geq 1. $$
Therefore, for any small enough  $\varepsilon >0$, there exists $\delta =\delta(\varepsilon)>0$ such that 
$$ \mathbb{P} \biggl( \biggl \rvert \sum_{i=1}^n u_iX_i\biggr \rvert \geq \delta\biggr) \geq \varepsilon, \; \forall\; (u_1, \ldots,u_n) \in \mathbb{R}^n \text{ with } \sum_{i=1}^n u_i^2=1,\; n \geq 1.$$
\end{lmm}

\begin{proof}
We shall prove by contradiction. First note that, the sequence $a_n$ is non-increasing and hence if the statement of the lemma is false, then $a_n \downarrow 0$ as $n \to \infty$. In particular, for all $k \geq 1$, we can find $\boldsymbol{u}_{k}=(u_{k,1},\ldots,u_{k,k}) \in \mathbb{R}^{k}$ such that $\sum_{i=1}^k X_iu_{k,i}$ converges to $0$ in $L^1$ as $k \to \infty$. Introduce the notation $\mathbf{X}_k=(X_1,\ldots,X_k)$. By \Cref{quadmeanvar}, we have
$$ \mathbb{E}\biggl[\biggl(\sum_{i=1}^k X_iu_{k,i} \biggr)^2\biggr] = \mathbb{E} \mathbf{X}_k^{\top}\boldsymbol{u}_{k}\boldsymbol{u}_{k}^{\top}\mathbf{X}_k = \sigma^2 \operatorname{Tr}(\boldsymbol{u}_{k}\boldsymbol{u}_{k}^{\top} )= \sigma^2 \|\boldsymbol{u}_{k}\|_2^2 =\sigma^2,$$
whereas
\begin{align*}
\operatorname{Var}\biggl[\biggl(\sum_{i=1}^k X_iu_{k,i} \biggr)^2 \biggr) &= \operatorname{Var}( \mathbf{X}_k^{\top}\boldsymbol{u}_{k}\boldsymbol{u}_{k}^{\top}\mathbf{X}_k ) \\
&= (\nu-3\sigma^4) \operatorname{Tr}((\boldsymbol{u}_{k}\boldsymbol{u}_{k}^{\top})_{\mathrm{diag}}^2) + \sigma^4 ( \operatorname{Tr}(\boldsymbol{u}_{k}\boldsymbol{u}_{k}^{\top}))^2 +  2\sigma^4 \operatorname{Tr}((\boldsymbol{u}_{k}\boldsymbol{u}_{k}^{\top})^2) \\
&= (\nu-3\sigma^4)\sum_{i=1}^k u_{k,i}^4 + 3\sigma^4 \leq (\nu-3\sigma^4)\|\boldsymbol{u}_{k} \|_{\infty}^2\|\boldsymbol{u}_{k} \|_{2}^2 + 3\sigma^4 \leq \nu.
\end{align*}
This shows that $\{(\mathbf{X}_k^{\top}\boldsymbol{u}_{k})^2 : k \geq 1\}$ is uniformly integrable and hence, by Vitali's convergence theorem (see \cite[Theorem 22.7]{schilling} for reference), $\mathbb{E}(\mathbf{X}_k^{\top}\boldsymbol{u}_{k})^2 \to 0$ as $k \to \infty$, which gives a contradiction.

For the second part, we apply the Paley--Zygmund inequality (see \cite[Exercise 2.4]{boucheron} for the statement) to note that for any $(u_1, \ldots,u_n) \in \mathbb{R}^n$ with $\sum_{i=1}^n u_i^2=1$ and $\theta \in [0,1]$,
\begin{align*}
\mathbb{P} \biggl( \biggl \rvert \sum_{i=1}^n u_iX_i\biggr \rvert \geq \theta c \biggr) &\geq \mathbb{P} \biggl( \biggl \rvert \sum_{i=1}^n u_iX_i\biggr \rvert \geq \theta \mathbb{E}  \biggl \rvert \sum_{i=1}^n u_iX_i\biggr \rvert \biggr) \\
& \geq (1-\theta)^2 \dfrac{( \mathbb{E} \rvert \sum_{i=1}^n u_iX_i\rvert )^2}{\mathbb{E}(\sum_{i=1}^k u_{i}X_i )^2} \geq \dfrac{c^2(1-\theta)^2}{\sigma^2}.
\end{align*} 
Therefore, for any $\varepsilon \in (0,c^2/\sigma^2)$, we can take $\delta(\varepsilon) = c-\sqrt{\varepsilon}\sigma >0$, which satisfies the the second assertion in the statement of the lemma.
\end{proof}

\begin{lmm}{\label{interlace}}
Let $A =(a_{ij})_{i,j \in [n]}$ be a real symmetric matrix of order $n \times n$ with eigenvalues $\lambda_1 \leq \cdots \leq \lambda_n$. For $k \in \{1,\ldots,n\}$, let $B_k$ be the symmetric matrix of order $ k \times k$ defined as $B_k := (a_{ij})_{i,j \in [k]}$, i.e., $B_k$ is obtained by considering only the first $k$ many rows and columns of $A$. Let $\mu_k^{+}$ and $\mu_k^-$ be the largest and the smallest eigenvalues of $B_k$ respectively. Then $\mu_{n-k+1}^- \leq \lambda_k \leq \mu_k^+$ for all $1 \leq k \leq n.$ As for $k=1,n$ case, we have one-sided equality : $\mu_{n}^- = \lambda_1, \lambda_n = \mu_n^{+}$. 
\end{lmm}

\begin{proof}
The assertions $\lambda_{1}=\mu_n^+$ and $\mu_{n}^- = \lambda_1$  are trivial since $A=B_n$. For $1 \leq k \leq n$, we shall use the following variational characterization of eigenvalues.
\begin{equation}
\lambda_k = \inf \biggl\{ \sup_{ \boldsymbol{x} \in U : \|\boldsymbol{x}\|_2=1} \boldsymbol{x}^{\top}A\boldsymbol{x} \; \biggl\rvert \; U \text{ linear subspace of } \mathbb{R}^n, \operatorname{dim}(U)=k \biggr\}.
\end{equation}
Now consider the subspace $U_k := \{\boldsymbol{x}=(x_1,\ldots,x_n) : \; x_i=0, \; \text{ for all } i > k\}$. This is a $k$-dimensional subspace of $\mathbb{R}^n$ and 
$$ \sup_{ \boldsymbol{x} \in U_k : \|\boldsymbol{x}\|_2=1} \boldsymbol{x}^{\top}A\boldsymbol{x} = \sup_{\substack{x_1, \ldots,x_k \in \mathbb{R} \\ \sum_{i=1}^k x_i^2 =1}} \sum_{i,j=1}^k a_{ij}x_ix_j = \sup_{ \boldsymbol{x} \in \mathbb{R}^k : \|\boldsymbol{x}\|_2=1} \boldsymbol{x}^{\top}B_k\boldsymbol{x} =\mu_k^+,$$
and hence 
$$ \mu_k^+ = \sup_{ \boldsymbol{x} \in U_k : \|\boldsymbol{x}\|_2=1} \boldsymbol{x}^{\top}A\boldsymbol{x} \geq \inf \biggl\{ \sup_{ \boldsymbol{x} \in U : \|\boldsymbol{x}\|_2=1} \boldsymbol{x}^{\top}A\boldsymbol{x} \; \biggl \rvert \; U \text{ linear subspace of } \mathbb{R}^n, \operatorname{dim}(U)=k \biggr\} = \lambda_k.$$

For the other inequality, we apply what we already proved to the symmetric matrix $-A$. Note that the eigenvalues of $-A$ are $-\lambda_n \leq \cdots \leq -\lambda_1$ and the largest eigenvalue of $-B_k$ is $-\mu_k^{-}$. Therefore, for any $1 \leq k \leq n$, we have $-\lambda_{n-k+1} \leq - \mu_{k}^-$ which yields $\mu_{k}^- \leq \lambda_{n-k+1}$. Substituting $k$ by $n-k+1$, we get our required inequality. 
\end{proof}

\begin{lmm}{\label{gapestimate}}
Consider the setup of Lemma~\ref{interlace}. Let $\boldsymbol{a}:=(a_{ni})_{1 \leq i \leq n-1}$, and  let $\boldsymbol{x}=(x_1,\ldots,x_{n-1})$ be an orthonormal eigenvector of $B_{n-1}$ corresponding the eigenvalue $\mu_{n-1}$. Then for any $\varepsilon \in [0,1]$, we have 
$$ \lambda_n - \mu_{n-1} \geq 2\sqrt{\varepsilon(1-\varepsilon)} \big \rvert \boldsymbol{a}^{\top}\boldsymbol{x}\big \rvert - \varepsilon \mu_{n-1} + \varepsilon a_{nn} .$$
\end{lmm}

\begin{proof}
Set $s= \operatorname{sgn}(\boldsymbol{a}^{\top}\boldsymbol{x})$ and define  $\boldsymbol{y}:=(\sqrt{1-\varepsilon}x_1,\ldots,\sqrt{1-\varepsilon}x_{n-1},s\sqrt{\varepsilon})$ is an unit vector in $\mathbb{R}^n$ and thus 
\begin{align*}
\lambda_n \geq \boldsymbol{y}^{\top}A\boldsymbol{y} &= (1-\varepsilon)\sum_{i=1}^{n-1}a_{ij}x_ix_j + 2s\sqrt{\varepsilon(1-\varepsilon)}\sum_{i=1}^{n-1}a_{ni}x_i + \varepsilon a_{nn} \\
& = (1-\varepsilon)\mu_{n-1} + 2\sqrt{\varepsilon(1-\varepsilon)}\big \rvert \boldsymbol{a}^{\top}\boldsymbol{x}\big \rvert+ \varepsilon a_{nn},
\end{align*}
which completes the proof.
\end{proof}

\begin{proof}[Proof of Proposition \ref{eigen:expecbound2}]
Without loss of generality, we can assume that $\sigma=1$. The proof of part (a) follows closely the arguments of \cite[Lemma 2.1.23]{anderson}.
\begin{align*}
\mathbb{E} [(\lambda_n(\mathbf{M}_n^{\top}\mathbf{M}_n) )^{k}] &\leq \mathbb{E}\biggl[ \sum_{i=1}^n (\lambda_i(\mathbf{M}_n^{\top}\mathbf{M}_n))^k\biggr] \nonumber \\ 
&= \mathbb{E} [ \operatorname{Tr}[(\mathbf{M}_n^{\top}\mathbf{M}_n )^{k}]] \nonumber \\
& = \mathbb{E} \sum_{i_1, \ldots,i_k \in [m_n]} \sum_{j_1, \ldots,j_k \in [n]} X_{n,i_1j_1}X_{n,i_1j_2}X_{n,i_2j_2}X_{n,i_2j_3}\cdots X_{n,i_kj_k}X_{n,i_kj_1}.  
\end{align*}
For any $i_1, \ldots,i_k \in [m_n], j_1, \ldots,j_k \in [n]$, we can consider the bipartite graph with vertex sets being $\{i_1,\ldots,i_k\}$ and $\{j_1,\ldots,j_k\}$, along with the $2k$ edges connecting $i_{\ell}$ to $j_{\ell}$ and $j_{\ell+1}$ for all $\ell=1,\ldots,k$ (where $j_{k+1}:=j_1$); of course the graph can have multiple edges connecting the same two vertices. The edges of this graph clearly form an Euler circuit. Let us denote by $\mathcal{BG}(n;k,t)$ to be the set of all such graphs with $t$ vertices, $2k$ will be called the length of the graph.  For any such graph $G \in \mathcal{BG}(n;k,t)$ and edge $e=\{i,j\} \in [m_n] \times [n]$, let $N_e(G)$ denotes the number of times the edge appears in the graph $G$. Then from our computation above and independence of matrix entries, we have 
\begin{equation}{\label{expec:wish}}
\mathbb{E} [(\lambda_n(\mathbf{M}_n^{\top}\mathbf{M}_n) )^{k}] \leq \sum_{t=2}^{2k} \sum_{G \in \mathcal{BG}(n;k,t)} \prod_{e \in [m_n] \times [n]} \mathbb{E} ( X_{n,e}^{N_e(G)}),
\end{equation}
where $X_{n,e}:=X_{n,ij}$ for $e=\{i,j\}$. For any $u \geq 1$, let $E(G,u):=\{e : N_e(G) = u\}$, whereas $E(G) :=\{e : N_e(G) >0\}$, the set of edges that appears in $G$. If $E(G,1) \neq \emptyset$, then the corresponding summand for $G$ in the right hand side of \Cref{expec:wish} will be $0$ (since the matrix entries have mean zero) and hence, to have non-zero contribution, all the edges in $G$ must appear at least twice. This means we can restrict our attention to the case when $G$ has at most $k$ many distinct edges, i.e., $\operatorname{card}(E(G)) \leq k$. Since the graphs in $\mathcal{BG}(n;k,t)$ are connected, we can restrict ourselves to $t \leq \operatorname{card}(E(G)) + 1 \leq k+1$. Moreover, if $\operatorname{card}(E(G,2))=l$, then we have $2k \geq 2l + 3(\operatorname{card}(E(G))-l)$ which implies $2k-2l \leq 6(k+1-t)$. Now, for such $G$, by H\"older's inequality, we have 
\begin{align*}
&\prod_{e \in [m_n] \times [n]} \mathbb{E} ( X_{n,e}^{N_e(G)}) \leq \mathbb{E}_{\mathcal{P}}( |X|^{\sum_{e} N_e(G)\mathbbm{1}(N_e(G)>2)}) \\
&\qquad \qquad = \mathbb{E}_{\mathcal{P}}( |X|^{2k-2l}) \leq (2k-2l)^{\beta(2k-2l)} \leq (6(k+1-t))^{6\beta (k+1-t)},
\end{align*}
and hence 
$$ \mathbb{E} [(\lambda_n(\mathbf{M}_n^{\top}\mathbf{M}_n) )^{k}] \leq \sum_{t=2}^{k+1} \sum_{G \in \mathcal{BG}(n;k,t) : E(G,1) = \emptyset} (6(k+1-t))^{6\beta (k+1-t)}.$$
Now any graph in $G \in \mathcal{BG}(n;k,t)$ with $E(G,1)=\emptyset$ can be thought of as a graph (with possibly self-edge and multiple edges) with $t$ vertices and length $2k$, with each edge appearing at least twice and the edges of the graph form an Euler circuit. By \cite[Lemma 2.1.23]{anderson}, number of such graphs is bounded above by $(n+m_n)^t4^k(2k)^{6(k+1-t)}$. Therefore, 
\begin{align*}
\mathbb{E} [(\lambda_n(\mathbf{M}_n^{\top}\mathbf{M}_n) )^{k}] & \leq \sum_{t=2}^{k+1} (n+m_n)^t4^k(2k)^{6(k-t+1)} (6(k+1-t))^{6\beta (k+1-t)} \\
& \leq (n+m_n)^{k+1} 4^k \sum_{t=2}^{k+1} \biggl( \dfrac{(2k)^6(6(k+1-t))^{6\beta}}{n+m_n}\biggr)^{k+1-t} \nonumber \\
& \leq (n+m_n)^{k+1} C^k \sum_{t=2}^{k+1} \biggl( \dfrac{k^{6(\beta+1)}}{n+m_n}\biggr)^{k+1-t} \leq (n+m_n)^{k+1}C^k \sum_{i=2}^{k} \biggl( \dfrac{k^{6(\beta+1)}}{n}\biggr)^{i},
\end{align*}
for some finite constant $C$, depending on $\beta$. Employing arguments similar to the ones we used in the proof of \Cref{eigen:expecbound}, we complete the proof of the first assertion.

For the part (b) in the statement, we shall use Paley--Zygmund inequality. Note that 
$$ \mathbb{E} \operatorname{Tr}( \mathbf{M}_n^{\top}\mathbf{M}_n) = \sum_{i,j} \mathbb{E} X_{n,ij}^2 = nm_n,$$
whereas if $\mathbb{E}_{\mathcal{P}}|X|^4 = \nu < \infty$, then 
\begin{align*}
\mathbb{E} [\operatorname{Tr}( \mathbf{M}_n^{\top}\mathbf{M}_n )^2] \sum_{i,j,k,l} \mathbb{E} X_{n,ij}^2X_{n,kl}^2 &= \sum_{i,j} \mathbb{E} X_{n,ij}^4 + \sum_{\substack{i,j,k,l \\ i \neq k, j \neq l}} \mathbb{E} X_{n,ij}^2X_{n,kl}^2  \\
& = \nu nm_n + n(n-1)m_n(m_n-1).
\end{align*} 
Hence, for any $\varepsilon_n =o(1)$, we have
\begin{align*}
\mathbb{P} [\operatorname{Tr}( \mathbf{M}_n^{\top}\mathbf{M}_n ) < \varepsilon_nnm_n ] \leq 1-(1-\varepsilon_n)^2 \dfrac{(\mathbb{E} \operatorname{Tr}( \mathbf{M}_n^{\top}\mathbf{M}_n ) )^2}{\mathbb{E} [\operatorname{Tr}( \mathbf{M}_n^{\top}\mathbf{M}_n)^2]} = o(1).
\end{align*}
Since $\operatorname{Tr}( \mathbf{M}_n^{\top}\mathbf{M}_n ) \leq n \lambda_n(\mathbf{M}_n^{\top}\mathbf{M}_n)$, this shows that $\mathbb{P}( \lambda_n(\mathbf{M}_n^{\top}\mathbf{M}_n) < \varepsilon_n m_n ) =o(1)$ for any $\varepsilon_n=o(1)$. Now, observing that the largest eigenvalue of $\mathbf{M}_n\mathbf{M}_n^{\top}$ is also $\lambda_n(\mathbf{M}_n^{\top}\mathbf{M}_n)$, we complete the proof.
\end{proof}



\begin{lmm}{\label{binombound}}
Fix $m \in \mathbb{N} \cup \{0\}$ and $q \in [m,m+1]$. Then for any $a, \delta \geq 0$, we have
$$ (a+\delta)^q - a^q \leq \sum_{j=1}^{m+1} q^j\delta^ja^{q-j}.$$
\end{lmm}

\begin{proof}
We prove by induction on $m$. For $m=0$, the assertion is true since for any $q \in (0,1]$, we have 
$$(a+\delta)^q - a^q = q\int_{a}^{a+\delta} x^{q-1}\, dx \leq q\delta a^{q-1},$$
while the inequality is trivially true for $q=0$. Assume that the assertion is true for $m=0, \ldots,n-1$ for some $n \in \mathbb{N}$ and fix $q \in [n,n+1]$. We can then write down the following by the induction hypothesis since $q-1 \in [n-1,n]$.
\begin{align*}
(a+\delta)^q - a^q = q\int_{a}^{a+\delta} x^{q-1}\, dx \leq q\delta(a+\delta)^{q-1} 
& \leq q\delta \biggl(a^{q-1} + \sum_{j=1}^{n}(q-1)^j\delta^ja^{q-1-j}  \biggr) \\
& \leq  q\delta \biggl(a^{q-1} + \sum_{j=1}^{n}q^j\delta^ja^{q-1-j} \biggr) \\
&= \sum_{j=1}^{n+1} q^j\delta^ja^{q-j},
\end{align*}
which completes the proof.
\end{proof}

\begin{lmm}{\label{binombound2}}
Fix $m \in \mathbb{N} \cup \{0\}$ and $q \in [m,m+1]$. Then for any $x,y \in \mathbb{R}^d$, we have
$$ \rvert \|x+y\|_2^q - \|x\|_2^q \rvert \leq \sum_{j=1}^{m+1} q^j\; \|y\|_2^j \;\|x\|_2^{q-j} + q^{m+1} \; \|y\|_2^{m+1}\; \|x+y\|_2^{q-m-1}.$$
\end{lmm}

\begin{proof}
It follows from \Cref{binombound} that 
\begin{equation}{\label{proofstep1}}
\|x+y\|_2^q \leq ( \|x\|_2+\|y\|_2)^q \leq \|x\|_2^q + \sum_{j=1}^{m+1} q^j\; \|y\|_2^j \;\|x\|_2^{q-j}.
\end{equation}
On the other hand, if $\|x\|_2 \geq \|x+y\|_2$, we have
\begin{align}
\|x\|_2^q &\leq ( \|x+y\|_2 +\|y\|_2)^q \leq \|x+y\|_2^q + \sum_{j=1}^{m+1} q^j \; \|y\|_2^j\; \|x+y\|_2^{q-j} \nonumber \\
&\leq \|x+y\|_2^q + \sum_{j=1}^{m} q^j \; \|y\|_2^j \;\|x\|_2^{q-j} + q^{m+1} \; \|y\|_2^{m+1}\; \|x+y\|_2^{q-m-1}. \label{proofstep2}
\end{align}
Combining \Cref{proofstep1} and \Cref{proofstep2}, we complete our proof.
\end{proof}

\begin{lmm}{\label{mindist}}
Let $X_1, \ldots, X_n \stackrel{i.i.d.}{\sim} f$, a density on $\mathbb{R}^d$ with $\|f\|_{\infty}< \infty$. Then for any $\alpha \in [0,d)$, we have
$$ \mathbb{E}\biggl[ \biggl(\min_{2 \leq j \leq n} \|X_1-X_j\|_2 \biggr)^{-\alpha}\biggr] = \bigO(n^{\alpha/d}).$$ 
\end{lmm}

\begin{proof}

It is enough to only consider the situation when $\alpha >0$. Introduce the notation 
\[
S_n := \min_{2 \leq j \leq n} \|X_1-X_j\|_2.
\]
For any $y \in \mathbb{R}^d$ and $t >0$,
\begin{align*}
	h_n(y,t):=\mathbb{P} \biggl( \min_{2 \leq j \leq n} \|y-X_j\|_2 \leq t \biggr) 
	= 1- \biggl( 1- \int_{B_d(y,t)} f(x) \, dx\biggr)^{n-1},
\end{align*}
where $B_d(y,t)$ is the closed ball of radius $t$ around $y$. 
It is easy to see
\begin{align}{ \label{bound}}
	h_n(y,t^{-1}n^{-1/d})  \leq  1- ( 1- C_d\|f\|_{\infty}t^{-d}n^{-1})_{+}^{n-1} 
	&\stackrel{(i)}{\leq} C_d\|f\|_{\infty}t^{-d} \wedge 1,
\end{align}
for universal constant $C_d = \operatorname{Vol}(B_d(\mathbf{0},1)) \in (0,\infty)$. The assertion $(i)$ follows from an application of Bernoulli's inequality. In particular,
\begin{align}{ \label{bound1}}
	\mathbb{P}(n^{1/d}S_{n} \leq t^{-1}) = \int_{\mathbb{R}^d} h_n(y,t^{-1}n^{-1/d}) \, f(y)\, dy &\leq C_d\|f\|_{\infty}t^{-d} \wedge 1.
\end{align}
Using integration-by-parts, we then obtain the following.
\begin{align*}
n^{-\alpha/d}	\mathbb{E} S_n^{-\alpha} = \int_{0}^{\infty} \mathbb{P} ( n^{1/d}S_n < t^{-1/\alpha})\, dt  & \leq 1+ \int_{1}^{\infty} \mathbb{P} ( n^{1/d}S_n < t^{-1/\alpha}) \, dt\\
& \leq 1+ \int_{1}^{\infty} (C_d\|f\|_{\infty}t^{-d/\alpha} \wedge 1 ) \, dt < \infty,
\end{align*}
for $0 < \alpha <d$. This concludes the proof.
\end{proof}

\begin{lmm}{\label{mindist2}}
Let $X_1, \ldots, X_n \stackrel{i.i.d.}{\sim} f$, a density on $\mathbb{R}^d$. Then we have
$$  \min_{2 \leq j \leq n} \|X_1-X_j\|_2  = O_p(n^{-1/d}).$$
\end{lmm}
\begin{proof}
We continue with the notations introduced in the proof of \Cref{mindist}. For any $t >0$, we have
\begin{align*}
\mathbb{P} \biggl( n^{1/d}\min_{2 \leq j \leq n} \|X_1-X_j\|_2 > t\biggr) &= \int_{\mathbb{R}^d} (1-h_n(y,tn^{-1/d})) f(y)\,dy \\
& \leq \int_{\mathbb{R}^d} \exp \biggl(-(n-1)\int_{B_d(y,tn^{-1/d})} f(x) \, dx \biggr) f(y)\, dy.
\end{align*}
By the Lebesgue differentiation theorem, we have
$$ \dfrac{1}{C_dt^dn^{-1}}\int_{B_d(y,tn^{-1/d})} f(x) \, dx \stackrel{a.e.}{\longrightarrow} f(y),$$
and hence the dominated convergence theorem yields that
\begin{equation}{\label{tightmin}}
\limsup_{n \to \infty} \mathbb{P} \biggl( n^{1/d}\min_{2 \leq j \leq n} \|X_1-X_j\|_2 > t\biggr) \leq \int_{\mathbb{R}^d} \exp (-C_df(y)t^d) f(y)\, dy.
\end{equation}
Since the right hand side goes to $0$ as $t \to \infty$ (by applying the dominated convergence theorem), this suffices to prove the required assertion.
\end{proof}

\begin{proof}[Proof of Theorem \ref{yukich:general:mst}]{\label{proofout}}
As mentioned in \Cref{general:proof}, it suffices to prove only the upper bound on the growth rate of the expectation. We shall actually prove a stronger statement, i.e., we shall write down the proof of the $L^1$ convergence of the sequence of random variables  $n^{-(d-q)/d}L_{\mathrm{MST}}(X_1, \ldots, X_n;q)$ as $n \to \infty$, since the rest can be found in the proof of \cite[Theorem 7.6]{yukich}. Recall that $B_d(\mathbf{0},2^r)$ be the ball of radius $2^r$ around $\mathbf{0}$ in $\mathbb{R}^d$; then the following expression holds true whose derivation can be found at \cite[pp. 87]{yukich}. The expression is a consequence of simple sub-additivity and growth bound for the functional $L_{\mathrm{MST}}(\cdot;q)$ defined in the sense as described in \cite[Chapter 3]{yukich}.
\begin{align}{\label{umbrella1}}
L_{\mathrm{MST}}(X_1, \ldots, X_n;q) &\leq L_{\mathrm{MST}}(\{X_1, \ldots, X_n\} \cap B_d(\mathbf{0},2^r);q) \notag\\
&\qquad \qquad + \sum_{r \leq k \leq s(n)} M_12^{kq} (N_k(n))^{(d-q)/d} + M_12^{qs(n)},
\end{align}
where $M_1$ is some finite constant (depending only on $q$ and $d$), $s(n)$ is the largest $k \in \mathbb{N}$ such that the set  $\{X_1, \ldots, X_n\} \cap A_k \neq \emptyset$ and $N_k(n) := \operatorname{card}(\{X_1, \ldots, X_n\} \cap A_k)$.  Here we have abused the notation by writing $L_{\mathrm{MST}}(S;q)$ for some set $S \subseteq \mathbb{R}^d$; as usual it is the length of the $q$-power weighted shortest spanning tree through the points in $S$. Arguments at the end of \cite[pp. 88]{yukich} shows that Set $1-\varepsilon(r) = \mathcal{P}(B_d(\mathbf{0},2^r))$ and let $Z_1,Z_2,\ldots$ be i.i.d.~random vectors from the density $f_r:=(1-\varepsilon(r))^{-1}f \rvert_{B_d(\mathbf{0},2^r)}$. It is then easy to see that 
\begin{equation}{\label{umbrella2}}
 \mathbb{E} L_{\mathrm{MST}}(\{X_1, \ldots, X_n\} \cap B_d(\mathbf{0},2^r);q ) = \mathbb{E} L_{\mathrm{MST}}(Z_1, \ldots, Z_{\texttt{Bin}(n,1-\varepsilon(r))};q)
\end{equation}
where $\texttt{Bin}(n,1-\varepsilon(r))$ is a Binomial random variables with those mentioned parameters, independent of the collection $\{Z_i : i \geq 1 \}$. By smoothness of the functional $L_{\mathrm{MST}}(\cdot;q)$ (see \cite[Chapter 3]{yukich} for the definition of smoothness), we can arrive at the following estimate. 
\begin{align}{\label{umbrella3}}
  &\mathbb{E} L_{\mathrm{MST}}(Z_1, \ldots, Z_{\texttt{Bin}(n,1-\varepsilon(r))};q) \notag \\
  &\leq \mathbb{E} L_{\mathrm{MST}}(Z_1, \ldots, Z_{n(1-\varepsilon(r))};q) \nonumber \\
  & \hspace{ 1 in }+ M_22^{rp}\big \rvert \texttt{Bin}(n,1-\varepsilon(r)) - n(1-\varepsilon(r)) \big \rvert^{(d-q)/d}, 
\end{align}
for some finite constant $M_2$, depending on $q$ and $d$ only.
Since the density $f_r$ is compactly supported,  we conclude the following using the argument made in \cite[Section 7.2]{yukich}. 
\begin{align}{\label{umbrella4}}
 &n^{-(d-q)/d}\mathbb{E} L_{\mathrm{MST}}(Z_1, \ldots, Z_{n(1-\varepsilon(r))};q) \notag\\
 & \stackrel{n \to \infty}{\longrightarrow} \beta_{\mathrm{MST}}(d,q) \int_{\mathbb{R}^d} f_r(x)^{(d-q)/d}\, dx \nonumber \\
 & = (1-\varepsilon(r))^{-(d-q)/d} \beta_{\mathrm{MST}}(d,q)  \int_{B_d(\mathbf{0},2^r)} f(x)^{(d-q)/d}\, dx, 
\end{align}
whereas by Jensen's inequality,
\begin{align}{\label{umbrella5}}
 &\mathbb{E} |\texttt{Bin}(n,1-\varepsilon(r))-n(1-\varepsilon(r))|^{(d-q)/d} \nonumber \\
 &\leq ( \mathbb{E} |\text{Binomial}(n,1-\varepsilon(r))-n(1-\varepsilon(r))|^2)^{(d-q)/(2d)}   \nonumber \\
 & \leq n^{(d-q)/(2d)}.
\end{align}
Observing that $N_k(n) \sim \text{Binomial}(n,\mathcal{P}(A_k))$, we obtain the following in a similar fashion.
\begin{align}{\label{umbrella6}}
\mathbb{E} \sum_{k=r}^{s(n)} 2^{kq} (N_k(n))^{(d-q)/d} &\leq \sum_{k \geq r} 2^{kq} (\mathbb{E}N_k(n))^{(d-q)/d} \notag\\
&= \sum_{k \geq r}  2^{kq}n^{(d-q)/d}( \mathcal{P}(A_k))^{(d-q)/d} \nonumber \\
& =  \sum_{k \geq r}  2^{kq}n^{(d-q)/d}(2^{-dkq/(d-q)}a_{k,q}(f))^{(d-q)/d} \nonumber \\
& = n^{(d-q)/d} \sum_{k \geq r} (a_{k,q}(f))^{(d-q)/d}.
\end{align}
Finally, from the definition of $s(n)$, we know that $2^{s(n)} \leq  \max_{i=1}^n \|X_i\|_2 \leq 2^{s(n)+1}$ and hence,
\begin{align}{\label{umbrella7}}
\mathbb{E} 2^{qs(n)} \leq \mathbb{E} \max_{i=1}^n \|X_i\|_2^q &= n^{(d-q)/d} \int_{0}^{\infty} \mathbb{P}\biggl(\max_{i=1}^n \|X_i\|_2^q > tn^{(d-q)/d} \biggr) \, dt. 
\end{align}
Fix $\delta, \eta>0$. Since $\sum_{k \geq 1} a_{k,q}(f)^{(d-q)/d} < \infty$, we can guarantee the existence of $N_{\delta} \in \mathbb{N}$ large enough such that  $\mathcal{P}(A_k) \leq \delta 2^{-kdq/(d-q)}$ for all $k \geq N_{\delta}$; and hence for  $\ell \geq N_{\delta}$ we have
\begin{align*}
1- \mathcal{P}( B_d(\mathbf{0},2^\ell)) = \sum_{k \geq \ell} \mathcal{P}(A_k) \leq  \delta \sum_{k \geq \ell} 2^{-kdq/(d-q)} \leq \delta M_3 2^{-\ell dq/(d-q)},
\end{align*}
for some finite constant $M_3$, depending on $d$ and $q$ only; in particular for $u \geq 2^{N_{\delta}}$,
$$ 1- \mathcal{P}( B_d(\mathbf{0},u)) \leq 1- \mathcal{P}( B_d(\mathbf{0},2^{\lfloor \log_2 u \rfloor})) \leq \delta M_3 2^{-\lfloor \log_2 u \rfloor dq /(d-q)} \leq \delta M_4u^{-dq/(d-q)}, $$
for some finite constant $M_4$, depending on $d$ and $q$ only. Plugging in the above estimate in \Cref{umbrella7} we obtain,
\begin{align}{\label{umbrella8}}
&\int_{0}^{\infty} \mathbb{P}\biggl(\max_{i=1}^n \|X_i\|_2^q > tn^{(d-q)/d} \biggr) \, dt \nonumber \\
 & \hspace{0.8 in} \leq \max(2^{qN_{\delta}}n^{-(d-q)/d},\eta) + \int_{\max(2^{qN_{\delta}}n^{-(d-q)/d},\eta)}^{\infty} \mathbb{P}\biggl(\max_{i=1}^n \|X_i\|_2^q > tn^{(d-q)/d} \biggr) \, dt \nonumber \\
 &\hspace{0.8 in} \leq 2^{qN_{\delta}}n^{-(d-q)/d} + \eta +  \int_{\max(2^{qN_{\delta}}n^{-(d-q)/d},\eta)}^{\infty} n(1-\mathcal{P}(B_d(\mathbf{0},t^{1/q}n^{(d-q)/dq}))) \, dt \nonumber \\
  &\hspace{0.8 in} \leq 2^{qN_{\delta}}n^{-(d-q)/d} + \eta + \delta M_4 \int_{\eta}^{\infty} t^{-d/(d-q)} \, dt \nonumber \\
  & \hspace{0.8 in} = 2^{qN_{\delta}}n^{-(d-q)/d} + \eta + \dfrac{\delta M_4 (d-q)}{q} \eta^{-q/(d-q)}.
\end{align}
Combining all of the estimates from \Cref{umbrella1,umbrella2,umbrella3,umbrella4,umbrella5,umbrella6,umbrella7,umbrella8}, we conclude that for any $r \in \mathbb{N}$ and  $\delta, \eta >0$,
\begin{align}{\label{umbrella9}}
&\limsup_{n \to \infty}  \dfrac{\mathbb{E} L_{\mathrm{MST}}(X_1, \ldots, X_n;q)}{n^{(d-q)/d}}\notag \\
& \leq (1-\varepsilon(r))^{-(d-q)/d} \beta_{\mathrm{MST}}(d,q)  \int_{B_d(\mathbf{0},2^r)} f(x)^{(d-q)/d}\, dx \nonumber \\
& \hspace{ 0.5 in} + M_1 \biggl(\sum_{k \geq r} (a_{k,q}(f))^{(d-q)/d} +  \eta + \dfrac{\delta M_4 (d-q)}{q} \eta^{-q/(d-q)} \biggr).
\end{align}
Now we first take $r \uparrow \infty$ in the right hand side of \Cref{umbrella9}, use the observation that $\varepsilon(r) \downarrow 0$ as $r \to \infty$ and the assumption that $\sum_{k \geq 1} a_{k,q}(f)^{(d-q)/d}$ is summable to conclude that for any $\delta,\eta >0$,
\begin{align}{\label{umbrella10}}
&\limsup_{n \to \infty}  \dfrac{\mathbb{E} L_{\mathrm{MST}}(X_1, \ldots, X_n;q)}{n^{(d-q)/d}}\notag \\
& \leq  \beta_{\mathrm{MST}}(d,q)  \int_{\mathbb{R}^d} f(x)^{(d-q)/d}\, dx  + M_1 \biggl(  \eta + \dfrac{\delta M_4 (d-q)}{q} \eta^{-q/(d-q)} \biggr).
\end{align}
Take $\delta \downarrow 0$ first and then $\eta \downarrow 0$ to yield
\begin{align}{\label{umbrella11}}
&\limsup_{n \to \infty}  \dfrac{\mathbb{E} L_{\mathrm{MST}}(X_1, \ldots, X_n;q)}{n^{(d-q)/d}}\notag \\
& \leq  \beta_{\mathrm{MST}}(d,q)  \int_{\mathbb{R}^d} f(x)^{(d-q)/d}\, dx  =: \beta_{\mathrm{MST}}(d,q,f).
\end{align}
By \cite[Theorem 7.6]{yukich}, we already know that $n^{-(d-q)/d}L_{\mathrm{MST}}(X_1, \ldots, X_n;q)$ converges almost surely to $\beta_{\mathrm{MST}}(d,q,f)$ and hence by Fatou's lemma, we have 
\begin{equation}{\label{umbrella12}}
 \liminf_{n \to \infty}  \dfrac{\mathbb{E} L_{\mathrm{MST}}(X_1, \ldots, X_n;q)}{n^{(d-q)/d}} \geq  \beta_{\mathrm{MST}}(d,q,f).
\end{equation}
Assertions \Cref{umbrella11} and \Cref{umbrella12}, aided by the a.s.~convergence of $n^{-(d-q)/d}L_{\mathrm{MST}}(X_1, \ldots, X_n;q)$ and Scheffe's lemma, imply that $n^{-(d-q)/d}L_{\mathrm{MST}}(X_1, \ldots, X_n;q)$ converges in $L^1$ to the number $\beta_{\mathrm{MST}}(d,q,f)$.
\end{proof}

\begin{lmm}{\label{mindistmean}}
Consider the setup of Lemma~\ref{mindist2}. Also assume that $f$ satisfies Assumption~\ref{suff}. Then for any $q \in [1,d)$, 
$$ \mathbb{E} \min_{i \in [n] : i \neq 1} \|X_i-X_1\|^q_2 = \bigO(n^{-q/d}).$$
\end{lmm}

\begin{proof}
For any $x_1, \ldots, x_n \in \mathbb{R}^d$, we have 
$$ L_{\mathrm{MST}}(x_1, \ldots,x_n;q) \geq \dfrac{1}{2}\sum_{i \in [n]} \min_{j \in [n] : j \neq i} \|x_j-x_i\|^q_2.$$
Hence,
$$ \mathbb{E} \min_{i \in [n] : i \neq 1} \|X_i-X_1\|^q_2 = \dfrac{1}{n} \mathbb{E} \sum_{i \in [n]} \min_{j \in [n] : j \neq i} \|X_j-X_i\|^q_2 \leq  \dfrac{2}{n} \mathbb{E}L_{\mathrm{MST}}(X_1, \ldots,X_n;q) = \bigO(n^{-q/d}),$$
where we have used \Cref{yukich:general:mst} in the last step.
\end{proof}

\section*{Acknowledgements} 
S.C.~was partially supported by NSF grants DMS-2113242 and DMS-2153654. A version of this work was submitted as the doctoral dissertation of S.R.~at Stanford University. The authors would like to thank Persi Diaconis and Amir Dembo for their useful comments and inputs.

\bibliographystyle{abbrvnat}
\bibliography{reference_souvik}


\end{document}